%% file: Preprint_ACC_StabilityGF.tex
\documentclass[11pt]{article}
\usepackage{amsmath,amsfonts,amssymb,amsthm}
\usepackage{graphics,graphicx}
\usepackage{epsfig}
\usepackage{fullpage}
\usepackage{cancel}
\usepackage{xcolor}
\usepackage{epsfig,psfrag}
\usepackage{comment}

\providecommand{\keywords}[1]
{
  \small	
  {\textit{Keywords:}} #1
}
\providecommand{\subjclass}[1]
{
{ \textit{2020 MSC:}} #1
}

\usepackage[margin=2cm]{geometry}

\usepackage{helvet} 

\usepackage{hyperref}

\usepackage{titlesec}
\titleformat{\paragraph}
{\normalfont\normalsize\bfseries}{\theparagraph}{1em}{}
\titlespacing*{\paragraph}
{0pt}{3.25ex plus 1ex minus .2ex}{1.5ex plus .2ex}

\newcommand{\email}[1]{\href{mailto:#1}{\tt #1}}

\title{Exponential
Stability of Large BV Solutions\\ in a Model of Granular flow}

\author{
Fabio Ancona 
\thanks{Dipartimento di Matematica ``Tullio Levi-Civita'', Universit\`a di Padova, Via Trieste 63, 35121 Padova, Italy. \email{ancona@math.unipd.it}}
\and
Laura Caravenna 
\thanks{Dipartimento di Matematica ``Tullio Levi-Civita'', Universit\`a di Padova, Via Trieste 63, 35121 Padova, Italy. \email{laura.caravenna@unipd.it}}
\and
Cleopatra Christoforou
\thanks{Department of Mathematics and Statistics,
 Univer. of Cyprus, 1678 Nicosia, Cyprus. \email{christoforou.cleopatra@ucy.ac.cy}}
}

\begin{document}
\maketitle

\baselineskip=18pt

\renewcommand{\footnote}{\endnote}

\newcounter{stepnb}
\newcounter{substepnb}
\newcommand{\firststep}{\setcounter{stepnb}{0}}
\newcommand{\firstsubstep}{\setcounter{substepnb}{0}}
\newcommand{\step}[1]{{{\sc \addtocounter{stepnb}{1}\noindent $\circleddash$ Step \arabic{stepnb}:} #1.}} 
\newcommand{\substep}[1]{\vskip.3\baselineskip{\bf \addtocounter{substepnb}{1} \arabic{stepnb}.\arabic{substepnb}: #1.} }

\newcommand{\sgn}{\mathrm{sgn}}
\newcommand{\wkarr}{\; \rightharpoonup \;}
\def\Weak{\,\,\relbar\joinrel\rightharpoonup\,\,}
\def\del{\partial}
\font\msym=msbm10
\def\Real{{\mathop{\hbox{\msym \char '122}}}}
\def\R{\Real}
\def\cS{\mathcal{S}}
\def\A{\mathbb A}
\def\Z{\mathbb Z}
\def\K{\mathbb K}
\def\J{\mathbb J}
\def\L{\mathbb L}
\def\D{\mathbb D}
\def\M{\mathbb M}
\def\Integers{{\mathop{\hbox{\msym \char '132}}}}
\def\Complex{{\mathop{\hbox{\msym\char'103}}}}
\def\C{\Complex}

\newcommand{\ptpS}[4]{ \frac{ \partial\mathfrak {p} \mathbf {S}_{1} }{\partial #1} \left(#3;#4\right)} 
\newcommand{\pthS}[4]{ \frac{ \partial\mathfrak {h} \mathbf {S}_{2} }{\partial #1} \left(#3;#4\right)} 
\newcommand{\ptlls}[2]{ \frac{\partial^{}#2}{\partial #1} }
\newcommand{\ptll}[1]{ \frac{\partial^{}}{\partial #1} }
\newcommand{\ptl}[2]{ \frac{\partial^{#2}}{\partial #1^{#2}} }
\newcommand{\ptltdu}[2]{ \frac{\partial^{3}}{\partial {(#1)^{2}}\partial{#2}} }
\newcommand{\ptld}[2]{ \frac{\partial^{2}}{\partial {#1}\partial{#2}} }
\newcommand{\ddn}[2]{ \frac{d^{#2}}{d #1^{#2}} }

\def\Fdot{\dot F}

\def\bU{\bar U}
\def\bF{\bar F}
\def\bFdot{ \dot {\bar F}}
\def\bQ{\bar Q}
\def\bZ{\bar Z}
\def\bSigma{\bar \Sigma}

\def\sC{{\gamma}}
\def\sR{{\rho}}

\def\bu{\bar u}
\def\bv{\bar v}
\def\btheta{\bar \theta}
\def\bsigma{\bar \sigma}
\def\bareta{\bar \eta}
\def\bkappa{\bar \kappa}
\def\bmu{\bar \mu}
\def\br{\bar r}
\def\barf{\bar f}

\newcommand{\Bild}[2]{\begin{center}\begin{tabular}{l}
\setlength{\epsfxsize}{#1}
\epsfbox{#2}
\end{tabular}\end{center}}
\newcommand{\lessim}{\stackrel{<}{\sim}}
\newcommand{\gesim}{\stackrel{>}{\sim}}
\newcommand{\ignore}[1]{}
\newcommand{\oh}{{\textstyle\frac{1}{2}}}
\newcommand{\red}[1]{{\color{purple}#1}}
\newcommand{\blue}[1]{{{{\color{blue}#1\frown}}}}

\newtheorem{lemma}{Lemma}
\newtheorem{theorem}[lemma]{Theorem}
\newtheorem{maintheorem}[lemma]{Main Theorem}
\newtheorem{corollary}[lemma]{Corollary}
\newtheorem{proposition}[lemma]{Proposition}
\newtheorem{definition}[lemma]{Definition}
\newtheorem{remark}[lemma]{Remark}
\newtheorem{remarks}[lemma]{Remarks}
\newtheorem{Notation}[lemma]{Notation}

\newcommand{\be}[1]{\begin{equation}\label{#1}}
\newcommand{\ee}{\end{equation}}
\newcommand{\bes}{\begin{equation*}}
\newcommand{\ees}{\end{equation*}}
\def\ba#1\ea{\begin{align}#1\end{align}}
\def\bas#1\eas{\begin{align*}#1\end{align*}}

\newcommand{\Suno}[2]{{\mathbf S_{1}\!\left(#1;#2\right)}}
\newcommand{\Sdue}[2]{{\mathbf S_{2}\!\left(#1;#2\right)}}
\newcommand{\SC}[3]{{\mathbf {S}_{#1}\!\left(#2;#3\right)}}
\newcommand{\hSC}[3]{{\mathfrak {h}\mathbf {S}_2\!\left(#2;#3\right)}}
\newcommand{\pSC}[3]{{\mathfrak {p}\mathbf {S}_1\!\left(#2;#3\right)}}
\newcommand{\dSdue}[2]{{\mathbf {\dot S}_{2}\!\left(#1;#2\right)}}
\newcommand{\dSuno}[2]{{\mathbf {\dot S}_{1}\!\left(#1;#2\right)}}
\newcommand{\dotSuno}[2]{{\dot \mathbf S_{1}\!\left(#1;#2\right)}}
\newcommand{\dotSdue}[2]{{\dot \mathbf S_{2}\!\left(#1;#2\right)}}

\newcommand{\PHI}[1]{{\left\lvert#1-1\right\rvert}}

\newcommand{\ca}{{\cal A}}
\newcommand{\cd}{{\cal D}}
\newcommand{\ce}{{\cal E}}
\newcommand{\cj}{{\cal J}}
\newcommand{\co}{{\cal O}}
\newcommand{\cs}{{\cal S}}
\newcommand{\cp}{{\cal P}}
\newcommand{\cc}{{\cal C}}
\newcommand{\cw}{{\cal W}}
\newcommand{\ccp}{{\cal P}}
\newcommand{\cb}{{\cal B}}
\newcommand{\cl}{{\cal L}}
\newcommand{\cg}{{\cal G}}
\newcommand{\cq}{{\mathcal Q}}
\newcommand{\cM}{{\cal M}}
\newcommand{\cK}{{\cal K}}
\newcommand{\er}{{z}}

\newcommand{\dist}{{\mbox{dist}}}
\newcommand{\m}{\mbox{\boldmath $ \mu$}}
\newcommand{\ebb}{\mbox{\boldmath $\epsilon$}}

\newcommand{\oeps}{\overline{\varepsilon}}
\newcommand{\kr}{\hbox{Ker}}
\newcommand{\qinfo}{\stackrel{\circ}{Q}_{\infty}}
\newcommand{\mat}{\hbox{Mat}}
\newcommand{\cof}{\hbox{cof}\,}
\renewcommand{\det}{\hbox{det}\,}
\renewcommand{\del}{\partial}
\newcommand{\delt}{\partial_t}
\newcommand{\dela}{\partial_{\alpha}}
\newcommand{\eps}{\varepsilon}
\newcommand{\barQT}{{\overline Q}_T}
\newcommand{\esssup}{\operatornamewithlimits{esssup}}
\newcommand{\essinf}{\operatornamewithlimits{essinf}}

\font\msym=msbm10
\def\Real{{\mathop{\hbox{\msym \char '122}}}}
\def\Integers{{\mathop{\hbox{\msym \char '132}}}}
\font\smallmsym=msbm7
\def\smr{{\mathop{\hbox{\smallmsym \char '122}}}}
\def\torus{{{\text{\rm T}} \kern-.42em {\text{\rm T}}}}
\def\T3{\torus^3}
\def\div{\hbox{div}\,}
%
%

\date{}
\numberwithin{equation}{subsection}
\numberwithin{lemma}{section}

%

%
%
\begin{abstract}
We consider a $2\times 2$ system of hyperbolic balance laws, in one-space dimension,
that describes the evolution of a granular material 
with slow erosion and deposition. The dynamics is expressed in terms of the thickness of a moving layer on top and of a standing layer at the bottom. 
The system is linearly degenerate along two straight lines in the phase plane and genuinely nonlinear in the subdomains
confined by such lines. In particular, the characteristic speed of the first characteristic family is strictly increasing 
in the region above the line of linear degeneracy and strictly decreasing in the region below such a line.
The non dissipative source term is the product of two quantities that are transported with the two different characteristic speeds.

The global existence of entropy weak solutions of the Cauchy problem for such a system was established by Amadori and Shen~\cite{AS}
for initial data with bounded but possibly large total variation, under the assumption that the initial height
of the moving layer be sufficiently small.

In this paper we establish the Lipschitz ${\bf L^1}$-continuous dependence of the solutions on the initial data
with a Lipschitz constant that grows exponentially in time.
The proof of the ${\bf L^1}$-stability of solutions is based on the construction of a Lyapunov like functional
equivalent to the ${\bf L^1}$-distance, in the same spirit of the functional introduced by Liu and Yang~\cite{LY} and then developed by Bressan, Liu, Yang~\cite{bly} for systems of conservation laws with genuinely nonlinear or linearly degenerate characteristic fields.
\end{abstract}

	\keywords{Balance laws; granular flow; stability;  large $BV$; weakly linearly degenerated system.}
	
	\subjclass{35L65; 76T25; 35L45.}
	
%
%
%
\tableofcontents 

%
%
\section{Introduction}
\label{intro}
We consider a model for the flow of granular material
proposed by Hadeler and Kuttler~\cite{HK} 
where the evolution of a moving layer on top and of a resting layer at the bottom
is described by the two balance laws:
\be{E:S1system-d}
\begin{aligned}
h_t&=\div(h\nabla \mathfrak{s})-(1-|\nabla \mathfrak{s}|)h,\\
\mathfrak{s}_t&=(1-|\nabla \mathfrak{s}|)h.
\end{aligned}
\ee
Here, the unknown $h=h(x,t)\in\R$ and $\mathfrak{s}(x,t)\in\R$ represent, respectively, the thickness of the rolling layer and the height of the standing layer, while $t\geq 0$ and $x\in \R^n$ are the time and space
variables.
 The evolution equations~\eqref{E:S1system-d} show that the moving layer slides downhill with speed proportional to the slope of the standing layer in the direction of steepest
descent. The model is written in normalised form, assuming that the critical slope is $|\nabla \mathfrak{s}|=1$. This means that, if $|\nabla \mathfrak{s}|>1$, then grains initially at rest are hit by rolling matter of the moving layer
and hence they start moving as well. As a consequence the moving layer gets thicker. On the other hand,
if $|\nabla \mathfrak{s}|<1$, then rolling grains can be deposited on the standing bed. Hence the moving
layer becomes thinner. Typical examples of granular material whose dynamics is described by such models are dry sand and gravel 
in dunes and heaps, or snow in avalanches. 

In the one-space dimensional setting, assuming that the thickness of the
moving layer and the slope of the resting layer remain non-negative, 
if we differentiate the second equation of~\eqref{E:S1system-d} with respect to $x\in\R$ and
set $p\doteq \mathfrak{s}_x$, we obtain the system of balance laws
%
\be{S1system}
\begin{aligned}
&h_t-(hp)_x=(p-1)h,\\
&p_t+((p-1)h)_x=0,
\end{aligned}
\ee
with $h\ge 0$ and $p\ge 0$. The purpose of the present paper is to study the 
well-posedness of the 
Cauchy problem for~\eqref{S1system}.

Observing that the Jacobian matrix of the flux
function $F(h, p)= \big(hp, (p-1)h\big)$ associated to~\eqref{S1system} is 
\bes
A(h,p)=\left[\begin{array}{cc}
-p & -h\\
p-1 & h\end{array}\right],
\ees
by a direct computation one finds that the system~\eqref{S1system} is strictly hyperbolic 
on the domain 
\begin{equation}
\label{Om-def}
\Omega\doteq\big\{(h,p) : h\geq 0,\ p>0\big\}
\end{equation}
and weakly linearly degenerate at the point $(h,p)=(0,1)$. Namely, 
letting $\lambda_1(h,p)<\lambda_2(h,p)$ denote the eigenvalues of $A(h,p)$
when $p>0$, 
one can verify that the first characteristic family is genuinely nonlinear on each
domain $\{(h,p)\, | \, h\geq 0,\, p>1\}$, $\{(h,p)\, | \, h\geq 0,\, 0<p<1\}$, and linearly degenerate on
the semiline $\{(h,1)\, | \, h\geq 0 \}$ since, along the rarefaction curves of the first family, 
$\lambda_1$ is strictly increasing for $p>1$, strictly decreasing for $p<1$
and constant for $p=1$. Moreover, each region $\{(h,p)\, | \, h\geq 0,\, p>1\}$, $\{(h,p)\, | \, h\geq 0,\, 0<p<1\}$
is an invariant domain for solutions of the Riemann problem.
Instead, the second characteristic family is genuinely nonlinear 
for $h>0$ and linearly degenerate along $h=0$. We recall that hyperbolic systems of balance laws
generally do not admit smooth solutions and, therefore, 
weak solutions in the sense of distributions are considered.
Moreover, for the sake of uniqueness, an entropy criterion for
admissibility is usually added.
In~\cite{tplrpnpn} T.P. Liu proposed an 
admissibility criterion 
valid for general 
systems of conservation laws with non genuinely nonlinear
characteristic fields. 
For system~\eqref{S1system}, since the characteristic families enjoy the above 
properties, this criterion 
is equivalent to the classical {\it Lax stability condition}:
\begin{description}
\item[]\hspace{19pt}
 A shock connecting the left state $(h_\ell,p_\ell)$ and the right state $(h_r,p_r)$,
 traveling with speed \ $s$ \ is an admissible discontinuity
 of the $k$-th family if
 \begin{equation}
 \label{Lax-adm}
 \lambda_k\big((h_\ell,p_\ell)\big)\geq s\geq\lambda_k\big((h_r,p_r)\big).
 \end{equation}
\end{description}
Thus, throughout the paper,
with an {\it entropy-admissible weak solution} of~\eqref{S1system}
we shall always mean a standard weak solution, admissible 
in the sense of Lax.

Global existence of classical smooth solutions to~\eqref{S1system} were established for a special class of initial data 
by Shen~\cite{S}. 
In the case of more general initial data with bounded but possible large total variation, the existence of entropy weak solutions
globally defined in time was proved by Amadori and Shen~\cite{AS}.
In the present paper we tackle the problem of stability of entropy weak solutions to~\eqref{S1system}
with respect to the ${\bf L}^1$--topology. 

For systems without source term and small BV data, the 
Lipschitz ${\bf L}^1$-continuous dependence of solutions 
 on the initial data, 
was first established by Bressan and collaborators in~\cite{BC,BCP}
 under the assumptions that all characteristic families
are genuinely nonlinear (GNL) or linearly degenerate (LD), relying on a (lengthy and technical) homotopy method.
This approach requires careful a-priori estimates on a suitable defined weighted norm of the generalized tangent vector
to the flow generated by the system of conservation laws.
These results were then extended with the same techniques in~\cite{AM} to a class of $2\times 2$ systems with non GNL characteristic fields
 that does not comprise the convective part of system~\eqref{S1system}. 
A much simpler, more transparent proof of the ${\bf L}^1$-stability of solutions for conservation laws
with GNL or LD characteristic fields 
was later achieved by a technique
introduced by Liu and Yang in~\cite{LY} and then developed in~\cite{bly}. 
The heart of the matter here is to construct a Lyapunov-like nonlinear functional, equivalent
to the ${\bf L}^1$-distance, which is decreasing in time along any pair of solutions. 

Extensions of ${\bf L}^1$-stability results to the setting of large BV data was obtained
 for systems of conservation laws with Temple type characteristic fields, adopting the homotopy
 approach in~\cite{BG,B1,B2}, and constructing a Lyapunov-like functional in~\cite{CC1}.
This latter approach was followed also in~\cite{L1,L2,LT} to prove ${\bf L}^1$-stability of
solutions for general systems with GNL or LD characteristic fields,
within a special class
of 
initial data with large total variation.

In the case of balance laws with GNL or LD characteristic families
and small BV data, the ${\bf L}^1$-stability of solutions was 
first obtained in~\cite{CP} 
for $2\times 2$ systems via the homotopy method and the a-priori bounds
on a weighted distance.
Next, 
this result was established for $N\times N$ systems producing
 a Lyapunov-like functional 
for balance
 laws with dissipative source~\cite{AG} and non-resonant source~\cite{AGG}.
 An extension of these results for systems of balance laws of Temple class
 with large BV data is given in~\cite{CC2}.
 
We remark that all of the above results, with the only exception of~\cite{AM, B1,CC2}, 
deal with 
 systems having GNL or LD characteristic fields. Unfortunately, 
 for systems as~\eqref{S1system} that do not fulfill these classical assumptions,
 the 
 derivation of a-priori bounds on generalized tangent vectors, already hindered by
 heavy technicalities in the classical Lax setting,
 is further hampered 
 by the occurrence of richer nonlinear wave phenomena exhibited by such equations.
 This is mainly due to the presence in the solutions of discontinuous waves of the first characteristic family,
 shocks or contact discontinuities, which may 
 turn into rarefaction waves (and vicecersa)
 after interactions with waves of the other family.

In order to establish the ${\bf L}^1$-stability of solutions
 for system~\eqref{S1system}, we have thus followed the second approach,
 introducing in the present paper a Lyapunov-like functional
 controlling the growth of the ${\bf L}^1$-distance between
 pairs of approximate solutions with large BV data.
 Namely, in the same spirit of~\cite{bly}, 
 we explicitly construct a functional $\Phi=\Phi(u,v)$ 
for piecewise constant functions $u,\,v\in {\bf L}^1(\R;\,\Omega)$,
such that:
\begin{enumerate}
\item[(i)]
it is equivalent to the $ {\bf L}^1$-distance. Namley, for every pair of piecewise constant functions
$u=(u_h, u_p), v=(v_h,v_p)\in {\bf L}^1(\R;\,\Omega)$, with bounded total variation, there holds
\begin{equation}
\frac{1}{C}\cdot
\big\|u-v\big\|_{\strut {\bf L}^1}\le
\Phi(u,v)\le C\cdot
\big\|u-v\big\|_{\strut {\bf L}^1}
\end{equation}
for some constant $C>0$ depending only on the system~\eqref{S1system},
on the total variation of $u,v$,
and on the ${\bf L}^\infty$ norm of $u_h, v_h$.
\item[(ii)]
it is exponentially increasing in time
along pairs of
approximate solutions of~\eqref{S1system} generated by a front-tracking algorithm
combined with an operator splitting scheme with time steps $t_k=k \Delta t$.
Namely, for every couple of such approximate solutions 
$u(x,t),\, v(x,t),$ the right limits 
of $u(\,\cdot,t), v(\,\cdot,t)$ at $t_h<t_k$
satisfy
%
\begin{equation}
\label{gammaest2}
\begin{aligned}
\Phi\big(u(\,\cdot,t_k+), v(\,\cdot,t_k+)\big)
&\leq
\Phi\big(u(\,\cdot,t_h+), v(\,\cdot,t_h+)\big)
\big(1+{\mathcal{O}}(1)\Delta t\big)^{(k-h)}+
\\
\noalign{\smallskip}
&\qquad + {\mathcal{O}}(1)\cdot \varepsilon\,\Delta t \sum_{i=1}^{k-h} \big(1+{\mathcal{O}}(1)\Delta t\big)^i
\qquad\quad\forall~0\leq h<k,
\end{aligned}
\end{equation}
where $\varepsilon$ denotes a small parameter that controls the errors in
the wave speeds and the maximum size of rarefaction fronts 
in $u$ and in $v$. 

In particular, $\Phi$ is ``almost decreasing'' in time if the only effect of the convective part of~\eqref{S1system}
is taken into account:
\begin{equation}
\label{gammaest1}
\Phi\big(u(\,\cdot, \tau_2), v(\,\cdot, \tau_2)\big)\leq
\Phi\big(u(\,\cdot, \tau_1), v(\,\cdot, \tau_1)\big)+
{\mathcal{O}}(1)\cdot \varepsilon
(\tau_2-\tau_1)\qquad\forall~t_k<\tau_1<\tau_2<t_{k+1}.
\end{equation}
Here, and throughout the paper, we use the Landau symbol $\co(1)$ to denote a 
quantity whose absolute value
satisfies a uniform bound that depends only on the system~\eqref{S1system}. In particular, this 
bound does not depend on the front tracking parameter $\varepsilon$, or on the two solutions $u,v$ considered.
\end{enumerate}

\noindent
The value of $\Phi$ is defined as follows. 
Given two piecewise constant functions $u,\,v\in {\bf L}^1(\R;\,\Omega)$,
for each $x\in\R$, connect $u(x)$ with $v(x)$
moving along the Hugoniot curves of the first and second families and 
let $\eta_i(x)$, $i=1,2$, denote the size of the corresponding $i$-shock in the jump $(u(x), v(x))$.
Then, define
 \begin{equation}
 \label{Phi-def-1}
 \Phi\big(u,v)\doteq \sum_{i=1}^2 \int_{-\infty}^\infty W_i(x)\big|\eta_i(x)\big|~dx \cdot \exp(\kappa_\cg\mathcal{B} )\,,
 \end{equation}
 where the weights $W_i$ have the following form:
 \begin{equation}
 \label{W-def-1}
 W_i(x)\doteq \exp(\kappa_{i\ca 1}\cdot\mathcal{A}_{i,1}(x)+ \kappa_{i\ca 2}\cdot\mathcal{A}_{i,2}(x))\,,
 \qquad i=1,2\,,
 \end{equation}
 with
  \begin{equation}
 \label{W-def-2}
 \allowdisplaybreaks
\begin{aligned}
\mathcal{A}_{1,1}(x)
&\doteq
\Big\{
\text{total}\ \Big[\big[\text{strength}\big]\cdot\big[\text{distance from}~1~\text{of the p-component of the left state}\big] \Big]
\\
\noalign{\vspace{-2 pt}}
&\hskip 0.6in
\text{of $1$-waves in $u$ and in $v$ which approach the $1$-wave $\eta_1(x)$}
\Big\}
\\
\noalign{\bigskip}
\mathcal{A}_{1,2}(x)
&\doteq \Big\{
\text{total}\ \big[\text{strength}\big] \
\text{of $2$-waves in $u$ and in $v$ which approach the $1$-wave $\eta_1(x)$}
\Big\}\,,
\\
\noalign{\bigskip}
\mathcal{A}_{2,j}(x)
&\doteq
\Big\{
\text{total}\ \big[\text{strength}\big] \
\text{of $j$-waves in $u$ and in $v$ which approach the $2$-wave $\eta_2(x)$}
\Big\}\,,\quad j=1,\,2
\\
\noalign{\bigskip}
\mathcal{B}
&\doteq
\Big\{
\text{total \big[\text{strength}\big] of $u$ and of $v$} 
\Big\} +
\Big\{\text{\big[\text{wave interaction potential}\big] of $u$ and of $v$} 
\Big\}\,.
\end{aligned}
 \end{equation}

Here, $\kappa_{i\ca j}$, $i,\,j=1,\,2$ and $\kappa_\cg$ denote suitable positive constants depending on the system~\eqref{S1system} that obey Conditions ${\bf \Sigma}$ given in the proof of Proposition~\ref{PropCond}.
A precise definition of $W_1, W_2$ is given in Subsection~\ref{Ss:Lyapfunct}.
 
Observe that the weights $W_i$ are defined similarly to the expression of the weight given in~\cite{bly} for GNL and LD characteristic fields in the sense that here we have the exponential version of them. However, notice that the main novelty of our functional is encoded in the weight $W_1$
 and in particular in $\mathcal{A}_{1,1}$, whereas
$\mathcal{A}_{1,2}$, $\mathcal{A}_{2,1}$, $\mathcal{A}_{2,2}$ have almost the same expression given in~\cite{bly}.
In addition, another difference between the above definition of $\Phi$ and the one given in~\cite{bly} is the presence  of the whole Glimm functional
 of $u$ and $v$ in $\mathcal{B}$, instead of their interaction potential alone that appears as well on the exponent of e.  This is due to the fact that, since the first characteristic family is not GNL, we adopt as in~\cite{AS}
 a definition of wave interaction potential, suited to~\eqref{S1system},
 that is in general not decreasing in presence of interactions
 of 1-waves of different sign
 (1-shocks with 1-rarefaction waves). 
 Therefore, one needs to exploit the decrease of the total strength of waves due
 to cancellation in order to control the possible increase of the potential interaction
 occurring at such interactions. 

 The key ingredient in the definition of $\mathcal{A}_{1,1}$ is the appropriate formulation of {\it approaching wave}
 of the first family 
 for a given wave $\eta_1(x)$ in the jump $(u(x), v(x))$, which extends to our case the 
 definition given in~\cite{bly} for GNL characteristic fields. Observe that, 
 letting $\gamma\mapsto\mathbf {S}_1(\gamma;\, h_0,p_0)$
 be the Rankine-Hugoniot curve
 of right states of the first family issuing from a given state $(h_0,p_0)\in\Omega$, and denoting
 $\lambda_1(\gamma;\, h_0,p_0)$ the Rankine-Hugoniot speed of the jump connecting 
 $(h_0,p_0)$ with $\mathbf {S}_1(\gamma;\, h_0,p_0)$, 
 by the properties of system~\eqref{S1system} it follows that $\gamma\mapsto \lambda_1(\gamma;\, h_0,p_0)$ is strictly increasing 
 on $\{p>1\}$, strictly decreasing 
 on $\{0<p<1\}$, and constant along $\{p=1\}$. 
 Therefore, if the size $\eta_1(x)$ is positive, we shall regard as approaching all the $1$-waves present in $v$ 
 which either have left state in the region $\{p>1\}$ and are located on the left of $\eta_1(x)$,
or have left state in the region $\{0<p<1\}$ and are located on the right of~$\eta_1(x)$.
On the contrary, we regard as approaching to $\eta_1(x)>0$ all the $1$-waves present in $u$ 
 which either have left state in the region $\{p>1\}$ and are located on the right of $\eta_1(x)$,
 or have left state in the region $\{0<p<1\}$ and are located on the left of $\eta_1(x)$.
 Similar definition is given in the case where $\eta_1(x)<0$.
 
Observe that, in the definition of the Lyapunov functional given in~\cite{bly}, the weights $W_i$
are expressed only in terms of the strength of the approaching waves. Instead here the terms of  $\mathcal{A}_{1,1}$ 
related to the approaching waves of the first family have the form of the product of the strength of the waves $|\sR_\alpha|$ times the distance from 
$\{p=1\}$ of 
the left state of the waves $|p_\alpha-1|$. The presence of the factor $|p_\alpha-1|$ is crucial to 
guarantee the decreasing property~\eqref{gammaest1} at times of interactions involving a $1$-wave, say of strength $|\sR_\alpha|$,
and a $2$-wave crossing 
 $\{p=1\}$ (i.e. connecting two states lying on opposite sides of $\{p=1\}$), say of strength $|\sR_\beta|$.
 In fact, in this case the possible increase of  $\mathcal{A}_{1,1}$  turns out to be of order $|p_\beta-1||\sR_\alpha|
 \approx |\sR_\alpha\sR_\beta|$, and thus it can be controlled by the decrease of $\mathcal{B}$
determined by the corresponding decrease of the interaction potential.
Unfortunately, because of the presence of these quadratic terms in the weight $W_1$, we are forced to establish sharp
fourth order interaction estimates in order to carry on the analysis of the variation of $\Phi(u(t,\cdot), v(t,\cdot))$.
This is achieved deriving accurate Taylor expansions of the Hugoniot and rarefaction curves of each famiy, 
which rely on the specific geometric features of system~\eqref{S1system}. Namely, the characteristic fields
of~\eqref{S1system} are ``almost Temple class'' (the rarefaction and Hugoniot curves through the same point are ``almost'' straight lines
and have ``almost'' third order tangency at their issuing point)
near $\{p=1\}$ for the first family and near $\{h=0\}$
for the second family.
 
The estimate~\eqref{gammaest1} implies the convergence of front-traking approximate
solutions of the homogeneous system 
\be{S1system-hom}
\begin{aligned}
&h_t-(hp)_x=0,\\
&p_t+((p-1)h)_x=0,
\end{aligned}
\ee
to a unique limit, depending Lipschitz continuously on the initial data in the ${\bf L}^1$-norm, that
defines a semigroup solution operator $\cs_t$, $t\geq 0$, on domains $\mathcal{D}$ 
of the form\\
%
\begin{equation}
\label{Domain-def1}
\begin{aligned}
\mathcal{D}(M_0 , \delta_0,\delta_p )=cl\big\{(h,p)\in{\bf L}^1 (\mathbb{R};\mathbb{R}^2):\ \, &h,p \ \ \text{are piecewise constant, }
\\
& 0\leq h(x)\leq \delta_0 ,
\ \ |p(x)-1|<\delta_p\ \ \text{for a.e.}~x, 
\\
& \text{and}\ \ 
\mathrm{TotVar}\{(h,p)\} \le M_0,
\quad \lVert{h}\rVert_{{\bf L}^{1}}+\lVert{p-1}\rVert_{{\bf L}^{1}}\leq M_0 \big\},
 \end{aligned}
\end{equation}
%
where $cl$ denotes the ${\bf L}^1$-closure, 
$\mathrm{TotVar}\{(h,p)\} \doteq \mathrm{TotVar}\{h\}+\mathrm{TotVar}\{p\}$,
and $M_0, 
\delta_0, \delta_p$ are positive constants.
For any given initial data $\overline u \doteq (\,\overline h, \overline p\,)\in \mathcal{D}(M_0 , 
\delta_0, \delta_p)$, 
the map $u(t,x)\doteq \cs_t\overline u (x)$ provides an entropy weak solution of the Cauchy problem for~\eqref{S1system-hom} with initial condition
\be{S1system-data}
h(x,0)=\overline h(x)\,,\qquad p(x,0)=\overline p(x)
\qquad\quad \text{for a.e.} \ \ x~\in\R\,.
\ee
Relying on the estimate~\eqref{gammaest2}, we then show that 
approximate solutions of~\eqref{S1system} generated by a front-tracking algorithm
combined with an operator splitting scheme, in turn, converge to 
a map
that
defines a Lipschitz continuous semigroup operator $\ccp_t$, $t\geq 0$, 
on domains as~\eqref{Domain-def1}, 
with a Lipschitz constant 
that grows exponentially in time.
The trajectories of $\ccp$ are entropy weak solution of the Cauchy problem~\eqref{S1system},~\eqref{S1system-data}. 
 
The uniqueness of the limit of approximate solutions to~\eqref{S1system} 
and of the semigroup operator $\ccp$, 
is achieved as in~\cite{AG}
deriving the key estimate 
\begin{equation}
\label{uniq-semigr-est-1}
\left\|\ccp_\theta \overline u-\cs_\theta \overline u - \theta\cdot \!
\Bigg(
\!\begin{aligned}
&(\overline p-1) \overline h
\\
&\quad\ \ 0
\end{aligned}
\,\Bigg)
\right\|_{{\bf L}^1}
= {\mathcal{O}}(1)\cdot\theta^2\qquad\text{as}\qquad \theta\to 0\,,
\end{equation}
relating the solutions operators of the homogeneous and nonhomogeneous systems,
and invoking a general uniqueness result
for quasidifferential equations in metric spaces~\cite{Bre2}. 

We point out that
the results established in the present paper 
provide the first construction of a Lipschitz continuous semigroup 
of entropy weak solutions for nonlinear hyperbolic systems via a Lyapuonv type functional 
for:
\begin{itemize}
\item[-] systems with characteristic families that are neither GNL nor LD (nor of Temple class)
\item[-] initial data with arbitrary large total variation.
\end{itemize}
It remains an open problem to analyze whether the Lipschitz constant of the solution operator $\ccp$
is actually uniformly bounded.

We conclude this section observing that in~\cite{AS} it was investigated the slow erosion/deposition limit
of~\eqref{S1system} as the height of the moving layer tends to zero. The limiting behaviour of the slope of the standing layer provides an entropy weak solution of a scalar integro-differential conservation law. 
A semigroup of solutions to such a nonlocal conservation law, depending Liptschitz continuously on the initial
data, was constructed in~\cite{cgw} as limit of generalized front tracking approximations, and in~\cite{BS}
by a flux splitting method alternating backward Euler approximations with a nonlinear projection operator.
Granular models different from the one derived in~\cite{HK} can be found in~\cite{BdG,Du,SH}.
An analysis of steady state solutions for~\eqref{S1system} was carried out in~\cite{CC,CCCG}.

The paper is organized as follows. In Section~\ref{S2}, we study the properties of system~\eqref{S1system}, describe the construction of the approximate solutions also employed in~\cite{AS}, define the wave size in the two coordinate systems, Eulerian or Lagrangian, and their relation, introduce the Lyapunov functionals and conclude with the statement of the theorems for the semigroups associated with both the homogeneous and the non-homogeneous systems.
Lets us note here that the stability functional $\Phi$ that is equivalent to the ${\bf L}^1$ norm between two solutions $u$ and $v$ is denoted by $\Phi_0$ from here and on depending on the type of estimates explored.
In Section~\ref{S:basicest}, we present the interaction estimates for the approximate solutions and the variation of the wave size at time steps. 
The main work of our analysis is in Section~\ref{S4old}, which is divided into four subsections and the analysis establishes Theorem~\ref{stability-Phy}. More precisely, after defining the functional $\Phi$ that is equivalent to the ${\bf L}^1$ norm between two solutions $u$ and $v$, we estimate the change of $\Phi$ in the following three regions: In \S~\ref{Ss:interactionTimes} at interaction times, in~\S~\ref{Ss:nointeractiontimesN} between interaction times and in \S~\ref{Ss:timestep} at time steps. Then in~\S~\ref{S5.4}, we generalize our analysis to treat the functional $\Phi_z$ and conclude the proof of Theorem~\ref{stability-Phy}. Last, in Section~\ref{S6}, we establish the uniqueness of the limit and obtain a Lipschitz continuous evolution operator for the non-homogeneous system, proving Theorem~\ref{exist-nonhom-smgr-1}. There are many technical steps employed throughout our analysis and for the convenience of the reader, these can be found in the Appendices ~\ref{S:shockReduction}-~\ref{App:C}. They involve the reduction to shock curves in the stability analysis, standard analysis on the wave curves, delicate interaction-type estimates for each characteristic family up to fourth order and a convenient auxiliary lemma.

%
%
%
%


%
%
\section{Preliminaries and 
main results}\label{S2}

Let $F(h, p)= \big(hp, (p-1)h\big)$ be the flux
function
associated to~\eqref{S1system}. Then, the Jacobian matrix
\bes
\label{jacob}
DF(h,p)\doteq A(h,p)=\left[\begin{array}{cc}
-p & -h\\
p-1 & h\end{array}\right]
\ees
has eigenvalues
\be{S2 evalues}
\lambda_1(h,p)=\frac{h-p-\sqrt{(p-h)^2+4h}}{2},\qquad\lambda_2(h,p)=\frac{h-p+\sqrt{(p-h)^2+4h}}{2}
\ee
with associated right eigenvectors
\be{S2 evectors}
{\bf r}_1(h,p)=\left(\begin{array}{c}
1 \\ \noalign{\smallskip}
-\dfrac{\lambda_1+1}{\lambda_1}\end{array}
\right),\qquad {\bf r}_2(h,p)=\left(\begin{array}{c}
 -\dfrac{\lambda_2}{\lambda_2+1}\\\noalign{\smallskip}1\end{array}
\right)\;.
\ee
Note that system~\eqref{S1system} is strictly hyperbolic in the domain 
$$\Omega=\{(h,p):\,h\ge 0,\,p>0\},$$
since, for every $0<p_0<1$, one has
\be{strict-hyp}
\lambda_1(h,p)\leq -\frac{p_0}{2},\qquad\quad \lambda_2(h,p)\geq 0\qquad\forall~h\geq 0,\, p\geq p_0\,.
\ee
Moreover, for $p=1$, one has
$$\lambda_1(h,1)=-1,\quad\lambda_2(h,1)=h,\quad{\bf r}_1(h,1)=\left(\begin{array}{c}
1\\0\end{array}
\right),\qquad {\bf r}_2(h,1)=\left(\begin{array}{c}
 -\dfrac{h}{h+1}\\\noalign{\smallskip}1\end{array}
\right),
$$
while, for $h=0$, there holds

\begin{equation}
\label{lambda10p}
\lambda_1(0,p)=-p,\quad\lambda_2(0,p)=0,\quad{\bf r}_1(0,p)=\left(\begin{array}{c}
1\\ \noalign{\smallskip} \dfrac{1-p}{p}\end{array}
\right),\qquad {\bf r}_2(0,p)=\left(\begin{array}{c}
0\\1\end{array}
\right).
\end{equation}
Moreover, by direct computations, we find that
$$D\lambda_1\, {\bf r}_1=-\frac{2(\lambda_1+1)}{\lambda_2-\lambda_1}\approx\frac{2(p-1)}{p},\quad
D\lambda_2\, {\bf r}_2=-\frac{2\lambda_2}{\lambda_2-\lambda_1}\approx-\frac{2h}{p^2}.
$$
Therefore, the first characteristic field is genuinely nonlinear on each domain $\{p<1\}$, $\{p>1\}$,
and linearly degenerate along the semiline $p=1$, while the quantity $D\lambda_1\, {\bf r}_1$ changes sign across the
semiline
$p=1$. On the other hand, the second characteristic field is genuinely nonlinear for $h\neq 0$ and linearly degenerate along $h=0$ (see Figure~\ref{S2:fig1}).

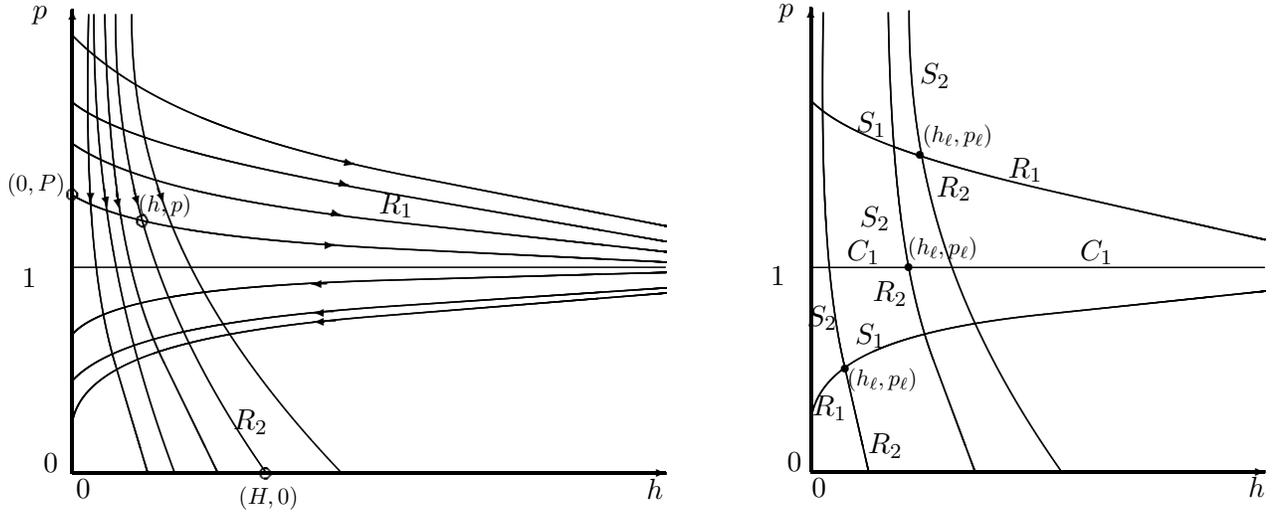
\begin{figure}[htbp]
{\centering \scalebox{1}{\input{characteristic-Riemann} } \par}
\caption{{\bf On the left:} The 
rarefaction curves 
of the two families, with the arrow pointing in the direction of increasing eigenvalues. {\bf On the right:} The curves 
of the right states 
that are connected to the left state $(h_{\ell},p_{\ell})$ by an entropy admissible
1-wave or 2-wave of the homogeneous system~\eqref{S1system-hom}. Here, $R_i, S_i, C_i$ denote rarefaction, Hugoniot and contact discontinuity curves of the $i$-th family,
respectively.
The three cases depending on the value of $p_\ell$ ($<\!1, =\!1, >\!1$) are indicated. \label{S2:fig1}}
\end{figure}

\subsection{Properties of Riemann solver and approximate solutions}
\label{Ss:Rsolv-appsol}
 
Let $\gamma\mapsto\mathbf {S}_k(\gamma;\, h^{\ell},p^{\ell})$
denote the Hugoniot curve of right states 
 of the $k^{\text{th}}$ family issuing from $(h^{\ell},p^{\ell})$, whose points $(h^{r},p^{r})\doteq \mathbf {S}_k(\gamma;\, h^{\ell},p^{\ell})$
satisfy the Rankine Hugoniot equations $$F((h^{r},p^{r}))-F((h^{\ell},p^{\ell}))=\lambda \cdot
\big((h^{r},p^{r})-(h^{\ell},p^{\ell})\big)$$ for 
$\lambda=\lambda_k\big((h^{\ell},p^{\ell}),\,(h^{r},p^{r})\big)$ where $\lambda_k\big((h^{\ell},p^{\ell}),\,(h^{r},p^{r})\big)$
denotes the $k$-th eigenvalue of the averaged matrix
\begin{equation}
A\big((h^{\ell},p^{\ell}),\,(h^{r},p^{r})\big)=
\int_0^1 A\big(s\, (h^{\ell},p^{\ell}) +(1-s) (h^{r},p^{r})\big) d s\,.
\end{equation}
We call $\lambda_k\big((h^{\ell},p^{\ell}),\,(h^{r},p^{r})\big)$ the 
Rankine Hugoniot speed associated to the left and right states 
$(h^{\ell},p^{\ell})$, $(h^{r},p^{r})$.
The analysis in~\cite{AS} shows that the Hugoniot curve 
 of the \emph{first} and \emph{second family} is given by
\begin{equation}\label{Suno}
\SC{1}{\sC}{ h^{\ell},p^{\ell}}=\bigg(h^{\ell}+\sC,\,p^{\ell}-\frac{(s_1+1)\,\sC}{s_1}\bigg)=
\bigg(h^{\ell}+\sC,\,p^{\ell}-\frac{(p^\ell-1)\,\sC}{h^\ell+\gamma-s_1}\bigg)
\qquad\gamma\geq - h^\ell\,,
\end{equation}
and
\begin{equation}\label{Sdue}
\SC{2}{\sC}{h^{\ell},p^{\ell}}=\bigg(h^{\ell}-\frac{s_2}{s_2+1}\sC,\,p^{\ell}+\sC\bigg)
=\bigg(h^\ell\bigg(1+\frac{\gamma}{\lambda_1(h^\ell,p^\ell+\gamma)-h^\ell}\bigg),
\,p^{\ell}+\sC\bigg)
\qquad\gamma\geq - p^\ell\,,
\end{equation}
respectively, where 
\begin{equation}
\label{rh-speed-def}
s_k=\lambda_k\big(\gamma;\,h^{\ell},p^{\ell}\big)
\doteq \lambda_k\bigg((h^{\ell},p^{\ell}),\, \SC{k}{\sC}{ h^{\ell},p^{\ell}}\bigg),\qquad k=1,2,
\end{equation}
are the corresponding Rankine Hugoniot speeds.
In fact, one finds that there holds
\begin{equation}
\label{rhspeed-eigenvalue}
s_1 =\lambda_1\big(h^{\ell}+\sC,p^{\ell}\big),\qquad\quad s_2=\lambda_2\big(h^{\ell},p^{\ell}+\sC\big)\,.
\end{equation}
The shock connecting the left state $(h_{\ell},p_{\ell})$ with the right state $ \mathbf {S}_1(\gamma;\, h^{\ell},p^{\ell})$
satisfies the Lax stability condition~\eqref{Lax-adm}, if $\gamma\cdot (p-p^\ell)\leq 0$,
while the shock with left state $(h^{\ell},p^{\ell})$ and right state $ \mathbf {S}_2(\gamma;\, h^{\ell},p^{\ell})$
is Lax admissible if $\gamma>0$.


We observe that the line $p=1$ separates the domain $\Omega$ into two invariant regions for solutions of the Riemann problem: the quarter $\{h\geq0,\ p>1\}$ and the half-strip $\{h\geq 0, \ 0<p<1\}$. Indeed, the rarefaction and Hugoniot curves of the first family through a point $(h^{\ell},p^{\ell})$, with $p^{\ell}\ne 1$, never meets the line $p=1$,
while the rarefaction and Hugoniot curves of the second family through a point $(h^{\ell},p^{\ell})$, with $h^{\ell}>0$,
never meets the line $h=0$. On the the other hand, the lines $p=1$ and $h=0$ are also invariant regions for solutions of the Riemann problem since they coincide with the rarefaction and Hugoniot curves of
the first and second family, respectively, passing through any of their points.
For convenience, 
we have drawn in Figure~\ref{S2:fig1} (on the right)
the elementary
curves of right states that are connected to a given left state $(h^{\ell}, p^{\ell})$ with entropy admissible 
waves of the first family (equivalently called {\bf${\bf 1}$-waves} or {\bf ${\bf h}$-waves})
and of the second family (equivalently called {\bf ${\bf 2}$-waves} or {\bf ${\bf p}$-waves})
 of the homogeneous system~\eqref{S1system-hom}.
Notice that, although the characteristic field of the first family does not satisfy the
classical GNL assumption, no composite waves are present in the solution 
of a Riemann problem for~\eqref{S1system-hom}
since in each invariant region $\{p>1\}$, $\{p<1\}$ the field is GNL.
In fact, the general solution of a Riemann problem for~\eqref{S1system-hom} consists of at most one simple wave
for each family which can be either a rarefaction or a compressive shock or a contact discontinuity. 

Global existence of 
entropy weak solutions to~\eqref{S1system} has been established by Amadori and Shen~\cite{AS}
using a front tracking algorithm in conjunction with the operator splitting. Their results can be summarized as follows:

\begin{theorem}[\cite{AS}]
\label{T:AS}
For any given $M_{0}, p_{0} > 0$, there exists $\delta_{0} > 0$ small enough 
 such that, if
\be{initial-data-bounds}
\mathrm{TotVar}\ \bar h+\mathrm{TotVar}\ \bar p\leq M_{0}, 
\quad \lVert\bar h\rVert_{{\bf L}^{1}}+\lVert\bar p-1\rVert_{{\bf L}^{1}}\leq M_{0},
\quad\lVert \bar h\Vert_{{\bf L}^{\infty}}\leq \delta_{0} ,\quad \bar p\geq p_{0},
\ee
hold, then the Cauchy problem~\eqref{S1system},~\eqref{S1system-data} has an entropy weak solution
$(h(x,t),\, p(x,t))$ defined for all $t \geq 0$, satisfying
%
\be{unifbvbound}
\mathrm{TotVar}\{h(\,\cdot\,,t)\} + \mathrm{TotVar}\{p(\,\cdot\,,t )\}\le M^*_0\,,
\qquad
\big\|h(\,\cdot\,,t)\big\|_{{\bf L}^{1}}+\big\|p(\,\cdot\,,t )-1\big\|_{{\bf L}^{1}}\le M^*_0\,,
\ee
and with values in a compact set
\be{E:reeeeeeeterete}
K=[0,\delta^*_0] \times [p^*_{0}, p^*_1]\,,
\ee
for some constants $M^*_0, \delta^*_0, p^*_0, p^*_1>0$.
\end{theorem}

%
In this article, we treat solutions $(h(x,t),\, p(x,t))$ to~\eqref{S1system},~\eqref{S1system-data} as established in~\cite{AS} stated inTheorem~\ref{T:AS}that have the $p$-component of the initial data close to $1$, i.e.
\be{unifbvboundLinftyp}\big\|\bar p-1\big\|_{{\bf L}^{\infty}}\le \delta_p\ee
for some constant $\delta_p>0$ sufficiently small. Hence, the solution $(h(x,t),\, p(x,t))$ satisfies in addition to~\eqref{unifbvbound}--\eqref{E:reeeeeeeterete} the bound
\be{unifbvboundLinftypsol}\big\|\bar p(\,\cdot\,,t )-1\big\|_{{\bf L}^{\infty}}\le \delta_p^*, \qquad \forall t>0\,.\ee


We shall consider here approximate solutions converging to an entropy weak solutions to~\eqref{S1system}--\eqref{S1system-data} constructed as in~\cite{AS}
by the following
operator splitting scheme. 
Fix a {\bf time step $\bf s=\Delta t>0$} and a {\bf parameter $\varepsilon>0$} that shall control:
\vspace{-5pt}
\begin{itemize}
\item[-] the size of the rarefaction fronts; 
\vspace{-5pt}
\item[-] the errors in speeds of shocks (or contact discontinuities) and rarefaction fronts;
\vspace{-5pt}
\item[-] the ${\bf L^1}$-distance between the piecewise constant initial data of the front-tracking approximation and the initial data $(\bar h,\bar p)$ in~\eqref{S1system-data}.
\end{itemize}
Let $t_k\doteq k \Delta t= k\, s$, 
$k=0,1,2,\dots$. Then,
an $s$-$\varepsilon${\bf -approximate solution} $(h^{s,\varepsilon},p^{s,\varepsilon})$ of~\eqref{S1system} is obtained as follows: on each time interval $[t_{k-1},t_k)$ the function $(h^{s,\varepsilon},p^{s,\varepsilon})$ is an $\varepsilon${\it -front tracking approximate solution} to the homogeneous system of conservation laws
~\eqref{S1system-hom}.
We recall that front-tracking solutions of a system of conservation laws are piecewise constant functions with
discontinuities occurring along finitely many lines in the $(t,x)$-plane, with interactions involving exactly two incoming
fronts. In general, the jumps can be of three types: shocks (or contact discontinuities),
rarefaction fronts and non-physical waves travelling at a constant speed faster than all characteristic speeds
(see~~\cite{Bre} for systems with GNL or LD characteristic fields and~\cite{AM3} for general systems).
In particular, we shall adopt here the simplified version of front tracking algorithm developed in~\cite{BD} 
for $2\times 2$ systems
as~\eqref{S1system-hom},
which does not require the introduction of non physical fronts.

Next, at time $t_k$ the function $(h^{s,\varepsilon},p^{s,\varepsilon})$ is updated as follows
\be{updated}
\left\{\begin{aligned}
& h^{s,\varepsilon}(t_k)=h^{s,\varepsilon}(t_k-)+\Delta t[p^{s,\varepsilon}(t_k-)-1]h^{s,\varepsilon}(t_k-)\\
&p^{s,\varepsilon}(t_k)=p^{s,\varepsilon}(t_k-).
\end{aligned}\right.
\ee
to account for the presence of source terms. 
In this way we construct an approximate solution of~\eqref{S1system}
defined for all times. 
It is shown in~\cite{AS} that, given $M_{0}, p_{0} > 0$, one can choose $\delta_{0} > 0$
sufficiently small so that the following holds. Consider a sequence of piecewise constant initial data 
$({\overline h}^{\varepsilon_m},{\overline p}^{\varepsilon_m})$
satisfying~\eqref{initial-data-bounds}, with $\varepsilon_m\to 0$ as $m \to \infty$,
that converges in ${\bf L}^1$ to the initial data~\eqref{S1system-data} as $m \to \infty$,
and let $\{s_m\}_m$ be a sequence of positive numbers converging to zero as $m \to \infty$.
Then, the above scheme provides a sequence of approximate solutions 
$(h^{s_m,\varepsilon_m},p^{s_m,\varepsilon_m})$ 
taking values in a compact set $K$ as in~\eqref{E:reeeeeeeterete},
and satisfying the a-priori bounds~\eqref{unifbvbound}.
A subsequence of $(h^{s_m,\varepsilon_m},p^{s_m,\varepsilon_m})$
converges, as $m \to \infty$, in~${\bf L}^1_{loc}$ to an 
entropy weak solution of the Cauchy problem~\eqref{S1system},~\eqref{S1system-data},
defined for all times $t>0$. 

The idea of adopting a time-splitting scheme to handle the effect of source terms was first introduced by Dafermos and Hsiao~\cite{DH}, 
in combination with the Glimm scheme (see also~\cite{C5, daf16} and references therein).
Subsequently, this scheme was implemented in conjunction with the front tracking method~\cite{AG,CP} and with the vanishing viscosity method~\cite{C1,C2}. 
Alternative methods to generate solutions of balance laws present in the literature are based
on a generalisation of the Glimm scheme~\cite{TPL2}
or of the front tracking algorithm~\cite{AGG}.

All these techniques provide local existence of solutions in presence of general source terms,
whereas global existence
is achieved either within the class of dissipative source terms~\cite{C5}, or for non-resonant systems
(having characteristic speeds bounded away from zero)
with source terms sufficiently small in~${\bf L}^1$~\cite{AGG,TPL2}. 
 However, although the system of balance laws~\eqref{S1system} does not belong to any of these 
 classes, an a-priori bound on the total variation of its weak solutions valid for all times is established in~\cite{AS},
 exploiting the particular geometric features of~\eqref{S1system}. 
This yields the existence of global solutions provided by Theorem~\ref{T:AS}.

\subsection{Wave size notation}
\label{Ss:strengthnotation}

The sizes of wave fronts of approximate solutions of~\eqref{S1system} are defined as the jumps between the left and right states which can be
 measured either with the original thickness and slope variables $(h,p)$, or with 
the corresponding Riemann coordinates $(H,P)$ associated to the $2\times 2$ system~\eqref{S1system}.
Such coordinates are defined as follows~\cite[Definition 1]{AS}.
Given any point $(h,p)\in\Omega$, let $(H,0)$ be the point on the $h-$axis connected with $(h,p)$ by a rarefaction curve 
of the second family and, similarly, let $(0,P)$ be the point on the $p-$axis connected to $(h,p)$ by a rarefaction curve
of the first family. Then the functions $(H,P)$ form a coordinate system of Riemann invariants
associated to~\eqref{S1system}.

So given a wave front with left and right states $(h^{\ell},p^{\ell})$ and $(h^{r},p^{r})$, respectively, let $(H^{\ell},P^{\ell})$ and $(H^{r},P^{r})$ be the corresponding Riemann coordinates. For simplicity, we drop from here and on the dependence of the approximate solution on $(s,\varepsilon)$.
%
Then, the wave size of the jump $\big((h^{\ell},p^{\ell}),\,(h^{r},p^{r})\big)$ can be defined 
either in the original or in the Riemann 
coordinate systems as follows:
\begin{itemize}
\item the size of a 1-wave (h-wave) is measured by 
\bes
\sR_{h} = H^{r} - H^{\ell}
\qquad
\text{or }
\qquad
\sC_{h} = h^{r} - h^{\ell}
\ees
in Riemann or 
original coordinates, respectively.
\item the size of a 2-wave (p-wave) is measured by
\bes
\sR_{p} = P^{r} -P^{\ell}
\qquad
\text{or }
\qquad
\sC_{p} = p^{r} - p^{\ell}
\ees
in Riemann or original
coordinates, respectively.
\end{itemize}
Notice that 1-rarefaction waves have positive size in the region $p>1$ and negative size in the region $p<1$,
whereas 1 admissible shocks have negative size in the region $p>1$ and positive size in the region $p<1$.
Instead 2-rarefaction waves have always negative sizes whereas 2 admissible shocks have always positive sizes (see Figure~\ref{S2:fig1}).
In particular, recalling~\eqref{Suno}--\eqref{Sdue}, when
the right state $(h^r,p^r)$ is connected with the left state $(h^{\ell},p^{\ell})$ via a shock wave of size $\sC$
in original coordinates, then one has:
\begin{equation}
(h^r,p^r)\doteq\mathbf{S}_1\big(\sC;h^{\ell}, p^{\ell}\big)
,\qquad \sC=\sC_h,\qquad (p-1)\cdot\gamma\leq 0\,,
\end{equation}
if $\big((h^{\ell},p^{\ell}),\,(h^{r},p^{r})\big)$ is an admissible shock of the first family, and
\begin{equation}
(h^r,p^r)\doteq\mathbf{S}_2\big(\sC;h^{\ell}, p^{\ell}\big)
,\qquad \sC=\sC_p,\qquad\gamma\geq 0\,,
\end{equation}
if $\big((h^{\ell},p^{\ell}),\,(h^{r},p^{r})\big)$ is an admissible shock of the second family.
\\
With the above notations, the size $\rho$ of the shock $\big((h^{\ell},p^{\ell}),\,(h^{r},p^{r})\big)$ expressed in Riemann coordinates is given by:
\bas
\rho=\sR_{h}=H\big(\mathbf{S}_1(\sC_{h};h^{\ell}, p^{\ell})\big)
-H(h^{\ell}, p^{\ell})\,,
&&
\rho=\sR_{p}=P\big(\mathbf{S}_2(\sC_{p};h^{\ell}, p^{\ell})
\big)-P(h^{\ell}, p^{\ell})\,.
\eas
\begin{remark}
\label{Riemann coordinates}
Since the change of variable $(h,p) \to (H,P)$ is a smooth map, the two ways of measuring the sizes
of the wave fronts are equivalent, provided $(h,p)$ lie within the compact set $K$ of Theorem~\ref{T:AS}.
Namely, in view of the analysis of~\cite[Section~3]{AS}, one can obtain the following relations
 among the sizes and the $p$-component in the two different coordinate systems
\ba\label{E:equivalncestrengths}
&\frac{|\sC_{h}|}{\mu}\leq |\sR_{h}|\leq \mu |\sC_{h}|,
&&\frac{|\sC_{p}|}{\mu}\leq |\sR_{p}|\leq \mu|\sC_{p} |,
&&\frac{|p^\ell-1|}{\mu}\leq |P(h^\ell,p^\ell)-1|\leq \mu|p^\ell-1|,
\ea
that hold for some $\mu>1$ and for all $(h^\ell,p^\ell), (h^r,p^r)\in K$.
\end{remark}


\subsection{Glimm functionals}
\label{Ss:glimmfunct}

Let $u=u(x,t)\doteq (h^{s,\varepsilon},p^{s,\varepsilon})(x,t)$ be a piecewise constant
$s$-$\varepsilon$-approximate solution of~\eqref{S1system}
constructed by the procedure described in Subsection~\ref{Ss:Rsolv-appsol}.
As customary, a-priori bounds on the total variation of $u(t)\doteq u(\cdot,t)$ outside the time steps are 
obtained in~\cite{AS} by analyzing suitable wave strength and wave interaction potential defines as follows. 
At any time $t>0$ where no interaction occurs and different from the time steps 
$t_k$
where $u$ is updated taking into account the source term, let $\cj_{i}\big(u(t)\big)$
denote a set of indexes $\alpha$ associated to the jumps of the $i$-th family of $u(t)$.
Assume that the jump of index $\alpha$ is located at $x_\alpha$, that
its strength 
measured in Riemann coordinates is $|\sR_\alpha|$, and let $P_\alpha^\ell\doteq P(u(x_\alpha-,t))$
denote the left limit of the \linebreak P-Riemann coordinate of $u(t)$ at $x_\alpha$.
Set $\cj\big(u(t)\big)\doteq \cj_{1}\big(u(t)\big)\bigcup \cj_{2}\big(u(t)\big)$ 
to denote the collection of indexes associated to all jumps of $u(t)$.
Denote by $k_{\alpha}\in\{1,2\}$ the characteristic family of the jump $\alpha\in \cj\big(u(t)\big)$, so that, in particular, one has
$\alpha\in \cj_{k_{\alpha}}\big(u(t)\big)$.
Then, we define the {\it total strength} of waves in $u(t)$ as:
\be{E:Vdef}
V_i\big(u(t)\big)\doteq\!\!\sum_{\alpha\in\cj_i((u(t))}|\sR_{\alpha}|,\qquad i=1,2,\qquad\quad V\big(u(t)\big)=V_1\big(u(t)\big)+V_2\big(u(t)\big)\doteq\sum_{\alpha\in\cj(u(t))}|\sR_{\alpha}|\,,
\ee
and the {\it interaction potential} as:
\be{E:Q}
\cq\big(u(t)\big)\doteq\cq_{hh}+\cq_{hp}+\cq_{pp}\;.
\ee
Here $\cq_{hh}$ is the modified interaction potential of waves of the first family (h-waves)
introduced in~\cite{AS}, defined as 
\be{E:Qhh}
\cq_{hh}\doteq\sum_{\stackrel{k_{\alpha}=k_\beta=1}{x_{\alpha}<x_\beta}}\omega_{\alpha,\beta}|\sR_{\alpha}||\sR_\beta|\,,
\ee
with the weights $\omega_{\alpha,\beta}$ given by
\be{E:omega}
\omega_{\alpha,\beta}\doteq
\begin{cases}
\overline\delta\cdot\min\{|P_{\alpha}^{\ell}-1|,|P_{\beta}^{\ell}-1|\}
\quad&\text{if}\quad 
\sR_{\alpha},\sR_\beta \ \ \ \text{are 1-shocks lying on the same side of} \ \ \ $p=1$,
\\
\noalign{\medskip}
\qquad\qquad 0\qquad &\text{otherwise,}
\end{cases}
\ee
%
for a suitable constant $\overline\delta>0$ sufficiently small
(depending on the bound $M_0$ on the total variation of the initial data in Theorem~\ref{T:AS}).
Instead, the interaction potential of waves of both families and of the second family (p-waves)
 are defined in the standard way as
\be{E:Qhp}
\cq_{hp}\doteq\sum_{\stackrel{k_{\alpha}=2,k_\beta=1}{ x_{\alpha}<x_\beta}}|\sR_{\alpha}\sR_\beta|
\ee
\be{E:Qpp}
\cq_{pp}\doteq\sum_{(\alpha,\beta)\in Appr_2}|\sR_{\alpha}\sR_\beta|
\ee
where $Appr_2$ denotes the collection of pairs of indexes of approaching p-waves.
We recall that two waves of the same family, located at $x_\alpha, x_\beta$
 with $x_{\alpha}<x_\beta$, are defined as {\it approaching} if at least one of them is a shock.
Notice that the interaction potentials defined for non GNL systems available in the literature (e.g.~the one in~\cite{AM2}) 
are suited only for solutions with sufficiently small total variation, whereas the functional~\eqref{E:Q} introduced in~\cite{AS}
allows to establish the existence of solutions with arbitrarily large total variation.
In fact, relying on the interaction estimates established in~\cite{AS} and collected here in Section~\ref{S:basicest},
one can show that 
 the {\em{Glimm functional}} 
\be{S3G}
\cg\big(u(t)\big)\doteq V\big(u(t)\big)+\cq\big(u(t)\big)
\ee
is nonincreasing in any time interval
$]t_k, t_{k+1}[$ between two consecutive time steps.
Instead, when the solution is updated with the source term,
we will exploit other estimates on the variation of the strength of waves 
which were derived in~\cite{AS} and are collected here in Section~\ref{S:basicest}.
Thanks to such estimates, we deduce that 
at any time step $t_k=k\Delta t = k\, s$
there holds
\begin{equation}
\label{G-timestep-est1}
\cg\big(u(t_k+)\big)\leq \big(1+\co(1)\Delta t\big)\cdot\cg^{-}\big(u(t_k-)\big)\,.
\end{equation}
Notice that, by definition~\eqref{E:Vdef}, oner clearly has 
\begin{equation}
\label{V-totvar-est}
 V\big(u(t)\big) \leq \co(1) \mathrm{TotVar}\{u(\,\cdot\,,t)\}\qquad\forall~t>0\,.
\end{equation}
\subsection{Lyapunov functional}
\label{Ss:Lyapfunct}

Let $u$ and $v:\R\times\R_+\to\R^2$ be two
$s$-$\varepsilon$-approximate solution of~\eqref{S1system},
with values in a compact set $K$ as in~\eqref{E:reeeeeeeterete},
and constructed as in Subsection~\ref{Ss:Rsolv-appsol}.
In order to:
\vspace{-5pt}
\begin{itemize}
\item[(I)] 
provide an a-priori bound on the ${\bf L^1}$-distance between $u$ and $v$
in the spirit of~\cite[\S~8]{Bre};
\item[(II)]
derive as in~\cite{AG} an estimate of the type~\eqref{uniq-semigr-est-1}
between approximate solutions of the non-homogeneous system~\eqref{S1system}
and of the homogeneous system~\eqref{S1system-hom};
\end{itemize}
we introduce here
a Lyapunov-like functional~$\Phi$ 
with the properties (i)-(ii) stated in the introduction.
Consider a piecewice constant function 
$\er:\R\to\R^2$
taking values in a given compact set $K'$
related to $K$. 
Assume also that $\er\in {\bf L^1}$
and 
\begin{equation}
\label{z-bound-1}
\mathrm{TotVar}\{\er_i\}\le \sigma\quad\ i=1,2\,,
\qquad
\forall~t>0\,,
\end{equation}
for some constant $\sigma>0$.
Notice that $\er$ is an arbitrary function with the aforementioned properties, not related to the system~\eqref{S1system},
which is introduced to accomplish~(II). Next, for every $(x,t)\in\R\times\R_+$,
we connect $u(x,t)$ with $w(x,t)\doteq v(x,t)+\er(x)$
through the Hugoniot curves of the first
and second family. In this way we define implicitly two scalar functions $\eta_{i}:\R\times\R_+\to\R$, $i=1,2$,
by the relation
\be{E:vu} 
w(x,t)=\mathbf {S}_2\big(\eta_{2}(x,t);\,\cdot\,\big)\circ\mathbf {S}_1\big(\eta_{1}(x,t);\,u(x,t)\big)\,.
\ee
According to this definition, the parameter $\eta_{i}$ can be regarded as the size of the i-shock wave
in the jump $\big(u(x,t),\, w(x,t)\big)$, where this size is measured in the original coordinates
(cfr.~Subsection~\ref{Ss:strengthnotation}).
Then, we clearly have
\be{S4:etaabs}
\frac{1}{C_0} \big|u(x,t)-w(x,t)\big|\le \sum_i \big|\eta_i(x,t)\big|\le C_0 \big|u(x,t)-w(x,t)\big|
\qquad\forall~x,t\,,
\ee
for some constant $C_0>0$.
Setting $u(t)\doteq u(\,\cdot\,,t)$, $v(t)\doteq v(\,\cdot\,,t)$, we now define the functional
\be{UpsilonNew}
\Phi_\er(u(t),v(t))\doteq
\left(\sum_{i=1}^2\int_{-\infty}^\infty|\eta_{i}(x,t)|W_i(x,t)\,dx\right)
 \cdot e^{\kappa_{\cg} \big[\cg(u(t))+\cg(v(t))\big]}
\ee
with weights $W_i$ of the form
\begin{subequations}
\label{S4Wi}
\ba
\label{S4WiN} 
W_1(x,t)&\doteq\exp\left(\kappa_{1\ca 1}\ca_{1,1}(x,t)+\kappa_{1\ca 2}\ca_{1,2}(x,t) \right), \\
\label{S4WiN2} 
W_2(x,t)&\doteq\exp\left(\kappa_{2\ca 1}\ca_{2,1}(x,t)+ \kappa_{2\ca 2}\ca_{2,2}(x,t) \right), 
\ea
\end{subequations}
for suitable positive constants $1<\kappa_{i\ca j}<\kappa_\cg$, $i,j=1,2$, to be specified later.
Here $\cg$ is the Glimm functional defined in~\eqref{E:Vdef}--\eqref{S3G},
while $\ca_{i,j}(x,t)$ measures the total amount of $j$-waves in $u(t)$ and $v(t)$ which approach the $i$-wave $\eta_{i}$ located at $x$.
Since the second characteristic family is GNL for $h>0$ and LD along $h=0$,
the expression of $\ca_{2,1}$ and $ \ca_{2,2}$ is the same as the 
one given in~\cite{bly} for GNL and LD characteristic fields.
Instead, 
because of the properties of the non GNL first characteristic family, the definition of 1-waves approaching $\eta_1$ varies if the left state of such waves 
lies on the left or on the right of $\{p=1\}$ (see Figure~\ref{figappr0}). 
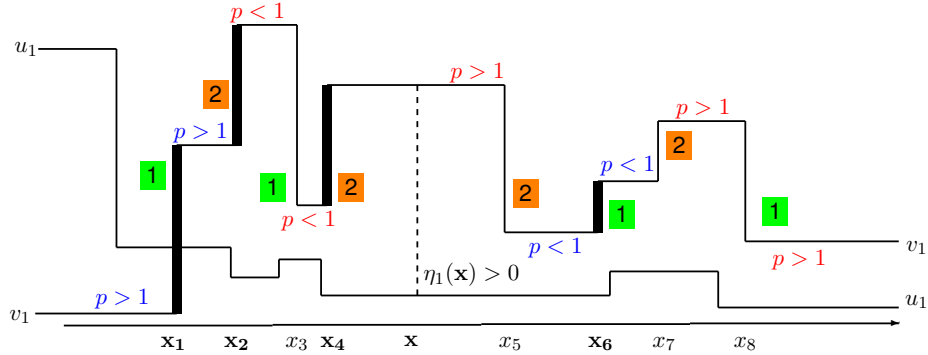
\begin{figure}[htbp]
\label{figappr0}
{\centering \scalebox{.8}{\input{uv.tex} } \par}
\caption{\small
\emph{Approaching waves} in $v$ towards $ \eta_1({\bf x})>0$ are indicated by the jumps marked with bolded lines.
Also, regions $p < 1$, $p > 1$ can only be connected by $2-$waves crossing the line $p=1$. The selected $1-$waves that are located at ${\bf x_\alpha}$ with ${\bf x_\alpha}<{\bf x}$ correspond to $\gamma\to \lambda_1(\gamma;\cdot)$ strictly increasing, i.e. $\{p>1\}$. On the other hand, the selected $1-$waves that are located at ${\bf x_\alpha}$ with ${\bf x_\alpha}>{\bf x}$ correspond to $\gamma\to \lambda_1(\gamma;\cdot)$ strictly decreasing, i.e. $\{p<1\}$.\label{figappr0}
}
\end{figure}

For this reason, it will be necessary to assign a weight to the strength of the
1-waves approaching $\eta_1$ which depends on the distance of their left state from $\{p=1\}$,
so to control the possible increase of $\Phi_\er$
at times of interactions (of $u$ or of $v$) involving a 1-wave and a 2-wave crossing $\{p=1\}$.
More precisely, the definitions of $\ca_{i,j}$, $i,j=1,2$, are the followings:
As in Subsection~\ref{Ss:glimmfunct}, let $\cj_{i}\big(u(t)\big)$
denote a set of indexes $\alpha$ associated to the jumps of the $i$-th family of $u(t)$
located at $x_\alpha$ and let $\cj_{i}\big(v(t)\big)$ denote a similar set of indexes for
the jumps of the $i$-th family of $v(t)$.
Denote by $p_\alpha^\ell$ 
the p-component (in original coordinates) of the left state of the jump located at $x_{\alpha}$, and by 
$\sR_{\alpha}$ 
the corresponding size of the jump
measured in Riemann coordinates. 
Then define
\be{S4A1}
\begin{aligned}
\ca_{1,1}&(x,t)\doteq\left\{
\!\!
\begin{array}{ll}
\displaystyle\Big[\sum_{\stackrel{\alpha\in\cj(u(t)),k_{\alpha}=1}{p_\alpha^\ell
>1,\,x_{\alpha}<x}}
+\sum_{\stackrel{\alpha\in\cj(u(t)),k_{\alpha}=1}{p_\alpha^\ell
<1,\,x_{\alpha}>x}}
+\sum_{\stackrel{\alpha\in\cj(v(t)),k_{\alpha}=1}{p_\alpha^\ell
>1,\,x_{\alpha}>x}}
+\sum_{\stackrel{\alpha\in\cj(v(t)),k_{\alpha}=1}{p_\alpha^\ell
<1,\,x_{\alpha}<x}}
\Big]\PHI{p}|\sR_{\alpha}| \qquad &\text{if }\quad\eta_{1}(x,t)<0,\\
\noalign{\medskip}
\displaystyle\Big[\sum_{\stackrel{\alpha\in\cj(v(t)),k_{\alpha}=1}{p_\alpha^\ell
>1,\,x_{\alpha}<x}}
+\sum_{\stackrel{\alpha\in\cj(v(t)),k_{\alpha}=1}{p_\alpha^\ell
<1,\,x_{\alpha}>x}}
+\sum_{\stackrel{\alpha\in\cj(u(t)),k_{\alpha}=1}{p_\alpha^\ell
>1,\,x_{\alpha}>x}}
+\sum_{\stackrel{\alpha\in\cj(u(t)),k_{\alpha}=1}{p_\alpha^\ell
<1,\,x_{\alpha}<x}}
\Big]\PHI{p}|\sR_{\alpha}| \qquad& \text{if }\quad\eta_{1}(x,t)>0,
\end{array}
\right.\\
\noalign{\medskip}
\ca_{1,2}&(x,t)\doteq\sum_{\stackrel{\alpha\in\cj(u(t))\cup\cj(v(t))}{k_{\alpha}=2,\,x_{\alpha}<x}}|\sR_{\alpha}|\,,
\end{aligned}
\ee
and
\begin{subequations}
\label{S4A2}
\begin{align}
\ca_{2,1}&(x,t)\doteq\sum_{\stackrel{\alpha\in\cj(u(t))\cup\cj(v(t))}{k_{\alpha}=1,\,x_{\alpha}>x}}|\sR_{\alpha}|\,,
\\
\noalign{\medskip}
\ca_{2,2}&(x,t)\doteq\left\{
\begin{array}{ll}
\displaystyle\Big[\sum_{\stackrel{\alpha\in\cj(u(t)),k_{\alpha}=2}{x_{\alpha}>x}}
+\sum_{\stackrel{\alpha\in\cj(v(t)),k_{\alpha}=2}{x_{\alpha}<x}}
\Big]|\sR_{\alpha}| \qquad &\text{if }\quad\eta_{2}(x,t)<0,\\
\noalign{\medskip}
\displaystyle\Big[\sum_{\stackrel{\alpha\in\cj(v(t)),k_{\alpha}=2}{x_{\alpha}>x}}
+\sum_{\stackrel{\alpha\in\cj(u(t)),k_{\alpha}=2}{x_{\alpha}<x}}
\Big]|\sR_{\alpha}| \qquad&\text{if }\quad\eta_{2}(x,t)>0.
\end{array}
\right.
\end{align}
\end{subequations}
Notice that, differently from the functional introduced in~\cite{bly}, 
here the sum of the whole Glimm functional
 of $u$ and $v$ instead of the sum of their interaction potential $\cq$ is present as a weight in the time variable.
This is due to the fact that one needs to exploit the decrease of the functional $V$ in~\eqref{E:Vdef}
at interactions of 1-shock with 1-rarefaction since, by the definitions~\eqref{E:Q}--\eqref{E:Qhh},
the interaction potential is not decreasing when such interactions occur (cfr.~\eqref{E:decreaseGlimm1s1r} and the analysis 
in Subsection~\ref{Sss:Interaction1nC}). 

We point out that:
\vspace{-5pt}
\begin{itemize}
\item[-] the definition of the functional $\Phi_\er$ at~\eqref{UpsilonNew}--\eqref{S4Wi} is given in
terms of waves
connecting $u(t)$ with $w(t)=v(t)+z$,
and of waves in $u(t), v(t)$,
which are measured with respect to different coordinate systems.
Namely, the size $\eta_i(t)$ of the i-waves connecting $u(t)$ with $w(t)$ is measured in original coordinates.
Instead, the size $\sR_\alpha$ of the waves in $u(t)$ or in $v(t)$ is measured in Riemann coordinates.
 Of course, one can express also $\eta_i(t)$ in Riemann coordinates because of~\eqref{E:equivalncestrengths}, but we choose to 
 keep them in original coordinates for technical reasons since this choice simplifies the computations
carried out in Appendix~\ref{S:finerInteractions} and applied in Section~\ref{Ss:nointeractiontimesN}.
\item[-] The function $z$ affects directly the definition of the waves $\eta_i$
connecting $u$ with $w\doteq u+z$ while enters only indirectly in the definition of the weights $W_i$
which are expressed in terms of waves of $u$ and $v$ which depend on the sign of $\eta_i$.
\end{itemize} 
We observe also that, thanks to the a-priori BV and ${\bf L}^{\infty}$-bounds established in~\cite{AS},
there will be constants $M^*, P^*>0$ such that, for all $t>0$, there hold
\begin{equation}
\label{mstar-bound}
\cg(u(t))\le M^*\,,\qquad\quad \cg(v(t))\le M^*\,,
\end{equation}
and
\begin{equation}
\label{pstar-bound}
\big\|P(u(t))-1\big\|_{\bf L^\infty}\le P^*\,,
\qquad\quad
\big\|P(v(t))-1\big\|_{\bf L^\infty}\le P^*\,.
\end{equation}
Hence, the functional $W_i$ in~\eqref{S4Wi} is uniformly bounded by
\begin{subequations}
\label{W-unif-boundN}
\ba
&\ca_{1,1}(x,t) 
\leq M^* \delta_p^* \,,
&& 1 \leq W_1(x,t) \leq W_1^*\doteq e^{\kappa_{1\ca1}\cdot M^*\delta_p^* + \kappa_{1\ca2}\cdot M^* }
\qquad\quad \forall~x,t 
\\
 &&&1 \leq W_2(x,t) \leq W_2^*\doteq e^{ \kappa_{2\ca1}\cdot M^*+ \kappa_{2\ca2}\cdot M^*}
\qquad\quad \forall~x,t \,.
\ea
\end{subequations}
Therefore, relying on~\eqref{S4:etaabs},~\eqref{W-unif-boundN}, 
we deduce that the functional $\Phi_\er$ is equivalent to the ${\bf L^1}$ distance
between $u(t)$ and $w(t)=v(t)+z$:
\be{S4-PhiL1}
\frac{1}{C_{0}}\big\| u(t)-w(t)\big\|_{L^1}
\leq\Phi_\er(u(t),v(t))
\leq C_{0}\cdot W^* \cdot \big\| u(t)-w(t)\big\|_{L^1}
\qquad\quad \forall~t>0\,.
\ee
For $z=0$, we automatically have that $\Phi_0(u,v)$ satisfies the corresponding relation between $u$ and $v$.

\subsection{Main results}
\label{Ss:mainres}

In view of~\eqref{S4-PhiL1}, ${\bf L}^1$ stability estimates for approximate solutions $u, v$ of~\eqref{S1system}
and~\eqref{S1system-hom}
can be established in terms of the functional $\Phi_z$ as stated in the following
\begin{theorem}
\label{stability-Phy}
Given $M_0>0$, there exist constants $\delta_0, \delta_p, \delta_0^*, \delta_p^*, M^*_0>0$, 
and $C_1, C_2>0$,
so that, letting $\Phi_z$ be the functional
defined in~\eqref{UpsilonNew}--\eqref{S4A2}, 
for suitable $\kappa_{i\ca j}>0$,\,$i,\ j=1,\ 2$ and $\kappa_\cg$, the following hold.
\begin{itemize}
\item[(i)] (Homogeneous case).
Let $u$ and $v:\R\times\R_+\to\R^2$ be two
$\varepsilon${\it -front tracking approximate solution} of the homogeneous system~\eqref{S1system-hom},
constructed as in Subsection~\ref{Ss:Rsolv-appsol},
with initial data $u(\,\cdot,0), v(\,\cdot,0)$ 
satisfying~\eqref{initial-data-bounds} and~\eqref{unifbvboundLinftyp} ,
and taking values in $[0,\delta^*_0] \times [p^*_{0}, p^*_1]$.
Let $\er : \R \to\R^2$ be a piecewise constant function, 
that takes values in the compact set
\be{E:im-z}
K'=[(p^*_{0}-1)\cdot\delta_0^*,\,p_1^*\cdot\delta^*_0] \times [p^*_{0}, p^*_1]\,,
\ee
and satisfies~\eqref{z-bound-1} for 
\begin{equation}
\label{sigma-def}
\sigma\doteq (\delta_0^*+p_1^*)\cdot M_0^*,
\end{equation}
where $p_0^*:=1-\delta_p^*>0$, $p_1^*:=1+\delta_p^*>1$.
Then, there holds
\begin{equation}
\label{Phi0est1}
\Phi_{\er}\big(u(\tau_2), v(\tau_2)\big)\leq
\Phi_{\er}\big(u(\tau_1), v(\tau_1)\big)+
C_1\cdot \big(\varepsilon + \sigma\big)
(\tau_2-\tau_1)\qquad\forall~ \tau_2>\tau_1>0\,.
\end{equation}

\item[(ii)] (Non-homogeneous case).
Let $u$ and $v:\R\times\R_+\to\R^2$ be two
$s$-$\varepsilon$-approximate solution of the non-homogeneous system~\eqref{S1system} constructed as in Subsection~\ref{Ss:Rsolv-appsol},
with initial data $u(\,\cdot,0), v(\,\cdot,0)$ satisfying~\eqref{initial-data-bounds} and~\eqref{unifbvboundLinftyp},
and taking values in $[0,\delta^*_0] \times [p^*_{0}, p^*_1]$.
Then, letting $t_k\doteq k \Delta t = k\, s$, \, $k \in \mathbb{N}$ be the time steps,
there holds
\begin{equation}
\label{Phi0est2}
\Phi_0\big(u(\tau_2), v(\tau_2)\big)\leq
\Phi_0\big(u(\tau_1), v(\tau_1)\big)+
C_1\cdot \varepsilon
(\tau_2-\tau_1)\qquad\forall~t_k<\tau_1<\tau_2<t_{k+1}\,,
\end{equation}
and
\begin{equation}
\label{Phi0est3}
\begin{aligned}
\Phi_0\big(u(t_k+), v(t_k+)\big)
&\leq
\Phi_0\big(u(t_h+), v(t_h+)\big)
\big(1+C_2\cdot\Delta t\big)^{(k-h)}+
\\
\noalign{\smallskip}
&\qquad + C_1\cdot \varepsilon\,\Delta t \sum_{i=1}^{k-h} \big(1+C_2\cdot \Delta t\big)^i
\qquad\quad\forall~0\leq h<k,
\end{aligned}
\end{equation}
for all $k,\  h\in\mathbb{N}$.
\end{itemize}
\end{theorem}

The estimate~\eqref{Phi0est3} is precisely the estimate~\eqref{gammaest2} stated in Section~\ref{intro},
with $\Phi_0$ in place of $\Phi$.
A proof of Theorem~\ref{stability-Phy} will be established in Section~\ref{S4old}.
Relying on Theorem~\ref{stability-Phy}-(i), one can easily derive the existence of a Lipschitz continuous 
semigroup of solutions of the homogeneous system~\eqref{S1system-hom}.
\begin{theorem}
\label{exist-hom-smgr}
Given $M_0>0$, there exist 
$\delta_0, \delta_p, \delta^*_0, \delta_p^*, M^*_0,  L>0$ and a unique (up to the domain) semigroup map
\begin{equation}
\label{E:sem}
\mathcal{S}: [0,+\infty)\times \mathcal{D}_{0}\to \mathcal{D}^*_{0},
\qquad\qquad
(\tau,\overline u)\mapsto \mathcal{S}_\tau \overline u\,,
\end{equation}
with $\mathcal{D}_{0}\doteq \mathcal{D}(M_0 , \delta_0 , \delta_p)$, 
$\mathcal{D}^*_{0}\doteq \mathcal{D}(M^*_0,  \delta^*_0, \delta_p^*)$
domains defined as in~\eqref{Domain-def1}, which enjoys the following properties:
\vspace{-5pt}
\begin{itemize}
\item[(i)] $\cs_{\tau_2}\big(\cs_{\tau_1}\, \overline u\big)\in \mathcal{D}^*_0\qquad\forall~
\overline u\in\mathcal{D}_{0},\ \
\forall~\tau_1, \tau_2\geq 0;$
\item[(ii)] $\cs_{0} \,\overline u= \overline u,\quad \cs_{\tau_1+\tau_2}\, u=\cs_{\tau_2}\big(\cs_{\tau_1}\, \overline u\big)\qquad
\forall~\overline u\in\mathcal{D}_{0},\ \ \forall~\tau_1, \tau_2\geq 0;$
\item[(iii)] $\big\Vert \cs_{\tau_2} \overline u - \cs_{\tau_1} \overline v \big\Vert_{{\bf L}^{1}}\leq L\cdot (|\tau_1-\tau_2|+\lVert \overline u-\overline v\rVert_{{\bf L}^{1}})\qquad
\forall~\overline u,\,\overline v \in\mathcal{D}_{0},\ \ \forall~\tau_1, \tau_2\geq 0;$
\item[(iv)] For any $\overline u\doteq (\overline h, \overline p)\in\mathcal{D}_{0}$, the map $\big(h(x,\tau),p(x,\tau)\big)\doteq 
 \mathcal{S}_\tau\, \overline u(x)$ provides an entropy weak solution of the Cauchy problem~\eqref{S1system-hom},
~\eqref{S1system-data}. Moreover, $\mathcal{S}_\tau\, \overline u(x)$ coincides with the unique limit of front tracking approximations.
 \item[(v)] If $\overline u\in \mathcal{D}_{0}$ is piecewise constant, then for $\tau$ sufficiently small $u(\,\cdot\,,\tau)\doteq 
 \cs_\tau\,\overline u$ coincides with the solution of the Cauchy problem~\eqref{S1system-hom},~\eqref{S1system-data} obtained
 by piecing together the entropy solutions of the Riemann problems determined by the jumps of~$\overline u$. 
\end{itemize}
\end{theorem}
\begin{proof}
The proof is entirely similar to~\cite[Proof of Theorem~2]{bly}. 
For sake of clarity, we provide it here. 
Let $\delta_0, \delta_p, \delta^*_0, \delta_p^*, M^*_0, p^*_0, p_1^*$ be the constants provided by Theorem~\ref{stability-Phy}.
Given $\overline u\in \mathcal{D}_{0}$, consider a sequence $\{u_m\}_m$ of $\varepsilon_m$-front tracking approximate solutions to~\eqref{S1system-hom} with values in $[0,\delta^*_0] \times [p^*_{0}, p^*_1]$,
with initial data $u_m(0)\in \mathcal{D}_{0}$, and such that
\begin{equation}
\label{indata-eps-conv-1}
\lim_{m\to\infty} \big\|u_m(0)-\overline u\big\|_{\bf L^1}=0\,.
\end{equation}
Then, assuming that $\{\varepsilon_m\}_m$ is decreasing to zero,
relying on~\eqref{S4-PhiL1} with $u=u_m$, $v=u_n$, $\er=0$, 
and applying~\eqref{Phi0est1} with $\er=0$, $\sigma=0$, we find 
\begin{equation}
\label{lip-est-appr-hom-1}
\begin{aligned}
\big\|u_m(\tau)-u_n(\tau)\big\|_{\bf L^1} &\leq C_0 \cdot 
\Phi_0\big(u_m(\tau), u_n(\tau)\big)
\\
&\leq C_0 \cdot \Phi_0\big(u_m(0), u_n(0)\big)+
C_0\,C_1\cdot \tau\cdot\varepsilon_m
\\
&\leq C_0^2\, W^* \cdot \big\|u_m(0)-u_n(0)\big\|_{\bf L^1}+
C_0\,C_1\cdot \tau\cdot\varepsilon_m\,,
\end{aligned}
\end{equation}
for $m\leq n$ and for all $\tau>0$.
Thus, it follows from~\eqref{indata-eps-conv-1}--\eqref{lip-est-appr-hom-1} that
$\{u_m(t)\}_m$ is a Cauchy sequence in~${\bf L^1}$ for all $t>0$, and hence it converges to a unique
limit
\begin{equation}
\label{hom-sem-lim-def}
\cs_\tau\, \overline u \doteq \lim_{m\to\infty} u_m(\tau)\,.
\end{equation}
With the same arguments of the analysis in~\cite{AS},
and by the
uniqueness of the limit~\eqref{hom-sem-lim-def}, one then deduces that $\cs_\tau\, \overline u\in
\mathcal{D}^*_0$ for all $\tau>0$ (cfr.~Theorem~\ref{T:AS}), and that properties (i)-(ii), (iv) are verified.
Next, the Lipschitz continuity property of $\cs_\tau$ is obtained as in~\cite{bly}.
Namely, given $\overline u, \overline v\in \mathcal{D}_{0}$, consider two sequences $\{u_m\}_m$,
$\{v_m\}_m$ of $\varepsilon_m$-front tracking approximate solutions to~\eqref{S1system-hom} with values in $[0,\delta^*_0] \times [p^*_{0}, p^*_1]$,
with initial data $u_m(0), v_m(0)\in \mathcal{D}_{0}$, and such that
\begin{equation}
\label{indata-eps-conv-2}
\lim_{m\to\infty} \varepsilon_m= 0\,, \qquad\quad
\lim_{m\to\infty} \big\|u_m(0)-\overline u\big\|_{\bf L^1}=
\lim_{m\to\infty} \big\|v_m(0)-\overline v\big\|_{\bf L^1}=0\,.
\end{equation}
Again, relying on~\eqref{S4-PhiL1} with $u=u_m$, $v=v_m$, $\er=0$, 
and applying~\eqref{Phi0est1} with $\er=0$, $\sigma=0$, we derive
\begin{equation}
\label{lip-est-appr-hom-2}
\begin{aligned}
\big\|u_m(\tau)-v_m(\tau)\big\|_{\bf L^1} &\leq C_0 \cdot 
\Phi_0\big(u_m(\tau), v_m(\tau)\big)
%
\\
&\leq C_0^2\, W^* \cdot \big\|u_m(0)-u_n(0)\big\|_{\bf L^1}+
C_0\,C_1\cdot \tau\cdot\varepsilon_m\,.
\end{aligned}
\end{equation}
Taking the limit as $m\to\infty$ in~\eqref{lip-est-appr-hom-2},
and relying on~\eqref{hom-sem-lim-def}--\eqref{indata-eps-conv-2}, we thus obtain
\begin{equation}
\label{lip-est-appr-hom-3}
\big\|\cs_\tau\, \overline u-\cs_\tau\, \overline v\big\|_{\bf L^1} \leq C_0^2\, W^* \cdot \big\|\overline u-\overline v\big\|_{\bf L^1}\,.
\end{equation}
This yields (iii) since the Lipschitz continuity with respect to time 
is a standard property enjoyed by limits of front tracking solutions 
 with finite speed of propagation and uniformly bounded
total variation (cfr~\cite[Section 7.4]{Bre2}). 
Finally, the consistency with the solutions of the Riemann problem
and with limits of front tracking approximations (v) as well as the uniqueness of the semigroup map 
can be established by
standard arguments in~\cite{Bre2}, \cite{BB} that 
remain valid for solutions with large total variation. This concludes the proof.
\end{proof}
\begin{remark}
Notice that the image of the map $\cs_t$ in~\eqref{E:sem} is the same for every $t>0$,
but the domain~$\mathcal{D}_{0}$ is not positively invariant under the action of $\cs$.
This is due to the fact that, although one can establish ${\bf L}^\infty$, ${\bf L}^1$ 
and BV bounds on the solutions of~\eqref{S1system-hom} which are uniform in time, it turns out that the
${\bf L}^\infty$, ${\bf L}^1$- norms as well as the total variation of the solution 
(that appear in the definition of the domain~\eqref{Domain-def1}) may well increase in presence of interactions
(see the analysis in~\cite[Section 5]{AS}).
\end{remark}
Employing Theorem~\ref{stability-Phy}-(ii) and Theorem~\ref{exist-hom-smgr}, 
we can now construct an approximate solution operator for the non homogeneous system~\eqref{S1system}
that depends Lipschitz continuously on the initial data, with a Lipschitz constant 
that grows exponentially in time.
\begin{theorem}\label{ThmPS5.4}
Given $M_0>0$, there exist $\delta_0, \delta_p, \delta^*_0, \delta^*_p, M^*_0>0$
so that the conclusions of Theorem~\ref{exist-hom-smgr} hold 
together with
the following. For all $s=\Delta t>0$ sufficiently small, setting $t_k\doteq k \Delta t = k\, s$, \, $k \in \mathbb{N}$, \
$g((h,p))=\big(0,(p-1)h\big)$, 
and letting $\mathcal{D}_{0}\doteq \mathcal{D}(M_0 ,\delta_0, \delta_p )$, 
$\mathcal{D}^*_{0}\doteq \mathcal{D}(M^*_0, \delta^*_0, \delta^*_p)$
 be domains as in~\eqref{Domain-def1},
the map $(\tau,\overline u)\mapsto \mathcal{P}^s_\tau \,\overline u$ given by
%
\be{app-nonhom-flow-def}
\begin{aligned}
\mathcal{P}^s_{0} \,\overline u &= \overline u\qquad \overline u\in\mathcal{D}_{0}\,,
\\
\noalign{\smallskip}
\mathcal{P}^s_\tau\,\overline u &= \cs_{\tau-t_k}\cp^s_{t_k}\,\overline u\qquad\forall~\tau\in (t_k, t_{k+1}),
\quad k\in\mathbb{N}\,,\quad \overline u\in\mathcal{D}_{0}\,,
\\
\noalign{\smallskip}
\mathcal{P}^s_{t_k}\,\overline u &= \mathcal{P}^s_{t_k-}\,\overline u + s\cdot g\big( \mathcal{P}^s_{t_k-}\,\overline u\big)
 \quad \text{with} \quad \displaystyle{\mathcal{P}^s_{t_k-}\,\overline u\doteq \lim_{\ \tau\to t_k-} \!\!\mathcal{P}^s_{\tau}\,\overline u=\cs_s \cp^s_{t_{k-1}}\,\overline u},
\quad k\in\mathbb{N}\,,\quad \overline u\in\mathcal{D}_{0}\,,
\end{aligned}
\ee
is well defined for all $\tau>0$
and takes values in $\mathcal{D}^*_{0}$. Moreover, 
there exist $L',C_3,C_4>0$ so that
the following properties hold.
%
\begin{itemize}
\item[(i)] $\mathcal{P}^s_{\tau_2}\big(\mathcal{P}^s_{\tau_1} \,\overline u\big)\in \mathcal{D}^*_0\qquad\forall~
\overline u\in\mathcal{D}_{0},\ \
\forall~\tau_1, \tau_2\geq 0\,;$
\item[(ii)]
$\big\|\cp_{\tau_1}^{s} \cp_{\tau_2}^{s} \,\overline u-\cp_{\tau_1+\tau_2}^{s} \,\overline u\big\|_{L^1}\le C_3\cdot s\qquad\forall~
\overline u\in\mathcal{D}_{0},\ \
\forall~\tau_1, \tau_2\geq 0\,;$
\item[(iii)]
$\big\|\cp_{\tau_1}^s \,\overline u-\cp_{\tau_2}^s \,\overline u\big\|_{L^1} \le
 L'\cdot |\tau_2-\tau_1|+C_3\cdot s\qquad\forall~
\overline u \in\mathcal{D}_{0},\ \
\forall~\tau_1, \tau_2\geq 0\,;$
\item[(iv)]
$\big\|\cp_{\tau}^s \,\overline u-\cp_{\tau}^s \,\overline v\big\|_{L^1} \le
 L'\cdot e^{C_4\cdot \tau}
 \cdot\|\overline u - \overline v\|_{\bf L^1}
 \qquad\forall~\overline u, \overline v\in\mathcal{D}_{0},\ \
\forall~\tau>0\,.$
\end{itemize}
%
\end{theorem}
\begin{proof} 
Given $M_0>0$, let $\delta_0, \delta_p \delta^*_0, \delta_p^*,M^*_0, p^*_0, p_1^*>0$ be constants 
so that the conclusions of Theorem~\ref{T:AS}, Theorem~\ref{stability-Phy} and Theorem~\ref{exist-hom-smgr}
are verified. By the analysis in~\cite{AS} it follows that, taking the time step $s$ sufficiently small, the approximate operator 
$\cp^s$ in~\eqref{app-nonhom-flow-def} is well defined for all $\tau>0$, $\overline u\in \mathcal{D}_0$,
and satisfies property (i). 
Moreover, consider a sequence $\{u^{s,\varepsilon_m}\}_m$ of $s$-$\varepsilon_m$-approximate solutions 
of~\eqref{S1system}
constructed as in subsection~\ref{Ss:Rsolv-appsol},
with values in $[0,\delta^*_0] \times [p^*_{0}, p^*_1]$,
with initial data $u^{s,\varepsilon_m}(0)\in \mathcal{D}_{0}$, and such that
\begin{equation}
\label{indata-eps-conv-3}
\lim_{m\to\infty} \varepsilon_m= 0\,, \qquad\quad
\lim_{m\to\infty} \big\|u^{s,\varepsilon_m}(0)-\overline u\big\|_{\bf L^1}=0\,.
\end{equation}
Relying on Theorem~\ref{exist-hom-smgr} and on the Lipshitz continuity of the source term $g((h,p))$ 
by the definition~\eqref{app-nonhom-flow-def} it follows that
\begin{equation}
\label{nonhom-sem-lim}
\cp^s_\tau\, \overline u = \lim_{m\to\infty} u^{s,\varepsilon_m}(\tau)
\quad\forall~\tau>0\,.
\end{equation}
Given $\overline u, \overline v\in \mathcal{D}_0$, consider now
two sequences $\{u^{s,\varepsilon_m}\}_m$,
$\{v^{s,\varepsilon_m}\}_m$ of $s$-$\varepsilon_m$-approximate solutions 
of~\eqref{S1system}
 with values in $[0,\delta^*_0] \times [p^*_{0}, p^*_1]$,
with initial data $u^{s,\varepsilon_m}(0), v^{s,\varepsilon_m}(0)\in \mathcal{D}_{0}$, and such that
\begin{equation}
\label{indata-eps-conv-4}
\lim_{m\to\infty} \varepsilon_m= 0\,, \qquad\quad
\lim_{m\to\infty} \big\|u^{s,\varepsilon_m}(0)-\overline u\big\|_{\bf L^1}=
\lim_{m\to\infty} \big\|v^{s,\varepsilon_m}(0)-\overline v\big\|_{\bf L^1}=0\,.
\end{equation}
Then, relying on~\eqref{S4-PhiL1} with $u=u^{s,\varepsilon_m}$, $v=v^{s,\varepsilon_m}$, $\er=0$, 
and applying~\eqref{Phi0est2}--\eqref{Phi0est3}, we derive
\begin{equation}
\label{lip-est-appr-nonhom-1}
\begin{aligned}
\big\|u^{s,\varepsilon_m}(\tau)-v^{s,\varepsilon_m}(\tau)\big\|_{\bf L^1} &\leq C_0 \cdot 
\Phi_0\big(u^{s,\varepsilon_m}(\tau), v^{s,\varepsilon_m}(\tau)\big)
\\
&\leq C_0 \cdot e^{C_2\cdot\tau}\cdot \Phi_0\big(u^{s,\varepsilon_m}((0), v^{s,\varepsilon_m}((0)\big)+
C_0\,C_1\cdot e^{C_2\cdot\tau}\cdot \tau\cdot\varepsilon_m
\\
&\leq C_0^2\cdot W^* \cdot e^{C_2\cdot\tau}\cdot \big\|u^{s,\varepsilon_m}(0)-u^{s,\varepsilon_m}(0)\big\|_{\bf L^1}+
\frac{C_0\,C_1}{C_2}\cdot (e^{C_2\cdot\tau}-1) \cdot\varepsilon_m\,.
\end{aligned}
\end{equation}
Taking the limit as $m\to\infty$ in~\eqref{lip-est-appr-nonhom-1},
and relying on~\eqref{nonhom-sem-lim}--\eqref{indata-eps-conv-4}, we thus obtain
\begin{equation}
\label{llip-est-appr-nonhom-2}
\big\|\cp^s_\tau\, \overline u-\cp^s_\tau\, \overline v\big\|_{\bf L^1} \leq C_0^2\cdot W^* \cdot e^{C_2\cdot\tau}\cdot \big\|\overline u-\overline v\big\|_{\bf L^1}\,,
\end{equation}
which proves property (iv).
To conclude the proof, observe that properties (ii)-(iii) can be derived with entirely similar arguments to
the proofs of~\cite[Proposition 3.2, Proposition 4.1]{AG}, relying on property (iv) and on Theorem~\ref{exist-hom-smgr}.
\end{proof}
Observe that, given $\overline u\in\mathcal{D}_{0}$, if we consider a sequence $\{s_n\}_n$
decreasing to zero, the limit function $\lim_{n\to\infty} \cp^{s_n}_\tau\,\overline u$
may be not well defined.
In fact, the estimates provided by Theorem~\ref{ThmPS5.4} do not guarantee the uniqueness
of such a limit. 
However, one can show that it is possible to extract a subsequence $\{s_{n_k}\}_k$ 
so that $\{\cp^{s_{n_k}}_\tau\,\overline u(x)\}_k$ converges, for all $\tau>0$ and a.e. $x\in\R$, to a function $u(x,\tau)$
which is an entropy weak solution of~\eqref{S1system},
~\eqref{S1system-data}.
 Next, relying on 
 Theorem~\ref{stability-Phy}-(i), and
 applying a uniqueness result on quasi differential equations in metric spaces,
 we derive as in~\cite{AG} the uniqueness of solutions to the Cauchy problem~\eqref{S1system},~\eqref{S1system-data},
 In turn, this implies the convergence of the complete sequence $\{\cp^{s_n}_\tau\,\overline u\}_n$.
 and thus defines a solution operator $\cp_\tau$ for~\eqref{S1system} as stated in the next theorem
 whose proof is given in Section~\ref{S6}.
\begin{theorem}
\label{exist-nonhom-smgr-1}
Given $M_0,>0$, 
 there exist $\delta_0, \delta_p, \delta^*_0, \delta_p^*, M^*_0, L>0$
so that the conclusions of Theorem~\ref{exist-hom-smgr} hold together with
the following.
There exist a map
\begin{equation}
\label{E:semP}
\mathcal{P}: [0,+\infty)\times \mathcal{D}_{0}\to \mathcal{D}^*_0,
\qquad\qquad
(\tau,\overline u)\mapsto \mathcal{P}_\tau \overline u\,,
\end{equation}
%
(with $\mathcal{D}_{0}, \mathcal{D}^*_{0}$ domains as in~\eqref{Domain-def1}),
which enjoys the properties:
\vspace{-5pt}
\begin{itemize}
\item[(i)] $\mathcal{P}_{\tau_1}\big(\mathcal{P}_{\tau_2} \,\overline u\big)\in \mathcal{D}^*_0\qquad\forall~
\overline u\in\mathcal{D}_{0},\ \
\forall~\tau_1,\tau_2\geq 0;$
\item[(ii)] $\mathcal{P}_{0} \overline u= \overline u,\quad \mathcal{P}_{\tau_1+\tau_2} u=\mathcal{P}_{\tau_2}\big(\mathcal{P}_{\tau_1} \overline u\big)\qquad\forall~
\overline u\in\mathcal{D}_{0},\ \
\forall \tau_1,\tau_2\geq 0;$
\item[(iii)] $\big\Vert \mathcal{P}_{\tau_1} \overline u - \mathcal{P}_{\tau_2} \overline v \big\Vert_{{\bf L}^{1}}\leq L'
 \big(e^{C_4\cdot \tau_2}\cdot\lVert \overline u-\overline v\rVert_{{\bf L}^{1}}+(\tau_2-\tau_1)\big)
\qquad
\forall~\overline u,\,\overline v \in\mathcal{D}_{0},\ \ \forall~\tau_2>\tau_1>0\,,$\\
($L', C_4>0$ being the constants provided by Theorem~\ref{ThmPS5.4});
\item[(iv)] For any $\overline u\doteq (\overline h, \overline p)\in\mathcal{D}_{0}$, the map $\big(h(x,\tau),p(x,\tau)\big)\doteq 
 \mathcal{P}_\tau \overline u(x)$ provides an entropy weak solution of the Cauchy problem~\eqref{S1system},
~\eqref{S1system-data}.
\end{itemize}
\end{theorem}

\bigskip

%
\begin{remark}
Notice that, although the source term of system~\eqref{S1system} is not dissipative,
relying on the global existence result established in~\cite{AS}, we construct an evolution operator
whose image $ \mathcal{D}^*_0$ is the same for every time $\tau>0$.
\end{remark}
%

%
%
\section{Basic interaction estimates}
\label{S:basicest}

We collect in the next lemma the interaction estimates on the change of strength of the wave fronts
of an approximate solution constructed as in Section~\ref{Ss:Rsolv-appsol} whenever an interaction between two fronts takes place
outside a time step. These estimates were established in~\cite[Lemma~3]{AS} and are sharper than
the classical ones for $2\times2$ systems of conservation or balance laws. We present here also a slight refinement
of the estimate in~\cite[Lemma~3]{AS} for the case of interactions between two fronts of the second characteristic family.

%
%

\begin{lemma}[Interaction Estimates]
\label{L:stimeinter}
Consider two interacting wavefronts, with left, middle and right states $(h^{\ell},p^{\ell})$, $(h^{m},p^{m})$,
$(h^{r},p^{r})$ before the interaction. Then, assuming that the sizes of the incoming fronts and of the two outgoing waves produced by this interaction are measured in Riemann coordinates, the followings hold true:
\begin{itemize}
\item[1-1] If the incoming fronts are 
two h-waves of sizes $\sR_{h}$, $\widetilde{\sR}_{h}$, then the sizes $\widehat\sR_{h}$ and $\widehat\sR_{p}$ of the outgoing h-wave and p-wave satisfy
\be{E:11interest}
\big|\widehat\sR_{h}-\sR_{h}-\widetilde{\sR}_{h}\big|+\big|\widehat\sR_{p}\big|\leq \co(1) 
\min\big\{|p^{\ell}-1\big|,\,|p^m-1\big|\big\}
\big(|\sR_{h}|+|\widetilde{\sR}_{h}|\big)\big|\sR_{h}\widetilde{\sR}_{h}\big|\;.
\ee

\item[1-2] If the incoming fronts are an h-wave and a p-wave of sizes $ \sR_{h}, \sR_{p}$, respectively, then the sizes $\widehat\sR_{h}$ and $\widehat\sR_{p}$ of the outgoing h-wave and p-wave satisfy 
\be{E:21interest}
\big|\widehat{\sR}_{h} -\sR_{h}\big|+\big|\widehat{\sR}_{p} -\sR_{p}\big|\leq \co(1) h_{\rm max}\cdot \big|\sR_{h}\sR_{p}\big|\;,
\ee
where $ h_{\text{max}}\doteq \max\{h^l,h^m,h^r\}$.
\item[2-2] If the incoming fronts are 
two p-waves of 
sizes $\sR_{p}$, $\widetilde\sR_{p}$, then the sizes $\widehat\sR_{h}$ and $\widehat\sR_{p}$ of the outgoing h-wave and p-wave satisfy
\be{E:22interest}
\big|\widehat\sR_{h}\big|+\big|\widehat\sR_{p}-\sR_{p}-\widetilde\sR_{p}\big|\leq \co(1) h^\ell\cdot
\big|\sR_{p}\widetilde\sR_{p}\big| \big(|\sR_{p}|+|\widetilde\sR_{p}|\big) \;.
\ee

\end{itemize}
\end{lemma}
\begin{proof} 
The proofs of~\eqref{E:11interest},~\eqref{E:21interest} can be found in~\cite{AS}.
We provide here only a proof of~\eqref{E:22interest} which is a slight refinement of the corresponding
estimate established in~\cite[Lemma 3]{AS}.
Consider the functional 
\bes
\Psi(h^\ell,p^\ell,\sR_{p},\widetilde\sR_{p}):=(\widehat\sR_h, \ \widehat\sR_p-\sR_p-\widetilde\sR_p),
\ees
which is smooth in $(h^\ell,p^\ell)$ and twice continuously differentiable w.r.t. $\sR_{p},\widetilde\sR_{p}$,
with Lipschitz continuous second derivatives.
Observe that
\begin{equation}
\label{psi-est-1-1-a}
\Psi(0,p^{\ell},\sR_{p},\widetilde\sR_{p})=
\Psi(h^{\ell},p^{\ell},0,\widetilde\sR_{p})=\Psi(h^{\ell},p^{\ell},\sR_{p},0)=0\qquad\forall~h^\ell\geq 0\,,
\end{equation}
which implies
\begin{equation}
\label{psi-est-1-1-a2}
\frac{\partial\Psi}{\partial h^{\ell}}(h^{\ell},p^{\ell},0,\widetilde\sR_{p})=
\frac{\partial\Psi}{\partial h^{\ell}}(h^{\ell},p^{\ell},\sR_{p},0)=0\,.
\end{equation}
Moreover, with the same arguments of~\cite[Lemma~3]{AS} one can show that
\begin{equation}
\nonumber
\frac{\partial^2\Psi}{\partial\sR_p\partial{\tilde{\sR}}_p}(h^{\ell},p^{\ell},\sR_p=0,{\widetilde{\sR}}_p=0)=(0,0)
\qquad\forall~h^{\ell}\geq 0\,,
\end{equation}
which in turn implies
\begin{equation}
\label{psi-est-1-1-b}
\frac{\partial^3\Psi}{\partial\sR_p\partial{\tilde{\sR}}_p\partial h^{\ell}}(h^{\ell},p^{\ell},\sR_p=0,{\widetilde{\sR}}_p=0)=(0,0)
\qquad\forall~h^{\ell}\geq 0\,.
\end{equation}
Hence, using~\eqref{psi-est-1-1-a} we find
\be{psi-est-1-1-h}
\Psi(h^\ell,p^\ell,\sR_{p},\widetilde\sR_{p})
=\int_0^{h_\ell} \frac{\partial \Psi}{\partial h} (h,p^\ell,\sR_{p},\widetilde\sR_{p})\,dh\,.
\ee
On the other hand, relying on~\eqref{psi-est-1-1-a2},~\eqref{psi-est-1-1-b},
and invoking~\cite[Lemma~2.5]{Bre}, we derive 
\be{psi-est-1-1-hh}
\frac{\partial\Psi}{\partial h^{\ell}}(h^{\ell},p^{\ell},\sR_{p},\widetilde\sR_{p})
\leq \co(1) \big|\sR_{p}\widetilde\sR_{p}\big| \big(|\sR_{p}|+|\widetilde\sR_{p}|\big)
\qquad\forall~h^{\ell}\geq 0\,.
\ee
Combining together~\eqref{psi-est-1-1-h},~\eqref{psi-est-1-1-hh}, we recover the estimate~\eqref{E:22interest}.
\end{proof}
\begin{remark}
Notice that, thanks to the relations~\eqref{E:equivalncestrengths}, the interaction estimates provided by the above lemma relative to 1-1 and 2-2 interactions 
remain valid if we measure the size of incoming fronts and outgoing waves in the original coordinates 
instead that in the Riemann ones . 
Instead, for the 1-2 interaction, the $h_{max}$ factor would be missing in the right hand side of~\eqref{E:21interest} if the size of waves is measured in the original 
coordinates.
\end{remark}

Observe that, thanks to the a-priori ${\bf L}^{\infty}$-bounds established in~\cite{AS},
given $M_0,>0$ and any $\delta_0^*>0$ and $\delta_p^*\in(0,1)$, there exists $\delta_0>0$ and $\delta_p\in (0,1)$ such that
for an approximate solution $u=(h^{s,\varepsilon},p^{s,\varepsilon})$ constructed as in Subsection~\ref{Ss:Rsolv-appsol},
with initial data that satisfy~\eqref{initial-data-bounds}, one has 
\be{linf-h-bound}
\|h^{s,\varepsilon}(t)\|_{{\bf L}^{\infty}}\le \delta_0^*\qquad
\forall\,\, t>0\,.
\ee
\be{linf-pnew-bound}
\|p^{s,\varepsilon}(t)-1\|_{{\bf L}^{\infty}}\le \delta_p^*\qquad
\forall\,\, t>0\,.
\ee
Hence, relying on the estimates stated in Lemma~\ref{L:stimeinter}, it is 
 shown in~\cite{AS} that
 one can 
choose $\overline\delta>0$ in~\eqref{E:omega} 
 and $\delta_0^*>0$ in~\eqref{linf-h-bound} sufficiently small
so that
 the Glimm functional defined in~\eqref{S3G} 
 is strictly decreasing at any interaction
 occurring between time steps. Namely, 
 at any time $t>0$ where an interaction takes places, the variation $\Delta\cg(t)\doteq \cg\big(u(t+)\big)-\cg\big(u(t-)\big)$
 of the functional $\cg$ satisfies the following bounds.
\begin{enumerate}
\item[(i)] Consider an interaction between two 1-shocks 
 with sizes $\sR_{\alpha}, \sR_\beta$.
 Notice that, by the properties of the rarefaction and Hugoniot curves of system~\eqref{S1system}
 recalled in subsection~\ref{Ss:Rsolv-appsol},
 such shocks 
 must have the same sign and lie on the same side with respect to $p=1$.
 Then, we have
\be{E:decreaseGlimm1s1s}
\Delta\cg(t)\le-\frac{\omega_{\alpha,\beta}}{4}\cdot
|\sR_{\alpha}\sR_\beta|
\ee
if we assume that $\delta_0^*$ and $\dfrac{\delta_0^*}{\overline\delta}$ are sufficiently small.
\item[(ii)] At interactions between a $1$-shock of size $\sR_{\alpha}$ with a $1$-rarefaction
of size $\sR_\beta$, we have a cancellation in the waves and the functional $V$ is 
strictly decreasing. 
Then, we have
\be{E:decreaseGlimm1s1r}
\Delta\cg(t)\le-\min\big\{|\sR_{\alpha}|,\,|\sR_{\beta}|\big\}\,,
\ee
if we assume that $\delta_0^*$ is sufficiently small.
\item[(iii)] At interactions between fronts of different families or between two $2$-fronts, we have
\be{E:decreaseGlimm2r}
\Delta\cg\le-
\frac{1}{4}\cdot|\sR_{\alpha}\sR_\beta|
\ee
if we assume that $\delta_0^*$ and ${\overline\delta}$ are sufficiently small.
\end{enumerate}
Instead, the bound~\eqref{G-timestep-est1} on the variation of the functional 
$\cg$ occurring at time steps is based on the following lemma 
established in~\cite[Lemma 1]{AS}.
%
%

\begin{lemma}[Time Step Estimates]
Consider a wavefront located at a point $x$ at a time step $t_k$, with left state $(h^{\ell},p^{\ell})$ and right state $(h^r,p^r)$ before the time step. After updating the approximate solution at time~$t_k$
according with~\eqref{updated}, the solution of the Riemann problem determined by the jump at $(x,t_k)$ will consist of two waves
of the first and second characteristic families, say of sizes $\sR_{h}^+$ and $\sR_{p}^+$ respectively, measured in Riemann coordinates. Then, the followings hold true:
\begin{itemize}
\item[1]
If the front connecting $(h^{\ell},p^{\ell})$ to $(h^r,p^r)$ is of the first family with size $\sR_{h}$, 
then we have
\be{E:timestepest1}
\sR_{h}^+=\sR_{h}+\co(1)\Delta t\cdot |p^{\ell}-1|\cdot|\sR_{h}|\,,
\ee
\be{E:timestepest2}
\sR_{p}^+=\co(1)\Delta t\cdot|p^{\ell}-1|\cdot|\sR_{h}|\,.
\ee
\item[2]
If the front connecting $(h^{\ell},p^{\ell})$ to $(h^r,p^r)$ is of the second family with size $\sR_{p}$, then we have
\be{E:timestepest3}
\sR_{h}^+=\co(1)\Delta t \cdot h^{\ell} \cdot|\sR_{p}|\,,
\ee
\be{E:timestepest4}
\sR_{p}^+=\sR_{p}+\co(1)\Delta t \cdot h^{\ell}\cdot|\sR_{p}|\,.
\ee
\end{itemize}
\end{lemma}

%
%
\section{\texorpdfstring{${\bf L^{1}}$}{L1}-stability estimates - Proof of Theorem~\ref{stability-Phy}}
\label{S4old}

Consider two 
$s$-$\varepsilon$-approximate solutions $u, v$ of the non-homogeneous system~\eqref{S1system} constructed as in Subsection~\ref{Ss:Rsolv-appsol},
with
initial data $(\overline h_u,\overline p_u)\doteq u(\,\cdot,0), (\overline h_v,\overline p_v)\doteq v(\,\cdot,0)$ satisfying~\eqref{initial-data-bounds}.
The heart of the matter to establish Theorem~\ref{stability-Phy} is to control the change in time of the functional~$\Phi_z$ defined in~\eqref{UpsilonNew}--\eqref{S4A2}, evaluated along $(u(t), v(t))$.
This is accomplished in the following subsections by first
analyzing the variation of $\Phi_z(u(t),v(t))$ when $\er\equiv 0$.
Namely, assuming that $\delta_0^*$ in~\eqref{linf-h-bound}
is sufficiently small,
we will analyze the change of $\Phi_0(u(t),v(t))$ at three different classes of times:
\begin{itemize}
\item[\S~\ref{Ss:interactionTimes}:] at times where two fronts of $u$ or $v$ interact, we show that $t\mapsto\Phi_0(u(t),v(t))$ does not increase;
\item[\S~\ref{Ss:nointeractiontimesN}:] 
at times between interactions, the function $t\mapsto\Phi_0(u(t),v(t))$ is Lipschitz continuous and we prove that
there holds
\begin{equation}
\label{est-phi0-decr}
\ddn{t}{}\Phi_0(u(t),v(t))\leq C_1\eps \;,
\end{equation}
where $C_1>0$ constant independent of $s$-$\varepsilon$.
\item[\S~\ref{Ss:timestep}:] at time steps $t_{k}$, we prove that
\bas
\Phi_0(u(t_k+),v(t_k+))&\le  (1+C_2\, \Delta t) \,\Phi_0(u(t_k-),v(t_k-))\;.
\eas
where $C_2>0$ constant independent again of $s$-$\varepsilon$.
\end{itemize}
The analysis in each class of times is performed in the Subsections \S~\ref{Ss:interactionTimes}--\S~\ref{Ss:timestep}. 
 Integrating~\eqref{est-phi0-decr} between two interaction times, and combining these three results of~\S~\ref{Ss:interactionTimes}--\S~\ref{Ss:timestep}, we 
 thus establish Theorem~\ref{stability-Phy}-(ii).
Next, we shall consider two $\varepsilon$-approximate solutions $u$ and $v$ of the homogeneous system~\eqref{S1system-hom} with
initial data satisfying~\eqref{initial-data-bounds},
 and we will analyze the variation of $\Phi_\er(u(t),v(t))$
 when $z\ne 0$ and~\eqref{z-bound-1} holds:
\begin{itemize}
\item[\S~\ref{S5.4}:] performing a similar analysis as in~\S~\ref{Ss:interactionTimes}-\S~\ref{Ss:nointeractiontimesN},
we show that
\bas
\Phi_\er(u(t_2),v(t_2))-\Phi_\er(u(t_1),v(t_1))&\le   C_1 \cdot (t_2-t_1) [\eps+\sigma]\;,
\eas
\end{itemize}
for all $t_2\ge t_1\ge 0$, where $C_1>0$ constant independent of $s$-$\varepsilon$, and this establishes Theorem~\ref{stability-Phy}-(i).

%
%
\subsection{Analysis at interaction times}
\label{Ss:interactionTimes}
In this section, we consider an interaction between waves of the approximate solution $v$ or $u$
occurring at time $t=\tau$ and show that the functional $\Phi_0(u(t),v(t))$, given in~\eqref{UpsilonNew}, does not increase across interaction for appropriate constants in the weights $W_i$, i.e. we prove that
\begin{equation}
\label{phiest0}
\Phi_0(u(\tau{+}),v(\tau{+}))\le \Phi_0(u(\tau{-}),v(\tau{-})).
\end{equation}
In the following lemma, we provide a condition under which the constants in the weights $W_i$ need to be controlled by the coefficient $\kappa_{\cg}$ of the Glimm functionals of $u$ and $v$. In preparation for this, we need the notation of the change $\Delta$ across an interaction accuring at $t=\tau$.
\begin{align*}
&\Delta \cg(\tau) :=\cg(u(\tau{+}))-\cg(u(\tau{-}))+\cg(v(\tau{+}))-\cg(v(\tau{-})) \ ,\\
&\Delta W_i(\tau):=W_{i}(\tau{+},x)- W_{i}( \tau{-}, x)\ ,\\
&\Delta \ca_{i,j}(\tau):= \ca_{i,j}(\tau{+},x)-  \ca_{i,j}( \tau{-}, x) \ ,\qquad   i,\,j=1,\,2\ .
\end{align*}
The next lemma states a sufficient condition that implies~\eqref{phiest0}

\begin{lemma}\label{lemma4.1}
Let $t=\tau$ be an interaction time for either $v$ or $u$. If
\be{E:constraint_interNew}
 \kappa_{\cg}\ge
 \kappa_{i\ca1} \frac{ \Delta \ca_{i,1}}{|\Delta\cg|}+\kappa_{i\ca2}\frac{\Delta \ca_{i,2}}{|\Delta\cg|}
 \qquad\text{for $i=1,2$}
 \ee
then~\eqref{phiest0} holds true. 
\end{lemma}
\begin{proof} First, we note that across an interaction the Glimm functional is not increasing, hence
$$|\Delta\cg(\tau) |=\cg(u(\tau-))+\cg(v(\tau-))-\cg(u(\tau+))-\cg(v(\tau+))>0.$$
Then, using the $\Delta$ notation, we have the identity 
\begin{align}
W_i(\tau+,x) \cdot & e^{\kappa_{\cg} \big[\cg(u(\tau+))+\cg(v(\tau+))\big]}
- W_i(\tau-,x) \cdot e^{\kappa_{\cg} \big[\cg(u(\tau-))+\cg(v(\tau-))\big]}\nonumber\\
=& \left(\Delta W_i (\tau)e^{-\kappa_{\cg} |\Delta\cg(\tau)|}+W_i(\tau-,x) (e^{-\kappa_{\cg} |\Delta\cg(\tau)|}-1) \right)\cdot e^{\kappa_{\cg} \big[\cg(u(\tau-))+\cg(v(\tau-))\big]}.
\end{align}
Using that 
\[
\Delta W_i(\tau)=\left( e^{ \kappa_{i\ca1}\Delta \ca_{i,1}+\kappa_{i\ca2}\Delta \ca_{i,2}}-1\right) W_i(\tau-)
\]
we get
\begin{align}
W_i(\tau+,x) \cdot & e^{\kappa_{\cg} \big[\cg(u(\tau+))+\cg(v(\tau+))\big]}
- W_i(\tau-,x) \cdot e^{\kappa_{\cg} \big[\cg(u(\tau-))+\cg(v(\tau-))\big]}
\nonumber\\
\le & \left( e^{-\kappa_{\cg}|\Delta\cg|} e^{ \kappa_{i\ca1}\Delta \ca_{i,1}+\kappa_{i\ca2}\Delta \ca_{i,2}}-1 \right) \cdot  W_i(\tau-,x) e^{\kappa_{\cg} \big[\cg(u(\tau-))+\cg(v(\tau-))\big]}\le0
\end{align}
under  condition~\eqref{E:constraint_interNew}.%
This implies~\eqref{phiest0} immediately, because the map 
\bes
\R^{+}\ni t\mapsto |\eta_{i}(\cdot,t)|\in L^{1}(\R)
\ees
is continuous. The proof is complete.
\end{proof}

The aim is to show that there exists $\kappa_{\cg}>0$ large enough, for sufficiently small $\delta_0>0$, such that~\eqref{E:constraint_interNew} holds true at all interaction times $\tau$. Our strategy is to examine all cases of interactions and prove that in each case 
\ba\label{E:rel}
\Delta \ca_{i,j} \le a |\Delta\cg|
 \qquad
 \text{for $i,j=1,2$}
 \ea
 where the factor $a>0$ depends on $\delta_0^*$, $M_0^*$, $\bar\delta$ and the coefficients $\kappa_{i\ca j}$, for $i,j=1,2$. From this, the conclusion from the analysis in the following subsections is that~\eqref{E:constraint_interNew} holds true as long as
\begin{equation}\label{S4.1conditionkappa}\kappa_{\cg}> 2a\max_{i,j=1,2} \kappa_{i\ca j} \ .
\end{equation}

From here and on, we consider an interaction between waves of the approximate solution $v$ and omit the analysis relative to interactions of waves of $u$ because it is entirely similar. We devote the rest of the section to the proof of~\eqref{E:constraint_interNew} and hence~\eqref{phiest0}.
Therefore, we 
analyze separately in each subsection the different type of interactions occurring between two wave-fronts of $v$ using that $\|h\|_\infty$ and $\delta_0^*$ are sufficiently small according to Theorem~\ref{T:AS}. We recall that by~\eqref{linf-h-bound}
 the ${\bf L^\infty}$-bound on $h$-component of the approximate solutions 
takes value in $[0,\delta_0^*]$, and that the $p$-component of the approximate solutions takes values in the interval 
$[p_0^*, p_1^*]$.
Throughout the section, we denote by $\sR_{\alpha}$ and $\sR_{\beta}$ the strengths
(in Riemann coordinates) of the incoming waves of $v$ before their interaction takes place at $t=\tau$, 
with $\sR_{\alpha}$ located on the left of $\sR_{\beta}$. 
We also let $p_{\alpha}^{\ell}$, $p_{\beta}^{\ell}$ denote the $p$-components (in the original coordinates) of their left states.

For the convenience of the reader, we note that in the following subsection we use that
\begin{enumerate}
\item[$\circ$] for $1$-waves, it holds: $k_\alpha=1$, $|\sR_{\alpha}|<\mu\, \delta_0^*< M_0^* $,\ $0<h_\alpha<\delta_0^*$\ and $p_0^*<1-\delta_p^*<p_\alpha<1+\delta_p^*<p_1^*$
\item[$\circ$] for $2$-waves, it holds:  $k_\alpha=2$, $|\sR_{\alpha}|< M_0^* $, $0<h_\alpha<\delta_0^*$ and $p_0^*<1-\delta_p^*<p_\alpha<1+\delta_p^*<p_1^*$\;.
\end{enumerate}

%
 \subsubsection{Case of $1-1$ interaction without cancellation}
\label{Sss:Interaction1nC}

Assume that $\sR_{\alpha}$, $\sR_{\beta}$ are the sizes of two interacting $1$-waves of $v$ that have the same sign.
Thus $\sR_{\alpha}$, $\sR_{\beta}$ are two shocks of the first family
lying on the same side of $\{p=1\}$.
We denote by $\sR_{h}'$, $\sR_{p}'$, the sizes (in Riemann coordinates) of the outgoing $1$-wave and $2$-wave, respectively, after the interaction. Notice that the left states of $\sR_{h}'$ and $\sR_{\alpha}$ are the same.
Therefore, if we denote by ${p'}^{\ell}$ the $p$-component (in original coordinates) of
the left state of $\sR_{h}'$, one has ${p'}^{\ell}=p_{\alpha}^{\ell}$.

By~\eqref{E:decreaseGlimm1s1s} the Glimm functional is decreasing across this type of interaction
 with the bound
$$
\cg(v(\tau{+}))\leq \cg(v(\tau{-}))-\frac{\omega_{\alpha,\beta}}{4}|\sR_{\alpha}\sR_\beta|, \qquad \cg(u(\tau{+}))=\cg(u(\tau{-})),
$$
hence,
\be{E:orerggrgeege}
\Delta\cg=-\frac{\omega_{\alpha,\beta}}{4}|\sR_{\alpha}\sR_\beta|<0\;.
\ee

Next, we note that at points $x$ where neither $u(\tau)$ nor $v(\tau)$ admits a jump,
one has $\eta_i(x,\tau-)=\eta_i(x,\tau+)$, $i=1,2$.
Then, by the definition of $\ca_{1,1}$ in~\eqref{S4A1} and relying on the interaction estimate~\eqref{E:11interest}, we get that at such points there holds
\begin{align*}
\ca_{1,1}(x,\tau{+})- \ca_{1,1}(x,\tau{-}) & \leq \Big|\PHI{{p'}^{\ell}}|\sR_{h}'| -\left( \PHI{p_{\alpha}^{\ell}}|\sR_{\alpha}| +\PHI{p_{\beta}^{\ell}}|\sR_\beta| \right)\Big|
\end{align*}
Considering that in this case ${p'}^{\ell}=p_{\alpha}^{\ell}$ the right hand side can be estimated by
\[
\Big|\PHI{{p'}^{\ell}}|\sR_{h}'| - \PHI{p_{\alpha}^{\ell}}|\sR_{\alpha}| -\PHI{p_{\beta}^{\ell}}|\sR_\beta| \Big|
 \leq \PHI{p_{\alpha}^{\ell}}\big|\sR_{h}'-\sR_{\alpha}-\sR_\beta\big| +\left| p_{\beta}^{\ell} -  p_{\alpha}^{\ell}\right||\sR_\beta| \ . 
\]
By~\eqref{E:equivalncestrengths} and recalling definition~\eqref{E:omega} we observe that
\be{E:omega-equivalence}
\min\big\{|p^{\ell}_\alpha-1\big|,\,|p^\ell_\beta-1\big|\big\}\leq\frac{\mu}{\overline\delta}\cdot\omega_{\alpha,\beta}
\ee
so that applying the interaction estimate~\eqref{E:11interest} we arrive at
\begin{align}
\label{E:orerggrgeege223prec}
\ca_{1,1}(x,\tau{+})- \ca_{1,1}(x,\tau{-}) \le \co(1)p_1^*\frac{ \,\mu}{\overline\delta}\cdot
\omega_{\alpha,\beta} \delta_0^*|\sR_{\alpha}\sR_{\beta}|+|p_{\alpha}^{\ell}-p_{\beta}^{\ell}| |\sR_{\beta}|.
\end{align}
Next, observe that $p_\beta^\ell=p_\alpha^r=\SC{1}{\sC_\alpha}{(h_\alpha^\ell,p_\alpha^\ell)}$ with $|\sC_\alpha|\leq \mu |\sR_\alpha|$  by~\eqref{E:equivalncestrengths}. To estimate the last term in~\eqref{E:orerggrgeege223prec}, we employ the explicit expression of $\SC{1}{\cdot}{}$ given in~\eqref{E:p1shock} and get
\be{pab-est1}
\big|p_{\alpha}^{\ell}-p_{\alpha}^{r}\big|=\big|p_{\alpha}^{\ell}-p_{\beta}^{\ell}\big|=\co(1)\frac{\mu}{\overline\delta}\cdot\omega_{\alpha,\beta}|\sR_{\alpha}|\,.
\ee
since $|p_{\alpha}^{\ell}-1|<\delta_p^*$ and $h-s_{1}(h,p_{\ell})>\frac{1}{2}$.

Substituting into~\eqref{E:orerggrgeege223prec}, we get the estimate
\begin{align}\label{E:orerggrgeege223a}
\ca_{1,1}(x,\tau{+})- \ca_{1,1}(x,\tau{-}) &\leq
\co(1) \cdot\omega_{\alpha,\beta}|\sR_{\alpha}\sR_{\beta}|\,.
\end{align}
Now, the change of $\ca_{1,2}$ given in~\eqref{S4A1} across such an interaction at $\tau$ is estimated
\begin{align}
\ca_{1,2}(x,\tau{+})- \ca_{1,2}(x,\tau{-}) & \leq |\rho'_{p}| 
\label{E:orerggrgeege223b}
\leq\co(1)\frac{\,\mu}{\overline\delta}\cdot
\omega_{\alpha,\beta}\delta_0^*|\sR_{\alpha}\sR_{\beta}| .
\end{align}
using again~\eqref{E:11interest}. 
By the definition of $\ca_{2,1}$ and of $\ca_{2,2}$ in~\eqref{S4A2} and estimates~\eqref{E:11interest},~\eqref{E:omega-equivalence}, we also bound the change
\ba
\label{E:irehgreugrg1}
\left|\ca_{2,1}(x,\tau{+})- \ca_{2,1}(x,\tau{-})\right|+ & \left|\ca_{2,2}(x,\tau{+})- \ca_{2,2}(x,\tau{-})\right|
\leq   |\sR_{p}'|+|\sR_{h}'-\sR_{\alpha}-\sR_{\beta}| 
\notag
\\
&\stackrel{\eqref{E:11interest}}{\leq} 
 \co(1)\cdot\frac{\,\mu}{\overline\delta}\cdot
\omega_{\alpha,\beta}\delta_0^*|\sR_{\alpha}\sR_{\beta}|\,.
\ea
Estimates~\eqref{E:orerggrgeege},~\eqref{E:orerggrgeege223a},~\eqref{E:orerggrgeege223b} and~\eqref{E:irehgreugrg1} with~\eqref{S4Wi} directly prove that \eqref{E:rel} holds with 
\[a\ge\co(1) (1+\frac{\mu}{\overline\delta} \delta_0^*)\;.\]

%
%
\subsubsection{Case of $1-1$ interaction with cancellation}
\label{Sss:Interaction1C}
Here, we consider an interaction between two incoming $1$-waves of the solution $v$ at time $t=\tau$ but with strengths $\sR_{\alpha}$, $\sR_{\beta}$ of opposite sign. By~\eqref{E:decreaseGlimm1s1r} the Glimm functional is decreasing across this type of interaction
 with the bound
\[
\cg(v(\tau{+}))\leq \cg(v(\tau{-}))-\min\{|\sR_{\alpha}|,|\sR_{\beta}|\}, \qquad \cg(u(\tau{+}))=\cg(u(\tau{-})).
\]
Assuming that the strengths satisfy $|\sR_{\alpha}|<|\sR_{\beta}|$, we have
\be{E:orerggrgeege30394}
\Delta\cg(\tau)\leq-|\sR_{\alpha}|\;.
\ee

We proceed following the same arguments and notations of \S~\ref{Sss:Interaction1nC}. Firstly, by~\eqref{S4A1} and the interaction estimate~\eqref{E:11interest}, we deduce that:
\begin{align}
 \left|\ca_{1,1}(x,\tau{+})\right.&\left.- \ca_{1,1}(x,\tau{-})\right|+ |\ca_{1,2}(x,\tau{+})- \ca_{1,2}(x,\tau{-})|
\notag\\
\notag&\leq |\rho'_{p}|+ \Big|\PHI{{p'}^{\ell}}|\sR_{h}'| - \PHI{p_{\alpha}^{\ell}}|\sR_{\alpha}| -\PHI{p_{\beta}^{\ell}}|\sR_\beta| \Big|
\notag
\\
&=\co(1) (p_1^*)^2\,\delta_0^*\cdot |\sR_{\alpha}\sR_{\beta}|\notag\\
&=\co(1) (p_1^*)^2\,(\delta_0^*)^2\cdot |\sR_{\alpha}|\,.
\label{E:orerggrgeege0293-2}
\end{align}
Secondly, by~\eqref{S4A2} we deduce that:
\begin{align}
\notag
 \left|\ca_{2,1}(x,\tau{+})-\right.&\left. \ca_{2,1}(x,\tau{-})\right|+ \left|\ca_{2,2}(x,\tau{+})- \ca_{2,2}(x,\tau{-})\right|
\\
\notag&\leq |\sR'_{p}|+|\sR_{h}'| -|\sR_{\alpha}| -|\sR_\beta|  \notag
\\
& \leq |\sR'_{p}|+|\sR_{h}' -\sR_{\alpha} -\sR_\beta| \notag
\\
&=\co(1) p_1^*\delta_0^*\cdot |\sR_{\alpha}\sR_{\beta}|\notag\\
&=\co(1)   p_1^*(\delta_0^*)^2\cdot |\sR_{\alpha}|\,.
\label{E:orerggrgeege2029}
\end{align}
Estimates~\eqref{E:orerggrgeege30394},~\eqref{E:orerggrgeege0293-2} and~\eqref{E:orerggrgeege2029} directly prove that \eqref{E:rel} holds, actually with $a$ small satisfying
with 
\[a\ge\co(1)  p_1^*(\delta_0^*)^2\;.\]

%
%
\subsubsection{Case of 2-2 interaction}
\label{Sss:Interaction22}
Now, we assume that $\sR_{\alpha}$, $\sR_{\beta}$ are the strengths of two interacting $2$-waves of $v$.
 By~\eqref{E:decreaseGlimm2r}, the Glimm functional is decreasing across such an interaction
 with the bound:
\[
\cg(v(\tau{+}))\leq \cg(v(\tau{-}))-\frac{1}{4}|\sR_{\alpha}\sR_\beta|, \qquad \cg(u(\tau{+}))=\cg(u(\tau{-})),
\]
hence, 
\be{E:orerggrgeege2029-1}
\Delta\cg(\tau)\leq -\frac{1}{4}|\sR_{\alpha}\sR_\beta|\;.
\ee
By the definition of the functionals $\ca_{i,j}$, given at~\eqref{S4A1}--\eqref{S4A2} and relying on the interaction estimates~\eqref{E:22interest}, it holds
\begin{align}\label{E:orerggrgeege2029-2a}
 \left|\ca_{1,1}(x,\tau{+})- \ca_{1,1}(x,\tau{-})\right| &\le p_1^*\cdot |\sR'_{h}|\notag\\
 &\le \co(1)\delta_0^* M_0^*\cdot |\sR_{\alpha}\sR_{\beta}|
 \end{align}
 and
 \begin{align}
 \left| \ca_{1,2}(x,\tau{+})- \ca_{1,2}(x,\tau{-}) \right| &\le 
|\sR_{p}' -\sR_{\alpha} -\sR_\beta|
\notag
\\
& \le \co(1)\delta_0^*\,M_0^*\cdot |\sR_{\alpha}\sR_{\beta}|
\label{E:orerggrgeege2029-2}
\end{align}
while
\begin{align}
\left|\ca_{2,1}(x,\tau{+})- \ca_{2,1}(x,\tau{-})\right|+& \left|\ca_{2,2}(x,\tau{+})- \ca_{2,2}(x,\tau{-})\right|\notag\\
&\leq  \big( |\sR'_{h}|+|\sR_{p}' -\sR_{\alpha} -\sR_\beta|\big)\qquad\qquad
\notag
\\
& \leq\co(1)\delta_0^*\,M_0^*\cdot |\sR_{\alpha}\sR_{\beta}|\,.
\label{E:orerggrgeege2029-2b}
\end{align}
Estimates~\eqref{E:orerggrgeege2029-1} and~\eqref{E:orerggrgeege2029-2a}--~\eqref{E:orerggrgeege2029-2b} directly prove that \eqref{E:rel} holds, actually with $a$ small being
\[
a\ge\co(1)\delta_0^* M_0^* (1+p_1^*)\;.
\]

%
%
\subsubsection{Case of 2-1 interaction}
\label{Sss:Interaction21sameRegion}
Here, we consider the case that two incoming waves of different families interact. The strength $\sR_{\alpha}$ corresponds to the 2-wave located at $x_\alpha$ and the strength $\sR_{\beta}$ is the 1-wave located at $x_\beta$. We observe that the left state of the outgoing 1-wave $\rho'_h$ is the same with the left state of the incoming 2-wave $\sR_\alpha$. Hence, if we denote by ${p'}^{\ell}$
the $p$-component of
the left state of $\sR_{h}'$, then ${p'}^{\ell}=p_{\alpha}^{\ell}$. Since the $p$-component of the solution $v$ attains values in $[p_0^*,p_1^*]$, with $p_0^*<1<p_1^*$, there are the following cases: (a) The $p$-components of the left states $p_{\alpha}^{\ell}$, $p_{\beta}^{\ell}$ of the incoming waves belong to the same interval 
$[p_0^*,1]$ or $[1, p_1^*]$, i.e. they both lie in the same region $\{p>1\}$ or $\{p< 1\}$ of the $h-p$ plane. (b) The $p$-components of the left states $p_{\alpha}^{\ell}$, $p_{\beta}^{\ell}$ of the incoming waves belong to different regions. 

$\bullet$ Case (a) with states not crossing $\{p=1\}$. In this case, we first note that the decrease of the Glimm functional~\eqref{E:decreaseGlimm2r},
is the same as of the $2-2$ interaction in the previous subsection, see~\eqref{E:orerggrgeege2029-1}. Next, relying on the interaction estimates~\eqref{E:21interest}, we obtain again similar estimates to~\eqref{E:orerggrgeege2029-2}--\eqref{E:orerggrgeege2029-2b} on the variation of $\ca_{i,j}$ around $\tau$ and show that
\[
\Delta  \ca_{i,j}
\leq
 \co(1)\delta_0^*\, \cdot |\sR_{\alpha}\sR_{\beta}|
\;.
\]
Hence, the same conclusion~\eqref{E:rel} holds here as well, actually with $a$ small, i.e. $a=\co(1)\delta_0^*$.


$\bullet$ Case (b) with states crossing $\{p=1\}$. In this case, the left states of the incoming waves do not lie in the same region $\{p<1\}$ or $\{p>1\}$. In other words, we assume that the left state of $\sR_\beta$ lies on $\{p<1\}$ and the left state of $\rho'_h$ lies on on $\{p>1\}$ or viceversa.
Then, one can deal with the variation of the functionals $\ca_{1,2}(x,t)$, $\ca_{2,1}(x,t)$, $\ca_{2,2}(x,t)$, $\cg(u(t))$ and $\cg(v(t))$ across the interaction time $t=\tau$ precisely as in \S~\ref{Sss:Interaction22}, so that there holds~\eqref{E:orerggrgeege2029-1} and
\be{E:orerggrgeege2029-4b}
\begin{split}
&\ca_{1,2}(x,\tau{+})- \ca_{1,2}(x,\tau{-})\leq \co(1)\delta_0^*\, \cdot |\sR_{\alpha}\sR_{\beta}|
\\
& \left|\ca_{2,1}(x,\tau{+})- \ca_{2,1}(x,\tau{-})\right|+ \left|\ca_{2,2}(x,\tau{+})- \ca_{2,2}(x,\tau{-})\right|
\leq \co(1) \delta_0^*\, \cdot |\sR_{\alpha}\sR_{\beta}|\,.
\end{split}
\ee
Instead, the variation of the functional $\ca_{1,1}(x,t)$ needs a different treatment.
Due to the change of region with respect to $\{p=1\}$
of the left state of the incoming and outgoing 1-wave, either the 1-wave located at $x_\beta$ is moving towards 
a 1-wave $\eta_1(x,\tau-)$, $x\neq x_\beta$, before the interaction and it is moving away from $\eta_1(x,\tau+)$ after the interaction, or viceversa.
This behaviour is precisely determined by the 
fact that the first characteristic family is not genuinely nonlinear
since we have $D\lambda_1\, {\bf r}_1<0$ on $\{p<1\}$ and $D\lambda_1\, {\bf r}_1>0$ on $\{p>1\}$.
Hence, to estimate the variation of the functional $\ca_{1,1}$, we proceed by studying two subcases: 
To fix the ideas, let a point $x<x_\beta$ where $\eta_1(x,\tau-)=\eta_1(x,\tau+)>0$.
If the 1-wave at $x_\beta$ is approaching $\eta_1(x,\tau\pm)$ before the interaction and it is moving away from~$\eta_1(x,\tau\pm)$ after
the interaction, 
i.e. if $p_{\beta}^{\ell}< 1< p_\alpha^{\ell}$,
then one has
\be{S4.A11<0}
\ca_{1,1}(x,\tau{+})-\ca_{1,1}(x,\tau{-})=-\PHI{p_{\beta}^{\ell}}|\sR_{\beta}| <0
\ee
and the functional $\ca_{1,1}$ decreases.
Instead, if the 1-wave at $x_{\beta}$ approaches $\eta_1(x,\tau\pm)$ after the interaction, but it was not approaching $\eta_1(x,\tau\pm)$ before the interaction, i.e. if $p_{\alpha}^{\ell}< 1< p_\beta^{\ell}$,
then, relying on~\eqref{E:21interest}, one has 
\ba\label{E:orerggrgeege2029-4f}
\ca_{1,1}(x,\tau{+})-\ca_{1,1}(x,\tau{-})&= |p_{\beta}^{\ell}-1||\sR_{h}'|\notag\\
&= |p_{\beta}^{\ell}-1|\left(|\sR_\beta|+\co(1)\delta_0^*\cdot |\sR_{\alpha}\sR_{\beta}|\right).
\ea
Since 
 $|p_{\beta}^{\ell}-1|\leq |p_{\beta}^{\ell}-p_{\alpha}^{\ell}|$, and because
$$
(h_\beta^\ell,p_\beta^\ell)=\mathbf{S}_2\big(\sR_\alpha; (h_\alpha^\ell,p_\alpha^\ell)\big),
$$
recalling~\eqref{Sdue},~\eqref{E:equivalncestrengths}, we get $|p_{\beta}^{\ell}-1|\leq\mu|\sR_{\alpha}|$. We thus conclude 
from~\eqref{E:orerggrgeege2029-4f} that 
\be{E:orerggrgeege2029-4X}
\ca_{1,1}(x,\tau{+})- \ca_{1,1}(x,\tau{-})\leq \mu\big(1+\co(1)\delta_0^*\big)|\sR_{\alpha}\sR_{\beta}|\;.
\ee

\begin{figure}[htbp]
{\centering \scalebox{1}{\input{2-1interactiontalk.tex} } \par}
\caption{{\bf On the left:} The $2-1$ interaction in \S~\ref{Sss:Interaction21sameRegion} in Case (a). {\bf On the right:} The $2-1$ interaction in \S~\ref{Sss:Interaction21sameRegion}\label{S4:fig1} in Case (b).}
\end{figure}
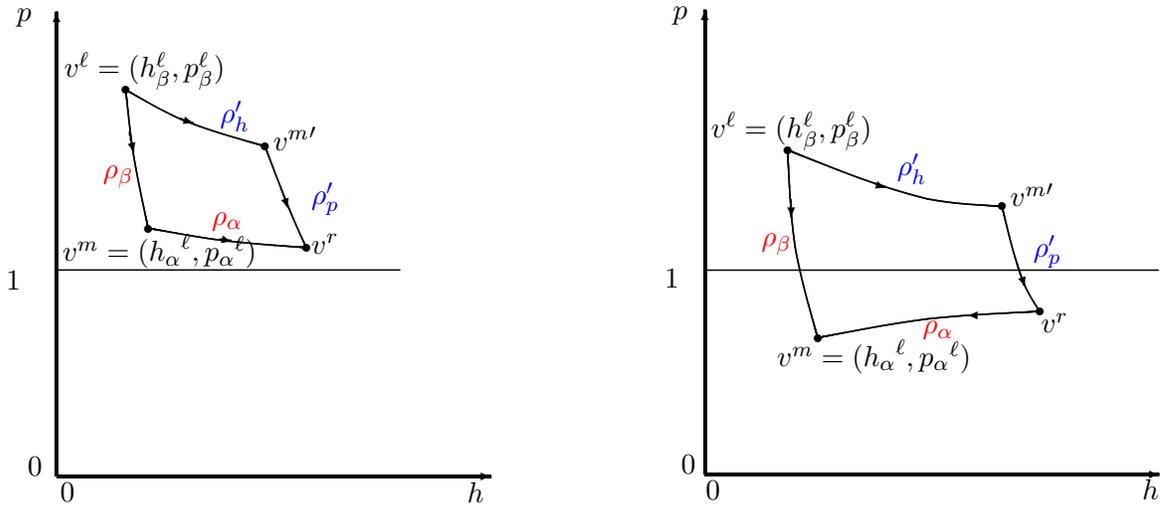
Estimates~\eqref{E:orerggrgeege2029-1},~\eqref{E:orerggrgeege2029-4b},~\eqref{S4.A11<0} and~\eqref{E:orerggrgeege2029-4X} directly prove that \eqref{E:rel} holds with
\[
a\ge\co(1)\left[(\delta_0^*+1)\mu+\delta_0^*\right]
\]

Combining all results of \S~\ref{Sss:Interaction1nC}--\S~\ref{Sss:Interaction21sameRegion}, estimate~\eqref{E:rel} holds for
\[
a=\co(1)\max\left\{ (1+\frac{\mu}{\overline\delta} \delta_0^*) ,  p_1^*(\delta_0^*)^2, \delta_0^* M_0^* (1+p_1^*),\left[(\delta_0^*+1)\mu+\delta_0^*\right]\right\}\;.
\]
By Lemma~\ref{lemma4.1}, this yields immediately that under the restriction~\eqref{S4.1conditionkappa} on the size of $\kappa_{\cg}$, the functional $\Phi_0$ is non-increasing at interaction times.

%


\subsection{Analysis at times between interactions}
\label{Ss:nointeractiontimesN}

When there is no interaction at time $t$, a short direct computation 
yields
\begin{equation}
\label{E:der-phi0-00N}
\ddn{t}{}\Phi_0(u(t),v(t))= \left[\sum_{\alpha\in \cj} \sum_{i=1}^{2} \big\{ \left|\eta_{i}\left(x_{\alpha}{-},t\right)\right| W_{i}(x_{\alpha}{-},t)-\left|\eta_{i}\left(x_{\alpha}{+},t\right)\right| W_{i}(x_{\alpha}{+},t)\big\} \dot x_{\alpha}\right]
 \cdot e^{\kappa_{\cg} \big[\cg(u(t))+\cg(v(t))\big]}
 \;,
\end{equation}
where $\cj$ 
denotes the set of indexes $\alpha$ associated to the jumps in $u(t)$ and $v(t)$.
As in \S~\ref{S2}, we denote $x_\alpha$ the location of the
the $\alpha$-jump, and let $\sC_\alpha$
denote its size, measured in original coordinate
(cfr.~\S~\ref{Ss:strengthnotation}).
Since $u(t), v(t)\in L^{1}(\R)$, we may assume that the piecewise constant maps $u(x,t), v(x,t)$
vanish when $x\to \pm \infty$, which implies that also $\eta_i(x,t)=0$ when $x\to \pm \infty$.
Hence, following~\cite[(8.18)-(8.19)]{Bre} 
we rewrite~\eqref{E:der-phi0-00N} in the equivalent form
\begin{equation}
\label{der-phi-sum}
\ddn{t}{}\Phi_0(u(t),v(t))= \left[\sum_{\alpha\in \cj} \sum_{i=1}^{2} E_{\alpha,i}\right]
 \cdot e^{\kappa_{\cg} \big[\cg(u(t))+\cg(v(t))\big]}
\end{equation}
where
\be{E:error}
E_{\alpha,i} \doteq W_{i}^{\alpha,r}|\eta_{i}^{\alpha,r}|(\lambda_{i}^{\alpha,r}-\dot x_{\alpha})- W_{i}^{\alpha,\ell}|\eta_{i}^{\alpha,\ell}|(\lambda_{i}^{\alpha,\ell}-\dot x_{\alpha})
\qquad
\alpha\in \cj
\ .
\ee
Here, and throughout the following, we adopt the notation 
\begin{subequations}
\label{GE:definitionslralpha}
\begin{equation}
\label{W-eta-lambda-rl-def}
\begin{gathered}
W_{i}^{\alpha,\ell}\doteq W_i(x_{\alpha}{-},t),\qquad
W_{i}^{\alpha,r}\doteq W_i(x_{\alpha}{+},t),\qquad
\eta_{i}^{\alpha,\ell}\doteq \eta_i(x_{\alpha}{-},t),\qquad
\eta_{i}^{\alpha,r}\doteq \eta_i(x_{\alpha}{+},t).
\end{gathered}
\end{equation}
Similarly, we set
\be{S5.2.1 intro}
u^\alpha\doteq (u^\alpha_{1},u^\alpha_{2})\doteq u(x_{\alpha},t)\ ,\qquad
v^{\alpha,\ell}\doteq (v^{\alpha,\ell}_{1},v^{\alpha,\ell}_{2})\doteq v(x_{\alpha}-,t) \ ,
\qquad
v^{\alpha,r}\doteq (v^{\alpha,r}_{1},v^{\alpha,r}_{2})\doteq v(x_{\alpha}+,t) \ .
\ee
Then, recalling~\eqref{E:vu} with $\er=0$, one has
\begin{equation}
\label{S5.2.1 intro-a}
\omega^{\alpha,\ell/r}\doteq \mathbf{S}_1\big(\eta_1^{\alpha,\ell / r}; u^\alpha\big),\qquad
v^{\alpha,\ell/r}= \mathbf{S}_2\big(\eta_2^{\alpha,\ell / r}; \omega^{\alpha,\ell/r}\big),
\end{equation}
and we set
\begin{equation}
\label{S5.2.1_introd}
 \lambda_{1}^{\alpha,\ell / r}\doteq \lambda_{1}\big(u^\alpha,\omega^{\alpha,\ell / r}\big),
 \qquad
 \lambda_{2}^{\alpha,\ell / r}\doteq\lambda_{2}\big(\omega^{\alpha,\ell / r}, v^{\alpha,\ell / r}\big),
\end{equation}
\end{subequations}
where $\lambda_i(u,w)$ denotes the speed of the $i$-wave $\eta_i$ connecting the left state $u$ with the right state $w$.
This means that $\lambda_1(u^\alpha,\omega^{\alpha,\ell / r})$
is the 
Rankine-Hugoniot speed of the $1$-shock connecting the left state $u^\alpha$ with the right state $\omega^{\alpha,\ell/r}$,
while $\lambda_2(\omega^{\alpha,\ell / r}, v^{\alpha,\ell/r})$
is the Rankine-Hugoniot speed of the $2$-shock
connecting the left state $\omega^{\alpha,\ell/r}$ with the right state $v^{\alpha,\ell / r}$
(see \S~\ref{Ss:Rsolv-appsol}). 

The goal of this section is to prove that, choosing the coefficeints $\kappa_{i\ca j}$, $i,\,j=1,2$,  in~\eqref{S4Wi} appropriately, together with the sizes $\delta_0^*$ and $\delta_p^*$ of the domain,
for every $\alpha\in \cj$, there holds
\be{E:mainest}
E_{\alpha,1}+E_{\alpha,2} \leq \co(1) \varepsilon |\sC_{\alpha}|\,,\quad\text{for every }\alpha\in \cj\,.
\ee
Actually, the selection of these parameters is performed to obey \emph{Conditions $(\Sigma)$} stated in the proof of Propostion~\ref{PropCond}. Next, summing up~\eqref{E:mainest} over all jumps $\alpha\in \cj$, we derive the general estimate~\eqref{est-phi0-decr} from~\eqref{der-phi-sum} and \eqref{E:mainest} relying on
~\eqref{unifbvbound},~\eqref{V-totvar-est}, since $\cg(u(t))<M^*$ and $\cg(v(t))<M^*$ for all times.

We will establish the basic estimate~\eqref{E:mainest} assuming that the term $E_{\alpha,i}$ in~\eqref{E:error} 
always refers to a
 jump in $v(t)$ at $x_\alpha$ that connects
two states along an Hugoniot curve. 
This means that, when the jump in $v(t)$ is actually a rarefaction front, we shall replace it
with a {\it rarefaction shock} (cfr.~\cite[\S~5.2]{Bre}) of the same size, connecting two states along an Hugoniot curve,
and travelling with the corresponding Hugoniot speed.
Following a similar argument as in~\cite[\S~8.2]{Bre}, 
one can show that, because of the second order tangency of 
Hugoniot and rarefaction curves, this reduction 
produces an error of size $\co(1) \varepsilon |\sC_{\alpha}|$. 
We shall discuss this reduction in Appendix~\ref{S:shockReduction}.
The estimate~\eqref{E:mainest} in the case when $E_{\alpha,i}$ refers to a jump of $u(t)$ rather than of $v(t)$ is entirely analogous.

The structure of this section is the following:
\begin{itemize}
\item[\S~\ref{Ss:estimPhys1wBis}] We derive~\eqref{E:mainest} when the wave of $v(t)$ at $x_{\alpha}$ is a 
(compressive or rarefaction) shock of the first famiy.
\item[\S~\ref{Ss:estimPhys2wNN}] We derive~\eqref{E:mainest} when the wave of $v(t)$ at $x_{\alpha}$ is a 
(compressive or rarefaction) shock of the second family.
\end{itemize}
The analysis of sections \S~\ref{Ss:estimPhys1wBis}-\S~\ref{Ss:estimPhys2wNN}
relies on refined interaction-type estimates that are obtained in Appendix~\ref{S:finerInteractions}: they involve the waves $\eta_i^{\alpha,\ell / r}$ connecting $u(t)$ with $v(t)$ at $x_\alpha$, the wave speeds  $\lambda_{2}^{\alpha,\ell / r}$ given at~\eqref{S5.2.1_introd} and the speed $\dot x_{\alpha}$ of the wave in $v$.

\subsubsection{Waves of the first family}
\label{Ss:estimPhys1wBis}

In this section, we derive the estimate~\eqref{E:mainest} on the sum of the errors $E_{\alpha,1}+E_{\alpha,2}$ 
defined in~\eqref{E:error} when the wave of $v(t)$ present at $x_{\alpha}$ belongs to the first family, i.e.~$k_{\alpha}=1$. 
To this end, we shall first provide an estimate of $E_{\alpha,1}$ and $E_{\alpha,2}$ separately, and 
then we will combine them to derive~\eqref{E:mainest}.
We shall adopt the notation given in~\eqref{GE:definitionslralpha},
dropping the superscript $\alpha$, and we will let $\sC_\alpha, \sR_\alpha$ denote the size of the
wave located at $x_\alpha$, measured in the original and Riemann coordinates, respectively (see 
\S~\ref{Ss:strengthnotation}).
Recall that, by~\eqref{E:reeeeeeeterete} in Theorem~\ref{T:AS}
the solutions~$u$, $v$, and the intermediate value $\omega$ 
defined in~\eqref{S5.2.1 intro-a}, take values in 
the compact set 
\begin{subequations}
\label{bound-eta1-gamma-eta1+gamma-pNN}
\ba
&\text{$K={[0,\delta_0^*]\times [p_0^*,p_1^*]}$}
\ea
with the parameters satisfying
\ba\label{:setKbd-1bound}
0<\delta_0^*<1-\delta_p^*< p_0^*<1<p_1^*<1+\delta_p^*
\ea
for $\delta_0^*>0$ and $\delta_p^*>0$ sufficiently small.
For convenience, we assume that both $\delta_0^*$ and $\delta_p^*$ are less than $\frac{1}{2}$. Having these restrictions on $K$, we obtain the following conditions on the variables
\ba\label{:cvetagamma-1bound}
v_1^{\alpha,\ell}, \omega_1^{\alpha,\ell}, |\eta_1^{\alpha,\ell}|, |\sC_\alpha|, |\eta_1^{\alpha,\ell}+\sC_\alpha| \le { \delta_0^*},\quad 
|v_2^{\alpha,\ell}-1|,|v_2^{\alpha,\ell}-1-\eta_2^{\alpha,\ell}|, 
|p_\alpha-1|\le \delta_p^* \,,\quad |\eta_2^{\alpha,\ell}|\le 2\delta_p^*.
\ea
Furthermore, we also require
\ba\label{:constrNewCoeff0}
\mathfrak K\doteq\kappa_{1\ca1}\mu\delta_0^*\delta_p^*<\tfrac{1}{4} \, .
\ea
and
  \ba\label{E:constrNewCoeff}
  &\left(1-e^{\mathfrak K}\frac{\kappa_{1\ca1}}{\mu^2}\delta_0^*\delta_p^* \right)  
>
\frac{1}{2}
 \ea
These additional conditions~\eqref{:constrNewCoeff0}--\eqref{E:constrNewCoeff} are imposed in order for the decay of the waves of the first family to dominate the possible increase of the waves of the second family and also to control error terms that arise while establishing~\eqref{E:mainest}.
\end{subequations}

\begin{figure}
\centering
\hfill
\includegraphics[width=.30\linewidth, height=.30\linewidth]{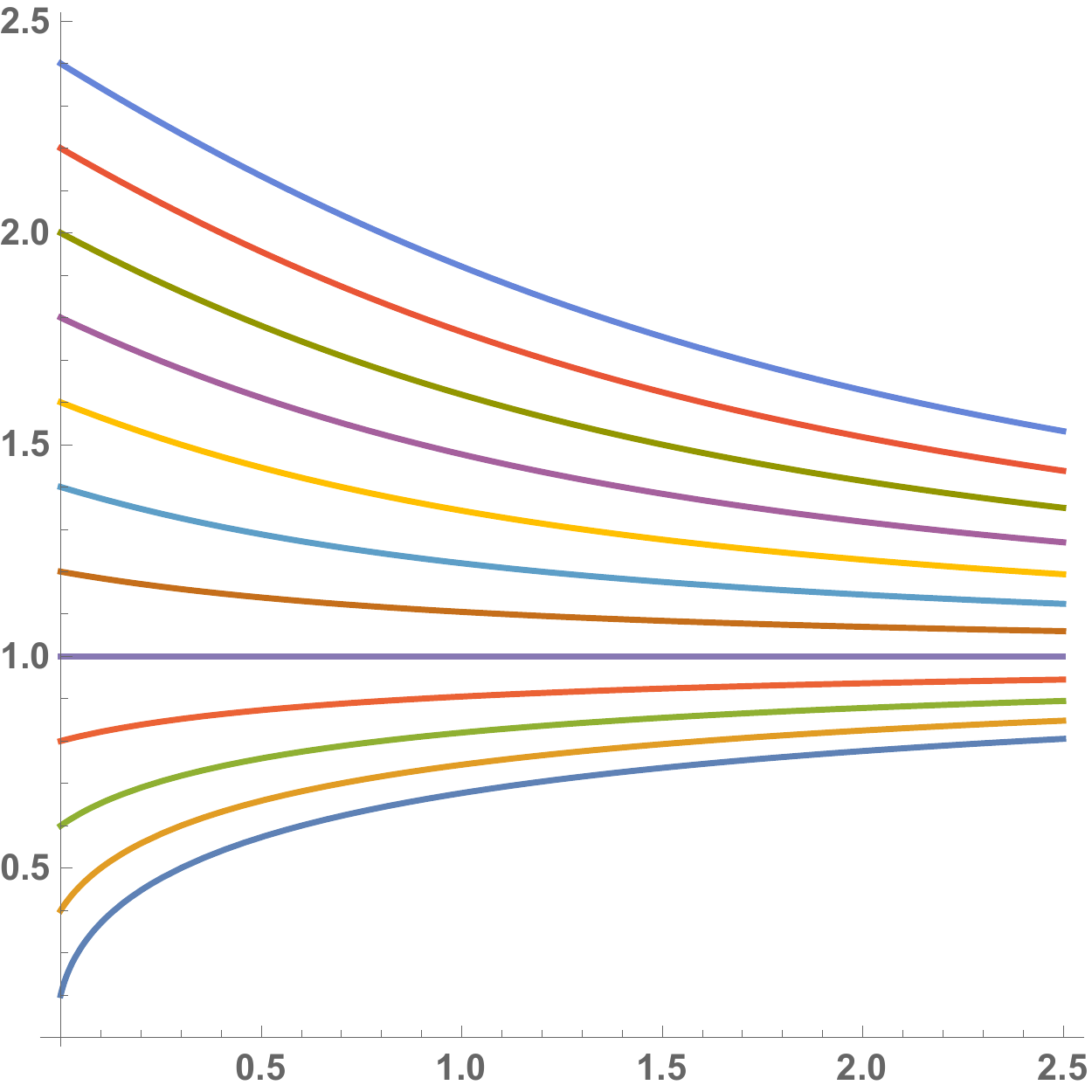}
\hfill
\includegraphics[width=.30\linewidth, height=.30\linewidth]{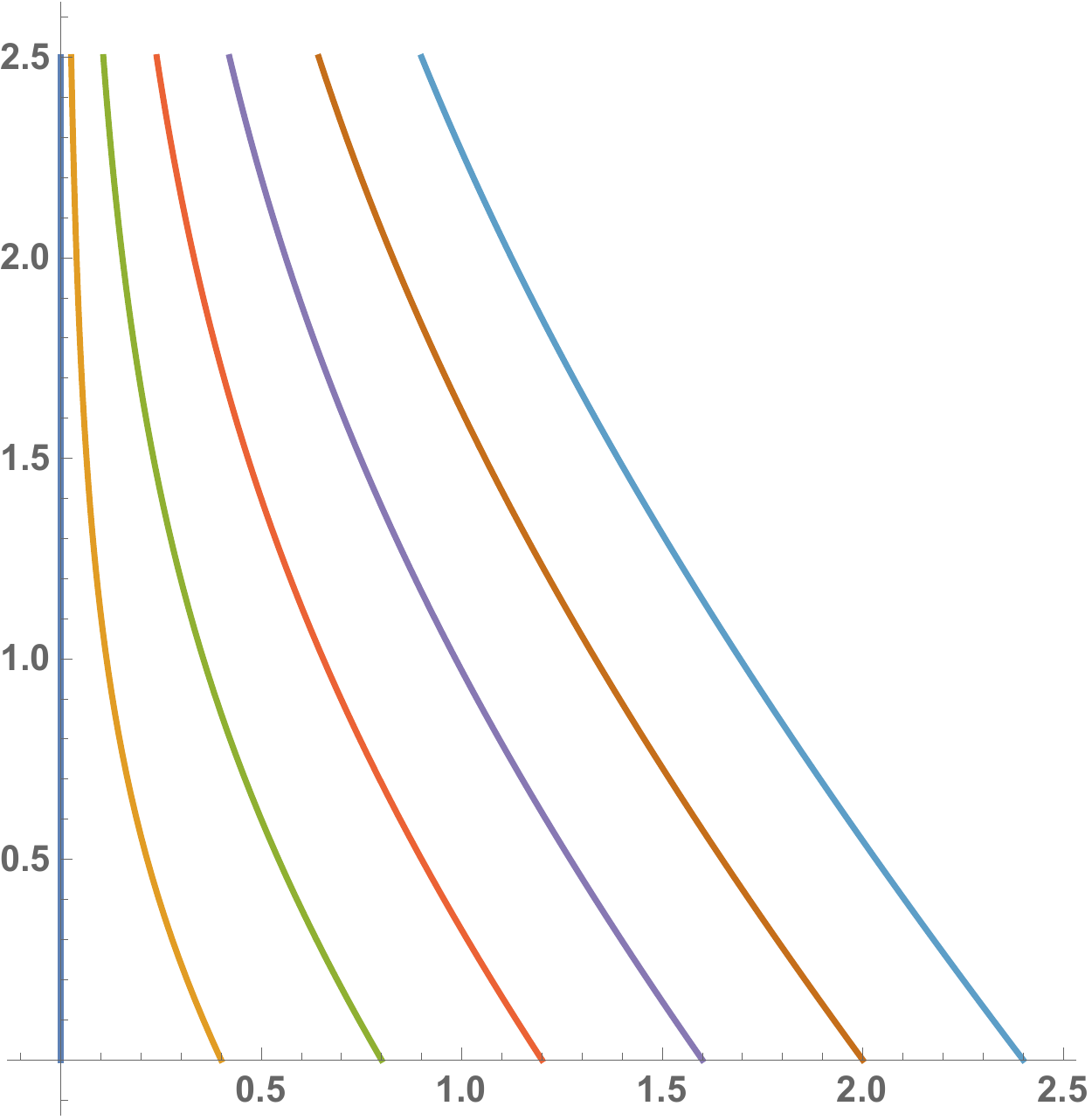}
\hspace{1in}
\caption{{\bf Left}: Hugoniot curves of the first family through $(0,p)$ \textemdash left to right are entropy admissible. {\bf Right}: Hugoniot curves of the second family throught $(h,0)$ \textemdash up to down are entropy admissible.}
\label{fig:shockCurves}
\end{figure}

\paragraph{Estimate of $E_{\alpha,1}$ for waves of the first family} \label{SubEalpha1-1}
The analysis is divided into three cases according to the signs of $\eta_{1}^{\alpha,r}$, $\eta_{1}^{\alpha,\ell}$ and to the type of the wave at $x_\alpha$: a) treating the case when $\eta_{1}^{\alpha,r}$ and $\eta_{1}^{\alpha,\ell}$ have the same sign, b) treating the case when $\eta_{1}^{\alpha,r}$ and $\eta_{1}^{\alpha,\ell}$ have opposite sign and the wave at $x_{\alpha}$ is an entropy admissible $1$-shock, c) treating the case when $\eta_{1}^{\alpha,r}$ and $\eta_{1}^{\alpha,\ell}$ have opposite sign and the wave at $x_{\alpha}$ is $1$-rarefaction shock.

\noindent\textbf{Case a)} Suppose that $\eta_{1}^{\alpha,r}$ and $\eta_{1}^{\alpha,\ell}$ have the same sign. 
Observe that by definitions~\eqref{S4Wi}--\eqref{S4A1} one has 
$$W_{1}^{\alpha,r}=W_{1}^{\alpha,\ell} e^{\mathfrak D}$$
where $ \mathfrak D$ is given by
\ba\label{E:rguiregregregrergeNN}
 \mathfrak D=\sgn\left( (v_{2}^{\alpha,\ell}-1)\eta_{1}^{\alpha,\ell}\right)\cdot \kappa_{1\ca1}\PHI{{v_{2}^{\alpha,\ell}}}|\sR_{\alpha}|\;.
\ea
By bounds~\eqref{bound-eta1-gamma-eta1+gamma-pNN},\,\eqref{E:equivalncestrengths}, it holds
\ba
 \left|\mathfrak D\right|\leq \mathfrak K<\tfrac14 \label{E:boundD}\;.
\ea
We thus estimate $E_{\alpha,1} $ in~\eqref{E:error} as follows:
\begin{align*}
E_{\alpha,1} 
&=
(W_{1}^{\alpha,r}-W_{1}^{\alpha,\ell})|\eta_{1}^{\alpha,\ell}|(\lambda_{1}^{\alpha,\ell}-\dot x_{\alpha})
+W_{1}^{\alpha,r}\left[|\eta_{1}^{\alpha,r}|(\lambda_{1}^{\alpha,r}-\dot x_{\alpha})-|\eta_{1}^{\alpha,\ell}|(\lambda_{1}^{\alpha,\ell}-\dot x_{\alpha})\right]
\notag \\
&\stackrel{\eqref{E:rguiregregregrergeNN}}{=}
W_{1}^{\alpha,r}\left(1-e^{-\mathfrak D}\right)(\lambda_{1}^{\alpha,\ell}-\dot x_{\alpha}) |\eta_{1}^{\alpha,\ell} |+W_{1}^{\alpha,r}\sgn{(\eta_{1}^{\alpha,\ell})}\left[\eta_{1}^{\alpha,r}(\lambda_{1}^{\alpha,r}-\dot x_{\alpha})-\eta_{1}^{\alpha,\ell}(\lambda_{1}^{\alpha,\ell}-\dot x_{\alpha})\right] \ .
\end{align*}
Using~\eqref{E:stimaetarbasefull1}, we approximate $\eta_{1}^{\alpha,\ell}+\sC_\alpha  $ by $\eta_{1}^{\alpha,r}  $ having some error terms
\[
\eta_{1}^{\alpha,\ell}+\sC_\alpha =\eta_{1}^{\alpha,r} (1+\co(1)\delta_0^*)+\co(1)\delta_0^*|\eta_2^{\alpha,\ell}|
\]
and then, we combine with~\eqref{E:stima11lfull} to estimate the range of the difference in speeds to get
\[
 \frac{ (v_{2}^{\alpha,\ell}-1) \eta_{1}^{\alpha,r}  }{ v_{2}^{\alpha,\ell}} \left(1-\co(1)  \delta_0^*\right)-\co(1)|\eta_2^{\alpha,\ell}| \leq
\dot x_{\alpha}- \lambda_{1}^{\alpha,\ell}
 \leq \frac{ (v_{2}^{\alpha,\ell}-1) \eta_{1}^{\alpha,r} }{ v_{2}^{\alpha,\ell}} \left(1+\co(1)  \delta_0^*\right)+ \co(1)|\eta_2^{\alpha,\ell}| \,.
\]
Next, we recall~\eqref{E:rguiregregregrergeNN} and observe that
\[
\left(1-e^{-\mathfrak D}\right) (v_{2}^{\alpha,\ell}-1) \eta_{1}^{\alpha,r}>0\;,
\]
since $\eta_{1}^{\alpha,\ell}\eta_1^{\alpha,r}>0$. Combining these estimates together with
~\eqref{E:fvgabfafvNewfull} and bounds~\eqref{bound-eta1-gamma-eta1+gamma-pNN}, we arrive at
\ba
E_{\alpha,1} 
\leq W_{1}^{\alpha,r}\cdot&\left(\frac{-1+\co(1) \delta_0^*}{p_1^*}\cdot
\left|\left(1-e^{-\mathfrak D}\right)(v_{2}^{\alpha,\ell}-1)\eta_{1}^{\alpha,\ell}\eta_{1}^{\alpha,r}\right|
 + \co(1) \left|\left(1-e^{-\mathfrak D}\right)\eta_{1}^{\alpha,\ell}\eta_2^{\alpha,\ell} \right|\notag
 \right.
\\
&\qquad\qquad\qquad\qquad\left. +\co(1)   \delta_0^* \left\lvert(v_{2}^{\alpha,\ell}-1)^2   \eta_{1}^{\alpha,\ell}\eta_{1}^{\alpha,r}\sC_{ \alpha} \right\rvert+\co(1)|\eta_2^{\alpha,\ell} \sC_{{{\alpha}}}| \right) \ \label{E:stimaintermedia}
 \ea
 using that $|v_1^{\alpha,\ell}+\sC_{{{\alpha}}}|\le \co(1)\delta_0^*$.
Next, we use the Maclaurin expansion to estimate $\left|e^{-\mathfrak D}-1\right|$ 
 \bas
0\leq e^{-\mathfrak D}-1+\sgn\left( (v_{2}^{\alpha,\ell}-1)\eta_{1}^{\alpha,\ell}\right)\cdot \kappa_{1\ca1}\PHI{{v_{2}^{\alpha,\ell}}}|\sR_{\alpha}| 
 \leq
 e^{\mathfrak K}\kappa_{1\ca1}^2 \PHI{{v_{2}^{\alpha,\ell}}}^2 \sR^2_{\alpha} \ ,
\eas
and control the error terms in~\eqref{E:stimaintermedia}. Indeed, by estimates~\eqref{E:equivalncestrengths}\,,~\eqref{E:boundD} we get the upper bound
  \ba\label{E:upperestE}
   \left|\left(1-e^{-\mathfrak D}\right)\eta_1^{\alpha,\ell}\right|\leq \left(1+e^{\mathfrak K}\mathfrak K\right) 
   \mathfrak K
   \left|\sC_\alpha\right|
   \leq4 \mathfrak K\left|\sC_\alpha\right| \,,
  \ea
while by~\eqref{E:constrNewCoeff}, we get the lower bound
 \ba\label{E:newlowerbound}
   {\left|1-e^{-\mathfrak D}\right|}&\ge|\mathfrak D|- e^{\mathfrak K}\kappa_{1\ca1}^2 \PHI{{v_{2}^{\alpha,\ell}}}^2 |\sR_{\alpha}|^2 \notag\\
 &\ge
  \left(1-e^{\mathfrak K}\frac{\kappa_{1\ca1}}{\mu^2}\delta_0^* \delta_p^*\right)\cdot \kappa_{1\ca1}\PHI{{v_{2}^{\alpha,\ell}}}|\sC_{\alpha}| \notag\\
&>
\frac{\kappa_{1\ca1} }{2}\PHI{{v_{2}^{\alpha,\ell}}}|\sC_{\alpha}| \;.
 \ea
Substituting estimates~\eqref{E:upperestE} and~\eqref{E:newlowerbound} into~\eqref{E:stimaintermedia} , we obtain
\begin{align}\label{E:sanfownFNN}
E_{\alpha,1} 
\leq W_{1}^{r}\cdot&\left(\left(\frac{-1+\co(1)  \delta_0^*(1+\frac{2}{\kappa_{1\ca1}})}{ p_1^*} \right)\cdot
\left|\left(1-e^{-\mathfrak D}\right)(v_{2}^{\alpha,\ell}-1)\eta_{1}^{\alpha,\ell}\eta_{1}^{\alpha,r}\right|
  +\co(1) \left(1+ 4 \mathfrak K  \right) \cdot\left| \sC_{\alpha} \eta_{2}^{\alpha,\ell}\right| \right)  \ .
 \end{align}

\noindent\textbf{Case b)} Suppose that $\eta_{1}^{\alpha,r}$ and $\eta_{1}^{\alpha,\ell}$ have opposite sign and 
that the wave at $x_{\alpha}$ is an entropy admissible $1$-shock. Due to the geometry of shock curves (see~Figure~
\ref{S2:fig1})
this implies that
\be{S4:C1b}
\sgn(v_{2}^{\alpha,\ell}-1)\,\eta_1^{\alpha,r}<0<
\sgn(v_{2}^{\alpha,\ell}-1)\,\eta_{1}^{\alpha,\ell} \ .
\ee
We also note that 
the interaction estimates~\eqref{E:stimaetarbasefull1}, 
together with~\eqref{bound-eta1-gamma-eta1+gamma-pNN},
 imply the bound
\be{etarl-strenght-opposite-signNN}
|\eta_{1}^{\alpha,r}|+|\eta_{1}^{\alpha,\ell}|=|\eta_{1}^{\alpha,r}-\eta_{1}^{\alpha,\ell}|
\leq \mathcal{O}(1)|\sC_{\alpha}| \ .
\ee

Applying~\eqref{E:stima11full} and bounds~\eqref{bound-eta1-gamma-eta1+gamma-pNN}, we can estimate the error $E_{\alpha,1} $ in~\eqref{E:error}:
\bas
E_{\alpha,1} 
&\stackrel{}{=} 
W_{1}^{\alpha,r} |\eta_1^{\alpha,r}|\left[- \frac{(v_{2}^{\alpha,\ell}-1) \eta_{1}^{\alpha,\ell} }{v_{2}^{\alpha,\ell}} (1+\co(1) \delta_0^*p_1^*)+ \co(1)|\eta_{2}^{\alpha,\ell} |\right]
\notag\\
 &\qquad\qquad-
W_{1}^{\alpha,\ell}|\eta_1^{\alpha,\ell}| \left[- \frac{(v_{2}^{\alpha,\ell}-1) (\eta_{1}^{\alpha,\ell}+\sC_{\alpha})}{v_{2}^{\alpha,\ell}} (1+\co(1) \delta_0^*  p_1^*)+ \co(1)|\eta_{2}^{\alpha,\ell} |\right]
\notag
\eas
If one now considers~\eqref{E:stimaetarbasefull1}--\eqref{etarl-strenght-opposite-signNN} and bounds~\eqref{bound-eta1-gamma-eta1+gamma-pNN},~\eqref{W-unif-boundN}, one arrives to
\ba\label{E:49nvbhqbdfdNN} 
E_{\alpha,1} 
&\stackrel{}{\le}
\frac{-1+\co(1) \delta_0^*  }{p_1^*} (W_{1}^{\alpha,r}+W_{1}^{\alpha,\ell} ) \cdot {\left|(v_{2}^{\alpha,\ell}-1) \eta_{1}^{\alpha,r}\eta_{1}^{\alpha,\ell}\right|} 
+ \co(1) W_{1}^{*}\cdot |\sC_\alpha\eta_{2}^{\alpha,\ell}| 
\,.
\ea
By~\eqref{E:stimaetarbasefull1}, we note that in the above, it holds
\[
\sgn\left((v_{2}^{\alpha,\ell}-1)\,(\eta_1^{\alpha,\ell}+\sC_\alpha)\right)=\sgn\left((v_{2}^{\alpha,\ell}-1)\,\eta_1^{\alpha,r}\right)<0
\]
for small $\delta_0^*$ and as in Case a), $\eta_{1}^{\alpha,\ell}+\sC_{\alpha}$ is approximated by $\eta_{1}^{\alpha,r}$ to establish~\eqref{E:49nvbhqbdfdNN}.

\noindent\textbf{Case c)} Suppose that $\eta_{1}^{\alpha,r}$ and $\eta_{1}^{\alpha,\ell}$ have opposite sign and 
that the wave at $x_{\alpha}$ is a 1-rarefaction shock (see discussion at beginning of \S~\ref{Ss:nointeractiontimesN}).
Then, by construction one has $|\sC_{\alpha}|\leq\varepsilon$ (see \S~\ref{Ss:Rsolv-appsol}).
By the estimate on velocities~\eqref{E:stima11lfull} and since~\eqref{etarl-strenght-opposite-signNN} remains valid, we deduce that
\be{wavespeed-opposite-sign-epsNN}
\begin{aligned}
|\lambda_1^{\alpha,\ell}-\dot{x}_\alpha|
&\leq \co(1) \left(\delta_p^*\cdot |\eta_1^{\alpha,\ell}+\sC_{\alpha}| +
|\eta_2^{\alpha,\ell}|\right)
\leq 
\co(1) \left(\delta_p^*|\sC_{\alpha} |+
|\eta_2^{\alpha,\ell}|\right)
\leq 
\co(1) \left(\delta_p^*\varepsilon +
|\eta_2^{\alpha,\ell}|\right) \;.
\end{aligned}
\ee
and similarly, from~\eqref{E:stima11full2},
\be{wavespeed-opposite-sign-epsNN2}
\begin{aligned}
|\lambda_1^{\alpha,r}-\dot{x}_\alpha|
&\leq \co(1) \left(\delta_p^*\cdot |\eta_1^{\alpha,\ell}| +
|\eta_2^{\alpha,\ell}|\right)
\leq 
\co(1) \left(\delta_p^*\varepsilon +
|\eta_2^{\alpha,\ell}|\right) \;.
\end{aligned}
\ee

Using~\eqref{etarl-strenght-opposite-signNN} and
\eqref{wavespeed-opposite-sign-epsNN}--\eqref{wavespeed-opposite-sign-epsNN2},
we can now derive the following estimate for $E_{\alpha,1} $ in~\eqref{E:error}:
\begin{align}\label{E:aefNN}
E_{\alpha,1} 
&\stackrel{\eqref{E:error}}{=}
(W_{1}^{\alpha,r}-W_{1}^{\alpha,\ell})|\eta_{1}^{\alpha,\ell}|(\lambda_{1}^{\alpha,\ell}-\dot x_{\alpha})
+W_{1}^{\alpha,r}\left[|\eta_{1}^{\alpha,r}|(\lambda_{1}^{\alpha,r}-\dot x_{\alpha})-|\eta_{1}^{\alpha,\ell}|(\lambda_{1}^{\alpha,\ell}-\dot x_{\alpha})\right]
\notag 
\\
&\leq
\left|W_{1}^{\alpha,r}-W_{1}^{\alpha,\ell}\right| |\eta_{1}^{\alpha,\ell}|
 \co(1) \left(\delta_p^* \varepsilon +
|\eta_2^{\alpha,\ell}|\right)
+\co(1)W_{1}^{\alpha,r}\left[|\eta_{1}^{\alpha,r}|+|\eta_{1}^{\alpha,\ell}|\right]  \left(\delta_p^*\varepsilon +
|\eta_2^{\alpha,\ell}|\right)\notag\\
&\stackrel{\eqref{etarl-strenght-opposite-signNN}}{\leq}
 \co(1) W_{1}^{*} \left( \ \delta_p^*  \varepsilon |\sC_\alpha|+
 |\eta_2^{\alpha,\ell}l\sC_\alpha| \right)\; .
\end{align}

All cases for $E_{\alpha,1}$ for waves of the first family have been investigated at this point.

\paragraph{Estimate of $E_{\alpha,2}$ for waves of the first family}\label{SubEalpha1-2}
Let us first point out that the Hugoniot curves
of the same family cannot cross each other, see Figure~\ref{fig:shockCurves}.
As a consequence, due to the geometric properties of such curves, the components $\eta_{2}^{\alpha,\ell},\,\eta_{2}^{\alpha,r}$ must have the same sign
and thus, by definitions~\eqref{S4WiN2} and~\eqref{S4A2}, there holds
\bes
W_{2}^{\alpha,\ell}=W_{2}^{\alpha,r} e^{\kappa_{2\ca 1}|\sR_{\alpha}|}\;.
\ees
We thus rewrite the error term in~\eqref{E:error} as
\bas
E_{\alpha,2} 
&= 
(W_{2}^{\alpha,r}-W_{2}^{\alpha,\ell})(\lambda_{2}^{\alpha,\ell}-\dot x_{\alpha})\big|\eta_{2}^{\alpha,\ell} \big|
	+W_{2}^{\alpha,r}\left[|\eta_{2}^{\alpha,r}|(\lambda_{2}^{\alpha,r}-\dot x_{\alpha})-|\eta_{2}^{\alpha,\ell}|(\lambda_{2}^{\alpha,\ell}-\dot x_{\alpha})\right]  
\\
&= 
W_{2}^{\alpha,r}\left[(1-e^{\kappa_{2\ca 1}|\sR_{\alpha}|})(\lambda_{2}^{\alpha,\ell}-\dot x_{\alpha})\big|\eta_{2}^{\alpha,\ell} \big|
	+ \left(\eta_{2}^{\alpha,r}(\lambda_{2}^{\alpha,r}-\dot x_{\alpha})-\eta_{2}^{\alpha,\ell}(\lambda_{2}^{\alpha,\ell}-\dot x_{\alpha})\right) \sgn(\eta_{2}^{\alpha,\ell})\right] \;.
\eas
Next, we observe that by convexity of the exponential and recalling bounds~\eqref{E:equivalncestrengths}, one has
\[
1-e^{\kappa_{2\ca 1}|\sR_{\alpha}|}\leq-\kappa_{2\ca 1}|\sR_{\alpha}|\leq-\frac{\kappa_{2\ca 1}}\mu|\sC_{\alpha}|\,.
\]
Also, by definition, we have $\lambda_{2}^{\alpha,\ell}-\dot x_{\alpha}=v_{2}^{\alpha,\ell} +\co(1)\delta_{0}^*\ge\frac{p_0^*}{2}$, using bounds~\eqref{bound-eta1-gamma-eta1+gamma-pNN} for small $\delta_{0}^*$. 
In view of the above analysis, 
we can estimate
the term $E_{\alpha,2} $  as:
\ba\label{E:fpNIVqbjkqsfNN}
E_{\alpha,2} 
&
\stackrel{\eqref{E:fvgabfafvNewfull2}}{\leq}W_{2}^{\alpha,r}\cdot\left[\left(
 - \kappa_{2\ca1} \frac{p_{0}^{*}}{2\mu} +  \co(1) \right)\big|\eta_{2}^{\alpha,\ell}\sC_{\alpha}\big|
 +\co(1) \cdot |(v_{2}^{\alpha,\ell}-1)^{ 2}\eta_{1}^{\alpha,r}  \eta_{1}^{\alpha,\ell}\sC_{\alpha} | \right]\;,
 \ea
using estimate~\eqref{E:fvgabfafvNewfull2}, bounds~\eqref{bound-eta1-gamma-eta1+gamma-pNN} and  noting again that~\eqref{E:stimaetarbasefull1} allowed us to replace $\eta_{1}^{\alpha,\ell}+\sC_\alpha  $ with $\eta_{1}^{\alpha,r}  $, with some error as in \S~\ref{SubEalpha1-1}.
For the case a) in \S~\ref{SubEalpha1-1} that both $\eta_1^{\alpha,\ell}$ and $\eta_1^{\alpha,r}$ have the same sign, estimate~\eqref{E:fpNIVqbjkqsfNN} further reduces to 
 \ba
E_{\alpha,2} 
 &\label{E:fpNIVqbjkqsfNN2}
\stackrel{\eqref{E:newlowerbound}}{\leq}
W_{2}^{\alpha,r}\cdot\left[\left(
 - \kappa_{2\ca1} \frac{p_{0}^{*}}{2\mu} +  \co(1) \right)\big|\eta_{2}^{\alpha,\ell}\sC_{\alpha}\big|
 +  \co(1) \frac{2}{\kappa_{1\ca1}}\left|\left(1-e^{-\mathfrak D}\right)(v_{2}^{\alpha,\ell}-1)\eta_{1}^{\alpha,\ell}\eta_{1}^{\alpha,r}\right|
 \right]\;.
 \ea


\paragraph{Derivation of~\eqref{E:mainest} for waves of the first family}\label{P4.2.1.3}
We conclude here the proof of estimate~\eqref{E:mainest} for a wave of $v(t)$ at $x_{\alpha}$ belonging to the first family.
In view of the above analysis, we combine now the estimates of $E_{\alpha,1}$ and $E_{\alpha,2}$
 distinguishing the three cases studied to estimate $E_{\alpha,1}$ in \S~\ref{SubEalpha1-1}--~\ref{SubEalpha1-2}. 

\vskip\baselineskip
\noindent\textbf{Case a)} Recall that in this case both $\eta_{1}^{\alpha,\ell}, \eta_{1}^{\alpha,r}$ are assumed to have the same sign. Adding together~\eqref{E:sanfownFNN} and~\eqref{E:fpNIVqbjkqsfNN2} yields
\be{4.26-1aNN}
\begin{aligned}
E_{\alpha,1} +E_{\alpha,2} 
&\leq 
\left( \frac{-1+\co(1) \delta_0^*}{ p_1^*}  \cdot W_1^{\alpha,r}+\co(1) \cdot\frac{1}{\kappa_{1\ca1}} W_2^*\right)
\left|\left(1-e^{-\mathfrak D}\right)(v_{2}^{\alpha,\ell}-1)\eta_{1}^{\alpha,\ell}\eta_{1}^{\alpha,r}\right|
\\
&+\left(
 - \kappa_{2\ca1}\,\frac{W_2^{\alpha,r} p_0^* }{2\mu}+\co(1)  ( 1+\mathfrak K )W_1^*+ \co(1)W_2^{\alpha,r} \right)
|\eta_{2}^{\alpha,\ell}\sC_{\alpha}| 
\end{aligned}
\ee
using~\eqref{W-unif-boundN} and taking $\kappa_{1\ca1}\ge 1$. 
Taking $\delta_0^*$ small enough and using that $W_i^{\alpha,r}\ge 1$, we can further bound this sum as 
\be{4.26-1aNN2}
\begin{aligned}
E_{\alpha,1} +E_{\alpha,2} 
\leq &
\left( \frac{-1+\co(1) \delta_0^*}{ p_1^*} +\co(1) \cdot\frac{1}{\kappa_{1\ca1}} W_2^*\right)
\left|\left(1-e^{-\mathfrak D}\right)(v_{2}^{\alpha,\ell}-1)\eta_{1}^{\alpha,\ell}\eta_{1}^{\alpha,r}\right|
\\
&+\left(
 - \kappa_{2\ca1}\,\frac{p_0^* }{4\mu}+\co(1)  W_1^*\right)
|\eta_{2}^{\alpha,\ell}\sC_{\alpha}| +\left(
 - \kappa_{2\ca1}\,\frac{ p_0^* }{4\mu}+ \co(1)\right) W_2^{\alpha,r}
|\eta_{2}^{\alpha,\ell}\sC_{\alpha}| \\
=:& I_1 \left|\left(1-e^{-\mathfrak D}\right)(v_{2}^{\alpha,\ell}-1)\eta_{1}^{\alpha,\ell}\eta_{1}^{\alpha,r}\right|+
\left(J_1  +J_2 W_2^{\alpha,r}\right)
|\eta_{2}^{\alpha,\ell}\sC_{\alpha}| \;.
\end{aligned}
\ee

\vskip\baselineskip

\noindent\textbf{Case b)} In this case, $\eta_{1}^{\alpha,\ell}$ and $\eta_{1}^{\alpha,r}$ have opposite sign
and the wave at $x_{\alpha}$ is an entropy admissible 1-shock.
Summing up~\eqref{E:49nvbhqbdfdNN},~\eqref{E:fpNIVqbjkqsfNN}, 
and relying on~\eqref{bound-eta1-gamma-eta1+gamma-pNN} and~\eqref{W-unif-boundN}, it follows
\be{4.26-1aNN3}
\begin{aligned}
E_{\alpha,1} +E_{\alpha,2} 
&\leq 
\left(\frac{-1+\co(1) \delta_0^*  }{p_1^*} W_{1}^{\alpha,r} 
+\co(1)W_2^{*}\delta_0^*p_1^*
\right)\cdot
 {\left|(v_{2}^{\alpha,\ell}-1) \eta_{1}^{\alpha,r}\eta_{1}^{\alpha,\ell}\right|} 
\\
&
+\left(- \kappa_{2\ca1}\,\frac{W_2^{\alpha,r} p_0^* }{2\mu}+\co(1)\,\left(W_1^*+W_2^{\alpha,r}\right)\right)
\cdot|\eta_{2}^{\alpha,\ell}\sC_{\alpha}| 
\end{aligned}
\ee
In a similar manner as before, this further reduces to
\be{4.26-1aNN4}
\begin{aligned}
E_{\alpha,1} +E_{\alpha,2} 
\leq &
\left( \frac{-1+\co(1) \delta_0^*  }{p_1^*} +\co(1) \cdot W_2^* \delta_0^*p_1^*\right)
\cdot
 {\left|(v_{2}^{\alpha,\ell}-1) \eta_{1}^{\alpha,r}\eta_{1}^{\alpha,\ell}\right|} 
\\
&+\left(
 - \kappa_{2\ca1}\,\frac{p_0^* }{4\mu}+\co(1)  W_1^*\right)
|\eta_{2}^{\alpha,\ell}\sC_{\alpha}| +\left(
 - \kappa_{2\ca1}\,\frac{ p_0^* }{4\mu}+ \co(1)\right) W_2^{\alpha,r}
|\eta_{2}^{\alpha,\ell}\sC_{\alpha}| \\
=:& I_2 \left|\left(1-e^{-\mathfrak D}\right)(v_{2}^{\alpha,\ell}-1)\eta_{1}^{\alpha,\ell}\eta_{1}^{\alpha,r}\right|+
\left(J_1 
 +J_2 W_2^{\alpha,r}\right)
|\eta_{2}^{\alpha,\ell}\sC_{\alpha}| \;.
\end{aligned}
\ee

\vskip\baselineskip

\noindent\textbf{Case c)} In this last case, $\eta_{1}^{\alpha,\ell}$ and $\eta_{1}^{\alpha,r}$ have opposite sign
and the wave at $x_{\alpha}$ is a 1-rarefaction shock. 
By~\eqref{E:aefNN},~\eqref{E:fpNIVqbjkqsfNN}, and relying on~\eqref{bound-eta1-gamma-eta1+gamma-pNN},~\eqref{etarl-strenght-opposite-signNN} and~\eqref{W-unif-boundN},
one thus has 
\be{4.26-1aNN5}
\begin{aligned}
E_{\alpha,1}+E_{\alpha,2} \leq &
\left(- \kappa_{2\ca1}\,\frac{W_2^{\alpha,r} p_0^* }{2\mu}+\co(1)\,\left(W_1^*
 	+ W_2^{\alpha,r}\right)\right)
\cdot|\eta_{2}^{\alpha,\ell}\sC_{\alpha}| 
\\
&+\co(1) (W_1^* +W_2^* \delta_p^*  \delta_0^*) \delta_p^* \cdot \varepsilon |\sC_{\alpha} | \,\\
\le & \left(J_1 
+J_2 W_2^{\alpha,r}\right)
|\eta_{2}^{\alpha,\ell}\sC_{\alpha}| + \co(1) (W_1^* +W_2^* \delta_p^*  \delta_0^*) \delta_p^* \cdot \varepsilon |\sC_{\alpha} | \;,
\end{aligned}
\ee
using $|\sC_{\alpha}|\le\varepsilon$.
\vskip\baselineskip
The next proposition guarantees the existence of suitable parameters $\kappa_{i\ca j}$, with $i,j=1,2$ and $\delta_p^*$, $\delta_0^*$ such that~\eqref{E:mainest} holds true.

\begin{proposition}\label{PropCond}
There exist coefficients $\kappa_{i\ca j}$, with $i,j=1,2$ of $W_1$ and $W_2$ in~\eqref{S4Wi} and positive constants $\delta_p^*$, $\delta_0^*$ such that 
estimate~\eqref{E:mainest} holds.
\end{proposition}
\begin{proof}
Let $\co(1)$ denote the maximum of all constants $\co(1)$ appearing in estimates~\eqref{4.26-1aNN2},~\eqref{4.26-1aNN4}\, and~\eqref{4.26-1aNN5}.
Then, we select the values $\kappa_{i\ca j}$ and $\delta_p^*$, $\delta_0^*$ in such a way that the terms $I_1$, $I_2$, $J_1$ and $J_2$ appearing in~\eqref{4.26-1aNN2},~\eqref{4.26-1aNN4}\,~\eqref{4.26-1aNN5} are all negative. This is accomplished
following the next steps under the so-called \emph{\bf Conditions ($\Sigma$)}:
\begin{enumerate}
\item[Step 1.] First, fix the positive constants $\kappa_{1\ca 2}$ and $\kappa_{2\ca 2}$. 
\item[Step 2.] Choose next $\kappa_{2\ca 1}$ large such that
\[- \kappa_{2\ca1}\,\frac{p_0^* }{4\mu}+\co(1)  e^{2 \kappa_{1\ca 2}M^*}\le 0.\label{Condition1}\tag{$\Sigma_1$}\]
\item[Step 3.] Now, select $\kappa_{1\ca 1}\ge 1$ large enough so that
\[ - \,\frac{1}{2p_1^*}+ \frac{\co(1)}{\kappa_{1\ca 1}} e^{(\kappa_{2\ca1}+\kappa_{2\ca2})M^*} <0\,.\label{Condition2}\tag{$\Sigma_2$}\]
\item[Step 4.] Next, let $\delta_p^*$ be a positive small constant that satisfies  \[0<\delta_p^*<\min\left\{\frac{1}{2},\,\frac{\kappa_{2\ca1}}{\kappa_{1\ca1}}\right\}\,\label{Condition3}\tag{$\Sigma_3$}\]
Then, combining with the previous step, we immediately have 
\[
J_2:= - \kappa_{2\ca1}\,\frac{ p_0^* }{4\mu}+ \co(1)\le J_1:= - \kappa_{2\ca1}\,\frac{p_0^* }{4\mu}+\co(1)  W_1^*<- \kappa_{2\ca1}\,\frac{p_0^* }{4\mu}+\co(1)  e^{2 \kappa_{1\ca 2}M^*}< 0
\]
recalling that $W_1^*= \co(1)e^{  (\kappa_{1\ca1} \delta_p^*+ \kappa_{1\ca2})\cdot M^*}$ from~\eqref{W-unif-boundN}.
\item[Step 5.] Last, choose $\delta_0^*$  small enough, depending on $\kappa_{i\ca j}$ and  $\delta_p^*$, so that 
\[
-1+\co(1)\delta_0^*<-\frac{1}{2}
,\quad\text{and}\quad-\frac{1}{ 2p_1^*} +\co(1) \cdot W_2^* \delta_0^*p_1^*< 0
\label{Condition4}\tag{$\Sigma_4$}
\]
as well as both conditions~\eqref{:constrNewCoeff0}--\eqref{E:constrNewCoeff} remain valid. 
The size of $\delta_0^*$ may even become smaller according to other requirements of the analysis in this paper. Then, we immediately deduce from \eqref{Condition2}  and~\eqref{Condition4} that
$I_1<0 $ and $I_2<0$.
\end{enumerate}
Having that the terms $I_1$, $I_2$, $J_1$ and $J_2$ are all negative and combining with~\eqref{4.26-1aNN2},~\eqref{4.26-1aNN4}\,~\eqref{4.26-1aNN5} , the proof is complete.
However, for completeness, we could add a last step in \emph{\bf Conditions ($\Sigma$)}: Having Steps 1--5, we proceed to
\begin{enumerate}
\item[Step 6.] 
From \S~\ref{Ss:interactionTimes} the parameter $\kappa_\cg$ can be chosen to satisfy
\begin{equation}\label{S4.1conditionkappa2}\kappa_{\cg}> 2a\max_{i,j=1,2} \kappa_{i\ca j}\tag{$\Sigma_5$} \ .
\end{equation}
where $a$ is determined in \S~\ref{Ss:interactionTimes} and may depend on $\delta_0^*$.
\end{enumerate}
\end{proof}

More precisely, under \emph{\bf Conditions ($\Sigma$)}, we get
\be{4.26-1aNN6}
\begin{aligned}
E_{\alpha,1}+E_{\alpha,2} \leq & \co(1) (W_1^* +W_2^* \delta_p^*  \delta_0^*) \delta_p^* \cdot \varepsilon |\sC_{\alpha} | \;,
\end{aligned}
\ee
that yields estimate~\eqref{E:mainest} for a wave of $v(t)$ at $x_{\alpha}$ belonging to the first family. Let us make some comments:\\
$\bullet$ According to the above steps, the coefficient $\kappa_{\cg}$ is selected obeying~\eqref{S4.1conditionkappa} after one completes step 5 above, hence in step $6$. 
It is should be pointed out that $\kappa_{\cg}$ is not involved in $\co(1)$ in steps 1--5 of \emph{\bf Conditions ($\Sigma$)}.\\
$\bullet$ One can realize that the smallness of the factor $\delta_p^*$ is crucial to balance the positive contribution of $E_{\alpha,2}$ with the negative part of $E_{\alpha,1}$, i.e. obtain $I_1<0$ and at the same time to balance the positive contribution of $E_{\alpha,1}$  with the negative part of $E_{\alpha,2}$, i.e. obtain $J_1<0$ in Case a). In addition, for this same reason, we write $W_i$ as exponential functions of the linear combination of $\ca_{i,j}$, recall~\eqref{S4Wi}, and not as linear functions of $\ca_{i,j}$, that is the standard way this analysis has been done in the literature so far. Actually the exponential allows us to better compensate the gain and the loss.

\subsubsection{Waves of the second family}
\label{Ss:estimPhys2wNN}

Here, we derive estimate~\eqref{E:mainest} on the sum of the errors $E_{\alpha,1}+E_{\alpha,2}$ defined in~\eqref{E:error} when the wave of $v(t)$ present at $x_{\alpha}$ belongs to the second family, i.e.~$k_{\alpha}=2$.
The strategy is similar to the previous subsection: we will first provide a separate estimate of $E_{\alpha,1}$, $E_{\alpha,2}$, and then we will combine such estimates to estblish~\eqref{E:mainest}.
We shall adopt the notation given in~\eqref{GE:definitionslralpha},
dropping the superscript $\alpha$, and we will let $\sC_\alpha, \sR_\alpha$ denote the size of the
wave located at $x_\alpha$, measured in the original and Riemann coordinates, respectively (see 
\S~\ref{Ss:strengthnotation}). Having the restrictions on the set $K$ given in \eqref{:setKbd-1bound}, we have
\ba\label{bound-eta1-gamma-eta1+gamma-pNN2}
v_1^{\alpha,\ell}, \omega_1^{\alpha,\ell}, |\eta_1^{\alpha,\ell}| \le { \delta_0^*},\qquad 
|v_2^{\alpha,\ell}-1|,|v_2^{\alpha,\ell}-1-\eta_2^{\alpha,\ell}|, 
\le \delta_p^* \,,
\qquad  |\sC_\alpha|, |\eta_2^{\alpha,\ell}+\sC_\alpha|, |\eta_2^{\alpha,\ell}| \le 2\delta_p^*.
\ea

\paragraph{Estimate of $E_{\alpha,1}$ for waves of the second family}\label{P:rergqokngrqongrqqgrwgqrqgg}

Again we divide the analysis into two cases according to the sign of $\eta_{1}^{\alpha,r}$ and $\eta_{1}^{\alpha,\ell}$: a) treating the case when $\eta_{1}^{\alpha,r}$ and $\eta_{1}^{\alpha,\ell}$ have the same sign, which is the most relevant case, while b) treating the case when $\eta_{1}^{\alpha,r}$ and $\eta_{1}^{\alpha,\ell}$ have opposite sign.

\noindent
\textbf{Case a)} Assume that $\eta_{1}^{\alpha,r}$ and $\eta_{1}^{\alpha,\ell}$ have the same sign. Hence, by~\eqref{S4WiN}--\eqref{S4A1}, one has the relation
\bes
W_{1}^{\alpha,r}=W_{1}^{\alpha,\ell} \exp\left({\kappa_{1\ca 2}|\sR_{\alpha}|}\right)\;.
\ees
Observe that by strict hyperbolicity~\eqref{strict-hyp} 
there holds $\lambda_{1}^{\alpha,\ell}-\dot x_{\alpha}<-p_0^*/2$.
Then, relying on~\eqref{E:sqFNDKAVDUAall} 
and on the uniform bounds~\eqref{bound-eta1-gamma-eta1+gamma-pNN2} and~\eqref{E:equivalncestrengths}, 
 it follows
\ba
E_{\alpha,1} &= W_{1}^{\alpha,r}\left[(1-e^{- \kappa_{1\ca 2}|\sR_{\alpha}|})\big(\lambda_{1}^{\alpha,\ell}-\dot x_{\alpha}\big)
\big|\eta_{1}^{\alpha,\ell}\big|+\left(\eta_{1}^{\alpha,r}(\lambda_{1}^{\alpha,r}-\dot x_{\alpha})-\eta_{1}^{\alpha,\ell}(\lambda_{1}^{\alpha,\ell}-\dot x_{\alpha})\right)\sgn(\eta_1^\ell) \right] \notag
\\
& { \leq} 
W_{1}^{\alpha,r}\left[- (1-e^{- \kappa_{1\ca 2}|\sR_{\alpha}|})\frac{p_0^*}{2} \big|\eta_{1}^{\alpha,\ell}\big|
+\left|\eta_{1}^{\alpha,r}(\lambda_{1}^{\alpha,r}-\dot x_{\alpha})-\eta_{1}^{\alpha,\ell}(\lambda_{1}^{\alpha,\ell}-\dot x_{\alpha})\right|
\right]\notag
\\
& { \leq} 
W_{1}^{\alpha,r}\left[- \frac{\kappa_{1\ca 2}\,p_0^*}{2\mu} \cdot e^{-2\kappa_{1\ca 2}\mu \delta_p^* }\cdot
\big|\eta_{1}^{\alpha,\ell}\sC_{\alpha}\big|
 +\co(1)\left| (\eta_{2}^{\alpha,\ell}+\sC_{{\alpha}})\eta_{2}^{\alpha,\ell}\sC_{{\alpha}} \right|  (v_{1}^{\alpha,\ell})^{2}  \right]\notag
\ea
Now, from~\eqref{E:primsfwewgrqgerqpall2}, we have the relation%
\be{etaell+g=r}
\eta_{2}^{\alpha,\ell}+\sC_{{\alpha}}=\eta_{2}^{\alpha,r}(1\pm\co(1)\delta_0^2)
\ee
and employing this, we arrive at the estimate
\ba
E_{\alpha,1} &{ \leq} 
W_{1}^{\alpha,r}\left[- \frac{\kappa_{1\ca 2}\,p_0^*}{2\mu} \cdot e^{-2\kappa_{1\ca 2}\mu \delta_p^* }\cdot
\big|\eta_{1}^{\alpha,\ell}\sC_{\alpha}\big|
 +\co(1)\left| \eta_{2}^{\alpha,r}\eta_{2}^{\alpha,\ell}\sC_{{\alpha}} \right|  (v_{1}^{\alpha,\ell})^{2}  \right].
 \label{E:efavadvNn}
 \ea

\noindent
\textbf{Case b)} Assume that $\eta_{1}^{\alpha,r}$ and $\eta_{1}^{\alpha,\ell}$ have opposite signs. Then using the interaction-type estimate~\eqref{E:primsfwewgrqgerqpall}, we have
\be{E:funfenfenfeNn}
|\eta_{1}^{\alpha,r}|+|\eta_{1}^{\alpha,\ell}| = |\eta_{1}^{\alpha,r}-\eta_{1}^{\alpha,\ell}|
\leq 
 \co(1) \left| \left(\eta_{2}^{\alpha,\ell}+\sC_{{\alpha}}\right)\eta_{2}^{\alpha,\ell}\sC_{{\alpha}} \right|  (v_{1}^{\alpha,\ell})^{2}   \ .
\ee
Since the speeds $\lambda_{1}^{\alpha,\ell}$, $\lambda_{1}^{\alpha,r}$, $\dot x_{\alpha}$ 
are uniformly bounded, expression~\eqref{E:error} of $E_{\alpha,1}$ can be estimated as follows:
\be{E:boundq1changesignNN}
\begin{split}
E_{\alpha,1}\leq& \co(1)W_1^*\cdot(|\eta_{1}^{\alpha,r}|+|\eta_{1}^{\alpha,\ell}| ) 
\\
\leq &  \co(1) W_1^*\left| \left(\eta_{2}^{\alpha,\ell}+\sC_{{\alpha}}\right)\eta_{2}^{\alpha,\ell}\sC_{{\alpha}} \right|  (v_{1}^{\alpha,\ell})^{2} \\
\leq& \co(1) W_1^*\cdot\left|\eta_{2}^{\alpha,r}\eta_{2}^{\alpha,\ell}\sC_{{\alpha}} \right|  (v_{1}^{\alpha,\ell})^{2} \ .
\end{split}
\ee
using again bounds~\eqref{W-unif-boundN} and~\eqref{etaell+g=r}.

We thus conclude that in both cases
\be{E:boundq1changesignNN2}
\begin{split}
E_{\alpha,1}
\leq& \co(1) W_1^*\cdot\left|\eta_{2}^{\alpha,r}\eta_{2}^{\alpha,\ell}\sC_{{\alpha}} \right|  (v_{1}^{\alpha,\ell})^{2} \ .
\end{split}
\ee

\paragraph{Estimate of $E_{\alpha,2}$ for waves of the second family}\label{per:rggqrgrqgrqgrqgrq}
In this subsection, we estimate the error term $E_{\alpha,2}$ for waves of the second family dividing the analysis into three cases again according to the signs of $\eta_{2}^{\alpha,r}$, $\eta_{2}^{\alpha,\ell}$ and to the type of the wave at $x_\alpha$ as follows: a) treating the case when $\eta_{2}^{\alpha,r}$ and $\eta_{2}^{\alpha,\ell}$ have the same sign, b) treating the case when $\eta_{2}^{\alpha,r}$ and $\eta_{2}^{\alpha,\ell}$ have opposite signs and the wave at $x_{\alpha}$ is an entropy admissible $2$-shock, c) treating the case when $\eta_{2}^{\alpha,r}$ and $\eta_{2}^{\alpha,\ell}$ have opposite signs and the wave at $x_{\alpha}$ is $2$-rarefaction shock.

\noindent\textbf{Case 2a)} We consider the case that $\eta_{2}^{\alpha,r}$ and $\eta_{2}^{\alpha,\ell}$ have the same sign. Then, by~\eqref{S4WiN} one has the identity
\be{E:rguiregregregrerge2}
W_{2}^{\alpha,r}
=W_{2}^{\alpha,\ell} e^{- \kappa_{2\ca2}|\sR_{\alpha}| \sgn(\eta_{2}^{\alpha,\ell})}\;,
\ee
since $\ca_{2,1}^{\alpha,\ell}=\ca_{2,1}^{\alpha,r}$ and $\ca_{2,2}^{\alpha,\ell}=\ca_{2,2}^{\alpha,r}+|\sR_{\alpha}| \sgn(\eta_{2}^{\alpha,\ell})$.

Thanks to estimates in Lemma~\ref{L:stimegenerali2}, by the uniform boundedness of $W_2^{\alpha,r}$, with the compact domain in bounds~\eqref{bound-eta1-gamma-eta1+gamma-pNN} from Theorem~\ref{T:AS} and the equivalence relation~\eqref{E:equivalncestrengths}, we estimate
\begin{align}
E_{\alpha,2}= &
(W_{2}^{\alpha,r}-W_{2}^{\alpha,\ell})|\eta_{2}^{\alpha,\ell}|(\lambda_{2}^{\alpha,\ell}-\dot x_{\alpha})
+W_{2}^{\alpha,r}\left[|\eta_{2}^{\alpha,r}|(\lambda_{2}^{\alpha,r}-\dot x_{\alpha})-|\eta_{2}^{\alpha,\ell}|(\lambda_{2}^{\alpha,\ell}-\dot x_{\alpha})\right]
\notag \\
 \stackrel{\eqref{E:rguiregregregrerge2}}{=} & 
 W_{2}^{\alpha,r}\left[
 (1-e^{\kappa_{2\ca2} \sgn(\eta_{2}^{\alpha,\ell})\left|\sR_{\alpha}\right|})| \eta_{2}^{\alpha,\ell}|(\lambda_{2}^{\alpha,\ell}-\dot x_{\alpha})+
 	\sgn(\eta_2^{\alpha,\ell})\left(\eta_{2}^{\alpha,r}(\lambda_{2}^{\alpha,r}-\dot x_{\alpha})-\eta_{2}^{\alpha,\ell}(\lambda_{2}^{\alpha,\ell}-\dot x_{\alpha})\right)\right]\notag\,.
	\end{align}
Next, using~\eqref{bound-eta1-gamma-eta1+gamma-pNN2} and~\eqref{E:8.50all}, we write
$$
\lambda_{2}^{\alpha,\ell}-\dot x_{\alpha}=
\frac{v_{1}^{\alpha,\ell} (\eta_2^{\alpha,\ell}+\sC_\alpha) }{(v_2^{\alpha,\ell}-\eta_2^{\alpha,\ell})(v_2^{\alpha,\ell}+\sC_\alpha)}(1+\co(1)|v_1^\ell |) 
$$
since $0<(1-\delta_p^*)^2<(v_2^{\alpha,\ell}-\eta_2^{\alpha,\ell})(v_2^{\alpha,\ell}+\sC_\alpha)\le(1+\delta_p^*)^2<4$ for $\delta_p^*<\frac{1}{2}$ by~\eqref{:cvetagamma-1bound} and also using that $v_2^{\alpha,\ell}+\gamma_\alpha=v_2^{\alpha,r}\in(1-\delta_p^*,1+\delta_p^*)$
Hence, combining with~\eqref{E:sqFNDKAVDUAall}, we deduce
\begin{align}	
E_{\alpha,2} \le
 	&W_{2}^{\alpha,r}\Big[  (1-e^{\kappa_{2\ca2} \sgn(\eta_{2}^{\alpha,\ell})\left|\sR_{\alpha}\right|}) |\eta_{2}^{\alpha,\ell}|  \frac{v_{1}^{\alpha,\ell} (\eta_2^{\alpha,\ell}+\sC_\alpha) }{(v_2^{\alpha,\ell}-\eta_2^{\alpha,\ell})(v_2^{\alpha,\ell}+\sC_\alpha)}(1+\co(1)\delta_0^* )  \notag
	\\ 
		&\qquad\qquad\qquad\qquad\qquad\qquad
+\co(1) \left| \left(\eta_{2}^{\alpha,\ell}+\sC_{{\alpha}}\right)\eta_{2}^{\alpha,\ell}\sC_{{\alpha}} \right|  (v_{1}^{\alpha,\ell})^{2} \Big]
\notag 
\end{align}
Now, from~\eqref{etaell+g=r}, we note that 
$$\sgn(\eta_{2}^{\alpha,\ell}+\sC_{{\alpha}})=\sgn(\eta_{2}^{\alpha,r})$$
and hence by~\eqref{E:equivalncestrengths} and~\eqref{bound-eta1-gamma-eta1+gamma-pNN2}, we get the estimate
$$
(1-e^{\kappa_{2\ca2} \sgn(\eta_{2}^{\alpha,\ell})\left|\sR_{\alpha}\right|}) |\eta_{2}^{\alpha,\ell}|  (\eta_2^{\alpha,\ell}+\sC_\alpha)\le  -\kappa_{2\ca2} e^{- 2\kappa_{2\ca 2}\mu \delta_p^* } \left|\sR_{\alpha}\right|  \eta_{2}^{\alpha,\ell}\cdot  (\eta_2^{\alpha,\ell}+\sC_\alpha)
$$
thus  from~\eqref{etaell+g=r}, we deduce
\begin{align}\label{E:case2.2.aNew}
E_{\alpha,2}
 {\le} 
 	&W_{2}^{\alpha,r}\left[ -\kappa_{2\ca2} e^{- 2\kappa_{2\ca 2}\mu \delta_p^* } \left|\sR_{\alpha}\right| \frac{v_{1}^{\alpha,\ell}  \eta_{2}^{\alpha,\ell}\eta_2^{\alpha,r}  }{(v_2^{\alpha,\ell}-\eta_2^{\alpha,\ell})(v_2^{\alpha,\ell}+\sC_\alpha)}(1+\co(1)\delta_0^* )
  +\co(1) \left| \eta_{2}^{\alpha,r}\eta_{2}^{\alpha,\ell}\sC_{{\alpha}} \right|  (v_{1}^{\alpha,\ell})^{2} \right]
\notag \\
 {\le} &W_{2}^{\alpha,r}\kappa_{2\ca2} e^{-2\kappa_{2\ca 2}\mu \delta_p^*}\frac{ (-1+\co(1)\delta_0^* ) }{ ({p_{1}^*)}^2\mu} 
 \left| \eta_{2}^{\alpha,r}\eta_{2}^{\alpha,\ell}\sC_{{\alpha}} \right|  v_{1}^{\alpha,\ell} +\co(1) W_2^*\delta_0^*\left| \eta_{2}^{\alpha,r}\eta_{2}^{\alpha,\ell}\sC_{{\alpha}} \right|  v_{1}^{\alpha,\ell}
 \ .\nonumber\\
  {\le} & 1\cdot \kappa_{2\ca2} e^{-2\kappa_{2\ca 2}\mu \delta_p^*}\frac{ (-\frac{1}{2}) }{ ({p_{1}^*)}^2\mu} 
 \left| \eta_{2}^{\alpha,r}\eta_{2}^{\alpha,\ell}\sC_{{\alpha}} \right|  v_{1}^{\alpha,\ell} +\co(1) W_2^*\delta_0^*\left| \eta_{2}^{\alpha,r}\eta_{2}^{\alpha,\ell}\sC_{{\alpha}} \right|  v_{1}^{\alpha,\ell}
 \ .
 \end{align}
Here, we also employed bounds $0\le (p_0^*)^2\le(v_2^{\alpha,\ell}-\eta_2^{\alpha,\ell})(v_2^{\alpha,\ell}+\sC_\alpha)\le (p_{1}^*)^2$, $W_{2}^{\alpha,r}\ge 1$ and take $\delta_0^*$ small enough so that $-1+\co(1)\delta_0^* <-\frac{1}{2}$.

\textbf{Case 2b)} 
In this case, we assume that $\eta_{2}^{\alpha,r}$ and $\eta_{2}^{\alpha,\ell}$ have opposite signs and the $\alpha$-wave is an admissible $2-$shock. This is equivalent to 
$\eta_{2}^{\alpha,\ell}<0<\eta_{2}^{\alpha,r}$. By interaction-type estimate~\eqref{E:primsfwewgrqgerqpall2}, we have
\bes
|\eta_{2}^{\alpha,r}|+|\eta_{2}^{\alpha,\ell}| = |\eta_{2}^{\alpha,r}-\eta_{2}^{\alpha,\ell}| \in \left(\frac{1}{2}|\sC_{\alpha}|, 2|\sC_{\alpha}|\right)\;.
\ees
using the smallness of the component $v_1^{\alpha,\ell}$; $0<v_1^{\alpha,\ell}<\delta_0^*$. Now, by estimates~\eqref{E:sstimaetarbase2all}, the error term $E_{\alpha,2}$ defined in~\eqref{E:error} can be estimated as 
\begin{align}
E_{\alpha,2}
{\le} 
&
W_{2}^{\alpha,r} |\eta_{2}^{\alpha,r}|\left[ \frac{v_1^{\alpha,\ell}\eta_2^{\alpha,\ell}}{ (v_2^{\alpha,\ell}-\eta_2^{\alpha,\ell})(v_2^{\alpha,\ell}+\sC_\alpha)} +\co(1) |\eta_2^{\alpha,\ell} |(v_1^{\alpha,\ell})^2 
\right] \notag\\
&-W_{2}^{\alpha,\ell} |\eta_{2}^{\alpha,\ell} | \left[ \frac{v_1^{\alpha,\ell}(\eta_2^{\alpha,\ell}+\sC_\alpha)}{ (v_2^{\alpha,\ell}-\eta_2^{\alpha,\ell})(v_2^{\alpha,\ell}+\sC_\alpha)} 
-\co(1) |\eta_2^{\alpha,\ell}+\sC_\alpha |(v_1^{\alpha,\ell})^2 
\right]\notag\;.
\end{align}
Next, recalling that $$0<(p_0^*)^2\le(v_2^{\alpha,\ell}-\eta_2^{\alpha,\ell})(v_2^{\alpha,\ell}+\sC_\alpha)\le (p_{1}^*)^2$$ given by the compact domain~\eqref{bound-eta1-gamma-eta1+gamma-pNN}, expression~\eqref{etaell+g=r} and the sign relations
$$
\sgn(\eta_2^{\alpha,\ell}+\sC_\alpha)=\sgn(\eta_2^{\alpha,r})=+1,\qquad \eta_{2}^{\alpha,\ell}<0<\eta_{2}^{\alpha,r}\;,
$$
we arrive at
\begin{align}\label{E:case22b}
E_{\alpha,2}
{\le} &(W_{2}^{\alpha,r}+W_{2}^{\alpha,\ell}) 
	\cdot \left(-\frac{1}{(p_1^*)^2}+\co(1)\delta_0^*\right)\cdot  \left| \eta_{2}^{\alpha,r}\eta_{2}^{\alpha,\ell} \right|  v_{1}^{\alpha,\ell}
\notag \\
{\le} &  2\left(-\frac{1}{(p_1^*)^2}+\co(1)\delta_0^*\right)\cdot  \left| \eta_{2}^{\alpha,r}\eta_{2}^{\alpha,\ell} \right|  v_{1}^{\alpha,\ell}\;,
\end{align}
choosinf $\delta^*_0$ small enough so that $-\frac{1}{(p_1^*)^2}+\co(1)\delta_0^* <0$ and using $W_i^{\alpha,r/\ell}\ge 1$, by definition.

\textbf{Case 2c)}
 In the last case, $\eta_{2}^{\alpha,r}$ and $ \eta_{2}^{\alpha,\ell}$ have opposite signs and the $\alpha$-wave is a $2-$rarefaction,
 which is replaced with a jump of the same size connecting states along a Hugoniot curve. This means that $\eta_{2}^{\alpha,r}\leq0\leq \eta_{2}^{\alpha,\ell}$ and $0<\sC_{\alpha}\leq\varepsilon$. As before in the previous case, by interaction-kind estimate~\eqref{E:primsfwewgrqgerqpall2} and for small enough $\delta_0^*$, 
 we have
\be{E:gregggfgg}
|\eta_{2}^{\alpha,r}|+|\eta_{2}^{\alpha,\ell}| = |\eta_{2}^{\alpha,r}-\eta_{2}^{\alpha,\ell}| 
\leq 2|\sC_{\alpha}|\leq 2\varepsilon \ .
\ee
Substituting this into~\eqref{E:sstimaetarbase2all}, it yields
\[
|\lambda_{2}^{\alpha,r}-\lambda_{2}^{\alpha,\ell}|\leq|\dot x_{\alpha}-\lambda_{2}^{\alpha,\ell}|+
|\dot x_{\alpha}-\lambda_{2}^{\alpha,r}|
\leq
\co(1) (|\eta_{2}^{\alpha,\ell}|+|\eta_{2}^{\alpha,\ell}+\sC_\alpha| )( v_{1}^{\alpha,\ell})^2 \leq
\co(1) \varepsilon (v_{1}^{\alpha,\ell} )^2\;.
\]
In view of the above, the error $E_{\alpha,2}$ defined in~\eqref{E:error} can be estimated as
\begin{align}\label{E:case2.2.c}
E_{\alpha,2}&\le W_2^{\alpha,r} |\eta_2^{\alpha,r}-\eta_2^{\alpha,\ell}| |\lambda_2^{\alpha,r}-\dot{x}_\alpha| 
+(W_2^{\alpha,r}+W_2^{\alpha,\ell}) |\eta_2^{\alpha,\ell}| |\lambda_2^{\alpha,r}-\dot{x}_\alpha| + W_2^{\alpha,\ell} |\eta_2^{\alpha,\ell}| |\lambda_2^{\alpha,r}-\lambda_2^{\alpha,\ell}| \notag\\
&\le \co(1) W_2^{*} \; \varepsilon (v_{1}^{\alpha,\ell})^2 |\sC_{\alpha}|\;,
\end{align}
recalling~\eqref{W-unif-boundN}.

\paragraph{Derivation of~\eqref{E:mainest} for 2-waves}

We combine now estimate~\eqref{E:boundq1changesignNN2} for $E_{\alpha,1}$ with the three cases presented in Paragraph~\ref{per:rggqrgrqgrqgrqgrq} for $E_{\alpha,2}$ to establish~\eqref{E:mainest} when the wave at $x_{\alpha}$ is a 2-wave.

\noindent
 {\bf(Case 2a)} In this case, we have $\eta_{2}^{\alpha,r}\,\cdot\eta_{2}^{\alpha,\ell}\ge 0$. Adding~\eqref{E:boundq1changesignNN2} and~\eqref{E:case2.2.aNew}, we get
\be{E:case2.2.anew}
E_{\alpha,1} +E_{\alpha,2} 
\leq
\left(-\kappa_{2\ca2}e^{-{\color{blue} 2}\kappa_{2\ca 2}\mu \delta_p^*}\frac{ 1}{2 ({p_{1}^*)}^2\mu} +(W_1^{*}+W_{2}^{*})\co(1)\delta_0^*\right)
 \left| \eta_{2}^{\alpha,r}\eta_{2}^{\alpha,\ell}\sC_{{\alpha}} \right|  v_{1}^{\alpha,\ell}  \ .
\ee
since $W_{2}^{\alpha,r}\ge 1$.

\noindent
{\bf(Case 2b)} Here, we recall that $\eta_{2}^{\alpha,\ell}<0<\eta_{2}^{\alpha,r}$. Summing up~\eqref{E:boundq1changesignNN2} and~\eqref{E:case22b}
\be{E:case2.2.bnew}
E_{\alpha,1} +E_{\alpha,2} 
\leq \left(-\frac{2}{(p_1^*)^2}+\co(1)\delta_0^*(W_1^*+1)\right)\cdot  \left| \eta_{2}^{\alpha,r}\eta_{2}^{\alpha,\ell} \right|  v_{1}^{\alpha,\ell}\;,
 \ee
%
%

\noindent
{\bf(Case 2c)} In the last case, we deal with $\eta_{2}^{\alpha,r}<0<\eta_{2}^{\alpha,\ell}$. Combining estimates~\eqref{E:boundq1changesignNN2} and~\eqref{E:case2.2.c} and taking into account~\eqref{E:gregggfgg}, we arrive at
\begin{align}\label{E:case2.2.cnew}
E_{\alpha,1}+E_{\alpha,2} &\leq
\co(1) (   W_1^{*}+W_2^{*}) \; \varepsilon (v_{1}^{\alpha,\ell})^2 |\sC_{\alpha}|\ .
\end{align}
In view of the above estimates, we thus deduce that~\eqref{E:mainest} holds by the smallness of $\delta_0^*$. Indeed, in Step $5$ in Conditions $(\Sigma$), (see Proposition~\ref{PropCond} in Paragraph~\ref{P4.2.1.3}), one can further shrink, if needed, $\delta_0^*$ in~\eqref{Condition4} depending on  $\delta_p^*$ and $\kappa_{2\ca 2}$ -- that are already fixed from the previous steps -- so that estimates ~\eqref{E:case2.2.anew}--\eqref{E:case2.2.bnew} are all negative.

\subsection{Analysis at time steps}
\label{Ss:timestep}
The aim in this section is to estimate the change of the function $\Phi_0(u(t),v(t)) $ across a time step $t=t_k$, when $u$ and $v$ are approximate solutions to~\eqref{S1system} constructed via the front-tracking algorithm in conjunction with the operator splitting method. Throughout this section, we denote by $u$ and $v$ the piecewise constant approximate solutions to the non-homogeneous system~\eqref{S1system} as described in Section~\ref{Ss:Rsolv-appsol}, with the update of the states given at~\eqref{updated} and the source denoted by $g(\theta)=((\theta_2-1)\theta_1,0)$ where $\theta$ stands for a generic state $\theta=(\theta_1,\theta_2) $.

Fix an index $k$ and set $t_k=k\Delta t$. For our convenience, we use the following notation
\[
u^+(x)\doteq u(x,t_k+) , \ u^-(x)\doteq u(x,t_k-) , \ \eta_i^+(x)\doteq\eta_i(x,t_k+) , \ \eta_i^-(x)\doteq\eta_i(x,t_k-)
\]
for the values before and after the update at the time step $t_k$. Similarly, we use 
\[
v^{\pm}(x) ,\ W_i^\pm(x) , \ \ca_i^\pm(x) , \ Q(u)^{\pm} , \ Q(v)^{\pm} , \ \cg^{\pm}(u) , \ \cg^{\pm}(v) \text{ etc.}
\]
for the corresponding values as $t\to t_k\pm$. Using this notation, we have
\begin{equation}
\label{S:etapm-def}
v^{\pm}(x)=\SC{2}{\eta_2^{\pm}}{ \SC{1}{\eta_1^{\pm}}{u^{\pm}(x)}}
\end{equation}
and
\begin{equation}
\label{S:vupm-def}
v^{+}(x)=v^{-}(x)+\Delta t g(v^-(x)),\qquad u^+(x)=u^{-}(x)+\Delta t g(u^-(x)) \;.
\end{equation}

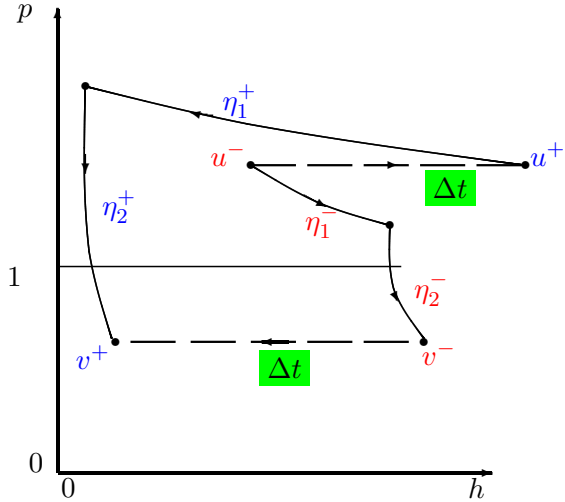
\begin{figure}[htbp]
{\centering \scalebox{1}{\input{timesttalk.tex} } \par}
\caption{The shock curves connecting the states of $u$ with $v$ before and after a time step of size $\Delta t$. \label{S4.5 :fig1}}
\end{figure}
The following lemma provides an estimate in the change of the strengths $\eta=(\eta_1,\eta_2)$ across the time step $t_k$.

\begin{lemma}\label{S:ts-lemma}
Let $u(t,x)$ and $v(t,x)$ be two approximate solutions to~\eqref{S1system} in $[0,\infty)\times\mathbb{R}$. Then there exists $s=\Delta t>0$ such that 
\be{S:ts-eta12}
|\eta_1^+-\eta_1^-|+|\eta_2^+-\eta_2^-|\le \co(1)\,\Delta t \left( |\eta_1^-|+|\eta_2^-| \right)
\ee
and 
\be{S-tsGbd}
\cg^{+}(u)\le (1+\co(1)\Delta t)\cg^{-}(u),\qquad
\cg^{+}(v)\le (1+\co(1)\Delta t)\cg^{-}(v)\;,
\ee
at every time step $t_k=k\Delta t$.
\end{lemma}
\begin{proof}
Given the state $u^-(x)=(u_1^-,u_2^-)$, the time step $\Delta t$ and the strengths $\eta^-=(\eta_1^-,\eta_2^-)$, one can determine $\eta^+=(\eta_1^+,\eta_2^+)$ relying on~\eqref{S:etapm-def}--\eqref{S:vupm-def}. Hence, by considering the independent variables $(u_1^-,u_2^-,\eta_1^-,\eta_2^-,\Delta t)$, we define the functional
\bes
\Psi^{*}(h_u^-,p_u^-,\eta_1^-,\eta_2^-,\Delta t)\doteq (\eta_1^+-\eta_1^-,\eta_2^+-\eta_2^-)\;.
\ees
It is easy to verify that
\bes
\Psi^{*}(u_1^-,u_2^-,\eta_1^-,\eta_2^-,0)=\Psi^{*}(u_1^-,u_2^-,0,0,\Delta t)=(0,0)
\ees
Hence, by Lemma~\ref{E:rewgwgeq}, 
we arrive immediately at~\eqref{S:ts-eta12}.
Relying on estimate~\eqref{S:ts-eta12} and a proof similar to the one of Lemmas $2.1$ and $2.2$ in~\cite{AG}, we obtain~\eqref{S-tsGbd}. Let us point out that the derivation of~\eqref{S-tsGbd} can be established following the aforementioned work in~\cite{AG} since it relies on the smallness of $\Delta t$ and the lack of genuine nonlinearity or linear degeneracy or the size of the initial data are not relevant at this point. One should only check the terms in $\cq$ with the weights $\omega_{\alpha,\beta}$ that are not present in~\cite{AG}. Indeed, one can follow the proof of estimate (2.14) in~\cite{AG} to show
$$
\omega_{\alpha,\beta}^+-\omega_{\alpha,\beta}^-=\co(1)\Delta t,\qquad {\mathcal{Q}_{hh}}^+-{\mathcal{Q}_{hh}}^-= \co(1) \Delta t {(V^-)}^2
$$
for each approximate solution $u$ and $v$. Combining the definition~\eqref{S3G} of the Glimm functional $\mathcal{G}$ and Lemma $2.2$ in~\cite{AG}, the proof of~\eqref{S-tsGbd} is complete.
\end{proof}

 The aim now is to estimate the change of the functional $\Phi_0(u,v)$ across $t=t_k$.

\begin{lemma} Let $u$ and $v$ be two approximate solutions to~\eqref{S1system}. Then there exists $s=\Delta t>0$ such that
\ba\label{S:ts-Phi2}
\Phi_0(u(t_k+),v(t_k+))-\Phi_0(u(t_k-),v(t_k-))\le\co(1) \Delta t \Phi_0(u(t_k-),v(t_k-))\;,
\ea
for every $k=0,1,2,\dots$.
\end{lemma}
\begin{proof}
The aim now is to estimate the change of the functional $\Phi_0(u,v)$ across $t=t_k$. By definition,
\be{S:ts-Phinew}
\Phi_0(u^+,v^+)-\Phi_0(u^-,v^-)=\sum_{i=1}^2\int_{-\infty}^\infty\left[|\eta_i^+(x)|W_i^+  \cdot e^{\kappa_{\cg} \cg^+} - |\eta_i^-(x)|W_i^-  \cdot e^{\kappa_{\cg} \cg^-}
\right] dx
\ee
where $ \cg^{\pm}:=\cg^{\pm}(u)+\cg^{\pm}(v)$
and we expand the above integrand as
\ba\label{S4.3term}
|\eta_i^+(x)|W_i^+(x) \cdot e^{\kappa_{\cg} \cg^+}- |\eta_i^-(x)|W_i^-(x) \cdot e^{\kappa_{\cg} \cg^-}
&
=e_i^-(x) \left[\eta_i^+(x) e^{\delta_i}-\eta_i^-(x)\right]\nonumber\\
&
=e_i^-(x) \left[\eta_i^-(x) (e^{\delta_i}-1)+(\eta_i^+(x)  -\eta_i^-(x)  ) e^{\delta_i}\right]
\ea
where 
$$
e_i^-(x):= \exp\left(\kappa_{i\ca 1}\ca_{i,1}^-(x)+\kappa_{i\ca 2}\ca_{i,2}^-(x) +\kappa_{\cg} \cg^-\right)
$$
$$
e^{\delta_i}:=\exp\left(\delta_i\right),\qquad \delta_i(x):=\exp\left(\kappa_{i\ca 1}\Delta\ca_{i,1}(x)+\kappa_{i\ca 2}\Delta\ca_{i,2}(x) +\kappa_{\cg} \Delta\cg\right)
$$
for $i=1,2$ and as usual $\Delta\cdot$ denotes the change across $t_k$, i.e. $\Delta\ca_{i,1}(x)=\ca_{i,1}^+(x)-\ca_{i,1}^-(x)$.

Now the aim is to estimate  term~\eqref{S4.3term}. Using that
the functional $\mathcal{G}$ is uniformly bounded for all times for both solutions $u,\,v\in\mathcal{D}_0^*$ by a positive constant, we first observe that
\be{S:ts-Abdnew}
e_i^-(x)\le\co(1),\qquad \ca_{i,j}^{\pm}(x)\le \co(1)\cg^\pm= \co(1)(\cg^{\pm}(u)+\cg^{\pm}(v))\le \co(1)\qquad i,\,j=1,2 \;,
\ee
and then we examine two possible cases: Either (i) $\eta_i(x)^- \eta_i(x)^+\le 0$ or (ii) $\eta_i(x)^- \eta_i(x)^+> 0$ for each family $i=1,2$.

\noindent
{\bf Case (i)} If $\eta_i(x)^- \eta_i(x)^+\le 0$, then by Lemma~\ref{S:ts-lemma} we get
\bes
|\eta_i^-(x)|+|\eta_i^+(x)|=|\eta_i^-(x)-\eta_i^+(x)|\le\co(1) \Delta t (|\eta_1^-(x)|+|\eta_2^-(x)|).
\ees
Combining~\eqref{S:ts-Abdnew} with the above, we arrive at
\ba\label{S:ts-I1new}
|\eta_i^+(x)|W_i^+(x) \cdot e^{\kappa_{\cg} \cg^+}- |\eta_i^-(x)|W_i^-(x)\cdot e^{\kappa_{\cg} \cg^-}&\le \co(1) 
([|\eta_i^+(x)| +|\eta_i^-(x)|)\nonumber\\
&\le \co(1) \Delta t(|\eta_1^-(x)|+|\eta_2^-(x)|)\;.
\ea
\noindent
{\bf Case (ii)} If $\eta_i(x)^- \eta_i(x)^+ > 0$, then we claim that
\be{S:ts-Abd2new}
\Delta \ca_{i,j}(x)=\ca_{i,j}^{+}(x)-\ca_{i,j}^{-}(x)\le \co(1)\Delta t ,\qquad j=1,\,2\;.
\ee
The proof of the above claim follows by a simpler argument of the one establishing (4.37) in~\cite[Section 4]{AG} and taking into account the presence of the factors $|p-1|$ in $\ca_{1,1}$. The argument here is slightly simpler since non-physical fronts do not appear in our approximate scheme. For the convenience of the reader, we present the argument. Let $\sR_{\alpha} ^{-}$ be a wave of either $u$ or $v$ at the point $x_{\alpha}$ before the time step $t_k$ and $\sR_{\alpha,j}^{+}$ a wave of the $j$-family at the same location $x_{\alpha}$ after the time step $t_k$. We call a newly generated wave $\sR_{\alpha,j}^{+}$ of the $j$- family at the point $x_{\alpha}$ if $\sR_{\alpha} ^{-}$ is not present in $\ca_{i,j}^{-}$, but $\sR_{\alpha,j}^{+}$ is present in $\ca_{i,j}^{+}$. Now let $\sR_{\alpha,j}^{+}$ be a wave front present in $\ca_{i,j}^+(x)$, then there are two possibilities:
\begin{enumerate}
\item[(a)] $\sR_{\alpha,j}^{+}$ is not a newly generated wave. Hence $\sR_{\alpha} ^{-}$ is present in $\ca_{i,j}^-(x)$ and therefore, 
\bes
|\sR_{\alpha,j}^{+}-\sR_{\alpha}^-|\le \co(1)\Delta t|\sR_{\alpha}^-|
\ees
by~\eqref{E:timestepest1},~\eqref{E:timestepest4}. Observe that both $\sR_{\alpha,j}^{+}$ and $\sR_{\alpha} ^{-}$ are waves of the same family. Furthermore, if $j=1$ then the $p$ components of the left state of both fronts, $\sR_{\alpha,1}^{+}$ and $\sR_{\alpha} ^{-}$, are equal using the update~\eqref{updated} at $t=t_k$. Hence,
\bes
|p_{\alpha,1}^{l,+}-1| |\sR_{\alpha,1}^{+}|-|p_{\alpha}^{l,-}-1| | \sR_{\alpha}^-|= |p_{\alpha}^{l,-}-1| ( |\sR_{\alpha,1}^{+}|-| \sR_{\alpha}^-| ) \le \co(1)\Delta t|\sR_{\alpha}^-|\;.
\ees
These are terms that may appear in $\Delta\ca_{1,1}$.
\item[(b)] $\sR_{\alpha,j}^{+}$ is a newly generated wave. So now,
\bes
|\sR_{\alpha,j}^{+}|=|\sR_{\alpha,j}^{+}-0|\le \co(1)\Delta t|\sR_{\alpha}^-|
\ees
by~\eqref{E:timestepest2},~\eqref{E:timestepest3}. If $j=1$, we also have $|p^{l,+}_\alpha-1| |\sR_{\alpha,1}^{+}|\le \co(1)\Delta t|\sR_{\alpha}^-|$ since the factor $|p-1|$ is uniformly bounded.
\end{enumerate}
Now bound~\eqref{S:ts-Abd2new} follows immediately by Cases (a)-(b) for both $i=1,2$ and it implies 
$e^{\delta_i}-1\le\delta_i \le\co(1)\Delta t$. Having~\eqref{S:ts-Abdnew},~\eqref{S:ts-Abd2new} and~\eqref{S:ts-eta12}, we get
\ba\label{S:ts-I2new}
|\eta_i^+(x)|W_i^+(x) \cdot e^{\kappa_{\cg} \cg^+}- |\eta_i^-(x)|W_i^-(x)\cdot e^{\kappa_{\cg} \cg^-} &\le\co(1)\left[|\eta_i^-(x)| {\delta_i}+|\eta_i^+(x)  -\eta_i^-(x)  | e^{\delta_i}\right]\nonumber\\
&\le\co(1) \,\Delta t (|\eta_1^-(x)|+|\eta_2^-(x)|)\;.
\ea
Combining~\eqref{S4.3term},\,\eqref{S:ts-I1new},~\eqref{S:ts-I2new} with~\eqref{S:ts-Phinew} , we arrive at
\ba
\Phi_0(u(t_k+),v(t_k+))-\Phi_0(u(t_k-),v(t_k-))&\le\co(1) \Delta t \int_{-\infty}^\infty (|\eta_1^-(x)|+|\eta_2^-(x)|) dx\notag\\
&\le\co(1) \Delta t \Phi_0(u(t_k-),v(t_k-))\;,
\ea
for sufficiently small $s=\Delta t$, where we use that $W_i\ge 1$. The proof is complete.
\end{proof}

%
%
\subsection{Stability of the functional $\Phi_z$ - Proof of Theorem~\ref{stability-Phy}-(i)}
\label{S5.4}

In order to complete the proof of Theorem~\ref{stability-Phy} it remains to establish the statement (i).
Let $u$, $v$ be two approximate solutions of the homogeneous system~\eqref{S1system-hom}
constructed as described in subsection~\ref{Ss:Rsolv-appsol}.
We shall extend here the estimate~\eqref{Phi0est2} 
on $\Phi_0(u,v)$
 to the general functional $\Phi_z(u,v)$, when $z\neq 0$ is
 an arbitrary piecewise constant function,
 that takes values in the compact set~\eqref{E:im-z},
 and satisfies~\eqref{z-bound-1} with $\sigma$ as in~\eqref{sigma-def}.
 Recall that the function $z$ affects the definition of $\Phi_z(u,v)$
 in~\eqref{E:vu},\,\eqref{UpsilonNew}, both through the value of the waves 
 $\eta_{i}$, $i=1,2$, that connect $u$ and $v+z$ via~\eqref{E:vu}, 
 and through the weights $W_i$, which depend on the sign of $\eta_i$.
 
 We first observe that, since $u, v$ are both approximate solutions to the homogeneous system~\eqref{S1system-hom},
the map $t\mapsto \Phi_\er(u(t),v(t))$ is continuous in $L^1$ except at times of interaction of the waves of $v$ or $u$.
Moreover, 
as in \S~\ref{Ss:interactionTimes},
one can verify that at interaction times for $u$ or $v$, the map $t\mapsto\Phi_\er(u(t),v(t))$ is decreasing. 
Indeed, this follows immediately by the observation that, at an interaction time $t=\tau$,
one has $\eta_i(x,\tau+)=\eta_i(x,\tau-)$ for all points $x$ where none of $u(\tau)$, $v(\tau)$
or $\er$ has a jump.
Hence, at such points~$x$ the value of $W_i(\tau\pm,x)$ depends only on the strengths and left states 
of the waves in $u$ and $v$ and is not affected by the presence of $\er$.
Therefore, employing the same analysis in \S~\ref{Ss:interactionTimes}, we deduce
%
\begin{equation}
\label{phi-est-z}
\Phi_\er(u(\tau{+}),v(\tau{+}))\le \Phi_\er(u(\tau{-}),v(\tau{-})).
\end{equation}

Next, away from interaction times of $v$ or $u$, 
relying also on~\eqref{unifbvbound},~\eqref{V-totvar-est},~\eqref{sigma-def}, we will prove that
\begin{equation}
\label{E:phier-est21}
\ddn{t}{}\Phi_\er(u(t),v(t))\le\co(1)\big(\eps V(u(t))+(\eps+\sigma) V(v(t))+\sigma\big)\le\co(1)(\eps+\sigma)\,.
\end{equation}
Integrating~\eqref{E:phier-est21} between two interaction times and relying on~\eqref{phi-est-z},
we thus derive the estimate~\eqref{Phi0est1}, proving Theorem~\ref{stability-Phy}-(i).

In order to establish the first estimate in~\eqref{E:phier-est21}, we write as in \S~\ref{Ss:nointeractiontimesN} the derivative
\begin{equation}
\label{E:phier-est21new}
\ddn{t}{}\Phi_\er(u(t),v(t))=\left( \sum_{\alpha\in \cj} \sum_{i=1}^{2} E_{\alpha,i}^\er\right) e^{\kappa_{\cg} (\cg(u)+\cg(v)) }
\end{equation}
where
\be{E:errorz}
E_{\alpha,i}^\er \doteq W_{i}^{\alpha,r}|\eta_{i}^{\alpha,r}|(\lambda_{i}^{\alpha,r}-\dot x_{\alpha})- W_{i}^{\alpha,\ell}|\eta_{i}^{\alpha,\ell}|(\lambda_{i}^{\alpha,\ell}-\dot x_{\alpha})
\qquad
\alpha\in \cj
\ .
\ee
Here $\cj=\cj(u)\cup\cj(v)\cup\cj(\er)$ and the other quantities $\eta_i^{\alpha,\ell/r}$, $W_i^{\alpha,\ell,r}$, $\lambda_i^{\alpha,\ell,r}$ are used as in \S~\ref{Ss:nointeractiontimesN} taking into account that here $u$ and $v+\er$ are connected via~\eqref{E:vu}.
The goal here is to prove that 
\begin{align}
\label{E:mainestz-z1} E_{\alpha,1}^\er+E_{\alpha,2}^\er & \leq \co(1) \varepsilon |\sC_{\alpha}|\,, \qquad \forall \alpha\in \cj(u)\\
\label{E:mainestz-z2} E_{\alpha,1}^\er+E_{\alpha,2}^\er & \leq \co(1) |z(x_\alpha+)-z(x_\alpha-) |\,, \qquad \forall \alpha\in \cj(\er)\\
\label{E:mainestz-z3} E_{\alpha,1}^\er+E_{\alpha,2}^\er & \leq \co(1) \left[\varepsilon+\sigma\right] |\sC_{\alpha}|\,, \qquad \forall \alpha\in \cj(v)
\end{align}
where $\sC_\alpha$ is the strength of the $\alpha$-wave at $x_\alpha$. Estimate~\eqref{E:mainestz-z1} follows immediately from the analysis in \S~\ref{Ss:nointeractiontimesN} exchanging $v$ with $v+\er$ in the process of proving~\eqref{E:mainest}. Hence, it remains to prove~\eqref{E:mainestz-z2} and~\eqref{E:mainestz-z3}. 

Given $\alpha\in \cj(\er)$, we have for $i=1,2$
\be{S5.4 etadiffz} |\eta_i^{\alpha,r}-\eta_i^{\alpha,\ell} |\le\co(1) |z(x_\alpha+)-z(x_\alpha-) |,\qquad | \eta_i^{\alpha,r}\lambda_i^{\alpha,r}-\eta_i^{\alpha,\ell} \lambda_i^{\alpha,\ell} |\le\co(1) |z(x_\alpha+)-z(x_\alpha-) |
\ee
by Lipschitz continuity. Now, if $\eta_i^{\alpha,\ell}\,\eta_i^{\alpha,r}\le 0$, then 
\be{S5.4 etadiffz2} E_{\alpha,i}^\er\le\co(1)\left( |\eta_i^{\alpha,r}|+|\eta_i^{\alpha,\ell} |\right)=\co(1)|\eta_i^{\alpha,r}-\eta_i^{\alpha,\ell} | \le\co(1) |z(x_\alpha+)-z(x_\alpha-) |
\ee
from~\eqref{E:errorz} and~\eqref{S5.4 etadiffz}. On the other hand, if $\eta_i^{\alpha,\ell}\,\eta_i^{\alpha,r}\ge 0$, then using that $W_i^{\alpha, r}=W_i^{\alpha,\ell}$ when $\alpha\in \cj(\er)$, we write
\begin{align}\label{S5.4 etadiffz3}
E_{\alpha,i}^\er&= W_{i}^{\alpha,r}\left[ ( |\eta_i^{\alpha,\ell} |-|\eta_i^{\alpha,r} |) \dot x_{\alpha} + |\eta_i^{\alpha,r} | \lambda_{i}^{\alpha,r}- |\eta_i^{\alpha,\ell}| \lambda_{i}^{\alpha,\ell}\right]\nonumber\\
&\le\co(1)\left[ \hat{\lambda}|\eta_i^{\alpha,r}-\eta_i^{\alpha,\ell} |+ | \eta_i^{\alpha,r}\lambda_i^{\alpha,r}-\eta_i^{\alpha,\ell} \lambda_i^{\alpha,\ell} |\right]\nonumber\\
& \le\co(1) |z(x_\alpha+)-z(x_\alpha-) |
\end{align}
from~\eqref{S5.4 etadiffz} again. Thus, the error estimate~\eqref{E:mainestz-z2} follows from~\eqref{S5.4 etadiffz2} and~\eqref{S5.4 etadiffz3}.

Last, let $\alpha\in \cj(v)$. Using notation similar to \S~\ref{Ss:nointeractiontimesN}, we consider a jump in $v$ at $x_\alpha$ of the family $k_\alpha$ connecting $v^{\alpha,\ell}=v(x_\alpha-)$ with $v^{\alpha,r}=v(x_\alpha+)$ and which is of strength $\sC_\alpha$. As before, we assume that this jump is along a shock and denote the quantities
\ba
&v^{\alpha,r}=\SC{k_{\alpha}}{\sC_{\alpha}}{v^{\alpha,\ell}} \ , 
&& \dot x_{\alpha}=\lambda_{k_{\alpha}}\left(v^{\alpha,\ell},v^{\alpha,r}\right) \ ,
\\
&\omega^{\alpha,\ell / r} =\SC{1}{\eta_{1}^{\alpha,\ell / r}}{u } \ ,
&&v^{\alpha,\ell / r}+z =\SC{2}{\eta_{2}^{\alpha,\ell / r}}{\omega^{\alpha,\ell / r} } \ ,
\\
&\lambda_{1}^{\alpha,\ell / r}=\lambda_{1}\left(u,\omega^{\alpha,\ell / r}\right) \ ,
 &&\lambda_{2}^{\alpha,\ell / r}=\lambda_{2}\left(\omega^{\alpha,\ell / r}, v^{\alpha,\ell / r}+z\right) \ .
\ea
and in addition to those, we also define
\ba
&\tilde{v}:=\SC{k_{\alpha}}{\sC_{\alpha}}{v^{\alpha,\ell}+z} \ , 
&& \tilde{s}:=\lambda_{k_{\alpha}}\left(v^{\alpha,\ell}+z,\tilde{v} \right) \ .
\ea
The definition of the intermediate state $\tilde{v}$ prompts the connection of $u$ with $\tilde{v}$ via the waves $\tilde{\eta}_i$, i.e.
$$
\tilde{v}(t,x)=\SC{2}{\tilde{\eta}_{2}(t,x)}{\cdot}\circ{ \SC{1}{\tilde{\eta}_{1}(t,x)}{u(t,x)}} \ ,
$$
and hence, we obtain the intermediate quantities, that are the speeds $\tilde{\lambda}_i$, the weights $\widetilde{W}_i$ and the corresponding errors
\be{E:errorztil}
\widetilde{E}_{\alpha,i}^z=\widetilde{W}_i|\tilde{\eta}_i|(\tilde{\lambda}_i-\tilde{s})-W_{\alpha,i}^{\ell}|\eta_i^{\alpha,\ell}|(\lambda_i^{\alpha,\ell}-\tilde{s})\;.
\ee
By exchanging the roles of $(v^{\alpha,\ell}, v^{\alpha,r}, \dot{x}_\alpha, \eta^{\alpha,r}_i,\lambda_i^{\alpha,r},W_i^{\alpha,r})$ with $(v^{\alpha,\ell}+z, \tilde{v}, \tilde{s}, \tilde{\eta}_i,\tilde{\lambda}_i,\widetilde{W}_i)$ in \S~\ref{Ss:nointeractiontimesN} and in particular at~\eqref{E:mainest}, we immediately deduce that
\be{Ez:mainest}
 |\widetilde{E}_{\alpha,i}^z|\le\co(1)|\sC_\alpha|\eps.
\ee
Next, we note that if $\er=0$, these intermediate quantities become
\be{Ez:mainesv1}
\tilde{v}=v^{\alpha,r},\quad \tilde{\eta}_i=\eta_i^{\alpha,r},\quad \tilde{\lambda}_i=\lambda_i^{\alpha,r},\quad \tilde{s}=\dot x_\alpha\;.
\ee
On the other hand, when $\sC_\alpha=0$, we have
\be{Ez:mainesv2}
v^{\alpha,r}=v^{\alpha,\ell},\quad \tilde{\eta}_i=\eta_i^{\alpha,r}=\eta_i^{\alpha,\ell},\quad \tilde{\lambda}_i=\lambda_i^{\alpha,r}=\lambda_i^{\alpha,\ell}\;.
\ee
Combining the expressions~\eqref{E:errorz} and~\eqref{E:errorztil} together with the values~\eqref{Ez:mainesv1}--\eqref{Ez:mainesv2}, we arrive at
\be{Ez:mainest2}
|\tilde{E}_{\alpha,i}^z-E_{\alpha,i}^z|\le\co(1)|\sC_\alpha|\cdot[\|\er\|_\infty+\eps]
\ee
using the same strategy as in~\cite[p.1002]{AG}. Just note that the term of $\er$ here is denoted by $\omega$ in~\cite{AG}. By~\eqref{Ez:mainest} and~\eqref{Ez:mainest2} and since $\|\er\|_\infty\le\sigma$ , we obtain~\eqref{E:mainestz-z3}.

Having now bounds~\eqref{E:mainestz-z1}--\eqref{E:mainestz-z3} that hold true
at times that no interaction in $u$ and $v$ occurs and combining them with~\eqref{E:phier-est21new}, we arrive at~\eqref{E:phier-est21} immediately.  As already mentioned, due to the non-increase of $\Phi_z(u,v)$ at interaction times, see~\eqref{phi-est-z}, we obtain~\eqref{Phi0est1} by integrating~\eqref{E:phier-est21}  over $[t_1,t_2]$. The proof of Theorem~\ref{stability-Phy}-(i) is now complete.
\qed

\begin{remark}
As $\eps\to0+$ in~\eqref{Phi0est1}, combining with the equivalence relation~\eqref{S4-PhiL1} we deduce
\be{}
\|\cs_{t_2}\bar u-\cs_{t_2}\bar v-z(t_2)\|_{L^1}\le 2C_0^2\, W^* \| \cs_{t_1}\bar u-\cs_{t_1}\bar v-z(t_1)\|_{L^1}+\co(1)(t_2-t_1)\sigma\;,
\ee
for $0<t_1<t_2$.
\end{remark}

%
%
\section{A Lipschitz continuous evolution operator
- Proof of Theorem~\ref{exist-nonhom-smgr-1}}\label{S6}

In this section, we conclude the proof of Theorem~\ref{exist-nonhom-smgr-1}. First, we consider the limit $u(t)=\cp_t\bar u$ of the approximate flow $\cp^s_t$ and prove that this limit is unique independent of the subsequence by using a uniqueness result on quasi differential equations in metric spaces. 

First, by Theorem~\ref{T:AS} of Amadori and Shen, we have that:

\begin{proposition}\label{S6Prop1}
For the constants of Theorem~\ref{ThmPS5.4}, given $\bar u\in\mathcal{D}_0$, then there exists a subsequence $\{s_m\}$, $m\in\mathbb{N}$ such that the functions $\cp^{s_m}_t \bar u \in \mathcal{D}_0^*$ converge in $L^1(\R)$ as $m\to\infty$ for any $t>0$ to an entropic weak solutions $u(t,\cdot)=(h(t,\cdot),p(t,\cdot))$ of the system~\eqref{S1system} with data $\bar u=(\bar h,\bar p)$ and the map $t\mapsto u(t,\cdot)$ is Lipschitz continuous in the $L^1$ norm.
\end{proposition}

Next we aim to prove the uniqueness of the solution $u(t,x)=(h(t,x),p(t,x))$. This result provides the fact that the whole sequence $\{\cp^s_t\}_s$ converges to $u(t,\cdot)$ as $s=\Delta t\to0+$. To establish uniqueness we follow a similar line to~\cite{AG}: We apply the uniqueness result on quasi differential equations in metric spaces established in~\cite{Bre2} in the case of the entropic weak solution $u$ of Proposition~\ref{S6Prop1}. For the convenience of the reader we state this uniqueness result:

\begin{theorem}\label{S6ThmBre}
Suppose that given a quasi-differential equation
\be{S6ThmBreeq}
\frac{du(t)}{dt}={\bf v}(u(t))
\ee
there exists a Lipschitz semigroup $\cp^*:E_0\times[0,\infty)\to E_0^*$, where $(E_0,d_0)$ and $(E_0^*,d_0^*)$ are two metric spaces, which enjoys the following properties $\forall \,\bar u,\bar v\in E_0$, and $\forall\, t_1,t_2\ge 0$:
\begin{enumerate}
\item[(a)] $\cp_0^*\bar u=\bar u$, $\cp^*_{t_1}\cp_{t_2}^*\bar u=\cp_{t_1+t_2}^* \bar u$, \qquad $\forall \,\bar u,\bar v\in E_0$,
\item[(b)] $\exists L>0$ such that $d_0^*(\cp_{t_1}^*\bar u,\cp_{t_2}^*\bar v)\le L\left( |t_1-t_2|+d_0(\bar u,\bar v) \right)$, 
\item[(c)] every trajectory $t\mapsto \cp_t^*\bar u$ provides a solution to the generalized Cauchy problem~\eqref{S6ThmBreeq} with initial data $u(0)=\bar u$.
\end{enumerate}
Then for every initial data $\bar u\in E_0$, the Cauchy problem~\eqref{S6ThmBreeq} with initial data $u(0)=\bar u$ has the unique solution $u(t)=\cp_t^*\bar u$.
\end{theorem}

We apply the above uniqueness result for $E_0=\mathcal{D}_0$, $E_1=\mathcal{D}_0^*$ and both metrics $d_0=d_0^*$ to be the $ L^{1}$ norm. Moreover, we consider ${\bf v} (u)$, to be the generalized tangent vector of the curve 
\be{S6ThmBreeqv}
\gamma(\theta)=\cs_\theta u+\theta g(u),\qquad \theta\ge 0,
\ee
with $\cs$ the semigroup accosiated with the homogeneous system~\eqref{S1system-hom}. The aim now is to prove that the solution $u(t,\cdot)$ obtained as limit of $\{\cp^{s_m}_t\bar u\}_m$ is also a solution to~\eqref{S6ThmBreeq}--\eqref{S6ThmBreeqv}, which admits a Lipschitz flow of solutions. As already mentioned, this is the strategy in~\cite{AG} for general systems of balance laws of small total variation that we also adopt here. However, we cannot quote the analysis in~\cite[\S6]{AG} since our stability functional and the metric spaces are different.

\begin{theorem}\label{S6Thm1}
The curve $t\mapsto u(t,\cdot)$, where $u(t,\cdot)$ is the entropy weak solution obtained in Proposition~\ref{S6Prop1}, satisfies the generalized differential equation~\eqref{S6ThmBreeq}--\eqref{S6ThmBreeqv} with initial data $u(0)=\bar u\in\mathcal{D}_0$ in the metric space $L^1(\R,\R^2)$ and there exists $\theta_0\in(0,1)$ so that it holds
\be{}
\|u(t+\theta)-\cs_\theta u(t)-\theta g(u(t))\|_{L^1}\le \co(1)\theta^2,\qquad 0<\theta<\theta_0\;,
\ee
where $\cs$ is the semigroup of the homogeneous system~\eqref{S1system-hom} and $g$ the source term of the inhomogeneous system~\eqref{S1system}. Moreover, the generalized differential equation~\eqref{S6ThmBreeq}--\eqref{S6ThmBreeqv} admits a Lipschitz semigroup of solutions.
\end{theorem}
\begin{proof} Let $\bar u\in\mathcal{D}_0$ and consider $u(t)\in\mathcal{D}_0^*$ the limit of $\{\cp^{s_m}_t\bar u\}$ for $t>0$ as $m\to\infty$ that is obtained in Proposition~\ref{S6Prop1}. By Theorem~\ref{ThmPS5.4}, note that $\cp_{t'}^{s_m}(\cp_t^{s_m}\bar u)\in \mathcal{D}_0^*$ $\forall \, t,t'>0$.
We first claim that there exists $\theta_0\in(0,1)$ such that for all $\theta\le \theta_0$ and $s>0$ satisfying $s<\theta^2$, it follows
\be{S6:claim1}
\|\cp_\theta^s u-\cs_\theta u-\theta g(u)\|_{L^1}\le \co(1)\theta^2
\ee
 for all $u\in\mathcal{D}_0^*$. To prove this, consider the approximate solution $u_\eps(\tau)$ to $P^s_\tau u$ of the non-homogeneous system~\eqref{S1system} and the approximate solution $v_\eps(\tau)$ to $\cs_\tau u$ of the homogeneous system~\eqref{S1system-hom} starting out from same initial data $ u$ and for a time interval of length $\theta>0$.					 Let $v_\eps(0,x)=u_\eps(0,x)$ and take $\omega(\tau,x):=u_\eps(0,x)$. Note that $u_\eps(\tau)$ is discontinuous in $L^1$ at the time steps $t_k=k s$, while $v_\eps(\tau)$ is continuous. On the other hand, $\omega$ is independent of time $\tau$ and discontinuous along vertical lines. In what follows, three limits are taken in the following order: first $\eps\to0+$ ,next $s\to0$ and last $\theta\to0+$. 
 To start with, fix $\theta_0$ to be determined, take $\theta<\theta_0$, choose $k_0$ so that $k_0s\le\theta<(k_0+1)s$ and consider $u_\eps$, $v_\eps$ and $\omega$ for $(\tau,x)\in[0,\theta]\times\mathbb{R}$. Also, define the quantities
 \be{}
 \varphi_k^+=\Phi_{ks g(\omega)}(u_\eps(ks),v_\eps(ks))
 \ee
 \be{}
 \varphi_k^-=\Phi_{(k-1)s g(\omega)}(u_\eps(ks-),v_\eps(ks))
 \ee
 for $k=1,\dots,k_0$. The aim is to estimate $\varphi_{k_0}^+$, and hence, we estimate the differences $\varphi_k^+-\varphi_k^-$ and $\varphi_k^--\varphi_{k-1}^+$ for $k=1,\dots, k_0$ and proceed by induction. First, we consider
 \be{}
 \varphi_k^--\varphi_{k-1}^+=\Phi_{(k-1)s g(\omega)}(u_\eps(ks-),v_\eps(ks))- \Phi_{(k-1)s g(\omega)}(u_\eps((k-1)s),v_\eps((k-1)s))
 \ee
and observe that the functions $u_\eps$, $v_\eps$, $(k-1) sg(\omega)$ within the time strip $(t_{k-1},t_{k})$ play the role of the functions $u$, $v$ and $z$, respectively of Theorem~\ref{stability-Phy}-(i). Hence, if $\theta_0$ is sufficiently small, we get
 \be{}
 \varphi_k^--\varphi_{k-1}^+\le \co(1) s[\eps+(k-1) s] \le \co(1) s[\eps+\theta]
 \ee
 since $\text{TotVar}\{ g(\omega)\}\le \co(1)$. By definition of $\Phi_\er$, the other difference can be rewritten as
 \be{}
 \varphi_k^+-\varphi_k^-=\sum_{i=1}^2 \int_{-\infty}^\infty \left[ |\eta_i^{+}(x) |W_i^+(x)-|\eta_i^-(x)| W_i^-(x) \right] dx
 \ee
 with $\eta_i^+$ connecting $u_\eps(ks)$ with $v_\eps(ks)+ks g(\omega)$ and $\eta_i^-$ connecting $u_\eps(ks-)$ with $v_\eps(ks)+(k-1)s g(\omega)$. Then using Lemma~\ref{S:ts-lemma} and a bound similar to~\eqref{S:ts-Abd2new}, we can reconstruct the work in~\cite[p. 1011]{AG} and get
 \be{}
 \varphi_k^+-\varphi_k^-\le\co(1) (\eps+s)\varphi_k^-+\co(1) s h_\eps
 \ee
 with $h_\eps=\max_{k=1,\dots,k_0}\|\omega-u_\eps(ks-)\|_{L^1}$. Proceeding as in~\cite[p. 1012]{AG}, by induction and the $L^1$ equivalence of the functional $\Phi_z$,
 \[
 \|u_\eps(k_0 s)-v_\eps(k_0 s)-k_0 s g(\omega)\|_{L^1}\le\co(1)\varphi_{k_0}^+\le\co(1)s(\eps+\theta+h_\eps+C(s+\eps)(\theta+\eps)) \frac{[1+C(s+\eps)]^{k_0}-1}{s+\eps}
 \]
 since $\varphi_0^+=0$. Letting $\eps\to0+$ and using property (iii) of Theorem~\ref{ThmPS5.4}
 and Theorems~\ref{exist-hom-smgr}--\ref{ThmPS5.4}, we arrive at~\eqref{S6:claim1} for $s<\theta^2$ and $\theta_0$ sufficiently small.
 
 Using now the subsequence $\{\cp^{s_m}_t \bar u \}$ of Proposition~\ref{S6Prop1} converging to $u=u(t)$, properties (ii), (iv)
 of Theorem~\ref{ThmPS5.4},
 and estimate~\eqref{S6:claim1} with $u=u(t)\in\mathcal{D}_0^*$ and letting $m\to\infty$, we get
 \be{}
\frac{\|u(t+\theta)-\cs_\theta u(t)-\theta g(u(t))\|_{L^1}}{\theta}\le \co(1)\theta,\qquad 0<\theta<\theta_0\;.
\ee
Taking the limit as $\theta\to0+$, this immediately implies that $u(t,\cdot)$ satisfies the generalized differential equation~\eqref{S6ThmBreeq}--\eqref{S6ThmBreeqv} with initial data $u(0)=\bar u$. Thus, to apply~Theorem~\ref{S6ThmBre} and conclude uniqueness and continuous dependence, it remains to prove the existence of a Lipschitz semigroup to~\eqref{S6ThmBreeq}--\eqref{S6ThmBreeqv}. Then arguing as in~\cite[p. 1013]{AG} and combining~\eqref{S6:claim1}, properties (ii)- (iv) of Theorem~\ref{ThmPS5.4}, 
we get that $\{\cp^{s_m}_t \bar u \}$ converges pointwise to a Lipschitz semigroup $\cp:[0,\infty)\times\mathcal{D}_0 \to \mathcal{D}_0^*$ enjoying the properties of Theorem~\ref{exist-nonhom-smgr-1} and its trajectories satisfy the generalized differential equation~\eqref{S6ThmBreeq}--\eqref{S6ThmBreeqv}. This immediately concludes the existence. Hence, applying Theorem~\ref{S6ThmBre} to $\cp^*=\cp$, the proof is complete.
\end{proof}

In view of the above analysis, Theorems~\ref{exist-hom-smgr} and~\ref{exist-nonhom-smgr-1} follow immediately. 
At the same time, the Lipschitz semigroup $u(t)=\cp_t\bar u$ obtained in Theorem~\ref{S6Thm1} satisfies the properties (i)--(iv) of Theorem~\ref{exist-nonhom-smgr-1} as indicated in the proof of Theorem~\ref{S6Thm1} by letting $s\to0+$ in Theorem~\ref{ThmPS5.4}-(iii)-(iv).

\newpage
\appendix

%
%
\section{Reduction to shock curves in the stability analysis}
\label{S:shockReduction}


In \S~\ref{Ss:nointeractiontimesN} we proved estimates for the stability analysis relatively to times which are neither of interaction nor time steps.
There, the right states and velocities were computed with the correct strength but along shock curves instead of along rarefaction curves even across a rarefaction discontinuity at $x_{\alpha}$ of the front-tracking approximation. The reason for this reduction is the same as in the reduction argument of \cite[\S~8.2, Page~161-162]{Bre} and we briefly remind this here, but we refer the reader to~\cite{Bre}  for the full explanations together with the computations.
Such values corresponding to approximated right states at rarefaction shall be more correctly denoted by $\eta_{i}^{\diamond},\lambda_{i}^{\diamond}, \dot x_{\alpha}^{\diamond}$ instead of $\eta_{i}^{r},\lambda_{i}^{r}, \dot x_{\alpha}^{r}$, as we do with an abuse of notation.
The goal was to prove~\eqref{E:mainest}: each correct error $E_{\alpha,i} $ can be written as in \cite[(8.44)]{Bre}, just by adding and subtracting equal terms, as
\[
E_{\alpha,i} = E_{\alpha,i}' +E_{\alpha,i}'' +E_{\alpha,i} '''
\quad,\quad
E_{\alpha,i}' \doteq W_{i}^{\alpha,r}|\eta_{i}^{\alpha,\diamond}|(\lambda_{i}^{\alpha,\diamond}-\dot x_{\alpha}^{\diamond})- W_{i}^{\alpha,\ell}|\eta_{i}^{\alpha,\ell}|(\lambda_{i}^{\alpha,\ell}-\dot x_{\alpha}^{\diamond})
\ .
\]
The error term $E_{\alpha,i}' $ is precisely what we estimate in \S~\ref{Ss:nointeractiontimesN}, where as mentioned with the abuse of notation we briefly write $\eta_{i}^{r},\lambda_{i}^{r}, \dot x_{\alpha}^{r}$ in place of $\eta_{i}^{\diamond},\lambda_{i}^{\diamond}, \dot x_{\alpha}^{\diamond}$.
The error term $E_{\alpha,i}''$ is estimated by 
\[
E_{\alpha,i}'' \doteq W_{i}^{\alpha,r}|\eta_{i}^{\alpha,\diamond}|(\lambda_{i}^{\alpha,r}-\lambda_{i}^{\alpha,\diamond})+ W_{i}^{\alpha,r}\left(|\eta_{i}^{\alpha,r}|-|\eta_{i}^{\alpha,\diamond}|\right)(\lambda_{i}^{\alpha,r}-\lambda_{i}^{\alpha,\diamond})
 \leq\co(1) |\sC_{\alpha}|^{3} \leq \co(1) \varepsilon |\sC_{\alpha}|
\]
just due to second orderer tangency among shock curves and rarefaction curves
which yields errors of order $\co(1) |\sC_{\alpha}|^{3}$, see \cite[(8.43)]{Bre}.
The error term $E_{\alpha,i}'''$, even with our functional~\eqref{S4Wi}--\eqref{S3G}--\eqref{S4A1}--\eqref{S4A2}, is estimated similarly to \cite[(8.46)]{Bre} by 
\[
E_{\alpha,i}''' \doteq 
(\dot x_{\alpha}^{\diamond}- \dot x_{\alpha} ) 
\left\{W_{i}^{\alpha,r}\left(|\eta_{i}^{\alpha,r}|-|\eta_{i}^{\alpha,\ell}|\right)
	+ \left( W_{i}^{\alpha,r}- W_{i}^{\alpha,\ell}\right) |\eta_{i}^{\alpha,\ell}| \right\}
 \leq \co(1) \varepsilon |\sC_{\alpha}| \ .
\]
We indeed stress, concerning the second addend within brackets, that by standard interaction estimates, 
$|\eta_{i}^{\alpha,\ell}|\leq \co(1) |\sC_{\alpha}|$ in case $\eta_{i}^{\alpha,\ell}$ and $\eta_{i}^{\alpha,r}$ have different sign. This follows immediately from~\eqref{E:stimaetarbasefull1}--\eqref{E:stimaetarbasefull} in Corollary~\ref{C:stimegenerali} and ~\eqref{E:primsfwewgrqgerqpall}--\eqref{E:primsfwewgrqgerqpall2} in Lemma~\ref{L:stimegenerali2}. On the other hand, if $\eta_{i}^{\alpha,\ell}$ and $\eta_{i}^{\alpha,r}$ have the same sign, it holds $| W_{i}^{\alpha,r}- W_{i}^{\alpha,\ell}|\leq \co(1) |\sC_{\alpha}|$ just by definition of $\ca_{i,j}$.

%
%
\section{Analysis of shock curves}\label{AppB}

\subsection{Analysis of 1-shock curves}
\label{S:1shocks}

By algebraic manipulations of the Rankine-Hugoniot equations, 
It is proven in \S~2 of~\cite{AS} that  from the Rankine-Hugoniot equations, the 1-shock curve $\SC{1}{h-h_{\ell}}{h_{\ell},p_{\ell}}$ through the point $(h_{\ell},p_{\ell})$ can take the form:
\begin{equation}
\label{E:p1shock}
p=p_{\ell}-\frac{s_{1}(h,p_{\ell})+1}{s_{1}(h,p_{\ell})}(h-h_{\ell})~\equiv p_{\ell}- \frac{ p_{\ell}-1}{ h-s_{1}(h,p_{\ell}) }(h-h_{\ell}) 
\end{equation}
and this wave connecting $(h_\ell,p_\ell)$ on the left with $(h,p)$ on the right has strength $h-h_\ell$ in Cartesian coordinates. Here, $s_{1}$ is the Rankine-Hugoniot speed of the 1-shock that is strictly negative and it has the following expression:
\begin{equation}
\label{E:p1shock2}
s_{1}=s_{1}(h,p_{\ell})\equiv\lambda_{1}(h,p_{\ell})\equiv 
-\frac{1}{2}\left[ p_{\ell}-h+\sqrt{(p_{\ell}-h)^{2}+4h}\right]\;.
\end{equation}
 and by~\eqref{E:p1shock}, we have
\bas
\frac{p-p_{\ell}}{h-h_{\ell}}=-\frac{s_{1}(h,p_{\ell})+1}{s_{1}(h,p_{\ell})} \;,
\eas
so that if we exchange $(h,p)$ and $(h^\ell,p^\ell)$ the left hand side remains the same. As a consequence,
\be{E:symmetry1HC}
s_{1}(h,p_{\ell})=s_{1}(h_{\ell},p)
\qquad \text{whenever }p=\SC{1}{h-h_{\ell}}{h_{\ell},p_{\ell}}\ .
\ee
Moreover, for small $h$, one has the expression
\bes
s_{1}(h,p_{\ell})= -p_{\ell }+\frac{p_{\ell}-1}{p_{\ell}}h+\mathcal{O}(h^{2})
\ees
and for $p_{\ell}=1$ one has that $s_{1}(h, p_{\ell}=1)\big|_{p_{\ell}=1}\equiv-1$.
Finally, one can compute
\ba
\label{E:coeff1s}
 \frac{s_{1}(h,p_{\ell})+1}{s_{1}(h,p_{\ell})} 
 &=\frac{2}{s_{1}(h,p_{\ell})}\cdot\frac{ p_{\ell}-1}{ p_{\ell}-h-2-\sqrt{(p_{\ell}-h)^{2}+4h} }
 = \frac{ p_{\ell}-1}{ h-s_{1}(h,p_{\ell}) }
 \\&=\frac{p_{\ell}-1}{p_{\ell}}-\frac{p_{\ell}-1}{p_{\ell}^{3}}h+o(h)\;.
\notag
\ea

Thanks to this explicit expression, one can compute the derivatives of the Rankine--Hugoniot speed when the 1-Hugoniot curve is parametrized by $h$:
\be{E:s1derh1}
\ptlls{h}{ s_{1}}(h,p_{\ell})=
\frac{1}{2} \left(\frac{-h+p_{\ell}-2}{\sqrt{(p_{\ell}-h)^{2}+4 h}}+1\right)
=
 \frac{ p_{\ell}-1 }{r(h,p_{\ell})(1+s_{2}(h,p_{\ell}))} 
\quad = \quad
\frac{p_{\ell}-1}{p_{\ell}}+2\frac{p_{\ell}-1}{ p_{\ell} ^{3}} h+o(h)\;,
\ee
\be{E:ddhs1}
\ddn{h}{2} s_{1}(h,p_{\ell})=-
\frac{2 (p_{\ell}-1)}{\left((p_{\ell}-h)^{2}+4 h \right)^{3/2}}\;.
\ee
We also have that
\be{E:s1derp1}
\ptlls{p}{ s_{1}}(h,p_{\ell})=\frac{1}{2} \left(\frac{h-p_{\ell}}{\sqrt{(h-p_{\ell})^2+4 h}}-1\right)
=
\frac{s_{1}(h,p_{\ell})}{\sqrt{(h-p_{\ell})^2+4 h}}
\quad
=
\quad
-1+\frac{h}{p_{\ell}^{2}}+o(h)\;,
\ee
\bes
\ptl{p}{2} s_{1}(h,p_{\ell})=-\frac{2 h}{\left((p_{\ell}-h)^{2}+4 h\right)^{3/2}}\;.
\ees
In the above, we have set \bes r(h,p):=\sqrt{(p-h)^{2}+4 h} \quad =\quad p+\frac{h}{2p}+o(h)\quad\text{for $h$ small}.\ees
Finally, we notice that if $p$ is defined by~\eqref{E:p1shock} then
\[
\ptlls{p_\ell}p=1- \frac{h-h_{\ell}}{ h-s_{1}(h,p_{\ell}) }- \frac{(p_{\ell}-1)(h-h_{\ell})}{ ( h-s_{1}(h,p_{\ell}))^2 }\frac{s_{1}(h,p_{\ell})}{\sqrt{(h-p_{\ell})^2+4 h}}\;,
\]
so that 
\ba\label{contiderivataS1}
&\ptlls{p_\ell}p{\Big|_{p_\ell=1}}=1- \frac{h-h_{\ell}}{ 1+ h }=\frac{1+h_\ell}{1+ h }
 &&\text{while}
&&\ptlls{h}p{\Big|_{p_\ell=1}}=\ptlls{h_\ell}p{\Big|_{p_\ell=1}}=0 \ .
\ea
It also holds that
\ba\label{contiderivataS1eta0}
&\ptlls{p_\ell}p{\Big|_{h_\ell=h}}=1 \ ,
&&\ptlls{h_\ell}p{\Big|_{h_\ell=h}}=0
 &&\text{while}
&&\ptlls{h}p{\Big|_{h_\ell=h}}=- \frac{ p_{\ell}-1}{ h-s_{1}(h,p_{\ell}) } \ .
\ea
When $0\leq h\leq \delta_0^*$ and $p_0^*\leq p\leq p_1^*$, with $p_0^*\leq 1\leq p_1^*$, it is also useful the bound
\be{stimadaldalvassos1}
p_{0}^*+h\leq\quad h-\lambda_{1}(h,p_{\ell})\equiv 
\frac{1}{2}\left[ p_{\ell}+h+\sqrt{(p_{\ell}-h)^{2}+4h} \right]\quad\leq p_{1}^*+h\;.
\ee

We also compute and estimate $\ptlls{h}p$ where $p$ is defined by~\eqref{E:p1shock}:
\ba
\ptlls{h}p&=-( p_{\ell}-1)\left( \frac{1}{ h-s_{1}(h,p_{\ell}) }-\frac{h-h_{\ell}}{ (h-s_{1}(h,p_{\ell}))^2 }\left(1-\ptlls{h}{s_{1}}(h,p_{\ell})\right) \right)\notag
\\
&\stackrel{\eqref{E:s1derh1}}{=}-\frac{p_{\ell}-1}{ h-s_{1}(h,p_{\ell}) }\left(1 -\frac{h-h_{\ell}}{ h-s_{1}(h,p_{\ell})  }\cdot\frac{1}{2} \left(\frac{h-p_{\ell}+2+\sqrt{(p_{\ell}-h)^{2}+4 h}}{\sqrt{(p_{\ell}-h)^{2}+4 h}}\right) \right)\notag\\
&\stackrel{\eqref{E:p1shock2}}{=}-\frac{p_{\ell}-1}{ h-s_{1}(h,p_{\ell}) }\left(1 -\frac{h-h_{\ell}}{ \sqrt{(p_{\ell}-h)^{2}+4 h}
 }\cdot\ \left(1+\frac{1-p_{\ell} }{h-s_{1}(h,p_{\ell}) } \right) \right)\ .
 \label{E:computationder}
\ea
For $0\leq h\leq \delta_0^*<p_0^*\leq p\leq p_1^*$,
 it is also useful to perform the estimates
 \ba\label{stimadaldalvassos1}
&\max\{p_{\ell}-h;2\sqrt h\}\leq\quad\sqrt{(p_{\ell}-h)^{2}+4h}\quad \leq\sqrt{(p_{1}^*-h)^{2}+4hp_{1}^*} =p_{1}^*+h\\
&p_{0}^*\leq p_{0}^*+h\leq\quad h-\lambda_{1}(h,p_{\ell})\equiv 
\frac{1}{2}\left[ p_{\ell}+h+\sqrt{(p_{\ell}-h)^{2}+4h} \right]\quad\leq p_{1}^*+h
\ea
Using now that $p_0^*=1-\delta_p^*$ and choosing $\delta_0^*$ and $\delta_p^*$ small enough so that $1-\delta_p^*-\delta_0^*>2\sqrt{\delta_0^*}$, we have $\max\{p_{\ell}-h;2\sqrt h\}\ge1-\delta_p^*-\delta_0^*$. This holds true for instance if we use the rough estimate that both $\delta_p^*$ and $\delta_p^*$ are less than $1/9$ in addition to Conditions ($\Sigma$) in Proposition~\ref{PropCond}. Combining these together with~\eqref{E:computationder}, we arrive at
\ba\label{esotedegaaggargrawgragar}
\left\lvert \ptlls{h}p\right\rvert&\leq
\frac{1}{p_{0}^*}\left(1+\frac{|h-h_\ell|}{\max\{p_{\ell}-h;2\sqrt h\}}\cdot\left(1+\frac{\delta_p^* }{p_{0}^*}\right)\right)
\cdot \left\lvert p_{\ell}-1\right\rvert\notag\\
& 
\leq \frac{1}{p_{0}^*} \left(1+\frac{\delta_0^*}{(1-\delta_p^*-\delta_0^*)p_{0}^*}\right)
\cdot \left\lvert p_{\ell}-1\right\rvert\notag\\
& 
\leq  \frac{1+\delta_0^*}{(p_{0}^*)^2(1-\delta_p^*-\delta_0^*)}\cdot \left\lvert p_{\ell}-1\right\rvert
\ea
where the coefficient of $ \left\lvert p_{\ell}-1\right\rvert$ in the above bound is a positive finite constant.

\subsection{Analysis of 2-shock curves}
\label{S:2shocks}

In \S~2 of~\cite{AS} it is shown from the Rankine-Hugoniot equations that the 2-shock curve~$\SC{2}{p-p_{\ell}}{h_{\ell},p_{\ell}}$ through the point $(h_{\ell},p_{\ell})$ has the following form:
\be{E:h2shock}
h=h_{\ell}-\frac{s_{2}(h_{\ell},p)}{s_{2}(h_{\ell},p)+1}(p-p_{\ell})\quad\equiv\quad h_{\ell}+\frac{h_{\ell}}{\lambda_{1}(h_{\ell},p)-h_{\ell}}(p-p_{\ell})
\quad\equiv\quad \left(1+\frac{p-p_{\ell}}{\lambda_{1}(h_{\ell},p)-h_{\ell}}\right)h_{\ell}
\ ,
\ee
 with the corresponding $2-$wave of strength $p-p_\ell$ in Cartesian coordinates.
Here, $s_{2}$ is the Rankine-Hugoniot speed of the 2-shock that is nonnegative and it has the following expression:
\be{E:s2proph}
s_{2}=s_{2}(h_{\ell},p)\equiv\lambda_{2}(h_{\ell},p)~\equiv\frac{1}{2}\left[-(p-h_{\ell})+\sqrt{(p-h_{\ell})^{2}+4h_{\ell}}\right]\equiv-\frac{h_{\ell}}{\lambda_{1}(h_{\ell},p)}
\ .
\ee
We know that $h=\SC{2}{p-p_{\ell}}{h_{\ell},p_{\ell}}$ is equivalent to $h_{\ell}=\SC{1}{p_{\ell}-p }{h,p }$ and from~\eqref{E:p1shock}, we have
\bas
\frac{p-p_{\ell}}{h-h_{\ell}}=-\frac{s_{2}(h,p_{\ell})+1}{s_{2}(h,p_{\ell})} 
\eas
so that if we exchange $(h,p)$ and $(h^\ell,p^\ell)$ the left hand side remains the same: as a consequence
\be{E:symmetry2HC}
s_{2}(h,p_{\ell})=s_{2}(h_{\ell},p)
\qquad \text{whenever }h=\SC{2}{p-p_{\ell}}{h_{\ell},p_{\ell}}\ .
\ee
Moreover, for small $h_{\ell}$ one has 
\bes
s_{2}(h_{\ell},p)= \frac{h_{\ell}}{p}+\frac{p-1}{p^{3}}{h_{\ell}^{2}}+\frac{(p-1)(p-2)}{p^{5}}h_{\ell}^{3}+\mathcal{O}(h_{\ell}^{4})
\ .
\ees
At $p=1$, one has that $s_{2}(h_{\ell},1)\equiv h_{\ell}$ while on the line $h_{\ell}=0$, it holds $s_{2}(0,p )\equiv0$. At $h_{\ell}=0$, we also have
\be{E:jacobianS2}
\ptl{s}{}\SC{2}{s}{0,p_{\ell}}=\left(\begin{array}{c}0 \\1\end{array}\right)
\qquad
J_{h_{\ell},p_{\ell}}\SC{2}{s}{0,p_{\ell}}=
\left(\begin{array}{cc}
 1-\frac{s}{p_{\ell}+s} & 0\\0&1\end{array}
\right)
\ .
\ee

Thanks to this explicit expression, one can compute again the derivatives of the Rankine-Hugoniot speed when the 2-Hugoniot curve is parametrized by $p$:
\ba
\label{E:ders2p1}
\ptlls{p}{s_{2}}(h_{\ell},p)&=\frac{1}{2} \left(\frac{p-h_{\ell}}{\sqrt{(h_{\ell}-p)^2+4 h_{\ell}}}-1\right)
= \frac{h_{\ell}}{s_{1}(h_{\ell},p)\cdot\sqrt{(h_{\ell}-p)^2+4 h_{\ell}}}
\\ \notag
&= 
-\frac{h_{\ell}}{p^{2}}+-2\frac{ 3-2 p }{ p^4 }h_{\ell}^{2}+ o(h_{\ell}^{2})
\ ,
\ea
\bas
\ptl{p}{2} s_{2}(h_{\ell},p)=\frac{2 h_{\ell}}{\left((h_{\ell}-p)^2+4 h_{\ell}\right)^{3/2}}
\ .
&&\ .
\eas
We also have that
\bas
\ptlls{h}{{s_{2}}}(h_{\ell},p)= \frac{1}{2} \left(\frac{h_{\ell}-p+2}{\sqrt{(h_{\ell}-p)^2+4 h_{\ell}}}+1\right)= \frac{\lambda_{2}(h_{\ell},p)+1}{\sqrt{(h_{\ell}-p)^2+4 h_{\ell}}} 
\ ,
\eas
\bas
\ptl{h}{2} s_{2}(h_{\ell},p)= \frac{2 (p-1)}{\left((h_{\ell}-p)^2+4 h_{\ell}\right)^{3/2}}\;.
\eas
In particular, when $h_\ell=0$, we get
\be{AppB.3s2h}
\ptlls{h}{{s_{2}}}(0,p)=  \frac{ 1}{ p } 
\ .
\ee

Finally, we notice that if $h$ is defined by~\eqref{E:h2shock} then
\[
\ptlls{h_\ell}h=1+\frac{p-p_{\ell}}{ \lambda_{1}(h_{\ell},p ) -h_\ell}- \frac{(p-p_{\ell})h_{\ell}}{ ( \lambda_{1}(h_{\ell},p) -h_\ell)^2 } \cdot \frac{1}{2} \left(\frac{-h_{\ell}+p-2}{\sqrt{(p -h_{\ell})^{2}+4 h_{\ell}}}-1\right)
\]
so that 
\be{contiderivataS2}
\ptlls{h_\ell}h{\Big|_{h_\ell=0}}=1- \frac{p-p_{\ell}}{ p }=\frac{p_\ell}{p }
\qquad \text{while}\qquad
\ptlls{p}h{\Big|_{h_\ell=0}}=0 \ .
\ee

\section{Finer interaction-type estimates}
\label{S:finerInteractions}
We consider the vector states $\underline v^{\alpha,\ell}$, $\underline v^{\alpha,r}$, $\underline u^\alpha$, $\underline \omega^{\alpha,\ell}$, $\underline \omega^{\alpha,r}$ in $K=[0,\delta^*_0] \times [p^*_{0}, p^*_1]$ as already defined in~\S~\ref{Ss:nointeractiontimesN} that are related as follows:
\begin{align*}
&\underline v^{\alpha,r}=\mathbf{S}_{k_\alpha}\big(\sC_\alpha; \underline v^{\alpha,\ell}\big),
&& 
\underline \omega^{\alpha,\ell}\doteq \mathbf{S}_1\big(\eta_1^{\alpha,\ell }; \underline u^\alpha\big),
&&
\underline v^{\alpha,\ell}= \mathbf{S}_2\big(\eta_2^{\alpha,\ell };\underline \omega^{\alpha,\ell}\big) \ ,
\\
&&& 
\underline \omega^{\alpha,r}\doteq \mathbf{S}_1\big(\eta_1^{\alpha, r};\underline u^\alpha \big),
&&
\underline v^{\alpha,r}= \mathbf{S}_2\big(\eta_2^{\alpha, r};\underline \omega^{\alpha,r}\big) 
\end{align*}
and we prove interaction-type estimates on the wave sizes $\eta_1^{\alpha,r}$, $\eta_2^{\alpha,r}$ and on their speeds. More precisely
\begin{itemize}
\item[$\bullet$] In \S~\ref{Ss:estimates1}, we derive auxiliary estimates when the jump of $v(t)$ at $x_{\alpha}$ is a $1$-wave and
\item[$\bullet$] In \S~\ref{Ss:estimates2}, we derives auxiliary estimates when the jump of $v(t)$ at $x_{\alpha}$ is a $2$-wave.
\end{itemize}
Let us recall that we only consider Hugoniot curves despite of the admissibility criteria and in a sense one could say that these are perturbed interaction estimates.
The estimates established in this appendix are fundamental for the error analysis derived in \S~\ref{Ss:nointeractiontimesN}.
For simplicity, as there is no ambiguity, we omit the index $\alpha$ for the remaining of the present appendix.

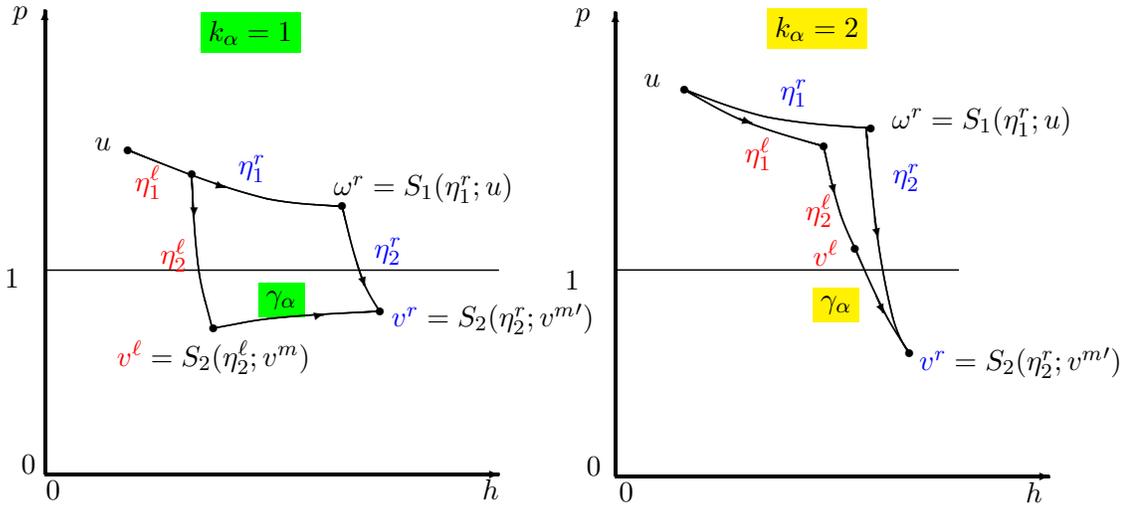
\begin{figure}[htbp]
 \label{fig:setting}
{\centering \scalebox{1}{\input{nointeractiontalk.tex} } \par}
\caption{ {\bf Left}: The jump $\gamma_\alpha$ at $x_{\alpha}$ is along the first Hugoniot curve:
$\underline v^{\alpha,r}=S_1(\sC_\alpha; \underline v^{\alpha,\ell})$. {\bf Right}: The jump $\gamma_\alpha$  at $x_{\alpha}$
is along the second Hugoniot curve
$\underline v^{\alpha,r}=S_2(\sC_\alpha; \underline v^{\alpha,\ell})$. \label{S4:fig2}}
\end{figure}

\subsection{Case of a $1$-wave}
\label{Ss:estimates1}

Let $\gamma$ be a wave that belongs to the first family joining the states $\underline v^{\ell}$ and $\underline v^{r}$ of $v$, i.e.
 \[\underline v^{r}=\mathbf{S}_{1}\big(\sC;\underline v^{\ell}\big)\] where the Hugoniot curve $\SC{1}{\cdot}{\cdot}$ is given in~\eqref{E:p1shock}, while $\SC{2}{\cdot}{\cdot}$ is given in~\eqref{E:h2shock}.
The states $\underline v^{\ell}$ and $\underline v^{r}$ are related to a third state $\underline u$ by
\[\underline v^\ell=\SC{2}{\eta_2^\ell}{\SC{1}{\eta_1^\ell}{\underline u}}
\qquad
\underline v^r=\SC{2}{\eta_2^r}{\SC{1}{\eta_1^r}{\underline u}},\]
via the waves $\eta^{\ell/r}=(\eta_1^{\ell/r}, \eta_2^{\ell/r})$, respectively. We will also need the speed $\dot x$ connecting the states $\underline v^\ell$ and $\underline v^r$, that is 
\be{C1xdot}\dot x:=\lambda_1\left(v_1^\ell+\sC_{},v_2^\ell\right) \ee
and the corresponding ones connecting $\underline u$ with $\underline v^\ell$ and $\underline v^r$ are:
\be{C1lambdaell} \lambda_1^\ell:=\lambda_1\left(u_1+\eta_1^\ell,u_2\right),\qquad \lambda_2^\ell:=\lambda_2(u_1+\eta_1^\ell,v_2^\ell)\ee
\be{C1lambdar}\lambda_1^r:=\lambda_1\left(u_1+\eta_{1}^{r}, u_2\right),\qquad
     \lambda_2^r=\lambda_2(u_1+\eta_1^r, v_2^r )\,,\ee
respectively. 

The aim is to establish auxiliary estimates on the wave strengths $\eta_1^r$, $\eta_1^\ell$, $\eta_2^r$, $\eta_2^\ell$, on the waves velocities and on a commutator among waves and velocities. First, in Lemma~\ref{L:sstimaetarbase} we deal with the simpler case when $\eta_2^\ell=0$ and afterwards, we prove  the general case in Corollary~\ref{C:stimegenerali}.

\newcommand{\evi}[1]{{\color{purple}#1}}

\begin{lemma} 
\label{L:sstimaetarbase}
\begin{subequations} 
\label{E:sstimaetarbase}
Let $\underline v^{\ell}$, $\underline v^{r}$, $\underline u$, $\underline \omega^{\ell}$, $\underline \omega^{r}$, $\gamma$, $\eta^{\ell}$ and $\eta^r$  as denoted above and $\gamma$ be a wave of the first family, i.e. $\underline v^{r}=\mathbf{S}_{1}\big(\sC; \underline v^{\ell}\big)$.  
Suppose that $\eta_2^\ell=0$, so that $\underline v^\ell=\SC{1}{\eta_1^\ell}{\underline u}$. Then 
\begin{align}
\label{E:stimaetarbasepartial1}
&|\eta_{1}^{r}-\eta_{1}^{\ell}-\sC_{{{}}}|
\leq \co(1)( v^{\ell}_{1}+\gamma) (v_{2}^{\ell}-1)^{2} \left\lvert(\eta_{1}^{\ell}+\sC_{ }) \eta_{1}^{\ell}\sC_{ } \right\rvert \ ,
\\
&\label{E:stimaetarbasepartial}
|\eta_{2}^{r}|
\leq \co(1) (v_{2}^{\ell}-1)^{2} \left\lvert(\eta_{1}^{\ell}+\sC_{ }) \eta_{1}^{\ell}\sC_{ } \right\rvert\ .
\end{align}
\end{subequations}
Moreover, the speed $\dot x$ given in~\eqref{C1xdot}  satisfies
\begin{subequations} 
\label{E:stima11part}
\begin{align}
\label{E:stima11lpart}
&\left\lvert\dot x_{}-\lambda_{1}^{\ell}-\frac{(v_{2}^{\ell}-1) (\eta_{1}^{\ell}+\sC_{})}{v_2^\ell}\right\rvert\leq \co(1)\delta_0^*\cdot \left\lvert (v_{2}^{\ell}-1) (\eta_{1}^{\ell}+\sC_{})\right\rvert
\\
\label{E:stima11rpart}
&\left\lvert\dot x_{}-\lambda_{1}^{r}- \frac{(v_{2}^{\ell}-1) \eta_{1}^{\ell} }{v_2^\ell} \right\rvert\leq \co(1) \delta_0^* \cdot \left\lvert(v_{2}^{\ell}-1) \eta_{1}^{\ell}\right\rvert
\end{align}
\end{subequations}
while the following commutators satisfy the bounds
\begin{subequations} 
\label{E:fvgabfafvGroup}
\begin{align}
\label{E:fvgabfafvNew}
&\left\lvert\eta_{1}^{r}( \lambda_{1}^{r}-\dot x_{})-\eta_{1}^{\ell}(\lambda_{1}^{\ell}-\dot x_{})\right\rvert\leq 
\co(1)(\sC+v_{1}^{\ell}) (v_{2}^{\ell}-1)^2 \left\lvert(\eta_{1}^{\ell}+\sC_{ }) \eta_{1}^{\ell}\sC_{ } \right\rvert
\;,
\\
&\left\lvert \eta_{2}^{r}(\lambda_{2}^{r}-\dot x_{})-\eta_{2}^{\ell}(\lambda_{2}^{\ell}-\dot x_{ })\right\rvert\leq 
\co(1) (v_{2}^{\ell}-1)^{2} \left\lvert(\eta_{1}^{\ell}+\sC_{ }) \eta_{1}^{\ell}\sC_{ } \right\rvert \;,
\end{align}
where $\lambda_i^r$, $\lambda_i^\ell$ correspond to the speeds of the $i$-family that are given in~\eqref{C1lambdaell}--\eqref{C1lambdar}.
\end{subequations}
\end{lemma}
\begin{proof}
Suppose that $\eta_2^\ell=0$, then we have $u_1+\eta_1^\ell=v_1^\ell$, and therefore, the speeds reduce to
\[ \lambda_1^\ell:=\lambda_1\left(v_1^\ell,u_2\right),\qquad \lambda_2^\ell=\lambda_2(\underline v ^\ell)\]
\[\lambda_1^r=\lambda_1\left(v_1^\ell-\eta_{1}^{\ell}+\eta_{1}^{r}, u_2\right),\qquad
     \lambda_2^r=\lambda_2(v_1^\ell-\eta_{1}^{\ell}+\eta_1^r, v_2^r )\,,\]
     and recall that \[\dot x:=\lambda_1\left(v_1^\ell+\sC_{},v_2^\ell\right)\;. \]
We now prove separately each estimate in the following steps.\\
{\sc Step 1}. Proof of~\eqref{E:stimaetarbasepartial}. \\
We consider $ \eta_{2}^{r}$ as smooth functions of the independent variables $\underline v^\ell=(v_{1}^{\ell},v_{2}^{\ell}), \eta_{1}^{\ell},\sC{}$ written in the form:
\bas
& \eta_{1}^{r}=\eta_{1}^{r}(v_{1}^{\ell},v_{2}^{\ell} ,\eta_{1}^{\ell}, \sC{}) \ ,
& \eta_{2}^{r}= \eta_{2}^{r}(v_{1}^{\ell},v_{2}^{\ell} ,\eta_{1}^{\ell}, \sC{}) \ .
\eas
Since $\underline w=\SC{i}{t}{\underline z}$ is equivalent to $\underline z=\SC{i}{-t}{\underline w}$, the functions $ \eta_{1}^{r}$ and $ \eta_{2}^{r}$ are implicitly defined also by the identity
\be{implicitdefinitionetar1}
 \SC{2}{-\eta_2^r}{\SC{1}{\gamma}{\underline v^{ \ell} }}\equiv\SC{1}{\eta_1^r}{\SC{1}{-\eta_1^\ell}{ \underline v^\ell}}\ .
\ee
We remark that although in the context of this model we consider $h\geq0$, the above functions are analytic in the larger domain
\ba
\label{E:domaknfknfgrjfwff1}
v_{1}^{\ell}, v_{1}^{\ell}-\eta_{1}^{\ell}, v_{1}^{\ell}+\sC\in[-\delta^*_0,\delta^*_0]\qquad v_{2}^{\ell}\in[p^*_{0}, p^*_1] \ .
\ea
Now, we observe that $\eta_2^r=\eta_{2}^{r}(v_{1}^{\ell},v_{2}^{\ell} ,\eta_{1}^{\ell},\sC)$ satisfies the following vanishing conditions
\begin{itemize}
\item[(i)] $\eta_{2}^{r}(v_{1}^{\ell},v_{2}^{\ell} ,\eta_{1}^{\ell},0)=0$. Indeed, if $\sC=0$, then $\eta_{2}^{r}\equiv\eta_{2}^{\ell}$, that is zero by assumption.
\item[(ii)]  $\eta_{2}^{r}(v_{1}^{\ell},v_{2}^{\ell} ,0,\sC)=0$. In this case, when $\eta_{1}^{\ell}=0$, we have $\underline v^\ell\equiv\underline u$ and thus $\underline v^{r}=\SC{1}{\sC}{ \underline u}{}$ and $\eta_{2}^{r}=0$.
\item[(iii)]  $\eta_{2}^{r}(v_{1}^{\ell}, 1 ,\eta_{1}^{\ell},\sC)=0$. This holds true because if $v_{2}^{\ell}=1$, then $\underline v^{r}$, $\underline v^{\ell}$ and $\underline u$ are connected following the $1$-Hugoniot curve that is a horizontal line at $p=1$ and we have $\eta_{2}^{r}\equiv\eta_{2}^{\ell}=0$ and $\eta_{1}^{r}\equiv\eta_{1}^{\ell}+\sC$.
\item[(iv)]  $\eta_{2}^{r}(v_{1}^{\ell},v_{2}^{\ell} ,\tau,-\tau)=0$ for $\tau\in\mathbb{R}$. This follows if 
$\sC+\eta_{1}^{\ell}=0$ since in this case $\underline v^{r}=\mathbf{S}_{1}\big(\sC; \underline v^{\ell}\big)=\mathbf{S}_{1}\left(- \eta_{1}^{\ell}; \mathbf{S}_{1}\left( \eta_{1}^{\ell};\underline u\right)\right)$ that implies $\underline u\equiv \underline v^r$
 and 
$\eta_{1}^{r}\equiv\eta_{2}^{r}\equiv \sC+\eta_{1}^{\ell}\equiv0 $.
\end{itemize}
In addition to these vanishing conditions, we claim that  
\be{E:claimvanishingderivative}
\ptlls{ {v_{2}^{\ell}}}{ } \eta_{2}^{r} \left(v_{1}^{\ell},v_2^\ell, \eta_{1}^{\ell},\sC{}\right) \Big|_{v_2^\ell=1}=0\ .
\ee
holds true as well. Before we prove the claim, we show that we can obtain~\eqref{E:stimaetarbasepartial}. Indeed, using Lemma~\ref{E:rewgwgeq} in Appendix~\ref{App:C},
we can express $\eta_2^r$ as
\begin{equation}
\label{psi1est-case1}
\begin{aligned}
\eta_2^r\left(v_{1}^{\ell},v_2^\ell, \eta_{1}^{\ell},\sC  \right)
&=\int_1^{v_2^\ell}\int_1^p \frac{\del^2 \eta_2^r}{\del\tilde p^2}\left(v_{1}^{\ell},\tilde{p},\eta_1^{\ell},\gamma\right) d\tilde p\,dp \;.
\end{aligned}
\end{equation}
On the other hand, relying on the vanishing conditions (i)-(ii) and (iv) and
applying Lemma 2.6 of~\cite{Bre}, we get
\bas 
\left|\frac{\del^2\eta_2^r}{\del\tilde p^2}\left(v_{1}^{\ell},\tilde{p},\eta_1^{\ell},\gamma\right)\right|= \co(1) |\eta_1^\ell| |\sC|   |\eta_1^{\ell} +\sC|\;.
\eas
Thus, from~\eqref{psi1est-case1} it follows immediately estimate~\eqref{E:stimaetarbasepartial}.

In view of the above, it remains to prove claim~\eqref{E:claimvanishingderivative}. To accomplish this, we write explicitly the second component of~\eqref{implicitdefinitionetar1}. Denoting by $\pSC{1}{\cdot}{\cdot}$ the second component of the Hugoniot curve $\SC{1}{\cdot}{\cdot}$ recalled in~\eqref{E:p1shock}, and recalling that we are using Cartesian coordinates, we get
\[
\pSC{1}{\gamma}{\underline v^{ \ell} }-\eta_2^r=\pSC{1}{\eta_1^r}{v_1^\ell-\eta_1^\ell,\pSC{1}{-\eta_1^\ell}{ \underline v^\ell}} \ .
\]
Now, we take the derivative $\ptll{v_2^\ell}$ of this equality and  by the chain rule, we arrive at
\be{accessrior030r2ke2}
\ptpS{v_2^\ell}{1}{\gamma}{\underline v^{\ell} }-\ptlls{v_2^\ell}{\eta_2^r}\left(\underline v^\ell , \eta_{1}^{\ell},\sC{}\right)
=\ptpS{\eta_1^r}{1}{\eta_1^r}{\underline u } 
\ptlls{v_2^\ell}{\eta_1^r}\left(\underline v^\ell , \eta_{1}^{\ell},\sC{}\right)
+\ptpS{u_2} {1}{\eta_1^r}{\underline u } 
\ptpS{v_2^\ell} {1}{-\eta_1^\ell}{ \underline v^{ \ell} }\ .
\ee
Next, using~\eqref{contiderivataS1}, we evaluate all terms for $v_2^\ell=u_2=1$ and hence, $\eta_2^r=0$ and $\eta_1^r=\gamma+\eta_1^\ell$ to get 
\bas
&\ptpS{v_2^\ell}{1}{\gamma}{v_1^{ \ell},v_2^{ \ell} }\Big|_{v_2^\ell=1}=\frac{1+v_1^{ \ell}}{1+v_1^r} \ ,
&&\ptpS{u_2}{1}{\eta_1^r=\gamma+\eta_1^\ell}{ u_1,u_2 }\Big|_{u_2=1}= \frac{1+u_1}{1+v_1^r} \ ,
\\
&\ptpS{v_2^\ell}{1}{-\eta_1^\ell}{ v_1^{ \ell},v_2^\ell }\Big|_{v_2^\ell=1}=\frac{1+v_1^\ell}{1+u_1} \ ,
&&\ptpS{\eta_1^r}{1}{\eta_1^r}{ u_1,u_2 }\Big|_{u_2=1} =0 \ .
\eas
Substituting these into~\eqref{accessrior030r2ke2} we compute 
\bas{}
\ptlls{v_2^\ell}{\eta_2^r}\left( v_1^\ell,v_2^\ell , \eta_{1}^{\ell},\sC{}\right)\Big|_{v_2^\ell=1}
&=
\ptpS{v_2^\ell}{1}{\gamma}{v_1^{ \ell},v_2^{ \ell} }\Big|_{v_2^\ell=1}
-\ptpS{\eta_1^r}{1}{\eta_1^r}{\underline u } \Big|_{u_2=1}
\ptlls{v_2^\ell}{\eta_1^r}\left( v_1^\ell,v_2^\ell , \eta_{1}^{\ell},\sC{}\right)\Big|_{v_2^\ell=1}
\\
&\qquad\qquad\qquad\qquad\qquad
-\ptpS{u_2}{1}{\eta_1^r}{\underline u } \Big|_{u_2=1}
\ptpS{v_2^\ell}{1}{-\eta_1^\ell}{ \underline v^{ \ell} }\Big|_{v_2^\ell=1}\\
&
=\frac{1+v_1^{ \ell}}{1+v_1^r}-0-\frac{\cancel{1+u_1}}{1+v_1^r}\cdot \frac{1+v_1^\ell}{\cancel{1+u_1}}
=0
\ .
\eas
and claim~\eqref{E:claimvanishingderivative} follows. This concludes the proof of~\eqref{E:stimaetarbasepartial}.

\noindent
{\sc Step 2}. Proof of~\eqref{E:stimaetarbasepartial1}. \\
Using the same notation as in the first step, we first recall that in the Cartesian coordinates $\eta_1^r$ represent the difference in the first components of the connected states. This means that $\eta_1^r$ is the difference of the $h$-component that is $\SC{2}{-\eta_2^r}{\underline v^r}-\underline u$. Denoting by $\hSC{1}{\cdot}{\cdot}$ the first component of the Hugoniot curve and 
using expression~\eqref{E:h2shock}, we get
\[
\eta_1^r=\hSC{2}{-\eta_2^r}{\underline v^r}-v_1^r+(v_1^r- u_1)\\
=\frac{-v_1^r\eta_2^r}{\lambda_{1}(v_1^r,v_2^r-\eta_2^r)-v_1^r} + (v_1^r-v_1^\ell)+(v_1^\ell-u_1)
\]
Since $\sC=v_1^r-v_1^\ell$ and $\eta_1^\ell=v_1^\ell-u_1$, we obtain
\[
|\eta_1^r-\eta_1^\ell -\sC|\stackrel{\eqref{stimadaldalvassos1}}{\leq} \frac{v_1^r|\eta_2^r|}{p^*_{0}}\stackrel{\eqref{E:stimaetarbasepartial}}{\leq}\co(1)(v_1^\ell+\gamma)|v^{ \ell}_{2}-1|^{2} | \eta_{1}^{\ell}+ \sC_{{{}}}||\eta_{1}^{\ell} \sC_{{{}}}|\ 
.\]
as claimed in~\eqref{E:stimaetarbasepartial1}.

\noindent
{\sc Step 3}. Proof of~\eqref{E:stima11part}. \\
We define the auxiliary functions of the independent variables $v_{1}^{\ell},v_{2}^{\ell}, \eta_{1}^{\ell} ,\sC{}$ as follows
\bas
&\Psi_{1,\ell} \left(v_{1}^{\ell},v_{2}^{\ell}, \eta_{1}^{\ell} ,\sC{}\right)=\dot x_{}-\lambda_{1}^{\ell}\equiv\lambda_1\left(v_1^\ell+\sC_{},v_2^\ell\right)-\lambda_1\left(v_1^\ell,u_2\right) \ ,
\\
&\Psi_{1,r}(v_{1}^{\ell},v_{2}^{\ell} ,\eta_{1}^{\ell} ,\sC{})=\dot x_{}-\lambda_{1}^{r}\equiv \lambda_1\left(v_1^\ell+\sC_{},v_2^\ell\right)-\lambda_1\left(v_1^\ell-\eta_{1}^{\ell}+\eta_{1}^{r}, u_2\right) \ .
\eas
and by evaluating each one of them and their derivatives at particular points of $ \left(v_{1}^{\ell},v_{2}^{\ell}, \eta_{1}^{\ell} ,\sC{}\right)$, we prove~\eqref{E:stima11part}. To begin with, we note that the functional $\Psi_{1,\ell}$ satisfies the following identities
\bas
\Psi_{1,\ell} \left(v_{1}^{\ell},v_{2}^{\ell}, s,-s\right)=0,\quad  \Psi_{1,\ell} \left(v_{1}^{\ell},1, \eta_{1}^{\ell} ,\sC\right)=0,\qquad\forall s\in\mathbb{R}
\eas
Indeed, if $\sC+\eta_{1}^{\ell} =0$, then $\underline u=\underline v^r$ and $\dot x=\lambda_1^r$ since $\eta_2^\ell=0$ by assumption and using the symmetry~\eqref{E:symmetry1HC}. On the other hand, if $v_{2}^{\ell}=1$, then $u_2=1$ and $\dot x=\lambda_1^\ell=-1$. Having now these two vanishing conditions, we can express $\Psi_{1,\ell}$ as
\be{Psi1lApC}
\Psi_{1,\ell} \left(v_{1}^{\ell},v_{2}^{\ell}, \eta_{1}^{\ell} ,\sC\right)= \int_{-\eta_{1}^{\ell} }^\sC \frac{\partial\Psi_{1,\ell} }{\partial\sC}  \left(v_{1}^{\ell},v_{2}^{\ell}, \eta_{1}^{\ell} ,\tau\right)d\tau\;.
\ee
Next, we show that
\bas
\ptlls{\sC_{}}{\Psi_{1,\ell} }\left(0,v_{2}^{\ell}, 0 ,0\right)=\frac{v_2^\ell-1}{v_2^\ell}\ .
\eas
This is a direct computation of the explicit expression of $\Psi_{1,\ell}$. Since $u_2$ is given by
\bas
 u_2=v_2^\ell+\frac{(v_2^\ell-1)\eta_1^\ell}{v_1^\ell-\eta_1^\ell-\lambda_1(v_1^\ell-\eta_1^\ell,v_2^\ell)} 
\eas
from $\underline u=\SC{1}{-\eta_1^\ell}{\underline v^\ell}$ computed using~\eqref{E:p1shock} and it is independent of $\sC_{}$, we get
\bas
\ptlls{\sC_{}}{\Psi_{1,\ell}}\left(0,v_{2}^{\ell}, 0 ,0\right)&=\ptll{\sC_{}} \left(\lambda_1\left(v_1^\ell+\sC_{},v_2^\ell\right)-\lambda_1\left(v_1^\ell,v_2^\ell+u_2\right) \right)\Big|_{\sC_{}=v_1^\ell=0}
\stackrel{\eqref{E:s1derh1}}{=}\frac{v_2^\ell-1}{v_2^\ell}\ .
\eas
Combining now with~\eqref{Psi1lApC}, we get
\be{}
\Psi_{1,\ell} \left(v_{1}^{\ell},v_{2}^{\ell}, \eta_{1}^{\ell} ,\sC\right)= (\eta_1^\ell+\sC) \frac{\partial\Psi_{1,\ell} }{\partial\sC} \left(0,v_{2}^{\ell}, 0 ,0\right)+\int_{-\eta_{1}^{\ell} }^\sC\int_1^{v_{2}^{\ell}} \left[\frac{\partial^2\Psi_{1,\ell} }{\partial\sC\partial v_2}  \left(v_{1}^{\ell},v_{2}, \eta_{1}^{\ell} ,\tau\right)-
\frac{\partial^2\Psi_{1,\ell} }{\partial\sC\partial v_2}  \left( 0 ,v_{2}, 0,0 \right)\right]
dv_2\,d\tau\;.
\ee
Then~\eqref{E:stima11lpart} follows since
\bas
\frac{\partial^2\Psi_{1,\ell} }{\partial\sC\partial v_2}  \left(v_{1}^{\ell},v_{2}, \eta_{1}^{\ell} ,\tau\right)-
\frac{\partial^2\Psi_{1,\ell} }{\partial\sC\partial v_2}  \left( 0 ,v_{2}, 0,0 \right)=\co(1)(|v_1^\ell |+|\eta_1^\ell|+|\sC|)=\co(1)\delta_0^*
\eas
in the domain~\eqref{E:domaknfknfgrjfwff1}.

We proceed in a similar way with $\Psi_{1,r}$, that is again analytic in~\eqref{E:domaknfknfgrjfwff1} and it vanishes as follows
\bas
\Psi_{1,r} \left(v_{1}^{\ell},v_{2}^{\ell}, 0,\sC\right)=0,\quad  \Psi_{1,r} \left(v_{1}^{\ell},1, \eta_{1}^{\ell} ,\sC\right)=0.
\eas
Indeed, if $\eta_1^\ell=0$, then $\eta_1^r=\sC$ and $\eta_2^r=0$ since by assumption $\eta_2^\ell=0$. Hence $\dot x=\lambda_1^r$. On the other hand, if $v_2^\ell=1$, then $u_2=1$ and $\dot x=\lambda_1^r=-1$. Thus the above two vanishing conditions hold true. Hence, we have the following expression
\be{Psi1rApC}
\Psi_{1,r} \left(v_{1}^{\ell},v_{2}^{\ell}, \eta_{1}^{\ell} ,\sC\right)= \int_{0}^{\eta_1^\ell} \frac{\partial\Psi_{1,r} }{\partial\eta_1^\ell}  \left(v_{1}^{\ell},v_{2}^{\ell}, \eta_{1} ,\sC\right)d\eta_1\;.
\ee
In order to find the leading coefficient in~\eqref{E:stima11rpart}, we now compute the derivative
\bas
\ptlls{\eta_1^\ell}{\Psi_{1,r}} \left(0,v_{2}^{\ell}, 0 ,0\right)=\ptll{\eta_1^\ell}\left(\lambda_1\left(v_1^\ell+\sC_{},v_2^\ell\right)-\lambda_1\left(v_1^\ell-\eta_{1}^{\ell}+\eta_{1}^{r}, u_2\right)\right)\Big|_{\sC_{}=v_1^\ell=\eta_1^\ell= 0}\ .
\eas
However, here not only $u_2$, but also $\eta_1^r$ depends on $\eta_1^\ell$.
Nevertheless, by~\eqref{E:stimaetarbasepartial1}, 
\bas
\ptlls{\eta_1^\ell}{\eta_1^r}\Big|_{\sC_{}= 0}=1
\qquad\Rightarrow\qquad
\ptll{\eta_1^\ell}\left(v_1^\ell-\eta_{1}^{\ell}+\eta_{1}^{r}\right)\Big|_{\sC_{}= 0}=0\ .
\eas
As this factor vanishes, by the explicit expression of $u_2$ and by~\eqref{E:s1derp1} we thus get, 
\bas
\ptlls{\eta_1^\ell}{\Psi_{1,r}} \left(0,v_{2}^{\ell}, 0 ,0\right)
\equiv 
-\ptll{\eta_1^\ell}\left(\lambda_1\left(v_1^\ell-\eta_{1}^{\ell}+\eta_{1}^{r}, u_2\right)\right)\Big|_{\sC_{}=v_1^\ell=\eta_1^\ell= 0}
=-\left(0-1\cdot \frac{v_2^\ell-1}{v_2^\ell}\right)
=\frac{v_2^\ell-1}{v_2^\ell}
\eas
since $\ptll{\eta_1^\ell}u_2=\frac{v_2^\ell-1}{v_2^\ell}$ when $v_1^\ell=\eta_1^\ell= 0$. Substituting into~\eqref{Psi1rApC}, we have
\be{}
\Psi_{1,r} \left(v_{1}^{\ell},v_{2}^{\ell}, \eta_{1}^{\ell} ,\sC\right)= \frac{\eta_1^\ell (v_2^\ell-1)}{v_2^\ell}+
\int_{0}^{\eta_1^\ell} \frac{\partial\Psi_{1,r} }{\partial\eta_1^\ell}  \left(v_{1}^{\ell},v_{2}^{\ell}, \eta_{1}^\ell ,\sC\right)-\ptlls{\eta_1^\ell}{\Psi_{1,r}} \left(0,v_{2}^{\ell}, 0 ,0\right)d\eta_1\;.
\ee
with
\[
\frac{\partial\Psi_{1,r} }{\partial\eta_1^\ell}  \left(v_{1}^{\ell},v_{2}^{\ell}, \eta_{1}^\ell ,\sC\right)-\ptlls{\eta_1^\ell}{\Psi_{1,r}} \left(0,v_{2}^{\ell}, 0 ,0\right)=\co(1) |v_2^\ell-1| (|v_1^\ell |+|\eta_1^\ell|+|\sC|)=\co(1)\delta_0^*|v_2^\ell-1| \;.
\]
This proves the desired estimate~\eqref{E:stima11rpart}.

\noindent
{\sc Step 4}. Proof of~\eqref{E:fvgabfafvGroup}. \\
Similarly to the previous step, we define the commutator functions of the independent variables $v_{1}^{\ell},v_{2}^{\ell}, \eta_{1}^{\ell} ,\sC{}$ as follows
\bas
&\widehat\Psi_{1,1} \left(v_{1}^{\ell},v_{2}^{\ell}, \eta_{1}^{\ell} ,\sC{}\right) :=\eta_{1}^{r}( \lambda_{1}^{r}-\dot x_{})-\eta_{1}^{\ell}(\lambda_{1}^{\ell}-\dot x_{}) \;,
\\
&\widehat\Psi_{1,2}\left(v_{1}^{\ell},v_{2}^{\ell} ,\eta_{1}^{\ell} ,\sC{}\right) :=\eta_{2}^{r}(\lambda_{2}^{r}-\dot x_{} ) -\eta_{2}^{\ell}(\lambda_{2}^{\ell}-\dot x_{})\ ,
\eas
which are analytic in the domain~\eqref{E:domaknfknfgrjfwff1}. These satisfy the following conditions:
\be{ApCvan111}
\widehat\Psi_{1,i} \left(v_{1}^{\ell},1, \eta_{1}^{\ell} ,\sC{}\right) =0,\quad \widehat\Psi_{1,i} \left(v_{1}^{\ell},v_{2}^{\ell}, 0 ,\sC{}\right)=0,\quad\widehat\Psi_{1,i} \left(v_{1}^{\ell},v_{2}^{\ell}, \eta_{1}^{\ell} ,0\right)=0\,\quad\widehat\Psi_{1,i} \left(v_{1}^{\ell},v_{2}^{\ell}, s,-s\right)=0
\ee
for $s\in\mathbb{R}$ and $i=1,\,2$. Indeed, recall that $\eta_{2}^{\ell}=0$ by assumption, then
\begin{itemize}
\item[(i)]  if $v_{2}^{\ell}=1$, we have $\lambda_{1}^{r}=\dot x_{}=\lambda_{1}^{\ell}=-1$, $\eta_{2}^{r}=0$ and $v_{2}^{r}=v_{2}^{\ell}=u_{2} =1$,
\item[(ii)]  if $\eta_{1}^{\ell}=0$, then $\underline u=\underline v^\ell$. Hence, $\underline v^r=\SC{1}{\sC_{}}{\underline u}$ so that $ \lambda_{1}^{r}=\dot x_{}$ and $\eta_{2}^{r}=0$,
\item[(iii)] if $\sC=0$, then $\underline v^r=\underline v^\ell$ and thus $\eta_{1}^{r}=\eta_{1}^{\ell}$, $\lambda_{1}^{r}=\lambda_{1}^{\ell}$ and $\eta_{2}^{r}=\eta_{2}^{\ell}=0$,
\item[(iv)]  if $\sC=-\eta_{1}^{\ell}$, then $\underline v^r=\underline u$ and hence $\eta_{1}^{r}=\eta_{2}^{r}=0$ and $\lambda_{1}^{\ell}=\dot x_{}$ by~\eqref{E:symmetry1HC}.
\end{itemize}
All the above cases imply immediately that $\widehat\Psi_{1,i}=0$ for $i=1,\,2$. For $i=1$, it also holds true that
\be{ApCvan112}\widehat\Psi_{1,1} \left(s,v_{2}^{\ell}, \eta_{1}^{\ell} ,-s\right)=0,\qquad \forall s\in\mathbb{R}\;.
\ee
To check this, assume that $\sC=-v_{1}^{\ell}$ then necessarily $v_1^r=u_1+\eta_1^r=0$ and this yields $\eta_1^r=\eta_1^\ell-v_{1}^{\ell}$, $ \lambda_{1}^{r}=-u_2$, $\dot x_{}=-v_2^\ell$. So 
\be{E:afqwrfggggggwggrw000}
\widehat\Psi_{1,1} \left(v_{1}^{\ell},v_{2}^{\ell}, \eta_{1}^{\ell} ,-v_{1}^{\ell}\right)= \eta_{1}^{r}( \lambda_{1}^{r}-\dot x_{})-\eta_{1}^{\ell}(\lambda_{1}^{\ell}-\dot x_{}) 
=-(\eta_1^\ell-v_{1}^{\ell}) (u_2-v_2^\ell)-\eta_{1}^{\ell} (\lambda_{1}^{\ell}+v_2^\ell) \ .
\ee
Since $\underline v^\ell= \SC{1}{\eta_1^\ell}{\underline u}$, it holds $\underline u= \SC{1}{-\eta_1^\ell}{\underline v^\ell}$ and by definition~\eqref{E:p1shock}, we get
\[
u_2-v_2^\ell=\frac{ v_2^\ell-1}{v_1^\ell-\eta_1^\ell-\lambda_{1}^{\ell}}\eta_1^\ell \ .
\]
We thus deduce
\be{E:afqwrfggggggwggrw}
\widehat\Psi_{1,1} \left(v_{1}^{\ell},v_{2}^{\ell}, \eta_{1}^{\ell} ,-v_{1}^{\ell}\right)
=-\frac{\eta_{1}^{\ell}}{v_1^\ell-\eta_1^\ell-\lambda_{1}^{\ell}} \left(( v_2^\ell-1)(\eta_1^\ell-v_{1}^{\ell})+(\lambda_{1}^{\ell}+v_2^\ell) (v_1^\ell-\eta_1^\ell-\lambda_{1}^{\ell}) \right)\ .
\ee
Now, by~\eqref{E:symmetry1HC}, the speed $\lambda_{1}^{\ell}$ is
\[
\lambda_{1}^{\ell}\equiv -\frac{1}{2}\left[ v_2^\ell-v_1^\ell+\eta_1^\ell+\sqrt{(v_2^\ell-v_1^\ell+\eta_1^\ell)^{2}+4(v_1^\ell-\eta_1^\ell)}\right]
\]
so that
\[
(\lambda_{1}^{\ell}+v_2^\ell) (v_1^\ell-\eta_1^\ell-\lambda_{1}^{\ell})=\frac{1}{4}(v_2^\ell+v_1^\ell-\eta_1^\ell)^2-\frac{1}{4}\left((v_2^\ell-v_1^\ell+\eta_1^\ell)^{2}+4(v_1^\ell-\eta_1^\ell)\right)
\equiv (v_2^\ell-1)(v_1^\ell-\eta_1^\ell)\ .
\] 
This immediately yields that the term within the parenthesis in~\eqref{E:afqwrfggggggwggrw} is zero and hence,~\eqref{E:afqwrfggggggwggrw000} holds true.

Next, we prove that
\be{ApCvan113}
\ptlls{v_2^\ell}{\widehat\Psi_{1,1}} \left(v_{1}^{\ell},v_2^\ell, \eta_{1}^{\ell} ,\sC{}\right)\Big|_{v_2^\ell=1}=\ptlls{v_2^\ell}{\widehat\Psi_{1,2}} \left(v_{1}^{\ell},v_2^\ell, \eta_{1}^{\ell} ,\sC{}\right)\Big|_{v_2^\ell=1}=0 \ .
\ee
To show these, we use~\eqref{E:claimvanishingderivative} and get $ \ptll{v_2^\ell}\eta_{2}^{r} =\eta_{2}^{r}=0$ when $v_2^\ell=1$. Then it follows easily
\[
\ptlls{v_2^\ell}{\widehat\Psi_{1,2}} \left(v_{1}^{\ell},v_2^\ell, \eta_{1}^{\ell} ,\sC{}\right)\Big|_{v_2^\ell=1}=
 \left(\ptlls{v_2^\ell}{\eta_{2}^{r}}\right)\Big|_{v_2^\ell=1}(\lambda_{2}^{r}-\dot x_{} )\Big|_{v_2^\ell=1}+\eta_{2}^{r}\Big|_{v_2^\ell=1}\ptll{v_2^\ell}\left(\lambda_{2}^{r}-\dot x_{} \right)\Big|_{v_2^\ell=1}
=0 
\]
Similarly $\ptll{v_2^\ell}\left(\eta_{1}^{r}( \lambda_{1}^{r}-\dot x_{})\right)=(\eta_1^\ell+\sC_{})\ptll{v_2^\ell}\left(\lambda_{1}^{r}-\dot x_{}\right)$ when $v_2^\ell=1$ since by~\eqref{E:stimaetarbasepartial1} we know that $ \ptll{v_2^\ell}\eta_{2}^{r}=0$. Moreover, we evaluate
\bas
\ptll{v_2^\ell}\left(\dot x_{}-\lambda_{1}^{\ell}\right)\Big|_{v_2^\ell=1}&=\ptlls{v_2^\ell}{\lambda_1}\left(v_1^\ell+\sC_{},v_2^\ell\right)\Big|_{v_2^\ell=1}-\ptlls{v_2^\ell}{\lambda_1}\left(v_1^\ell,u_2\right)\Big|_{v_2^\ell=1} \\
&\stackrel{\eqref{E:s1derp1}}{=}\frac{-1}{1+v_1^\ell+\sC_{}}-\frac{-1}{1+v_1^\ell} \cdot \left(1+\frac{\eta_1^\ell}{1+v_1^\ell-\eta_1^\ell}\right)
=\frac{\sC_{}+\eta_1^\ell}{(1+v_1^\ell-\eta_1^\ell)(1+v_1^\ell+\sC_{})}\;,
\\
\ptll{v_2^\ell}\left(\lambda_{1}^{r}-\dot x_{}\right) \Big|_{v_2^\ell=1}
&\stackrel{\eqref{E:s1derp1}}{=}\frac{-1}{1+v_1^\ell+\sC_{}} \cdot \left(1+\frac{\eta_1^\ell}{1+v_1^\ell-\eta_1^\ell}\right)-\frac{-1}{1+v_1^\ell+\sC_{}}
=\frac{ -\eta_1^\ell}{(1+v_1^\ell-\eta_1^\ell)(1+v_1^\ell+\sC_{})}\;.
\eas
Using these values, we are now able to compute
\[
\ptlls{v_2^\ell}{\widehat\Psi_{1,1}} \left(v_{1}^{\ell},v_2^\ell, \eta_{1}^{\ell} ,\sC{}\right)\Big|_{v_2^\ell=1}=(\eta_1^\ell+\sC_{})\cdot \frac{ -\eta_1^\ell}{(1+v_1^\ell-\eta_1^\ell)(1+v_1^\ell+\sC_{})}+\eta_1^\ell\cdot\frac{\sC_{}+\eta_1^\ell}{(1+v_1^\ell-\eta_1^\ell)(1+v_1^\ell+\sC_{})}=0\ .
\]
Now, estimates~\eqref{E:fvgabfafvGroup} follow immediately combining~\eqref{ApCvan111},~\eqref{ApCvan112} and~\eqref{ApCvan113} with Lemma~\ref{E:rewgwgeq} in Appendix~\ref{App:C}.
\end{proof}

Now, we extend the previous result to the general case when $\eta_2^\ell$ is not necessarily zero.
\begin{corollary} 
\label{C:stimegenerali}
Suppose $\underline v^\ell=\SC{2}{\eta_2^\ell}{\SC{1}{\eta_1^\ell}{\underline u}}$,  and $\underline v^r=\SC{2}{\eta_2^r}{\SC{1}{\eta_1^r}{\underline u}}$  and let $\gamma$ be a wave of the first family, i.e. $\underline v^{r}=\mathbf{S}_{1}\big(\sC; \underline v^{\ell}\big)$. 
Then 
\begin{subequations} 
\begin{align}
\label{E:stimaetarbasefull1}
|\eta_{1}^{r}-\eta_{1}^{\ell}-\sC_{{{}}}|
&\leq \co(1) ( v^{\ell}_{1}+\gamma)|v^{ \ell}_{2}-1|^{2} | \eta_{1}^{\ell}+  \sC_{{{}}}||\eta_{1}^{\ell} \sC_{{{}}}| +\co(1)|\eta_2^\ell \sC_{{{}}}| \ ,
\\
 \label{E:stimaetarbasefull}
|\eta_{2}^{r}-\eta_2^\ell|
&\leq \co(1)   |v^{ \ell}_{2}-1|^{2} | \eta_{1}^{\ell}+ \sC_{{{}}}||\eta_{1}^{\ell} \sC_{{{}}} |+\co(1) |\eta_2^\ell \sC_{{{}}}| \ .
\end{align}
\end{subequations}
Moreover, the speed $\dot x$ given in~\eqref{C1xdot}  satisfies
\begin{subequations} 
\label{E:stima11full}
\begin{align}
\label{E:stima11lfull}
&\left\lvert\dot x_{}-\lambda_{1}^{\ell}-\frac{(v_{2}^{\ell}-1) (\eta_{1}^{\ell}+\sC_{})}{v_2^\ell}\right\rvert\leq \co(1)\delta_0^*\cdot \left\lvert (v_{2}^{\ell}-1) (\eta_{1}^{\ell}+\sC_{})\right\rvert+\co(1)\left\lvert \eta_2\right\rvert
\\
\label{E:stima11full2}
&\left\lvert\dot x_{}-\lambda_{1}^{r}- \frac{(v_{2}^{\ell}-1) \eta_{1}^{\ell} }{v_2^\ell} \right\rvert\leq \co(1) \delta_0^* \cdot \left\lvert(v_{2}^{\ell}-1) \eta_{1}^{\ell}\right\rvert+\co(1)\left\lvert \eta_2\right\rvert\;,
\end{align}
\end{subequations}
while the commutators satisfy the estimates
\begin{subequations} 
\label{E:fvgabfafvGroupfull}
\begin{align}
\label{E:fvgabfafvNewfull}
&\left\lvert\eta_{1}^{r}( \lambda_{1}^{r}-\dot x_{})-\eta_{1}^{\ell}(\lambda_{1}^{\ell}-\dot x_{})\right\rvert\leq 
\co(1)(\sC+v_{1}^{\ell}) (v_{2}^{\ell}-1)^2 \left\lvert(\eta_{1}^{\ell}+\sC_{ }) \eta_{1}^{\ell}\sC_{ } \right\rvert+\co(1)|\eta_2^\ell \sC_{{{}}}|
\;,
\\
&\label{E:fvgabfafvNewfull2}
\left\lvert \eta_{2}^{r}(\lambda_{2}^{r}-\dot x_{})-\eta_{2}^{\ell}(\lambda_{2}^{\ell}-\dot x_{ })\right\rvert\leq 
\co(1) (v_{2}^{\ell}-1)^{2} \left\lvert(\eta_{1}^{\ell}+\sC_{ }) \eta_{1}^{\ell}\sC_{ } \right\rvert +\co(1)|\eta_2^\ell \sC_{{{}}}| \;,
\end{align}
\end{subequations}
where $\lambda_i^r$, $\lambda_i^\ell$ correspond to the speeds of the $i$-family that are given in~\eqref{C1lambdaell}--\eqref{C1lambdar}.
\end{corollary}
\begin{proof}
The estimates can be derived using the following general argument: Given a smooth function $\Psi(\underline x, \eta_2^\ell)$, one can write \[
\Psi(\underline x, \eta_2^\ell)=\Psi(\underline x, 0)+\int_0^{\eta_2^\ell}\ptlls{\eta_2^\ell}{\Psi}(\underline x,\xi)\,d\xi
\]
and then use Lemma~\ref{L:sstimaetarbase} to control the term $\Psi(\underline x, 0)$ and an estimate on $\ptlls{\eta_2^\ell}{\Psi}$ to conclude
\[
\left|\Psi(\underline x, \eta_2^\ell)- \Psi(\underline x, 0)\right|\leq 
\max_{0\leq \xi\,\sgn\eta_2^\ell\leq |\eta_2^\ell|}\ptlls{\eta_2^\ell}{\Psi}(\underline x,\xi) \cdot \left\lvert\eta_2^\ell\right\rvert \ .
\]
We apply this argument to prove~\eqref{E:stimaetarbasefull1}. Set \[\Psi(v_{1}^{\ell}, v_{2}^{\ell} , \eta_{1}^{\ell} , \sC{},\eta_2^\ell):=\eta_{1}^{r}(v_{1}^{\ell}, v_{2}^{\ell} , \eta_{1}^{\ell} , \sC{},\eta_2^\ell)-\eta_{1}^{\ell}-\sC_{{{}}},\]
and recall from Lemma~\ref{E:sstimaetarbase} that
\[|\Psi(v_{1}^{\ell}, v_{2}^{\ell} , \eta_{1}^{\ell} , \sC{},0)|\le \co(1)( v^{\ell}_{1}+\gamma) (v_{2}^{\ell}-1)^{2} \left\lvert(\eta_{1}^{\ell}+\sC_{ }) \eta_{1}^{\ell}\sC_{ } \right\rvert \ ,
\]
We need to prove that
\[\left\lvert\ptlls{\eta_2^\ell} {\eta_{1}^{r}}(v_{1}^{\ell}, v_{2}^{\ell} , \eta_{1}^{\ell} , \sC{}, \xi)\right\rvert\leq \co(1)|\sC_{}|\
\qquad\text{for $v_{1}^{\ell} , \ v_{1}^{\ell} -\eta_{1}^{\ell} ,\ v_{1}^{\ell} +\sC{}\in [0,\delta^*_0] $ and $v_{2}^{\ell},v_{2}^{\ell}+\xi\in [p^*_{0}, p^*_1]$}
 .\]
When $\sC_{}=0$ of course $\eta_{1}^{r}=\eta_{1}^{\ell}$ for all $\eta_2^\ell$, thus $\ptll{\eta_2^\ell} \eta_{1}^{r}(v_{1}^{\ell},v_{2}^{\ell} ,\eta_{1}^{\ell} , 0,\xi)=0$. Therefore, there exists $\tilde{\sC}$, that satisfies $\tilde{\sC}\in(0,\sC)$ for $\sC>0$ or $\tilde{\sC}\in(\sC,0)$ for $\sC<0$, so that 
\[
\ptlls{\eta_2^\ell}{ \Psi}=\ptlls{\eta_2^\ell}{ \eta_{1}^{r}}(v_{1}^{\ell},v_{2}^{\ell} ,\eta_{1}^{\ell}, \sC{},\ \xi)=\cancel{\ptlls{\eta_2^\ell} {\eta_{1}^{r}}(v_{1}^{\ell},v_{2}^{\ell} ,\eta_{1}^{\ell}, 0, \xi)}+ \frac{\partial^2  {\eta_{1}^{r}} }{\partial\eta_2^\ell\partial\sC} 
(v_{1}^{\ell},v_{2}^{\ell} ,\eta_{1}^{\ell},\tilde{\sC}, \xi)\cdot \gamma \ .
\]
We conclude the proof of~\eqref{E:stimaetarbasefull1} since $\partial^2_{\eta_2^\ell\sC_{}} \eta_{1}^{r}$ is bounded on the given compact domain.

The proofs of~\eqref{E:stimaetarbasefull} and of~\eqref{E:fvgabfafvGroupfull} are completely analogous since for $\sC_{}=0$, the following functions vanish since left and right states are the same:
\[
\eta_{2}^{r}-\eta_2^\ell\ , \qquad
\eta_{1}^{r}( \lambda_{1}^{r}-\dot x_{})-\eta_{1}^{\ell}(\lambda_{1}^{\ell}-\dot x_{})\ , 
\qquad
\eta_{2}^{r}( \lambda_{2}^{r}-\dot x_{})-\eta_{2}^{\ell}(\lambda_{2}^{\ell}-\dot x_{}) \ .
\]
In a similar manner, we also treat estimate~\eqref{E:stima11full} since the derivatives $\ptll{\eta_2^\ell}{(\dot x_{}-\lambda_{1}^{\ell})}$, $\ptll{\eta_2^\ell}{(\dot x_{}-\lambda_{1}^{r})}$ are continuous and hence bounded on the given compact domain.
\end{proof}

\subsection{Case of a $2$--wave}\label{Ss:estimates2}
Let now a wave $\sC$ that belongs to the second family joining states $\underline v^{\ell}$ and $\underline v^{r}$, this means \[\underline v^{r}=\mathbf{S}_{2}\big(\sC;\underline v^{\ell}\big)\ .\] 
and recall that the states $\underline v^{\ell}$ and $\underline v^{r}$ are related to a third state $\underline u$ via
\[\underline v^\ell=\SC{2}{\eta_2^\ell}{\SC{1}{\eta_1^\ell}{\underline u}}
\qquad
\underline v^r=\SC{2}{\eta_2^r}{\SC{1}{\eta_1^r}{\underline u}}.\]
and recall  in~\eqref{E:p1shock} the expression of the Hugoniot curve $\SC{1}{\cdot}{\cdot}$, while n~\eqref{E:h2shock}  the expression of Hugoniot curve $\SC{2}{\cdot}{\cdot}$.
We will also need the speed $\dot x$ connecting the states $\underline v^\ell$ ind $\underline v^r$, that is 
\be{C2xdot}\dot x:=\lambda_2\left(v_1^\ell,v_2^\ell+\sC_{}\right)  \ee
and the corresponding ones connecting $\underline u$ with $\underline v^\ell$ and $\underline v^r$ are:
\be{C2lambdaell}\lambda_1^\ell:=\lambda_1\left(u_1+\eta_1^\ell,u_2\right),\qquad \lambda_2^\ell:=\lambda_2(u_1+\eta_1^\ell,v_2^\ell)\ee
\be{C2lambdar}\lambda_1^r:=\lambda_1\left(u_1+\eta_{1}^{r}, u_2\right),\qquad
     \lambda_2^r=\lambda_2(u_1+\eta_1^r, v_2^r )\,,\ee
respectively. 

In the rest of the section, we work as in the previous subsection and derive auxiliary estimates on the wave strengths $\eta_1^r$, $\eta_1^\ell$, $\eta_2^r$, $\eta_2^\ell$, on the waves velocities and on a commutator among waves and velocities. These are established in the next lemma.
\begin{lemma} 
\label{L:stimegenerali2}
\begin{subequations}\label{E:8.49improved2Nall}

Let $\underline v^{\ell}$, $\underline v^{r}$, $\underline u$, $\underline \omega^{\ell}$, $\underline \omega^{r}$, $\gamma$, $\eta^{\ell}$ and $\eta^r$  be as denoted above and $\gamma$ be a wave of the second family, i.e. $\underline v^{r}=\mathbf{S}_{2}\big(\sC; \underline v^{\ell}\big)$.  Then 
\begin{align}
\label{E:primsfwewgrqgerqpall}
|\eta_{1}^{r}-\eta_{1}^\ell | 
\leq 
& \co(1) \left| \left(\eta_{2}^{\ell}+\sC_{{}}\right)\eta_{2}^{\ell}\sC_{{}} \right|  (v_{1}^{\ell})^{2}  \ ,
\\
 |\eta_{2}^{r}-\eta_{2}^{\ell}-\sC_{{}}|
\leq 
&\co(1)  \left| \  \left(v_2^\ell-\eta_2^\ell-1\right) \right| \left| \left(\eta_{2}^{\ell}+\sC_{{}}\right) \eta_{2}^{\ell}\sC_{{}} \right|  (v_{1}^{\ell})^{2}  \ .
\label{E:primsfwewgrqgerqpall2}
\end{align}
\end{subequations}
Moreover, the speed $\dot x$ given in~\eqref{C2xdot}  satisfies
\begin{subequations} 
\label{E:sstimaetarbase2all}
\ba
\label{E:8.50all}
&\left|\dot x_{{}}-\lambda_{2}^{\ell}
+ \frac{ v_{1}^{\ell}\cdot(\eta_{2}^{\ell}+\sC_{{}})}{(v_{2}^{\ell}-\eta_{2}^{\ell})(v_{2}^{\ell}+\sC_{{}}) }
\right|
\leq
 \co(1) |\eta_{2}^{\ell}+\sC_{{}}|(v_{1}^{\ell})^{2}  \ ,
\\
\label{E:8.51all}
&\left|\dot x_{{}}-\lambda_{2}^{r}
+
 \frac{ v_{1}^{\ell}\cdot\eta_{2}^{\ell} }{(v_{2}^{\ell}-\eta_{2}^{\ell})(v_{2}^{\ell}+\sC_{{}}) }
\right|
\leq
  \co(1 ) |\eta_{2}^{\ell}| (v_{1}^{\ell})^{2}  \ ,
\ea
\end{subequations}
while the following estimate of the commutators holds true
\ba
\label{E:sqFNDKAVDUAall}
&|\eta_{1}^{r}( \lambda_{1}^{r}-\dot x_{{}})-\eta_{1}^{\ell}(\lambda_{1}^{\ell}-\dot x_{{}})|+
|\eta_{2}^{r}(\lambda_{2}^{r}-\dot x_{{}})-\eta_{2}^{\ell}(\lambda_{2}^{\ell}-\dot x_{{}})|
\leq 
\co(1) \left| \left(\eta_{2}^{\ell}+\sC_{{}}\right)\eta_{2}^{\ell}\sC_{{}} \right|  (v_{1}^{\ell})^{2} \ ,
\ea
where $\lambda_i^r$, $\lambda_i^\ell$ correspond to the speeds of the $i$-family that are given in~\eqref{C2lambdaell}--\eqref{C2lambdar}.
\end{lemma}
\begin{proof}
We now prove separately each estimate in the following steps.\\
{\sc Step 1}. Proof of~\eqref{E:primsfwewgrqgerqpall}. \\
We consider $ \eta_{1}^{r}$ and $ \eta_{2}^{r}$ as smooth functions of the independent variables $\underline v^\ell=(v_{1}^{\ell},v_{2}^{\ell}), \eta_{1}^{\ell},\eta_{2}^{\ell},\sC{}$ 
given implicitly via the relation
\be{implicitdefinitionetar}
 \SC{2}{-\eta_2^r}{\SC{2}{\gamma}{\underline v^{ \ell} }}\equiv\SC{1}{\eta_1^r}{\SC{1}{-\eta_1^\ell}{\SC{2}{-\eta_2^\ell}{ \underline v^\ell}}}\ ,
\ee
since $\underline w=\SC{i}{t}{\underline z}$ is equivalent to $\underline z=\SC{i}{-t}{\underline w}$. We now study the functions
\bas
& \Psi_{1}^{r}(v_{1}^{\ell},v_{2}^{\ell} ,\eta_{1}^{\ell},\eta_{2}^{\ell}, \sC{}) :=\eta_{1}^{r}(v_{1}^{\ell},v_{2}^{\ell} ,\eta_{1}^{\ell},\eta_{2}^{\ell}, \sC{}) -\eta_1^\ell\ ,
& \Psi_{2}^{r}(v_{1}^{\ell},v_{2}^{\ell} ,\eta_{1}^{\ell},\eta_{2}^{\ell}, \sC{}) := \eta_{2}^{r}(v_{1}^{\ell},v_{2}^{\ell} ,\eta_{1}^{\ell}, \eta_{2}^{\ell},\sC{}) -\left(\eta_2^\ell+\sC\right)\ ,
\eas
that are  analytic in the larger domain
\ba\label{E:domaknfknfgrjfwff}
v_{1}^{\ell} \ , \ v_{1}^{\ell} -\eta_1^\ell\in[-\delta^*_0,\delta^*_0]\qquad v_{2}^{\ell}\ ,\  v_{2}^{\ell}-\eta_{2}^{\ell}\ , \ v_{2}^{\ell}+\sC\in[p^*_{0}, p^*_1] \ .
\ea
First,  we observe that $\Psi_i^r$, for $i=1,\,2$, satisfy the following vanishing conditions:
\begin{itemize}
\item[(i)] $ \Psi_{i}^{r}(v_{1}^{\ell},v_{2}^{\ell} ,\eta_{1}^{\ell},\eta_{2}^{\ell}, 0{})=0$. Indeed, if $ \sC=0 $, we have
 $\underline v^\ell\equiv\underline v^r$ and therefore, $\eta_{1}^{r}\equiv\eta_{1}^{\ell}$ and $\eta_{2}^{r}\equiv\eta_{2}^{\ell}=\eta_{2}^{\ell}+\sC$, since $ \sC=0 $.
\item[(ii)] $ \Psi_{i}^{r}(v_{1}^{\ell},v_{2}^{\ell} ,\eta_{1}^{\ell},0, \sC{})=0$. This is true because if
$ \eta_{2}^{\ell}=0 $, then it holds $\underline v^{r}=\SC{2}{\sC}{ \SC{1}{\eta_1^\ell}{ \underline u}}{}$. This implies
$\eta_1^r=\eta_1^\ell$ and $\eta_2^r=\sC$.
\item[(iii)] $ \Psi_{i}^{r}(0,v_{2}^{\ell} ,\eta_{1}^{\ell},\eta_{2}^{\ell}, \sC{})=0$. In this case we use that the $2$-Hugoniot curve is vertical at $h=0$. So if
$ v_{1}^{\ell}=0 $, then  $\eta_{2}^{r}\equiv\eta_{2}^{\ell}+\sC$ and $\eta_{1}^{r}\equiv\eta_{1}^{\ell}$.
\item[(iv)] $ \Psi_{i}^{r}(v_{1}^{\ell},v_{2}^{\ell} ,\eta_{1}^{\ell},\tau, -\tau)=0$ for $\tau\in\mathbb{R}$. Since we have
$\underline v^\ell=\SC{2}{\eta_2^\ell}{ \SC{1}{\eta_1^\ell}{ \underline u}}{}$, then if
$ \sC= -\eta_{2}^{\ell} $, it holds 
$$\underline v^r=\SC{2}{-\eta_2^\ell}{\SC{2}{\eta_2^\ell}{ \SC{1}{\eta_1^\ell}{ \underline u}}{}}\equiv \SC{1}{\eta_1^\ell}{ \underline u}$$ and this implies $\eta_{1}^{r}=\eta_{1}^{\ell}$ and $\eta_{2}^{r}\equiv \sC+\eta_{2}^{\ell}\equiv0 $.
\end{itemize}
All the above cases yield that $\Psi_i^r=0$. Next, we show that
\be{E:claimvanishingderivative2}
\ptlls{ {v_{1}^{\ell}}}
{\Psi_{1}^{r}}\left(v_{1}^{\ell},v_2^\ell, \eta_{1}^{\ell}, \eta_{2}^{\ell},\sC{}\right)\Big|_{v_{1}^{\ell}=0}=0\ .
\ee
To begin with, 
we write explicitly the first component of~\eqref{implicitdefinitionetar}, which is
\[{\hSC{2}{-\eta_2^r}{\hSC{2}{\sC}{\underline v^\ell},v_2^\ell-\sC }}
= \hSC{2}{-\eta_2^\ell}{ \underline v^\ell}-\eta_1^\ell+\eta_1^r\ .
\]
where $\hSC{2}{\cdot}{\cdot}$ denotes
 the first component of the Hugoniot curve $\SC{2}{\cdot}{\cdot}$ given at~\eqref{E:h2shock} and we recall that we use 
 Cartesian coordinates. Differentiating the above identity with respect to $v_1^\ell$, we obtain
\bas 
\ptlls{ {v_{1}^{\ell}}}{\eta_{1}^{r}}\left(v_1^\ell, v_2^\ell , \eta_{1}^{\ell}, \eta_{2}^{\ell},\sC{}\right)
=-\pthS{v_{1}^{\ell}}{2}{-\eta_2^\ell}{\underline v^\ell} 
&-\pthS{\eta_2^r}{2}{-\eta_2^r}{\underline v^r } 
\ptlls{v_1^\ell}{\eta_2^r}\left(\underline v^\ell ,\eta_{1}^{\ell},  \eta_{2}^{\ell},\sC{}\right)
\\
& 
+\pthS{v_1^r} {1}{-\eta_2^r}{\underline v^r } 
\pthS{v_1^\ell} {1}{\sC}{ \underline v^{ \ell} }\ .
\eas
Here, if  $v_1^\ell=0$, then $v_1^r=0$, $\eta_1^r=\eta_1^\ell=u_1$ and $\eta_2^r=\gamma+\eta_2^\ell$ and $v_2^r=v_2^\ell+\sC$. Therefore, by~\eqref{contiderivataS2}, we compute
\bas
\ptlls{ {v_{1}^{\ell}}}{\eta_{1}^{r}}\left( v_1^\ell,v_2^\ell , \eta_{1}^{\ell}, \eta_{2}^{\ell},\sC{}\right)\Big|_{v_1^\ell=0}
&=-\frac{v_2^\ell}{v_2^\ell-\eta_2^\ell}+
0\cdot
\ptlls{v_1^\ell}{\eta_2^r}\left(\underline v^\ell , \eta_{1}^{\ell}, \eta_{2}^{\ell},\sC{}\right)
+
\frac{\cancel{v_2^r}}{v_2^r-\eta_2^r}\cdot
\frac{v_2^\ell}{\cancel{v_2^\ell+\gamma}}
\\
&=-\frac{v_2^\ell}{v_2^\ell-\eta_2^\ell}+\frac{v_2^\ell}{v_2^r-\eta_2^r}=0\ ,
\eas
and~\eqref{E:claimvanishingderivative2} follows immediately. 
In view of the vanishing conditions (i)-(iv) and~\eqref{E:claimvanishingderivative2}, estimate~\eqref{E:primsfwewgrqgerqpall} is established using Lemma~\ref{E:rewgwgeq} in Appendix~\ref{App:C} and repeating a similar argument as the one in Step 1 in the proof of Lemma~\ref{L:sstimaetarbase}.

\noindent
{\sc Step 2}. Proof of~\eqref{E:primsfwewgrqgerqpall2}. \\
Here, we write explicitly the second component of the implicit relation~\eqref{implicitdefinitionetar} using as before
$\pSC{1}{\cdot}{\cdot}$  for the second component of the Hugoniot curve $\SC{1}{\cdot}{\cdot}$ given in~\eqref{E:p1shock} and $\hSC{2}{\cdot}{\cdot}$ for the first component of the Hugoniot curve $\SC{1}{\cdot}{\cdot}$ given in~\eqref{E:h2shock}. Using Cartesian coordinates, we have 
\be{impllicitsecond}
 v_2^\ell+\gamma-\eta_2^r= {\pSC{1}{\eta_1^r}{\underline u}}
 \qquad\text{where $\underline u=\SC{1}{-\eta_1^\ell}{\SC{2}{-\eta_2^\ell}{ \underline v^\ell}}$.}
\ee
and also
$
v_2^\ell=\pSC{1}{\eta_1^\ell}{\underline u}+\eta_2^\ell$. This means that we obtain
\[
\eta_2^\ell+\gamma-\eta_2^r
= 
\pSC{1}{\eta_1^r}{\underline u}-\pSC{1}{\eta_1^\ell}{\underline u}\ .
\]
Since by~\eqref{esotedegaaggargrawgragar}, the derivative of $\pSC{1}{\eta}{}$ with respect to the strength $\eta$ is bounded, we apply mean value theorem to get immediately
\bas
\left|\eta_2^\ell+\gamma-\eta_2^r\right|\stackrel{\eqref{esotedegaaggargrawgragar}}{\leq}  \frac{1+\delta_0^*}{(p_{0}^*)^2(1-\delta_p^*-\delta_0^*)} \left| u_2-1\right| \left|(\eta_1^r-\eta_1^\ell)\right|
=\co(1)\left|  v_2^\ell-\eta_2^\ell-1\right| \left| \eta_1^r-\eta_1^\ell\right|  \ ,
\eas
where we used that $|u_2-1|\leq (1+\delta_0^*/p_0^*) |v_2^\ell-\eta_2^\ell-1|$ by definition~\eqref{E:p1shock} of $\SC{1}{\cdot}{\cdot}$ because $\underline u$ is as in \eqref{impllicitsecond}. We also note that from the calculations in the end of Appendix ~\ref{AppB}, we know that $1-\delta_p^*-\delta_0^*>\frac{7}{9}$ if $\delta_0^*<\frac{1}{9} $ and $\delta_p^*<\frac{1}{9} $. This is to assure the reader that the coefficient in~\eqref{esotedegaaggargrawgragar} remains uniformly bounded.
Substituting now~\eqref{E:primsfwewgrqgerqpall} into this, estimate~\eqref{E:primsfwewgrqgerqpall2} is proven.

\noindent
{\sc Step 3}. Proof of~\eqref{E:sstimaetarbase2all}. \\
Here, we proceed as in the Lemma~\ref{L:sstimaetarbase} and define the auxiliary functions of the independent variables $v_{1}^{\ell},v_{2}^{\ell}, \eta_{1}^{\ell},  \eta_{2}^{\ell} ,\sC{}$ as follows
\bas
&\Psi_{2,\ell} \left(v_{1}^{\ell},v_{2}^{\ell},\eta_1^\ell, \eta_{2}^{\ell} ,\sC{}\right)=\dot x_{}-\lambda_{2}^{\ell}\equiv \lambda_2\left(v_1^\ell,v_2^\ell+\sC_{}\right)-\lambda_2\left(\omega_1^\ell,v_2^\ell\right) \ ,
\\
&\Psi_{2,r}(v_{1}^{\ell},v_{2}^{\ell} ,\eta_1^\ell, \eta_{2}^{\ell} ,\sC{})=\dot x_{}-\lambda_{2}^{r}\equiv \lambda_2\left(v_1^\ell,v_2^\ell+\sC_{}\right)-\lambda_2\left(u_1+\eta_{1}^{r}, v_2^\ell+\sC\right) \ ,
\eas
where $\omega_1^\ell=u_1+\eta_1^\ell $, $v_2^r=v_2^\ell+\sC$ and
\ba\label{E:expressionu1}
u_1=\hSC{}{-\eta_2^\ell}{\SC{1}{-\eta_1^\ell}{\underline v^\ell}}\ ,
&&
\omega_1^\ell&=\hSC{}{-\eta_2^\ell}{\underline v^\ell} \ ,
&&
v_1^r=\hSC{}{\sC}{\underline v^\ell}\ .
\ea
It is useful to recall the implicit relations~\eqref{implicitdefinitionetar} and the Hugoniot curve $\SC{1}{\cdot}{\cdot}$  given in~\eqref{E:p1shock}, while $\SC{2}{\cdot}{\cdot}$ is given in~\eqref{E:h2shock}.\\
\noindent
{\sc Step 3A}. We note that $\Psi_{2,\ell}$ is analytic in the domain~\eqref{E:domaknfknfgrjfwff} and it satisfies:
\begin{itemize}
\item $\Psi_{2,\ell} \left(0,v_{2}^{\ell},\eta_1^\ell, \eta_{2}^{\ell} ,\sC{}\right)=0$. If $ v_{1}^{\ell}=0 $, then $\omega_1=0$, because the $2$-Hugoniot curve is vertical at $h=0$, and hence $\dot x_{}=\lambda_{2}^{\ell}=0$;
\item  $\Psi_{2,\ell} \left(v_{1}^{\ell},v_{2}^{\ell},\eta_1^\ell, \tau ,-\tau\right)=0$, for all $\tau\in\mathbb{R}$. Indeed if $ \sC= -\eta_{2}^{\ell} $, then $\underline v^r=\underline \omega^\ell$ and the identity follows using 
 the symmetry condition~\eqref{E:symmetry2HC}.
\end{itemize}
Now, from relation~\eqref{E:h2shock}, we write the first component $w_1^\ell$
$$w_1^\ell=v_1^\ell-\frac{v_1^\ell \eta_2^\ell}{\lambda_1(v_1^\ell,v_2^\ell-\eta_2^\ell)-v_1^\ell}$$
and compute its derivative
\be{E:aboveexpressionw1ell}
\ptlls{v_1^\ell}{w_1^\ell}\Big|_{v_1^\ell=0}=\frac{v_2^\ell}{v_2^\ell-\eta_2^\ell}\;.
\ee
Moreover, we calculate the derivative of the speeds for $v_1^\ell=0$ to find
\[
\ptll{v_1^\ell}\dot x\Big|_{v_1^\ell=0}= \ptll{v_1^\ell} \left(\lambda_2\left(v_1^\ell,v_2^\ell+\sC_{}\right) \right)\Big|_{v_1^\ell=0}\stackrel{\eqref{AppB.3s2h}}{=}\frac{1}{v_2^\ell+\sC}\;,
\]
and%
\[
\ptll{v_1^\ell} \lambda_{2}^{\ell}\Big|_{v_1^\ell=0}=\ptll{w_1^\ell} \lambda_2\left(w_1^\ell,v_2^\ell\right) \cdot \ptlls{v_1^\ell}{w_1^\ell} \Big|_{v_1^\ell=0}\stackrel{\eqref{E:aboveexpressionw1ell}}{=}\frac{1}{v_2^\ell-\eta_2^\ell}\;.
\]
Then
\be{E:vrkjfenjkafnjkavnjkvnkfv}
\ptlls{v_1^\ell}{\Psi_{2,\ell}}\left(v_1^\ell,v_{2}^{\ell},\eta_{1}^{\ell} ,\eta_{2}^{\ell} ,\sC{}\right)\Big|_{v_1^\ell=0}=\frac{1}{v_2^\ell+\sC} - \frac{1}{v_2^{\ell}-\eta_2^\ell}
=
-\frac{ \eta_{2}^{\ell}+\sC_{{}} }{\left(v_2^\ell+\sC\right)\left(v_2^{\ell}-\eta_2^r\right)}\ .
\ee
This value together with the two vanishing conditions of $\Psi_{2,\ell}$ yield estimate~\eqref{E:8.50all} immediately.

\noindent
{\sc Step 3B}. Next, we check that
$\Psi_{2, r}$ satisfies the following conditions:
\begin{itemize}
\item  $\Psi_{2,r} \left(0,v_{2}^{\ell},\eta_1^\ell, \eta_{2}^{\ell} ,\sC{}\right)=0$. Here, if $v_1^\ell=0$ then we have $v_1^r=u_1+\eta_1^r=0$, since $\SC{2}{\cdot}{0,p}{}$ is vertical for $p>0$, and thus $\dot x_{}=\lambda_{2}^{r}=0$; 
\item  $\Psi_{2,r} \left(v_1^\ell,v_{2}^{\ell},\eta_1^\ell, 0 ,\sC{}\right)=0$. For $\eta_2^\ell=0$, we observe that
 $ \underline v^r\equiv\SC{2}{\eta_2^r}{\SC{1}{\eta_1^r}{\underline u}}\equiv\SC{2}{\sC}{\SC{1}{\eta_1^\ell}{\underline u}}$, so $\eta_1^r=\eta_1^\ell$ and so $v_1^\ell=w_1^\ell=u_1+\eta_1^\ell=u_1+\eta_1^r$.
\end{itemize}
and also the property
\ba\label{E:rvonveianriqbrbqbqbqb}
\ptlls{v_1^\ell}{\Psi_{2,r} }\left(v_1^\ell,v_{2}^{\ell},\eta_{1}^{\ell} ,\eta_{2}^{\ell} ,\sC{}\right)\Big|_{v_1^\ell=0}
=- \frac{  \eta_{2}^{\ell} }{(v_{2}^{\ell}-\eta_{2}^{\ell})(v_{2}^{\ell}+\sC_{{}}) }\ .
\ea
in the domain~\eqref{E:domaknfknfgrjfwff}. Indeed, 
we apply the chain rule in the expression of $\Psi_{2, r}$ and get
\bas
\ptlls{v_1^\ell}{\Psi_{2,r}}\left(v_1^\ell,v_{2}^{\ell},\eta_{1}^{\ell} ,\eta_{2}^{\ell} ,\sC{}\right)\Big|_{v_1^\ell=0}&=\ptll{v_1^\ell} \left(\lambda_2\left(v_1^\ell,v_2^\ell+\sC_{}\right)-\lambda_2\left(u_1+\eta_{1}^{r}, v_2^\ell+\sC\right) \right)\Big|_{v_1^\ell=0}
\\
&\stackrel{\eqref{AppB.3s2h}}{=}\frac{1}{v_2^\ell+\sC}-\frac{1}{v_2^\ell+\sC}\cdot\ptll{v_1^\ell}{\left(u_1+\eta_{1}^{r}\right)} \Big|_{v_1^\ell=0} \ .
\eas
Now from ~\eqref{E:primsfwewgrqgerqpall} $\ptll{v_1^\ell}{\eta_{1}^{r}}=0$ if $v_1^\ell=0$ and the identity $u_1=w_1^\ell-\eta_1^\ell$, we compute
\[\ptll{v_1^\ell}{\left(u_1+\eta_{1}^{r}\right)} \Big|_{v_1^\ell=0}=\ptll{v_1^\ell}{u_1} \Big|_{v_1^\ell=0}=\ptlls{v_1^\ell}{w_1^\ell}\Big|_{v_1^\ell=0}=\frac{v_2^\ell}{v_2^\ell-\eta_2^\ell}\;,
\]
which follows from~\eqref{E:aboveexpressionw1ell}. Substituting above, we get~\eqref{E:rvonveianriqbrbqbqbqb}. As before,~\eqref{E:rvonveianriqbrbqbqbqb} and the vanishing conditions of  ${\Psi_{2,r}}$ imply~\eqref{E:8.51all}.

\noindent
{\sc Step 4}. Proof of~\eqref{E:sqFNDKAVDUAall}. \\
We define in the domain~\eqref{E:domaknfknfgrjfwff} the auxiliary functions $\widehat{\Psi}_{2,1}$ and $\widehat{\Psi}_{2,2}$ as follows:
\bes
\widehat{\Psi}_{2,1}(v_{1}^{\ell},v_{2}^{\ell}, {}\eta_{1}^{\ell},\eta_{2}^{\ell},\sC_{\alpha})
=
\eta_{1}^{r}( \lambda_{1}^{r}-\dot x_{\alpha})-\eta_{1}^{\ell}(\lambda_{1}^{\ell}-\dot x_{\alpha}) \ ,
\ees
\bes
\widehat{\Psi}_{2,2}(v_{1}^{\ell},v_{2}^{\ell},{}\eta_{1}^{\ell},\eta_{2}^{\ell},\sC_{\alpha})=\eta_{2}^{r}(\lambda_{2}^{r}-\dot x_{\alpha})-\eta_{2}^{\ell}(\lambda_{2}^{\ell}-\dot x_{\alpha}) \ .
\ees 
We observe that $\widehat{\Psi}_{2,i}$, for $i=1,\,2$, is analytic and vanishes whenever:
\begin{itemize}
\item[(i)] $\widehat{\Psi}_{2,i}(v_{1}^{\ell},v_{2}^{\ell},{}\eta_{1}^{\ell},\eta_{2}^{\ell},0)=0$. We can see this as follows: if $\sC=0$, then $\underline v^\ell\equiv \underline v^r$ and thus, by definition, $\eta_{1}^{r}=\eta_{1}^{\ell}$, $\eta_{2}^{r}=\eta_{2}^{\ell}$ and $\lambda_{1}^{r}=\lambda_{1}^{\ell}$, $\lambda_{2}^{r}=\lambda_{2}^{\ell}$;
\item[(ii)] $\widehat{\Psi}_{2,i}(v_{1}^{\ell},v_{2}^{\ell},{}\eta_{1}^{\ell},0,\sC_{\alpha})=0$. Here, if $\eta_2^\ell=0$, then $\underline v^{r}=\SC{2}{\sC}{ \SC{1}{\eta_1^\ell}{ \underline u}}{}$ since $\sC$ is a $2$- wave and therefore, 
$\eta_1^r=\eta_1^\ell$, $\eta_2^r=\sC$ and $\lambda_1^r=\lambda_1^\ell$, $\lambda_2^r= \dot{x}_\alpha$;
\item[(iii)] $\widehat{\Psi}_{2,i}(0,v_{2}^{\ell},{}\eta_{1}^{\ell},\eta_{2}^{\ell},\sC_{\alpha})=0$. In this case, if $v_{1}^{\ell}=0$, then $w_1^\ell=v_1^r=0$
and hence, $\lambda_1^r=\lambda_1^\ell\equiv u_2$ and $\lambda_2^r=\dot{x}_\alpha=\lambda_2^\ell\equiv0$;
\item[(iv)] $\widehat{\Psi}_{2,i}(v_{1}^{\ell},v_{2}^{\ell},{}\eta_{1}^{\ell},\tau,-\tau)=0$ for $\tau\in\mathbb{R}$. In the last case, if
$\eta_2^\ell=-\sC$, then $\underline v^r=\SC{1}{\eta_1^\ell}{\underline u}$ and therefore, we get
 $\eta_1^r=\eta_1^\ell$, $\lambda_{1}^{r}=\lambda_{1}^{\ell}$ and $\eta_2^r=0$, $\lambda_2^r=\dot{x}_\alpha$.\end{itemize}
We can conclude now~\eqref{E:sqFNDKAVDUAall} in a similar way as the previous estimates as long as we show
\bas
&
\ptlls{ {v_{1}^{\ell}}}{\widehat{\Psi}_{2,1}} \left(v_{1}^{\ell},v_{2}^{\ell}, \eta_{1}^{\ell}, \eta_{2}^{\ell},\sC_{}\right)\Big|_{v_{1}^{\ell}=0}=\ptlls{ {v_{1}^{\ell}}}{\widehat{\Psi}_{2,2}} \left(v_{1}^{\ell},v_{2}^{\ell}, \eta_{1}^{\ell}, \eta_{2}^{\ell},\sC_{}\right)\Big|_{v_{1}^{\ell}=0}= 0 \ .
\eas
To justify this identity, we differentiate $\widehat{\Psi}_{2,2}$ and use the values~\eqref{E:vrkjfenjkafnjkavnjkvnkfv}--\eqref{E:rvonveianriqbrbqbqbqb} for $v_{1}^{\ell}=0$, to obtain
\bas
\ptlls{ {v_{1}^{\ell}}}{\widehat{\Psi}_{2,2}} \left(v_{1}^{\ell}, v_{2}^{\ell},\eta_{2}^{\ell},\eta_{2}^{\ell},\sC{}\right)\Big|_{v_{1}^{\ell}=0}&=
\left[\eta_{2}^{r}\ptl{ {v_{1}^{\ell}}}{}( \lambda_{2}^{r}-\dot x_{\alpha})-\eta_{2}^{\ell}\ptl{ {v_{1}^{\ell}}}{}(\lambda_{2}^{\ell}-\dot x_{\alpha})\right]_{v_{1}^{\ell}=0}
+\frac{\partial\eta_{2}^{r}}{ \partial v_{1}^{\ell}} \Big|_{v_{1}^{\ell}=0}\cdot ( \lambda_{2}^{r}-\dot x_{\alpha}) \Big|_{v_{1}^{\ell}=0} \\
&= \eta_2^r\, \Big|_{v_{1}^{\ell}=0}\cdot \left(-\frac{ \eta_{2}^{\ell} }{ (v_{2}^{\ell}-\eta_{2}^{\ell})(v_{2}^{\ell}+\sC) }\right)\, -\eta_2^\ell\cdot \left(-\frac{ \eta_{2}^{\ell} +\sC{}}{(v_2^\ell-\eta_2^\ell)(v_2^\ell+\sC)}\right)+0=0\;,
\eas
since $\eta_2^r=\eta_2^\ell+\sC$ and $\lambda_2^\ell=\lambda_2^r=\dot{x}_\alpha=0$ when $v_1^\ell=0$ and from~\eqref{E:primsfwewgrqgerqpall2}, we know that the derivative  $\partial_{v_1^\ell}\eta_{2}^{r}$ is zero at $v_1^\ell=0$. On the other hand, the derivative of $\widehat{\Psi}_{2,1}$ is
\bas
\ptlls{ {v_{1}^{\ell}}}{\widehat{\Psi}_{2,1}} \left(v_{1}^{\ell}, v_{2}^{\ell},\eta_{2}^{\ell},\eta_{2}^{\ell},\sC{}\right)\Big|_{v_{1}^{\ell}=0}&=
\left[\eta_{1}^{r}\ptl{ {v_{1}^{\ell}}}{}( \lambda_{1}^{r}-\dot x_{\alpha})-\eta_{1}^{\ell}\ptl{ {v_{1}^{\ell}}}{}(\lambda_{1}^{\ell}-\dot x_{\alpha})\right]_{v_{1}^{\ell}=0}
+0=\eta_{1}^{\ell}\left[\ptl{ {v_{1}^{\ell}}}{}( \lambda_{1}^{r}- \lambda_{1}^{\ell})
\right]_{v_{1}^{\ell}=0}
\eas
since $\eta_1^\ell=\eta_1^r$ and $\ptll{v_{1}^{\ell}}{\eta_1^\ell}=\ptll{v_{1}^{\ell}}{\eta_1^r}=0$ for $v_1^\ell=0$
 by~\eqref{E:primsfwewgrqgerqpall}.  Now
 \[
\ptll{v_{1}^{\ell}}{ \lambda_1^r}\Big|_{v_{1}^{\ell}=0}=\ptl{ {v_{1}^{\ell}}}{}\lambda_1(u_1+\eta_1^r,u_2)\Big|_{v_{1}^{\ell}=0}\stackrel{\eqref{E:s1derh1}}{=}\left(\frac{u_2-1}{u_2} \right) 
\frac{\partial u_1}{ \partial v_{1}^{\ell}}\Big|_{v_{1}^{\ell}=0} 
 \]
 and 
  \[
\ptll{v_{1}^{\ell}}{ \lambda_1^\ell}\Big|_{v_{1}^{\ell}=0}=\ptl{ {v_{1}^{\ell}}}{}\lambda_1(u_1+\eta_1^\ell,u_2)\Big|_{v_{1}^{\ell}=0}\stackrel{\eqref{E:s1derh1}}{=}\left(\frac{u_2-1}{u_2} \right)
\frac{\partial u_1}{ \partial v_{1}^{\ell}}\Big|_{v_{1}^{\ell}=0} 
 \]
Thus, we deduce
\bas
\ptlls{ {v_{1}^{\ell}}}{\widehat{\Psi}_{2,1}} \left(v_{1}^{\ell}, v_{2}^{\ell},\eta_{2}^{\ell},\eta_{2}^{\ell},\sC{}\right)\Big|_{v_{1}^{\ell}=0} &=
\eta_{1}^{\ell}\frac{\partial u_1}{ \partial v_{1}^{\ell}}\Big|_{v_{1}^{\ell}=0} 
\left[\frac{u_2-1}{u_2} -\frac{u_2-1}{u_2} \right]=0\;.
\eas
The proof is now complete.
\end{proof}

%
%
\section{Auxiliary lemma}\label{App:C}

\begin{lemma}
\label{E:rewgwgeq}
Let $J\subset\{0,1,\dots,m\}$ be a set of indices.
Suppose that, for each $j\in J$, $F:\R^{m}\to\R$ is $k_{j}$-times differentiable in the variable $x_{j}$ with 
$\ptl{x_{j}}{k_{j}}$ derivative continuous on $\R^{m}$ and
\[
 F\Big|_{x_{j}=0} =\ptl{x_{j}}{} F\Big|_{x_{j}=0} =\dots= \ptl{x_{j}}{k_{j}} F\Big|_{x_{j}=0} =0
 \ .
\]
One has then that for $(x_{1},\dots,x_{m})\in[-M,M]^{m}$ there is a constant $C=C(M)$ such that
\[
|F(x_{1},\dots,x_{m})|
\leq
C\prod_{j\in J}
|x_{j}|^{k_{j}+1} \ .
\]
\end{lemma}

\begin{proof}
Given a vector $\underline{x}=(x_{1},\dots,x_{m})$, consider the projected vector $\tau_{j}(s)[\underline{x}]:=\underline{x}+(s-x_{j})\hat{\mathrm{e}}_{j}$ that has the same components of $\underline x$ but the $j$-th one which is set to be $s$.
By the smoothness of $F$, and since
\[
F\left(\tau_{j}(0)[\underline{x}]\right)=\ptl{x_{i}}{h} F\left(\tau_{j}(0)[\underline{x}]\right) =0
\qquad\text{for $h=1,\dots,k_{j}$}, \quad j\in J
\]
one has that for every $j\in J$ the function $F$ can be written as
\bas
 F(\underline{x})
 &=\int_{0}^{ x_{{j}}} \ptl{x_{j}}{} F\left(\tau_{j}(s_{1})[\underline{x}]\right) \, ds_{1}=\dots
 \\
 &=\int_{0}^{x_j}\int_0^{s_1}\int_0^{s_2}\dots\int_0^{s_{k_j}} \ptl{x_{j}}{k_{j}+1} F\left(\tau_{j}(s_{k_j+1})[\underline{x}]\right) ds_{k_{j}+1}\dots\, ds_{3}\,ds_{2} ds_{1}
 \\
 &=\ptl{x_{j}}{k_{j}+1} F\left(\tau_{j}(\hat s)[\underline{x}]\right) \cdot 
 \frac{x_{j}^{k_{j}+1}}{(k_{j}+1)!}
 \qquad \text{for some $\hat s\in(0,x_{j})$}
\eas
and therefore 
\[
\left| F(\underline{x}) \right|
\leq  \max_{[\![\tau_{j}(0)[\underline{x}],\underline{x}]\!]}\left| \ptl{x_{j}}{k_{j}} F  \right| \cdot \left|x_{j}\right|^{k_{j}+1}
\leq  \max_{[-M,M]^{m}}\left| \ptl{x_{j}}{k_{j}} F  \right| \cdot \left|x_{j}\right|^{k_{j}+1}
\qquad\text{for $j\in J$, $\underline{x}\in[-M,M]^{m}$.}
\]
Recursively, for the same reason for $j\in J$ the function
\[
F_{j}(\underline{x})=\frac{ F(\underline{x})}{x_{j}^{k_{j}+1}}
\]
is continuous and it satisfies the hypotheses of the theorem in each index $i\in J\setminus\{j\}$.
Applying the same argument as above we conclude that 
\[
\left| F_{j}(\underline{x}) \right|\leq  \max_{[-M,M]^{m}}\left| \ptl{x_{i}}{k_{i}} F_{j}  \right| \cdot \left|x_{i}\right|^{k_{i}+1}
\qquad\text{for $i\in J\setminus\{j\}$, $\underline{x}\in[-M,M]^{m}$}
\]
and therefore
\[
\left| F(\underline{x}) \right|\leq  \max_{[-M,M]^{m}}\left| \ptl{x_{i}}{k_{i}} F_{j}  \right| \cdot \left|x_{i}\right|^{k_{i}+1}\cdot \left|x_{j}\right|^{k_{j}+1}
\qquad\text{for $i,j\in J $, $i\neq j$, $\underline{x}\in[-M,M]^{m}$.}
\]
Repeating the argument recursively for the indices in $J\setminus\{i,j\}$, we deduce the result.
\end{proof}

\section*{Acknowledgement}
Fabio Ancona and Laura Caravenna are partially supported by the Gruppo
Nazionale per l’Analisi Matematica, la Probabilità e le loro Applicazioni (GNAMPA)
of the Istituto Nazionale di Alta Matematica (INdAM), and by the PRIN 2020 "Nonlinear
evolution PDEs, fluid dynamics and transport equations: theoretical foundations and
applications".  Christoforou was partially supported by the Internal grant SBLawsMechGeom \#21036 from University of Cyprus . 

%
%

\end{document}

%% file: characteristic-Riemann
\ifx\JPicScale\undefined\def\JPicScale{1}\fi
\unitlength \JPicScale mm
\begin{picture}(160.85,67.07)(0,0)
\linethickness{0.4mm}
\put(100.49,5.33){\line(1,0){60.35}}
\put(160.84,5.33){\vector(1,0){0.12}}
\linethickness{0.4mm}
\put(100.49,5.33){\line(0,1){61.74}}
\put(100.49,67.07){\vector(0,1){0.12}}
\linethickness{0.2mm}
\put(100.5,32.5){\line(1,0){60.35}}
\put(159.75,3.28){\makebox(0,0)[cc]{$h$}}

\put(97.2,66.38){\makebox(0,0)[cc]{$p$}}

\put(129,45.5){\makebox(0,0)[cc]{$R_1$}}

\put(96.1,31.4){\makebox(0,0)[cc]{$1$}}

\put(98.3,6.71){\makebox(0,0)[cc]{$0$}}

\linethickness{0.2mm}
\qbezier(100.49,54.72)(103.56,51.52)(111.54,48.55)
\qbezier(111.54,48.55)(119.53,45.58)(133.68,42.38)
\qbezier(133.68,42.38)(147.86,39.15)(154.39,37.67)
\qbezier(154.39,37.67)(160.92,36.18)(160.84,36.2)
\linethickness{0.2mm}
\qbezier(100.49,12.19)(101.26,17.91)(108.72,21.29)
\qbezier(108.72,21.29)(116.18,24.68)(131.49,26.26)
\qbezier(131.49,26.26)(146.81,27.86)(153.87,28.61)
\qbezier(153.87,28.61)(160.93,29.35)(160.84,29.34)
\linethickness{0.2mm}
\qbezier(102.14,66.38)(101.84,49.6)(102.5,38.38)
\qbezier(102.5,38.38)(103.16,27.15)(104.88,19.74)
\qbezier(104.88,19.74)(106.6,12.22)(107.39,8.75)
\qbezier(107.39,8.75)(108.18,5.28)(108.17,5.33)
\linethickness{0.2mm}
\qbezier(110.78,66.38)(111.05,49.96)(112.5,38.82)
\qbezier(112.5,38.82)(113.96,27.68)(116.82,20.08)
\qbezier(116.82,20.08)(119.69,12.38)(121.01,8.83)
\qbezier(121.01,8.83)(122.33,5.28)(122.31,5.33)
\put(101.59,3.28){\makebox(0,0)[cc]{$0$}}

\put(120,50){\makebox(0,0)[cc]{\scalebox{0.8}{$(h_\ell,p_\ell)$}}}

\put(118,35){\makebox(0,0)[cc]{\scalebox{0.8}{$(h_\ell,p_\ell)$}}}

\put(107.35,33.11){\makebox(0,0)[cc]{}}

\put(110,17.5){\makebox(0,0)[cc]{\scalebox{0.8}{$(h_\ell,p_\ell)$}}}

\put(117,58){\makebox(0,0)[cc]{$S_2$}}

\put(119.5,43){\makebox(0,0)[cc]{$R_2$}}

\put(109,39){\makebox(0,0)[cc]{$S_2$}}

\put(111,29){\makebox(0,0)[cc]{$R_2$}}

\put(108.5,51.5){\makebox(0,0)[cc]{$S_1$}}

\put(138.5,34.5){\makebox(0,0)[cc]{$C_1$}}

\put(102,25.8){\makebox(0,0)[cc]{$S_2$}}

\put(110.4,8.8){\makebox(0,0)[cc]{$R_2$}}

\put(108.4,23.4){\makebox(0,0)[cc]{$S_1$}}

\put(103,13.8){\makebox(0,0)[cc]{$R_1$}}

\put(107.5,34.5){\makebox(0,0)[cc]{$C_1$}}

\put(103.51,21.8){\makebox(0,0)[cc]{}}

\linethickness{0.4mm}
\put(2.31,5.09){\line(1,0){78.9}}
\put(81.21,5.09){\vector(1,0){0.12}}
\linethickness{0.4mm}
\put(2.31,5.09){\line(0,1){61.72}}
\put(2.31,66.81){\vector(0,1){0.12}}
\linethickness{0.2mm}
\put(2.31,32.52){\line(1,0){78.9}}
\linethickness{0.2mm}
\qbezier(2.31,23.6)(5.19,26.82)(15.37,28.56)
\qbezier(15.37,28.56)(25.56,30.29)(44.63,30.81)
\qbezier(44.63,30.81)(63.72,31.35)(72.52,31.6)
\qbezier(72.52,31.6)(81.32,31.84)(81.21,31.84)
\linethickness{0.2mm}
\qbezier(2.31,17.43)(5.94,21.36)(16.3,23.76)
\qbezier(16.3,23.76)(26.66,26.15)(45.35,27.38)
\qbezier(45.35,27.38)(64.06,28.63)(72.69,29.21)
\qbezier(72.69,29.21)(81.32,29.79)(81.21,29.78)
\linethickness{0.2mm}
\qbezier(2.31,42.12)(7.07,39.27)(17.69,37.53)
\qbezier(17.69,37.53)(28.3,35.8)(46.42,34.92)
\qbezier(46.42,34.92)(64.58,34.03)(72.95,33.62)
\qbezier(72.95,33.62)(81.32,33.2)(81.21,33.21)
\linethickness{0.2mm}
\qbezier(2.31,48.98)(7.44,45.41)(18.14,42.85)
\qbezier(18.14,42.85)(28.84,40.3)(46.78,38.35)
\qbezier(46.78,38.35)(64.75,36.38)(73.03,35.48)
\qbezier(73.03,35.48)(81.32,34.57)(81.21,34.58)
\put(79.77,3.03){\makebox(0,0)[cc]{$h$}}

\put(-1.99,66.13){\makebox(0,0)[cc]{$p$}}

\put(45.35,40.75){\makebox(0,0)[cc]{$R_1$}}

\put(25.98,11.95){\makebox(0,0)[cc]{$R_2$}}

\put(-3.42,31.15){\makebox(0,0)[cc]{$1$}}

\put(-0.56,6.46){\makebox(0,0)[cc]{$0$}}

\linethickness{0.2mm}
\qbezier(2.31,54.47)(6.32,51.27)(16.76,48.3)
\qbezier(16.76,48.3)(27.2,45.33)(45.7,42.12)
\qbezier(45.7,42.12)(64.23,38.9)(72.78,37.42)
\qbezier(72.78,37.42)(81.32,35.93)(81.21,35.95)
\linethickness{0.2mm}
\qbezier(2.31,63.38)(8.57,56.96)(19.53,52.42)
\qbezier(19.53,52.42)(30.49,47.89)(47.86,44.52)
\qbezier(47.86,44.52)(65.26,41.12)(73.29,39.56)
\qbezier(73.29,39.56)(81.31,37.99)(81.21,38.01)
\linethickness{0.2mm}
\qbezier(2.31,11.95)(3.32,17.67)(13.07,21.05)
\qbezier(13.07,21.05)(22.82,24.43)(42.84,26.01)
\qbezier(42.84,26.01)(62.86,27.61)(72.1,28.36)
\qbezier(72.1,28.36)(81.33,29.1)(81.21,29.09)
\linethickness{0.2mm}
\multiput(38.17,46.65)(0.53,-0.14){3}{\line(1,0){0.53}}
\put(39.75,46.24){\vector(4,-1){0.12}}
\linethickness{0.2mm}
\multiput(37.46,43.63)(0.79,-0.14){2}{\line(1,0){0.79}}
\put(39.03,43.36){\vector(4,-1){0.12}}
\linethickness{0.2mm}
\multiput(36.6,39.65)(0.71,-0.13){2}{\line(1,0){0.71}}
\put(38.03,39.38){\vector(4,-1){0.12}}
\linethickness{0.2mm}
\multiput(35.59,35.54)(1.58,-0.14){1}{\line(1,0){1.58}}
\put(37.17,35.4){\vector(1,-0){0.12}}
\linethickness{0.2mm}
\multiput(34.16,30.33)(2.01,0.13){1}{\line(1,0){2.01}}
\put(34.16,30.33){\vector(-1,-0){0.12}}
\linethickness{0.2mm}
\multiput(34.44,26.49)(1.3,0.14){2}{\line(1,0){1.3}}
\put(34.44,26.49){\vector(-1,-0){0.12}}
\linethickness{0.2mm}
\multiput(34.44,25.25)(1.44,0.14){2}{\line(1,0){1.44}}
\put(34.44,25.25){\vector(-1,-0){0.12}}
\linethickness{0.2mm}
\qbezier(4.47,66.13)(4.08,49.35)(4.94,38.13)
\qbezier(4.94,38.13)(5.8,26.9)(8.05,19.49)
\qbezier(8.05,19.49)(10.29,11.98)(11.33,8.51)
\qbezier(11.33,8.51)(12.36,5.04)(12.35,5.09)
\linethickness{0.2mm}
\qbezier(5.18,66.13)(5.15,49.39)(7.74,34.71)
\qbezier(7.74,34.71)(10.32,20.02)(15.94,5.09)
\linethickness{0.2mm}
\qbezier(6.62,66.13)(6.97,49.71)(8.87,38.57)
\qbezier(8.87,38.57)(10.76,27.43)(14.5,19.84)
\qbezier(14.5,19.84)(18.25,12.14)(19.98,8.59)
\qbezier(19.98,8.59)(21.7,5.04)(21.68,5.09)
\linethickness{0.2mm}
\qbezier(8.05,66.13)(7.99,49.39)(12.82,34.71)
\qbezier(12.82,34.71)(17.65,20.02)(28.13,5.09)
\linethickness{0.2mm}
\qbezier(10.2,66.13)(10.11,49.76)(16.84,35.07)
\qbezier(16.84,35.07)(23.57,20.38)(38.17,5.09)
\linethickness{0.2mm}
\multiput(13.65,43.08)(0.12,-0.3){6}{\line(0,-1){0.3}}
\put(14.36,41.3){\vector(1,-3){0.12}}
\linethickness{0.2mm}
\multiput(10.49,42.95)(0.11,-0.38){5}{\line(0,-1){0.38}}
\put(11.06,41.03){\vector(1,-3){0.12}}
\linethickness{0.2mm}
\multiput(8.19,42.95)(0.15,-1.1){2}{\line(0,-1){1.1}}
\put(8.48,40.75){\vector(1,-4){0.12}}
\linethickness{0.2mm}
\multiput(6.62,42.95)(0.14,-1.03){2}{\line(0,-1){1.03}}
\put(6.9,40.89){\vector(1,-4){0.12}}
\linethickness{0.2mm}
\multiput(4.61,42.95)(0.14,-1.79){1}{\line(0,-1){1.79}}
\put(4.75,41.16){\vector(0,-1){0.12}}
\put(3.75,3.03){\makebox(0,0)[cc]{$0$}}

\linethickness{0.3mm}
\put(115,47.5){\circle*{1}}

\linethickness{0.3mm}
\put(113.5,32.5){\circle*{1}}

\linethickness{0.3mm}
\put(105,19){\circle*{1}}

\linethickness{0.2mm}
\qbezier(113.55,66.63)(113.49,49.89)(118.32,35.21)
\qbezier(118.32,35.21)(123.15,20.52)(133.63,5.59)
\put(14.6,40.7){\makebox(0,0)[cc]{\scalebox{0.8}{$(h,p)$}}}

\linethickness{0.3mm}
\put(2.95,41.91){\line(0,1){0.57}}
\multiput(2.57,41.51)(0.13,0.14){3}{\line(0,1){0.14}}
\put(2.03,41.51){\line(1,0){0.54}}
\multiput(1.65,41.91)(0.13,-0.14){3}{\line(0,-1){0.14}}
\put(1.65,41.91){\line(0,1){0.57}}
\multiput(1.65,42.49)(0.13,0.14){3}{\line(0,1){0.14}}
\put(2.03,42.89){\line(1,0){0.54}}
\multiput(2.57,42.89)(0.13,-0.14){3}{\line(0,-1){0.14}}

\linethickness{0.3mm}
\put(28.65,4.81){\line(0,1){0.57}}
\multiput(28.27,4.41)(0.13,0.14){3}{\line(0,1){0.14}}
\put(27.73,4.41){\line(1,0){0.54}}
\multiput(27.35,4.81)(0.13,-0.14){3}{\line(0,-1){0.14}}
\put(27.35,4.81){\line(0,1){0.57}}
\multiput(27.35,5.39)(0.13,0.14){3}{\line(0,1){0.14}}
\put(27.73,5.79){\line(1,0){0.54}}
\multiput(28.27,5.79)(0.13,-0.14){3}{\line(0,-1){0.14}}

\linethickness{0.3mm}
\put(12.28,38.49){\line(0,1){0.41}}
\multiput(12.09,38.13)(0.09,0.18){2}{\line(0,1){0.18}}
\multiput(11.78,37.93)(0.16,0.1){2}{\line(1,0){0.16}}
\put(11.42,37.93){\line(1,0){0.36}}
\multiput(11.11,38.13)(0.16,-0.1){2}{\line(1,0){0.16}}
\multiput(10.92,38.49)(0.09,-0.18){2}{\line(0,-1){0.18}}
\put(10.92,38.49){\line(0,1){0.41}}
\multiput(10.92,38.91)(0.09,0.18){2}{\line(0,1){0.18}}
\multiput(11.11,39.27)(0.16,0.1){2}{\line(1,0){0.16}}
\put(11.42,39.47){\line(1,0){0.36}}
\multiput(11.78,39.47)(0.16,-0.1){2}{\line(1,0){0.16}}
\multiput(12.09,39.27)(0.09,-0.18){2}{\line(0,-1){0.18}}

\put(-2.5,43.5){\makebox(0,0)[cc]{\scalebox{0.8}{$(0,P)$}}}

\put(28.4,1.8){\makebox(0,0)[cc]{\scalebox{0.8}{$(H,0)$}}}

\end{picture}

%% file: uv.tex
\ifx\JPicScale\undefined\def\JPicScale{1}\fi
\unitlength \JPicScale mm
\begin{picture}(143,52)(0,0)
\linethickness{0.3mm}
\multiput(1.25,0)(34.69,0.12){4}{\line(1,0){34.69}}
\put(140,0.5){\vector(1,0){0.12}}
\linethickness{1.5mm}
\put(20,2){\line(0,1){28}}
\linethickness{0.3mm}
\put(20,30){\line(1,0){10}}
\put(5,3.5){\makebox(0,0)[cc]{}}

\linethickness{0.3mm}
\put(-3.5,2){\line(1,0){23.5}}
\linethickness{1.5mm}
\put(30,30){\line(0,1){20}}
\linethickness{0.3mm}
\put(40,20){\line(1,0){5}}
\linethickness{1.5mm}
\put(45,20){\line(0,1){20}}
\linethickness{0.3mm}
\put(45,40){\line(1,0){29.5}}
\linethickness{0.3mm}
\put(74.5,15.5){\line(0,1){24.5}}
\linethickness{0.3mm}
\put(74.5,15.5){\line(1,0){15.5}}
\linethickness{1.5mm}
\put(90,15.5){\line(0,1){8.5}}
\linethickness{0.3mm}
\put(90,24){\line(1,0){10}}
\linethickness{0.3mm}
\multiput(60,5)(0,2){18}{\line(0,1){1}}
\put(59,-2.5){\makebox(0,0)[cc]{${\bf x}$}}

\put(69,8.5){\makebox(0,0)[cc]{$\eta_1({\bf x})>0$}}

\put(70,42){\makebox(0,0)[cc]{${\color{red} p>1}$}}

\put(83,13){\makebox(0,0)[cc]{${\color{blue} p<1}$}}

\put(11,4.5){\makebox(0,0)[cc]{${\color{blue} p>1}$}}

\put(33.5,52){\makebox(0,0)[cc]{${\color{red} p<1}$}}

\put(42,17.5){\makebox(0,0)[cc]{${\color{red} p<1}$}}

\put(-6,1.5){\makebox(0,0)[cc]{$v_1$}}

\put(19.5,-3){\makebox(0,0)[cc]{${\bf x_1}$}}

\put(30,-3){\makebox(0,0)[cc]{${\bf x_2}$}}

\put(40,-3){\makebox(0,0)[cc]{$x_3$}}

\put(75.5,-3){\makebox(0,0)[cc]{$x_5$}}

\put(90.5,-3){\makebox(0,0)[cc]{${\bf x_6}$}}

\linethickness{0.3mm}
\put(30,50){\line(1,0){10}}
\linethickness{0.3mm}
\put(40,20){\line(0,1){30}}
\put(24,32){\makebox(0,0)[cc]{${\color{blue} p>1}$}}

\put(16,25){\makebox(0,0)[cc]{$\colorbox{green}{1}$}}

\put(26.5,38.5){\makebox(0,0)[cc]{$\colorbox{orange}{2}$}}

\put(49,23){\makebox(0,0)[cc]{$\colorbox{orange}{2}$}}

\put(17.5,20){\makebox(0,0)[cc]{}}

\put(36,23){\makebox(0,0)[cc]{$\colorbox{green}{1}$}}

\put(78,22){\makebox(0,0)[cc]{$\colorbox{orange}{2}$}}

\put(95,26.5){\makebox(0,0)[cc]{${\color{blue} p<1}$}}

\put(94,18.5){\makebox(0,0)[cc]{$\colorbox{green}{1}$}}

\linethickness{0.3mm}
\put(100,24){\line(0,1){10}}
\linethickness{0.3mm}
\put(100,34){\line(1,0){14.5}}
\put(103.5,30){\makebox(0,0)[cc]{$\colorbox{orange}{2}$}}

\put(107.5,36){\makebox(0,0)[cc]{${\color{red} p>1}$}}

\linethickness{0.3mm}
\put(114.5,14){\line(0,1){20}}
\linethickness{0.3mm}
\put(114.5,14){\line(1,0){25.5}}
\put(123.5,11){\makebox(0,0)[cc]{${\color{red} p>1}$}}

\put(119.5,19){\makebox(0,0)[cc]{$\colorbox{green}{1}$}}

\put(143,13){\makebox(0,0)[cc]{$v_1$}}

\put(46,-3){\makebox(0,0)[cc]{${\bf x_4}$}}

\put(101,-3){\makebox(0,0)[cc]{$x_7$}}

\put(114.5,-3){\makebox(0,0)[cc]{$x_8$}}

\linethickness{0.3mm}
\put(44,5){\line(1,0){48}}
\linethickness{0.3mm}
\put(44,5){\line(0,1){6}}
\linethickness{0.3mm}
\put(37,11){\line(1,0){7}}
\linethickness{0.3mm}
\put(37,8){\line(0,1){3}}
\linethickness{0.3mm}
\put(29,8){\line(1,0){8}}
\linethickness{0.3mm}
\put(29,8){\line(0,1){5}}
\linethickness{0.3mm}
\put(10,13){\line(1,0){19}}
\linethickness{0.3mm}
\put(10,13){\line(0,1){33}}
\linethickness{0.3mm}
\put(-3,46){\line(1,0){13}}
\put(-5,46){\makebox(0,0)[cc]{$u_1$}}

\linethickness{0.3mm}
\put(92,5){\line(0,1){4}}
\linethickness{0.3mm}
\put(92,9){\line(1,0){18}}
\linethickness{0.3mm}
\put(110,3){\line(0,1){6}}
\linethickness{0.3mm}
\put(110,3){\line(1,0){30}}
\put(143,4){\makebox(0,0)[cc]{$u_1$}}

\end{picture}

%% file: 2-1interactiontalk.tex
\ifx\JPicScale\undefined\def\JPicScale{1}\fi
\unitlength \JPicScale mm
\begin{picture}(160.85,67.07)(0,0)
\linethickness{0.4mm}
\put(100.49,5.33){\line(1,0){60.35}}
\put(160.84,5.33){\vector(1,0){0.12}}
\linethickness{0.4mm}
\put(100.49,5.33){\line(0,1){61.74}}
\put(100.49,67.07){\vector(0,1){0.12}}
\linethickness{0.2mm}
\put(100.5,32.5){\line(1,0){60.35}}
\put(159.75,3.28){\makebox(0,0)[cc]{$h$}}

\put(97.2,66.38){\makebox(0,0)[cc]{$p$}}

\put(96.1,31.4){\makebox(0,0)[cc]{$1$}}

\put(98.3,6.71){\makebox(0,0)[cc]{$0$}}

\put(101.59,3.28){\makebox(0,0)[cc]{$0$}}

\put(112,51){\makebox(0,0)[cc]{\scalebox{1}{$v^\ell=(h_\beta^\ell,p_\beta^\ell)$}}}

\put(123,21){\makebox(0,0)[cc]{\scalebox{1}{$v^{m}=({h_\alpha}^\ell,{p_\alpha}^\ell)$}}}

\put(107.35,33.11){\makebox(0,0)[cc]{}}

\put(103.51,21.8){\makebox(0,0)[cc]{}}

\linethickness{0.4mm}
\multiput(14.31,5.09)(57.69,-0.09){1}{\line(1,0){57.69}}
\put(72,5){\vector(1,-0){0.12}}
\linethickness{0.4mm}
\put(14.31,5.09){\line(0,1){61.72}}
\put(14.31,66.81){\vector(0,1){0.12}}
\linethickness{0.2mm}
\put(14.31,32.52){\line(1,0){45.69}}
\put(70,3){\makebox(0,0)[cc]{$h$}}

\put(10.01,66.13){\makebox(0,0)[cc]{$p$}}

\put(8.58,31.15){\makebox(0,0)[cc]{$1$}}

\put(11.44,6.46){\makebox(0,0)[cc]{$0$}}

\linethickness{0.2mm}
\multiput(30.67,52.65)(0.37,-0.13){5}{\line(1,0){0.37}}
\put(32.5,52){\vector(3,-1){0.12}}
\linethickness{0.2mm}
\multiput(35.96,36.63)(0.79,-0.14){2}{\line(1,0){0.79}}
\put(37.53,36.36){\vector(4,-1){0.12}}
\linethickness{0.2mm}
\multiput(24,51)(0.12,-0.62){4}{\line(0,-1){0.62}}
\put(24.5,48.5){\vector(1,-4){0.12}}
\linethickness{0.2mm}
\multiput(44.19,43.45)(0.12,-0.35){7}{\line(0,-1){0.35}}
\put(45,41){\vector(1,-3){0.12}}
\put(15.75,3.03){\makebox(0,0)[cc]{$0$}}

\linethickness{0.3mm}
\put(111.5,48.5){\circle*{1}}

\linethickness{0.3mm}
\put(115.5,23.5){\circle*{1}}

\linethickness{0.3mm}
\put(145,27){\circle*{1}}

\linethickness{0.3mm}
\put(23.5,56.5){\circle*{1}}

\linethickness{0.3mm}
\put(26.5,38){\circle*{1}}

\linethickness{0.3mm}
\put(47.5,35.5){\circle*{1}}

\linethickness{0.3mm}
\put(42,49){\circle*{1}}

\put(22.5,45){\makebox(0,0)[cc]{${\color{red}\rho_\beta}$}}

\put(37,39){\makebox(0,0)[cc]{${\color{red}\rho_\alpha}$}}

\linethickness{0.3mm}
\put(140,41){\circle*{1}}

\put(110,36){\makebox(0,0)[cc]{${\color{red}\rho_\beta}$}}

\put(131.5,24.5){\makebox(0,0)[cc]{${\color{red}\rho_\alpha}$}}

\put(128,46){\makebox(0,0)[cc]{${\color{blue}\rho_h'}$}}

\put(146,35){\makebox(0,0)[cc]{${\color{blue}\rho_p'}$}}

\put(147,25.5){\makebox(0,0)[cc]{\scalebox{1}{$v^r$}}}

\put(144,42.5){\makebox(0,0)[cc]{\scalebox{1}{${v^m}'$}}}

\put(26,59){\makebox(0,0)[cc]{\scalebox{1}{$v^\ell=(h_\beta^\ell,p_\beta^\ell)$}}}

\put(38,53){\makebox(0,0)[cc]{${\color{blue}\rho_h'}$}}

\put(50,42){\makebox(0,0)[cc]{${\color{blue}\rho_p'}$}}

\put(28,35){\makebox(0,0)[cc]{\scalebox{1}{$v^m=({h_\alpha}^\ell,{p_\alpha}^\ell)$}}}

\put(50,36){\makebox(0,0)[cc]{\scalebox{1}{$v^r$}}}

\put(46,50){\makebox(0,0)[cc]{\scalebox{1}{${v^m}'$}}}

\linethickness{0.2mm}
\qbezier(111.5,48.5)(111.49,48.1)(111.73,43.66)
\qbezier(111.73,43.66)(111.96,39.23)(112.5,35.5)
\qbezier(112.5,35.5)(113.14,31.99)(114.26,27.92)
\qbezier(114.26,27.92)(115.38,23.86)(115.5,23.5)
\linethickness{0.2mm}
\qbezier(115.5,23.5)(115.93,23.6)(120.86,24.52)
\qbezier(120.86,24.52)(125.79,25.45)(130,26)
\qbezier(130,26)(134.32,26.47)(139.43,26.73)
\qbezier(139.43,26.73)(144.55,27)(145,27)
\linethickness{0.2mm}
\qbezier(140,41)(140.05,40.67)(140.95,37.05)
\qbezier(140.95,37.05)(141.85,33.43)(143,30.5)
\qbezier(143,30.5)(143.48,29.44)(144.2,28.27)
\qbezier(144.2,28.27)(144.93,27.1)(145,27)
\linethickness{0.2mm}
\qbezier(111.5,48.5)(112.04,48.23)(118.27,45.8)
\qbezier(118.27,45.8)(124.51,43.37)(130,42)
\qbezier(130,42)(132.85,41.44)(136.27,41.22)
\qbezier(136.27,41.22)(139.69,40.99)(140,41)
\linethickness{0.2mm}
\qbezier(23.5,56.5)(23.72,56.36)(26.24,54.93)
\qbezier(26.24,54.93)(28.76,53.5)(31,52.5)
\qbezier(31,52.5)(34.1,51.27)(37.88,50.17)
\qbezier(37.88,50.17)(41.66,49.07)(42,49)
\linethickness{0.2mm}
\qbezier(27,38)(27.27,37.94)(30.33,37.4)
\qbezier(30.33,37.4)(33.39,36.85)(36,36.5)
\qbezier(36,36.5)(39.31,36.11)(43.23,35.81)
\qbezier(43.23,35.81)(47.15,35.52)(47.5,35.5)
\linethickness{0.2mm}
\qbezier(42,49)(42.1,48.73)(43.27,45.65)
\qbezier(43.27,45.65)(44.44,42.58)(45.5,40)
\qbezier(45.5,40)(46.05,38.69)(46.74,37.16)
\qbezier(46.74,37.16)(47.44,35.63)(47.5,35.5)
\linethickness{0.2mm}
\qbezier(23.5,56.5)(23.52,56.24)(23.83,53.34)
\qbezier(23.83,53.34)(24.13,50.45)(24.5,48)
\qbezier(24.5,48)(24.99,45.09)(25.71,41.7)
\qbezier(25.71,41.7)(26.43,38.3)(26.5,38)
\linethickness{0.2mm}
\multiput(123,44)(0.38,-0.12){4}{\line(1,0){0.38}}
\put(124.5,43.5){\vector(3,-1){0.12}}
\linethickness{0.2mm}
\multiput(111.94,41.97)(0.06,-2.47){1}{\line(0,-1){2.47}}
\put(112,39.5){\vector(0,-1){0.12}}
\linethickness{0.2mm}
\multiput(142,33.5)(0.12,-0.38){8}{\line(0,-1){0.38}}
\put(143,30.5){\vector(1,-3){0.12}}
\linethickness{0.2mm}
\put(135.5,26.5){\line(1,0){1}}
\put(135.5,26.5){\vector(-1,0){0.12}}
\end{picture}

%% file: timesttalk.tex
\ifx\JPicScale\undefined\def\JPicScale{1}\fi
\unitlength \JPicScale mm
\begin{picture}(79.5,66.81)(0,0)
\put(64.85,28.11){\makebox(0,0)[cc]{}}

\linethickness{0.4mm}
\multiput(14.31,5.09)(57.69,-0.09){1}{\line(1,0){57.69}}
\put(72,5){\vector(1,-0){0.12}}
\linethickness{0.4mm}
\put(14.31,5.09){\line(0,1){61.72}}
\put(14.31,66.81){\vector(0,1){0.12}}
\linethickness{0.2mm}
\put(14.31,32.52){\line(1,0){45.69}}
\put(70,3){\makebox(0,0)[cc]{$h$}}

\put(10.01,66.13){\makebox(0,0)[cc]{$p$}}

\put(8.58,31.15){\makebox(0,0)[cc]{$1$}}

\put(11.44,6.46){\makebox(0,0)[cc]{$0$}}

\linethickness{0.2mm}
\put(18,45){\line(0,1){2.5}}
\put(18,45){\vector(0,-1){0.12}}
\linethickness{0.2mm}
\multiput(47,42)(0.23,-0.12){13}{\line(1,0){0.23}}
\put(50,40.5){\vector(2,-1){0.12}}
\put(15.75,3.03){\makebox(0,0)[cc]{$0$}}

\linethickness{0.3mm}
\put(76.5,46){\circle*{1}}

\linethickness{0.3mm}
\put(63,22.5){\circle*{1}}

\linethickness{0.3mm}
\put(40,46){\circle*{1}}

\linethickness{0.3mm}
\put(18,56.5){\circle*{1}}

\linethickness{0.3mm}
\put(22,22.5){\circle*{1}}

\linethickness{0.3mm}
\put(58.5,38){\circle*{1}}

\put(19,20.5){\makebox(0,0)[cc]{${\color{blue}v^+}$}}

\put(37,47.5){\makebox(0,0)[cc]{$\color{red}{u^-}$}}

\put(49.5,38.5){\makebox(0,0)[cc]{${\color{red}\eta_1^-}$}}

\put(64,29.5){\makebox(0,0)[cc]{$\color{red}{\eta_2^-}$}}

\put(65,21){\makebox(0,0)[cc]{\color{red}{$v^-$}}}

\linethickness{0.3mm}
\put(40,46){\line(1,0){6}}
\put(22,25.5){\makebox(0,0)[cc]{}}

\put(22.5,25){\makebox(0,0)[cc]{}}

\put(24,23){\makebox(0,0)[cc]{}}

\linethickness{0.3mm}
\put(47.5,46){\line(1,0){6}}
\linethickness{0.3mm}
\put(55,46){\line(1,0){6}}
\linethickness{0.3mm}
\put(62.5,46){\line(1,0){6}}
\linethickness{0.3mm}
\put(70,46){\line(1,0){6}}
\linethickness{0.3mm}
\put(55,22.5){\line(1,0){6}}
\linethickness{0.3mm}
\put(47.5,22.5){\line(1,0){6}}
\linethickness{0.3mm}
\put(40,22.5){\line(1,0){6}}
\linethickness{0.3mm}
\put(32,22.5){\line(1,0){6}}
\linethickness{0.3mm}
\put(24,22.5){\line(1,0){6}}
\linethickness{0.2mm}
\qbezier(76.5,46)(75.38,46.12)(62.75,47.85)
\qbezier(62.75,47.85)(50.13,49.57)(39.5,51.5)
\qbezier(39.5,51.5)(33.25,52.75)(25.94,54.53)
\qbezier(25.94,54.53)(18.64,56.32)(18,56.5)
\linethickness{0.2mm}
\qbezier(58.5,38)(58.45,37.73)(58.42,34.8)
\qbezier(58.42,34.8)(58.39,31.86)(59,29.5)
\qbezier(59,29.5)(59.77,27.45)(61.3,25.31)
\qbezier(61.3,25.31)(62.82,23.17)(63,23)
\linethickness{0.2mm}
\qbezier(18,56.5)(17.93,55.83)(17.86,48.32)
\qbezier(17.86,48.32)(17.78,40.81)(18.5,34.5)
\qbezier(18.5,34.5)(19.06,31.12)(20.21,27.23)
\qbezier(20.21,27.23)(21.37,23.34)(21.5,23)
\put(79.5,47.5){\makebox(0,0)[cc]{${\color{blue}u^+}$}}

\put(38.5,54.5){\makebox(0,0)[cc]{${\color{blue} \eta_1^+}$}}

\put(22.5,40.5){\makebox(0,0)[cc]{${\color{blue}\eta_2^+}$}}

\linethickness{0.2mm}
\multiput(32,53)(0.75,-0.12){4}{\line(1,0){0.75}}
\put(32,53){\vector(-4,1){0.12}}
\linethickness{0.2mm}
\multiput(59,29)(0.12,-0.25){4}{\line(0,-1){0.25}}
\put(59.5,28){\vector(1,-2){0.12}}
\linethickness{0.4mm}
\put(41.5,22.5){\line(1,0){3.5}}
\put(41.5,22.5){\vector(-1,0){0.12}}
\linethickness{0.2mm}
\put(55.5,46){\line(1,0){4}}
\put(59.5,46){\vector(1,0){0.12}}
\linethickness{0.2mm}
\qbezier(40,46)(40.26,45.82)(43.28,44.03)
\qbezier(43.28,44.03)(46.3,42.23)(49,41)
\qbezier(49,41)(51.66,39.92)(54.93,38.99)
\qbezier(54.93,38.99)(58.2,38.06)(58.5,38)
\put(49,18.5){\makebox(0,0)[cc]{}}

\put(44.5,19){\makebox(0,0)[cc]{\colorbox{green}{$\Delta t$}}}

\put(66.5,43){\makebox(0,0)[cc]{\colorbox{green}{$\Delta t$}}}

\end{picture}

%% file: nointeractiontalk.tex
\ifx\JPicScale\undefined\def\JPicScale{1}\fi
\unitlength \JPicScale mm
\begin{picture}(140.2,66.67)(0,0)
\linethickness{0.4mm}
\put(6.69,4.93){\line(1,0){60.35}}
\put(67.04,4.93){\vector(1,0){0.12}}
\linethickness{0.4mm}
\put(6.69,4.93){\line(0,1){61.74}}
\put(6.69,66.67){\vector(0,1){0.12}}
\linethickness{0.2mm}
\put(6.7,32.1){\line(1,0){60.35}}
\put(65.95,2.88){\makebox(0,0)[cc]{$h$}}

\put(3.4,65.98){\makebox(0,0)[cc]{$p$}}

\put(2.3,31){\makebox(0,0)[cc]{$1$}}

\put(4.5,6.31){\makebox(0,0)[cc]{$0$}}

\put(7.79,2.88){\makebox(0,0)[cc]{$0$}}

\put(14.4,48.9){\makebox(0,0)[cc]{$u$}}

\put(29.1,20.5){\makebox(0,0)[cc]{${\color{red}v^\ell}=S_2(\eta_2^\ell;v^m)$}}

\put(13.55,32.71){\makebox(0,0)[cc]{}}

\put(9.71,21.4){\makebox(0,0)[cc]{}}

\linethickness{0.4mm}
\multiput(82.51,4.69)(57.69,-0.09){1}{\line(1,0){57.69}}
\put(140.2,4.6){\vector(1,-0){0.12}}
\linethickness{0.4mm}
\put(82.51,4.69){\line(0,1){61.72}}
\put(82.51,66.41){\vector(0,1){0.12}}
\linethickness{0.2mm}
\put(82.51,32.12){\line(1,0){45.69}}
\put(138.2,2.6){\makebox(0,0)[cc]{$h$}}

\put(78.21,65.73){\makebox(0,0)[cc]{$p$}}

\put(76.78,30.75){\makebox(0,0)[cc]{$1$}}

\put(79.64,6.06){\makebox(0,0)[cc]{$0$}}

\linethickness{0.2mm}
\multiput(98.87,52.25)(0.37,-0.13){5}{\line(1,0){0.37}}
\put(100.7,51.6){\vector(3,-1){0.12}}
\linethickness{0.2mm}
\multiput(117.2,28.8)(0.12,-0.3){8}{\line(0,-1){0.3}}
\put(118.2,26.4){\vector(1,-2){0.12}}
\linethickness{0.2mm}
\multiput(111,45.2)(0.12,-0.52){5}{\line(0,-1){0.52}}
\put(111.6,42.6){\vector(1,-4){0.12}}
\put(83.95,2.63){\makebox(0,0)[cc]{$0$}}

\linethickness{0.3mm}
\put(17.7,48.1){\circle*{1}}

\linethickness{0.3mm}
\put(29.1,24.4){\circle*{1}}

\linethickness{0.3mm}
\put(51.2,26.6){\circle*{1}}

\linethickness{0.3mm}
\put(91.7,56.1){\circle*{1}}

\linethickness{0.3mm}
\put(121.6,21.1){\circle*{1}}

\linethickness{0.3mm}
\put(114.38,35){\circle*{1}}

\linethickness{0.3mm}
\put(110.2,48.6){\circle*{1}}

\linethickness{0.3mm}
\put(46.2,40.6){\circle*{1}}

\put(23.8,34.1){\makebox(0,0)[cc]{${\color{red}\eta_2^\ell}$}}

\put(38.12,28.12){\makebox(0,0)[cc]{\colorbox{green}{$\gamma_\alpha$}}}

\put(34.2,45.6){\makebox(0,0)[cc]{${\color{blue}\eta_1^r}$}}

\put(52.3,34.6){\makebox(0,0)[cc]{${\color{blue}\eta_2^r}$}}

\put(66.25,25.62){\makebox(0,0)[cc]{${\color{blue}v^r}=S_2(\eta_2^r;{v^m}')$}}

\put(57,43){\makebox(0,0)[cc]{${\omega^r}=S_1(\eta_1^r; u)$}}

\put(87.4,57.4){\makebox(0,0)[cc]{$u$}}

\put(106.25,55.62){\makebox(0,0)[cc]{${\color{blue}\eta_1^r}$}}

\put(110.5,34.4){\makebox(0,0)[cc]{${\color{red}v^\ell}$}}

\linethickness{0.2mm}
\qbezier(26.2,44.6)(26.2,44.2)(26.44,39.77)
\qbezier(26.44,39.77)(26.68,35.33)(27.2,31.6)
\qbezier(27.2,31.6)(27.51,29.75)(27.97,28.03)
\qbezier(27.97,28.03)(28.43,26.31)(28.9,24.5)
\linethickness{0.2mm}
\qbezier(29.28,24.42)(31.07,24.74)(32.76,25.07)
\qbezier(32.76,25.07)(34.45,25.4)(36.25,25.62)
\qbezier(36.25,25.62)(40.61,26.09)(45.76,26.35)
\qbezier(45.76,26.35)(50.92,26.62)(51.38,26.62)
\linethickness{0.2mm}
\qbezier(46.2,40.6)(46.25,40.27)(47.15,36.65)
\qbezier(47.15,36.65)(48.05,33.03)(49.2,30.1)
\qbezier(49.2,30.1)(49.68,29.04)(50.4,27.87)
\qbezier(50.4,27.87)(51.13,26.7)(51.2,26.6)
\linethickness{0.2mm}
\qbezier(17.7,48.1)(18.24,47.83)(24.47,45.4)
\qbezier(24.47,45.4)(30.71,42.97)(36.2,41.6)
\qbezier(36.2,41.6)(39.05,41.04)(42.47,40.82)
\qbezier(42.47,40.82)(45.89,40.59)(46.2,40.6)
\linethickness{0.2mm}
\qbezier(91.7,56.1)(91.92,55.96)(94.44,54.53)
\qbezier(94.44,54.53)(96.96,53.1)(99.2,52.1)
\qbezier(99.2,52.1)(102.3,50.87)(106.08,49.77)
\qbezier(106.08,49.77)(109.86,48.67)(110.2,48.6)
\linethickness{0.2mm}
\qbezier(110.2,48.6)(110.25,48.31)(110.96,45.15)
\qbezier(110.96,45.15)(111.67,41.99)(112.5,39.38)
\qbezier(112.5,39.38)(112.96,38.08)(113.63,36.6)
\qbezier(113.63,36.6)(114.31,35.13)(114.38,35)
\linethickness{0.2mm}
\multiput(29.2,43.6)(0.38,-0.12){4}{\line(1,0){0.38}}
\put(30.7,43.1){\vector(3,-1){0.12}}
\linethickness{0.2mm}
\multiput(26.3,41.7)(0.1,-1.25){2}{\line(0,-1){1.25}}
\put(26.5,39.2){\vector(0,-1){0.12}}
\linethickness{0.2mm}
\multiput(48.2,33.1)(0.12,-0.38){8}{\line(0,-1){0.38}}
\put(49.2,30.1){\vector(1,-3){0.12}}
\linethickness{0.2mm}
\multiput(42.7,26.1)(1.05,0.15){1}{\line(1,0){1.05}}
\put(43.75,26.25){\vector(4,1){0.12}}
\put(20.4,43.8){\makebox(0,0)[cc]{${\color{red}\eta_1^\ell}$}}

\put(34,49.7){\makebox(0,0)[cc]{}}

\put(58.4,37.9){\makebox(0,0)[cc]{}}

\linethickness{0.3mm}
\put(26.2,44.8){\circle*{1}}

\put(34,63.7){\makebox(0,0)[cc]{\colorbox{green}{$k_\alpha=1$}}}

\put(109.3,64.2){\makebox(0,0)[cc]{\colorbox{yellow}{$k_\alpha=2$}}}

\linethickness{0.2mm}
\qbezier(114.38,35)(114.49,34.71)(115.92,31.49)
\qbezier(115.92,31.49)(117.36,28.27)(118.75,25.62)
\qbezier(118.75,25.62)(119.46,24.37)(120.38,22.95)
\qbezier(120.38,22.95)(121.31,21.52)(121.4,21.4)
\put(136.4,19.8){\makebox(0,0)[cc]{${\color{blue}v^r}=S_2(\eta_2^r;{v^m}')$}}

\put(59.8,34.6){\makebox(0,0)[cc]{}}

\put(121.2,44.4){\makebox(0,0)[cc]{${\color{blue}\eta_2^r}$}}

\put(109.6,40){\makebox(0,0)[cc]{${\color{red}\eta_2^\ell}$}}

\put(101.5,47.6){\makebox(0,0)[cc]{${\color{red}\eta_1^\ell}$}}

\linethickness{0.2mm}
\qbezier(115.88,51)(115.87,50.22)(116.75,41.57)
\qbezier(116.75,41.57)(117.63,32.92)(119.4,25.8)
\qbezier(119.4,25.8)(119.71,24.67)(120.14,23.63)
\qbezier(120.14,23.63)(120.58,22.59)(121.2,21.6)
\put(111.88,27.5){\makebox(0,0)[cc]{\colorbox{yellow}{$\gamma_\alpha$}}}

\put(131.25,51.88){\makebox(0,0)[cc]{${\omega^r}=S_1(\eta_1^r; u)$}}

\linethickness{0.2mm}
\qbezier(92,56)(92.31,55.86)(95.85,54.55)
\qbezier(95.85,54.55)(99.4,53.25)(102.5,52.5)
\qbezier(102.5,52.5)(106.49,51.73)(111.28,51.36)
\qbezier(111.28,51.36)(116.07,51)(116.5,51)
\linethickness{0.3mm}
\put(116.5,51){\circle*{1}}

\linethickness{0.2mm}
\multiput(117,39.2)(0.13,-0.8){3}{\line(0,-1){0.8}}
\put(117.4,36.8){\vector(1,-4){0.12}}
\end{picture}